\documentclass[graybox,envcountsame,envcountsect]{svmult}

\spnewtheorem{examples}[theorem]{Examples}{\itshape}{}
\spnewtheorem{examplekh}[theorem]{Example}{\itshape}{}

\usepackage{amsmath}
\usepackage{latexsym}
\usepackage{amssymb}
\usepackage[all]{xy}
\usepackage{graphicx}

\usepackage[pagebackref, colorlinks=true,linkcolor=red,citecolor=blue]{hyperref}

\usepackage{color}

\usepackage{type1cm}        
\usepackage{makeidx}         

\newcommand{\scheiding}{ \dotfill } 

\usepackage{multicol}        
\usepackage[bottom]{footmisc}

\usepackage{newtxtext}       %
\usepackage[varvw]{newtxmath}       

\makeindex             

\newcommand{\pmat}[1]{\begin{pmatrix} #1 \end{pmatrix}}

\font\tengoth=eufm10 at 10pt
\font\sevengoth=eufm7 at 6pt
\newfam\gothfam
\textfont\gothfam=\tengoth
\scriptfont\gothfam=\sevengoth

\newcommand{\mlabel}[1]{\marginpar{#1}\label{#1}}

\newcommand{\fB}{{\mathfrak B}}

\newcommand{\fS}{{\mathfrak S}}

\newcommand{\g}{{\mathfrak g}}

\newcommand{\fa}{{\mathfrak a}}
\newcommand{\fb}{{\mathfrak b}}
\newcommand{\fc}{{\mathfrak c}}

\newcommand{\fe}{{\mathfrak e}}

\newcommand{\fg}{{\mathfrak g}}
\newcommand{\fh}{{\mathfrak h}}

\newcommand{\fj}{{\mathfrak j}}
\newcommand{\fk}{{\mathfrak k}}
\newcommand{\fl}{{\mathfrak l}}
\newcommand{\fm}{{\mathfrak m}}
\newcommand{\fn}{{\mathfrak n}}

\newcommand{\fq}{{\mathfrak q}}
\newcommand{\fp}{{\mathfrak p}}
\newcommand{\fr}{{\mathfrak r}}
\newcommand{\fs}{{\mathfrak s}}

\newcommand{\fu}{{\mathfrak u}}

\newcommand{\fz}{{\mathfrak z}}

\renewcommand\sp{\mathfrak {sp}} 
\newcommand\hsp{\mathfrak {hsp}} 
 
\newcommand\heis{\mathfrak {heis}}

\renewcommand{\:}{\colon}
\newcommand{\0}{{\bf 0}}

\newcommand{\cA}{\mathcal{A}}
\newcommand{\cB}{\mathcal{B}}
\newcommand{\cC}{\mathcal{C}}
\newcommand{\cD}{\mathcal{D}}
\newcommand{\cE}{\mathcal{E}}
\newcommand{\cF}{\mathcal{F}}
\newcommand{\cG}{\mathcal{G}}
\newcommand{\cH}{\mathcal{H}}
\newcommand{\cK}{\mathcal{K}}
\newcommand{\cL}{\mathcal{L}}
\newcommand{\cM}{\mathcal{M}}
\newcommand{\cN}{\mathcal{N}}
\newcommand{\cO}{\mathcal{O}}
\newcommand{\cP}{\mathcal{P}}

\newcommand{\cR}{\mathcal{R}}
\newcommand{\cS}{\mathcal{S}}

\newcommand{\cU}{\mathcal{U}}

\newcommand{\cW}{\mathcal{W}}
\newcommand{\cZ}{\mathcal{Z}}

\newcommand\bx{{\bf{x}}}
\newcommand\by{{\bf{y}}}
\newcommand\bz{{\bf{z}}}

\newcommand{\bO}{\mathbf{O}}

\newcommand{\eset}{\emptyset}

\newcommand{\derat}[1]{\frac{d}{dt}\big\vert_{t = #1}}

\newcommand{\dd}{{\tt d}}

\newcommand{\trile}{\trianglelefteq}
\newcommand{\subeq}{\subseteq}
\newcommand{\supeq}{\supseteq}

\newcommand{\into}{\hookrightarrow}
\newcommand{\eps}{\varepsilon}

\newcommand{\shalf}{{\textstyle{\frac{1}{2}}}}

\newcommand{\N}{{\mathbb N}}
\newcommand{\Z}{{\mathbb Z}}
\newcommand{\R}{{\mathbb R}}
\newcommand{\C}{{\mathbb C}}

\newcommand{\K}{{\mathbb K}}

\newcommand{\bP}{{\mathbb P}}

\renewcommand{\H}{{\mathbb H}}
\newcommand{\T}{{\mathbb T}}

\newcommand{\bB}{{\mathbb B}}

\newcommand{\bH}{{\mathbb H}}
\newcommand{\bS}{{\mathbb S}}

\renewcommand{\hat}{\widehat}

\renewcommand{\tilde}{\widetilde}

\renewcommand{\L}{\mathop{\bf L{}}\nolimits}

\newcommand{\Aff}{\mathop{{\rm Aff}}\nolimits}

\newcommand{\Deck}{\mathop{{\rm Deck}}\nolimits}

\newcommand{\GL}{\mathop{{\rm GL}}\nolimits}
\newcommand{\SL}{\mathop{{\rm SL}}\nolimits}

\newcommand{\AU}{\mathop{{\rm AU}}\nolimits}

\newcommand{\PSL}{\mathop{{\rm PSL}}\nolimits}
\newcommand{\SO}{\mathop{{\rm SO}}\nolimits}
\newcommand{\SU}{\mathop{{\rm SU}}\nolimits}
\newcommand{\OO}{\mathop{\rm O{}}\nolimits}

\newcommand{\U}{\mathop{\rm U{}}\nolimits}

\newcommand{\Sp}{\mathop{{\rm Sp}}\nolimits}
\newcommand{\Sym}{\mathop{{\rm Sym}}\nolimits}

\newcommand{\gl}  {\mathop{{\mathfrak{gl} }}\nolimits}

\newcommand{\fsl} {\mathop{{\mathfrak{sl} }}\nolimits}
\newcommand{\fsp} {\mathop{{\mathfrak{sp} }}\nolimits}
\newcommand{\su}  {\mathop{{\mathfrak{su} }}\nolimits}
\newcommand{\so}  {\mathop{{\mathfrak{so} }}\nolimits}

\newcommand{\Exp}{\mathop{{\rm Exp}}\nolimits}
\newcommand{\Fix}{\mathop{{\rm Fix}}\nolimits}

\newcommand{\ad}{\mathop{{\rm ad}}\nolimits}
\newcommand{\Ad}{\mathop{{\rm Ad}}\nolimits}

\renewcommand{\Re}{\mathop{{\rm Re}}\nolimits}
\renewcommand{\Im}{\mathop{{\rm Im}}\nolimits}

\newcommand{\tr}{\mathop{{\rm tr}}\nolimits}

\newcommand{\Alt}{\mathop{{\rm Alt}}\nolimits}

\newcommand{\Pol}{\mathop{{\rm Pol}}\nolimits}

\newcommand{\Herm}{\mathop{{\rm Herm}}\nolimits}
\newcommand{\Aherm}{\mathop{{\rm Aherm}}\nolimits}
\newcommand{\Heis}{\mathop{{\rm Heis}}\nolimits}

\newcommand{\Mp}{\mathop{\rm Mp{}}\nolimits}

\newcommand{\<}{<\!\!<}
\renewcommand{\>}{>\!\!>}

\newcommand{\Aut}{\mathop{{\rm Aut}}\nolimits}

\newcommand{\Diff}{\mathop{{\rm Diff}}\nolimits}

\newcommand{\diag}{\mathop{{\rm diag}}\nolimits}
\newcommand{\End}{\mathop{{\rm End}}\nolimits}
\newcommand{\id}{\mathop{{\rm id}}\nolimits}

\newcommand{\rad}{\mathop{{\rm rad}}\nolimits}
\renewcommand{\dim}{\mathop{{\rm dim}}\nolimits}

\newcommand{\im}{\mathop{{\rm im}}\nolimits}

\newcommand{\supp}{\mathop{{\rm supp}}\nolimits}

\newcommand{\Tr}{\mathop{{\rm Tr}}\nolimits}

\newcommand{\Inn}{\mathop{{\rm Inn}}\nolimits}

\newcommand{\Cone}{\mathop{{\rm Cone}}\nolimits}

\newcommand{\cone}{\mathop{{\rm cone}}\nolimits}

\newcommand{\sgn}{\mathop{{\rm sgn}}\nolimits}

\newcommand{\Spann}{\mathop{{\rm span}}\nolimits}

\newcommand{\der}{\mathop{{\rm der}}\nolimits}

\newcommand{\Bil}{\mathop{{\rm Bil}}\nolimits}

\newcommand{\dS}{\mathop{{\rm dS}}\nolimits}
\newcommand{\PSO}{\mathop{{\rm PSO}}\nolimits}

\newcommand{\indlim}{{\displaystyle \lim_{\longrightarrow}}\ }

\newcommand{\Rarrow}{\Rightarrow}
\newcommand{\nin}{\noindent} 
\newcommand{\oline}{\overline}

\newcommand{\Larrow}{\Leftarrow}
\newcommand{\la}{\langle}
\newcommand{\ra}{\rangle}

\newcommand{\up}{\mathop{\uparrow}}

\newcommand{\res}{\vert}

\newcommand{\spann}{{\rm span}}

\newcommand{\Spec}{{\rm Spec}}

\newcommand{\ssssarr}{\hbox to 15pt{\rightarrowfill}}
\newcommand{\sssarr}{\hbox to 20pt{\rightarrowfill}}
\newcommand{\ssarr}{\hbox to 30pt{\rightarrowfill}}
\newcommand{\sarr}{\hbox to 40pt{\rightarrowfill}}
\newcommand{\arr}{\hbox to 60pt{\rightarrowfill}}
\newcommand{\sssslarr}{\hbox to 15pt{\leftarrowfill}}
\newcommand{\ssslarr}{\hbox to 20pt{\leftarrowfill}}
\newcommand{\sslarr}{\hbox to 30pt{\leftarrowfill}}
\newcommand{\slarr}{\hbox to 40pt{\leftarrowfill}}
\newcommand{\larr}{\hbox to 60pt{\leftarrowfill}}

\newcommand{\Arr}{\hbox to 80pt{\rightarrowfill}}

\newcommand{\mapright}[1]{\smash{\mathop{\arr}\limits^{#1}}}

\newcommand{\qand}{\quad\text{and}\quad}
\newcommand{\qfor}{\quad\text{for}\quad}

\newcommand{\pitwo}{\frac{\pi}{2}}

\newcommand{\difftev}{\left.\frac{d}{dt}\right\vert_{t=0}}

\newcommand{\Cay}{\mathop{{\rm Cay}}\nolimits}
\renewcommand{\max}{{\rm max}}

\newcommand\be{{\bf{e}}}

\renewcommand\up{{\uparrow}}
\newcommand{\ie}{i} 
\newcommand{\ee}{e} 

\newcommand{\ood}{{\overline{1}}\,}
\newcommand{\eev}{{\overline{0}}\,}
\newcommand{\Smi}{\mathsf S}

\newcommand{\bone}{\mathbf 1}

\newcommand{\jV}{{\rm V}} 

\newcommand{\sH}{{\sf H}}

\newcommand{\sV}{{\tt V}}
\newcommand{\sF}{{\tt F}}
\newcommand{\sE}{{\tt E}}

\newcommand{\PSU}{\mathop{{\rm PSU}}\nolimits}

\newcommand\RR{{\mathbb R}}
\newcommand\bD{{\mathbb D}}

\newcommand{\cY}{\mathcal Y} 
 
\renewcommand{\bO}{\mathbb O} 
\newcommand{\csp}{{\mathfrak{csp}}} 
\newcommand{\hcsp}{{\mathfrak{hcsp}}}

\newcommand{\Mod}{\mathop{{\rm Mod}}\nolimits}

\renewcommand{\phi}{\varphi}

\newcommand{\AdS}{\mathop{{\rm AdS}}\nolimits}

\newcommand{\Stand}{\mathop{{\rm Stand}}\nolimits}

\newcommand{\bHy}[1]{\mathrm{Hyp}^{#1}}

\begin{document}

\tableofcontents 

\title*{Nets of real subspaces on homogeneous spaces\\ and
  Algebraic Quantum Field Theory} 
\author{Karl-Hermann Neeb} 
\institute{Karl-Hermann Neeb \at Department Mathematik, Friedrich--Alexander-Universit\"at Erlangen-N\"urnberg, Cauerstrasse 11, 91056 Erlangen, \email{neeb@math.fau.de}}
%
%
\maketitle

\abstract*{In these notes, we describe an interesting connection between
unitary representations of Lie groups and nets of local algebras,
as  they appear in Algebraic Quantum Field Theory (AQFT).
It is based on first translating the axioms
for nets of operator algebras parameterized by regions in a
space-time manifold into those for nets of real subspaces, 
and then study this structure from a perspective based on
geometry and representation theory of Lie groups.
}

\section*{Introduction} 
\label{sec:1}

In these notes, we describe an interesting connection between
unitary representations of Lie groups and nets of local algebras,
as  they appear in Algebraic Quantum Field Theory (AQFT).
It is based on first translating the axioms
for nets of operator algebras parameterized by regions in a
space-time manifold into those for nets of real subspaces, 
and then study this structure from a perspective based on
geometry and representation theory of Lie groups.

This topic owes much of its fascination to the close relations 
between operator algebraic concepts, such as
Kubo--Martin--Schwinger (KMS) conditions and spectral conditions,
and the complex geometry related to unitary Lie group representations.
To be more concrete, suppose that
$U_t = e^{it H}$ is a unitary one-parameter group on the complex
Hilbert space $\cH$, $H = H^*$ is its selfadjoint generator, 
and $\xi \in \cH$. We are interested in analytic continuations of the
orbit map $U^\xi \: \R \to \cH, t \mapsto U_t \xi$.
If a bounded analytic extension exists on the upper half-plane
$\C_+ = \{ z \in\C \: \Im z > 0\}$, then its range lies in an invariant
subspace on which the operator $H$ is non-negative (spectral condition).
This is rather restrictive, and it is much more common that
$U^\xi$ only extends to the closure of a strip  
$\cS_\beta = \{ z \in \C \: 0 < \Im z < \beta\}$. Here the most interesting
context arises if the upper boundary values are coupled to the
lower ones by a conjugation $J$ via
\[ J U^\xi(i\beta + t) = U^\xi(t) \quad \mbox{ for } \quad t \in \R.\]
This is precisely the situation one finds in the modular theory
of operator algebras if $\xi$ represents a KMS state
(thermal state), and the
case of positive spectrum corresponds to so-called ground states
(\cite{BGN20}, \cite{BN24}, \cite{NO19}, \cite{ANS25}). 
Below we shall see that such conditions also specify so-called standard subspaces $\sV \subeq \cH$ (for $\beta = \pi$) if $(U_t)_{t \in \R}$ is the corresponding
modular group. 

On the geometric side, an action $\sigma$ of a Lie group $G$ on a
manifold $M$ often has a ``complexification'' in the sense that
$M$ sits in the boundary of a complex manifold $\Xi$ that locally
looks like a tube domain $\R^n + i \Omega \subeq \C^n$,
i.e., $\Omega \subeq \R^n$ is a pointed open convex cone.
In this context, one may also ask for extensions of orbit maps  
$\sigma^m \:\R \to M, t \mapsto \exp(tx).m$ ($m \in M, x\in\g = \L(G)$),
to the upper half-plane $\sigma^m \: \C_+ \to \Xi$, or to a strip
$\sigma^m \: \cS_\beta \to \Xi$. In the latter case, we typically
have an antiholomorphic involution $\tau_\Xi$ satisfying
$\tau_\Xi(\sigma^m(i\beta  + t)) = \sigma^m(t)$ for $t \in \R$.
In the context of semisimple Lie groups, such situations are well-known 
for non-compactly causal symmetric spaces $M = G/H$, 
sitting in the boundary of the so-called complex crown of 
the Riemannian symmetric space~$G/K$ (\cite{GK02}).
Then the existence of such analytic
extensions specifies so-called wedge regions $W \subeq M$, 
that can be characterized in many different ways (\cite{NO23a, NO23b, NO23c}). Here 
the ``imaginary tangent cone'', specifying how $M$ sits in the
boundary of $\Xi$, determines the causal structure on~$M$.
So  $M$ carries similar geometric structures as the
spacetimes in Mathematical Physics. Our goal is to connect
the analytic extension phenomena in unitary group
representations and the underlying geometry with
structures in AQFT.

These notes consist of four main sections whose contents are as follows.
In {\bf Section~\ref{sec:2}} we discuss axioms for nets of local observables,
as they appear in Algebraic Quantum Field Theory (AQFT).
This involves a symmetry group $G$ (a connected Lie group) acting on a
manifold $M$ (spacetime in the physics context) and, for each open subset
$\cO \subeq M$, a von Neumann algebra $\cM(\cO)$ on some complex
Hilbert space~$\cH$, on which we also have a unitary representation
$(U,\cH)$ of $G$, i.e., a continuous homomorphism $U \: G \to \U(\cH)$.

Open subsets $\cO \subeq M$ may be considered as laboratories,
in which experiments are performed that correspond to the evaluation
of quantum observables, represented by  hermitian elements of
$\cM(\cO)$. This leads to families, also called {\it nets},
of von Neumann algebras $(\cM(\cO))_{\cO \subeq M}$.

The axioms that we discuss here are: 
\begin{enumerate} 
\item[(Iso)] {\bf Isotony:}   
  $\cO_1 \subeq \cO_2$ implies $\cM(\cO_1) \subeq \cM(\cO_2)$.
\item[(RS)] {\bf Reeh--Schlieder property:} There exists a unit vector
  $\Omega\in \cH$ that is cyclic for $\cM(\cO)$ if $\cO \not=\eset$,
  i.e., $\cM(\cO)\Omega$ is dense in $\cH$ (cf.\ \cite{RS61}, \cite{Ja00}, \cite{Ja00b}),
    \index{Reeh--Schlieder property (RS) \scheiding }
\item[(Cov)] {\bf Covariance:} $U_g \cM(\cO) U_g^{-1} = \cM(g\cO)$
  for $g \in G$. 
\item[(Vi)] {\bf Invariance of the  vacuum:} $U(g)\Omega = \Omega$ for $g \in G$.  
\item[(BW)] {\bf Bisognano--Wichmann property:}
    \index{Bisognano--Wichmann property (BW) \scheiding }
  There exists a Lie algebra element $h \in \g$ and
  an open subset $W \subeq M$ (called a wedge region), 
  such that the orbit map $\cM(W) \to \cH, A \mapsto A\Omega$
  is injective with dense range ($\Omega$ is cyclic and separating)
  and the corresponding   modular operator $\Delta$, associated
  to the pair $(\cM,\Omega)$ by the Tomita--Takesaki
  Theorem (Theorem~\ref{thm:tom-tak}), satisfies
  $\Delta^{-it/2\pi} = U(\exp th)$ for $t \in \R$. In this sense, the
  modular group is geometrically implemented by a one-parameter
  subgroup of~$G$ (cf.\ \cite{BDFS00}, \cite{BMS01}, \cite{HL82}).
\end{enumerate}

A first step in our analysis is to simplify this situation
by replacing the algebra $\cM(\cO)$ by the real subspace
\[   \sH(\cO) := \sV_{\cM(\cO),\Omega}
  = \oline{\{ A\Omega \: A = A^* \in \cM(\cO)\}}.\]
To formulate our axioms for real subspaces,
recall that a closed real subspace $\sV\subeq \cH$ is called
{\it standard} if $\sV + i \sV$ is dense and $\sV \cap i \sV= \{0\}$.
For any standard subspace, there exists a unique
positive selfadjoint operator $\Delta_\sV$ and a conjugation
(an antilinear involutive isometry) $J_\sV$, such that
$\sV= \Fix(J_\sV \Delta_\sV^{1/2})$
(see Definition~\ref{def:mod-obj} for details).

We are now ready to formulate the basic axioms for the family
$(\sH(\cO))_{\cO \subeq M}$:
\begin{enumerate}
\item[(Iso)] {\bf Isotony}: $\cO_1 \subeq \cO_2$ implies $\sH(\cO_1) \subeq \sH(\cO_2)$
\item[(RS)] {\bf Reeh--Schlieder property:} $\sH(\cO)$ is cyclic  if $\cO \not=\eset$. 
\item[(Cov)] {\bf Covariance:}
  $U_g \sH(\cO) = \sH(g\cO)$ for $g \in G$. 
\item[(BW)] {\bf Bisognano--Wichmann property:} 
  There exists a Lie algebra element $h \in \g$ and an open connected subset 
  $W \subeq M$, such that 
  $\sH(W)$ is standard and the corresponding modular operator
  satisfies $\Delta^{-it/2\pi} = U(\exp th)$ for $t \in \R.$ 
\end{enumerate}
Our goal is to understand such nets and the requirements on the
$G$-space $M$, its geometry, the structure of $G$ and the representation
$(U,\cH)$ for which such nets exist. Eventually, one would like
to ``classify'' all these nets in a suitable sense,
but first one has to specify which structures
we are dealing with. Key questions are: 
\begin{enumerate}
\item[(Q1)] Which elements $h \in \g$ can arise in the
Bisognano--Wichmann (BW) condition?
\item[(Q2)] What $G$-invariant structure do we need on
  $M$ as a fertile ground for nets of real subspaces? 
\item[(Q3)] How to find the domains $W \subeq M$ arising in the
  (BW) condition?
\end{enumerate}
In these notes we shall not go deeper into locality
requirements, but see Subsection~\ref{subsec:6.2}. We refer to \cite{MNO26} and \cite{NO25}
for recent progress in this direction. 

A key result, described in {\bf Section~\ref{sec:3}}, answers (Q1), namely that
$h$ has to be an {\it Euler element}, i.e., $\ad h$
is non-zero and diagonalizable with $\Spec(\ad h) \subeq \{-1,0,1\}$.
In the physical context of the Lorentz and Poincar\'e group,
these are suitably normalized generators of Lorentz boosts. 

In {\bf Section~\ref{sec:3b}} we further argue that it is natural to require
$M$ to carry a {\it causal structure}, i.e., a field of pointed generating
closed convex  cones $C_m \subeq T_m(M)$, invariant under the $G$-action. 
Given an Euler element $h$ and a causal structure on $M$,
the natural candidates for $W$ are the connected components of the
{\it positivity region} 
\[ W_M^+(h) =
  \Big\{ m \in M \: \frac{d}{dt}\Big|_{t = 0} \exp(th).m \in C_m^\circ \Big\}\]
of the vector field on $M$ corresponding to~$h$. 
We discuss these structures for various examples.
Since it will play an important role later in
the construction of nets of real subspaces,
we describe the compression semigroups
\[ S_W := \{ g \in G \: g.W \subeq W \} \]
for several types of wedge regions~$W$. The most important
examples of causal homogeneous spaces~$M$ are causal
symmetric spaces and causal flag manifolds.

In {\bf Section~\ref{sec:4}} we turn to constructions of nets
for a given antiunitary representation $(U,\cH)$ and an
Euler element $h\in\g$.
This is motivated by the consequence
of the Euler Element Theorem~\ref{thm:2.1}, according to which we may assume
that the Lie algebra involution $\tau_h^\g = e^{\pi i \ad h}$
integrates to a group involution $\tau_h$ (e.g., 
if $G$ is simply connected), so that we can form the group
\[ G_{\tau_h} = G \rtimes \{\id_G, \tau_h \} \]
and assume that $U$ extends to an antiunitary representation of~$G_{\tau_h}$. 
This specifies in particular a
standard subspace $\sV = \sV(h,U)$ by 
\begin{equation}
  \label{eq:def-V(h,U)}
 \Delta_\sV = e^{2 \pi i \cdot \partial U(h)} \quad \mbox{ and } \quad
 J_\sV = U(\tau_h)
\end{equation}
(Definitions~\ref{def:mod-obj} and~\ref{def:bgl-net}). 

To find a net $\sH$ satisfying (BW) with $\sH(W) = \sV$,
it is instructive to observe that the elements of $\sV$ are characterized
by the (abstract) Kubo--Martin--Schwinger (KMS) condition:
The orbit map $U^v(t) := U(\exp th)v$ extends analytically to the
closure of the strip $\cS_\pi = \{ z \in \C \: 0 < \Im z< \pi\}$, such that
\[  U^v(\pi i) = J_\sV v\]
(cf.\ Proposition~\ref{prop:standchar}). 

This suggests to look for domains $W \subeq M$
and a complex manifold $\Xi$ with $M \subeq \partial \Xi$
on which $G$ acts by holomorphic maps, such that
$W$ consists of elements $m \in M$ whose orbit map
$\alpha^m(t) = \exp(th).m$ yields by analytic extension
a map 
$\cS_\pi \to \Xi$ satisfying $\alpha^m(\pi) = \oline\tau_h(m)$, where
$\oline\tau_h$ is an antiholomorphic involution on
$\Xi$ satisfying $\oline\tau_h(g.z) = \tau_h(g).\oline\tau_h(z)$ for
$z \in \Xi$. We call these points the {\it KMS points of $M$}. 

For the case where $G$ is contained in its universal complexification 
$G_\C$, we describe in Section~\ref{sec:4} conditions on a domain
$\Xi \subeq G_\C$ (crown domains for $G$), so that the following construction
leads to nets. We start with a real subspace $\sF$ of
$J_\sV$-fixed vectors $v \in \cH$, whose orbit map $U^v \: G \to \cH$
extends analytically to a map $U^v \:\Xi \to \cH$ such that the limit 
\begin{equation}
  \label{eq:betalim}
  \beta^+(v) = \lim_{t \to \frac{\pi}{2}} U^v(\exp(-it h))
\end{equation}
exists in the space $\cH^{-\infty}(U_h)$ of distribution vectors
for the one-parameter group $U_h(t) = U(\exp th)$. We have natural inclusions
\[ \cH^{\infty} \subeq \cH^{\infty}(U_h) \subeq \cH
  \subeq \cH^{-\infty}(U_h) \subeq \cH^{-\infty}\]
(see Appendix~\ref{subsec:app1} for details).
Then 
\[ \sE := \beta^+(\sF) \subeq \cH^{-\infty}(U_h) \subeq \cH^{-\infty} \]
is a real subspace. For $\phi \in C^\infty_c(G,\R)$, the 
operator $U^{-\infty}(\phi)= \int_G \phi(g) U^{-\infty}(g)\, dg$ 
maps $\cH^{-\infty}$ to $\cH$. We thus obtain by 
\begin{equation}
  \label{eq:HEnet}
  \sH_\sE^G(\cO) := \oline{\spann_\R\{  U^{-\infty}(\phi)\sE \:
    \phi \in  C^\infty_c(\cO,\R)\}},
\end{equation} 
a net of real subspaces on $G$ satisfying
(Iso) and (Cov) for trivial reasons, but also (RS) and (BW).
Here the main point is to show that
$\sH_\sE^G(W^G) = \sV$ for a suitable open subset $W^G \subeq G$.\\

\nin{\it Example.} Elementary
  particles in the sense of E.~Wigner \cite{Wgn39} (see also \cite{Ni20})  
  are classified by irreducible unitary
  representations of the Poincar\'e group
  $G = \R^{1,d-1} \rtimes \SO_{1,d-1}(\R)_e$. We write $\jV := \R^{1,d-1}$
  for the corresponding translation group.
  For scalar particles, the Hilbert space is of the form
  $\cH = L^2(\R^{1,d-1},\mu)$, where $\mu$ is a Lorentz invariant measure
  on the dual space $\jV^*$ (often identified with $\jV$ via the Lorentzian form).
  Here the space  $\sE = \R 1$ of real-valued constant functions
  represents distribution vectors, 
  and for test functions $\phi \in C^\infty_c(\jV,\R)$, 
    we have $U^{-\infty}(\phi)\bone = \hat\phi$ (Fourier transform).
  So the real subspace $\sH_\sE^\jV(\cO)$ from \eqref{eq:HEnet} is generated by  Fourier transforms of
  test functions supported in~$\cO$.\\

This leaves us with the question of how to find the crown domains
$\Xi$ and subspaces $\sF \subeq \cH^J$.
For semisimple groups, this can be done with the theory of
crown domains for Riemannian symmetric spaces $G/K$.
They provide natural domains $\Xi \subeq G_\C$ to which
orbit maps of $K$-finite vectors\footnote{Here $K \subeq G$ is a
    maximal compact subgroup and $K$-finiteness means that
    $U(K)\xi$ is contained in a finite-dimensional subspace.}
of irreducible representations
extend, and a recent result by T.~Simon (\cite{Si24}) ensures
that they have a sufficiently well-behaved boundary behavior
at $\partial \Xi$ to ensure \eqref{eq:betalim}
and $\sE = \beta^+(\sF) \subeq \cH^{-\infty}(U_h)$.
Here an important point is that {\bf no further requirement beyond
  irreducibility} is needed for $U$ 
to obtain these nets, and they all descend in a natural way to
the {\it non-compactly causal symmetric spaces}
$M = G/H$, associated to the Euler element~$h$
(cf.\ Sections~\ref{subsec:css}, \ref{subsubsec:reg-reduct},
\cite[Thm.~4.21]{MNO23}).

{\bf Section~\ref{sec:5}} develops a global perspective on
these results. Here we are dealing with representations that are
not necessarily irreducible. 
Starting with a homogeneous space $M = G/H$, a domain $W \subeq M$ and
an antiunitary representation $(U,\cH)$, we directly obtain two nets
$\sH_M^{\rm max}$ and $\sH_M^{\rm min}$  on~$M$, such that any net
$\sH$ on $M$ satisfying (Iso), (Cov) and $\sH(W) = \sV :=  \sV(h,U)$ 
(cf.~\eqref{eq:def-V(h,U)}) also satisfies
\[ \sH_M^{\rm min}(\cO) \subeq \sH(\cO) \subeq \sH_M^{\rm  max}(\cO) \]
for every open subset $\cO \subeq M$ (Lemma~\ref{lem:maxnet-larger}).
From this perspective, the question is, whether a net~$\sH$
satisfying (Iso), (Cov) and (BW) exists at all. This is equivalent
to $\sH^{\rm max}(W) = \sV$, which in turn is equivalent to
the inclusion of semigroups
\begin{equation}
  \label{eq:swinsv}
 S_W = \{ g \in G \: g.W \subeq W \}
 \subeq S_\sV = \{g \in G \:U(g)\sV \subeq \sV\}.
\end{equation}
The semigroup $S_W$ has already been described in Section~\ref{sec:3}
for several types of wedge regions.
If $\ker U$ is discrete, then  
\begin{equation}
  \label{eq:sv-intro}
 S_\sV= \exp(C_+) G_\sV \exp(C_-) \quad\mbox{ holds for } \quad
 C_\pm = \pm C_U \cap \g_{\pm 1}(h),
\end{equation}
where
\[  C_U := \{ x \in \g \: -i \cdot \partial U(x) \geq 0\}, \quad \mbox{ with }
  \qquad 
  \partial U(x) = \frac{d}{dt}\Big|_{t = 0} U(\exp tx),\]
is the {\it positive cone of the unitary representation $U$}, and $\g_\lambda(h)
= \ker(\lambda \bone - \ad h)$ are the eigenspaces of~$\ad h$
(Section~\ref{subsec:endo-v}).

If $G$ is semisimple and $M$ the non-compactly causal symmetric space
associated to the Euler element $h$, then $S_W$ is a group,
hence equal to $G_W$, 
so that \eqref{eq:swinsv} reduces to the inclusion $G_W \subeq G_\sV$,
which boils down to the implication
\[ g.W = W \Rarrow U(g) J = J U(g),\]
which is
equivalent to $\tau_h(g)^{-1}g \in \ker U$ for $g \in G_W$. 

If $S_W$ is not a group, it is of the form
\[ S_W = \exp(C_+) G_W \exp(C_-), \]
where the convex cones $C_\pm$ are specified in terms of an 
$\Ad(G)$-invariant cone $C_\g \subeq \g$ by
\[ C_\pm = \pm C_\g \cap \g_{\pm 1}(h).\]
Therefore \eqref{eq:sv-intro} implies that $S_W \subeq S_\sV$
is equivalent to 
$G_W \subeq G_\sV$ and the {\it spectral condition} 
\[ C_\pm \subeq C_U \]
on the representation $U$, i.e., the operators
$-i \partial U(x)$ are positive for $x \in \pm C_\pm$.
For the Poincar\'e group, acting on Minkowski space
(Remark~\ref{rem:poin}), this corresponds to the positivity of the energy. 

We conclude these notes with a discussion of perspectives and open
problems in {\bf Section~\ref{sec:6}.} \\

\nin {\bf Some history:} The starting point for the development that led to 
fruitful applications of modular theory in QFT 
was the Bisognano--Wichmann Theorem, asserting that  
the modular automorphisms $\alpha_t(M) = \Delta^{-it/2\pi}M \Delta^{it/2\pi}$ 
associated to the algebra $\cM(W_R)$ of observables corresponding 
to the {\it Rindler wedge} 
\[ W_R = \{ (x_0,x_1, \ldots, x_{d-1}) \: x_1 > |x_0|\} \]
in $d$-dimensional Minkowski space $\R^{1,d-1}$ are implemented by the
action of a 
one-parameter group of Lorentz boosts preserving $W_R$. 
This geometric implementation of modular automorphisms in terms of 
Poincar\'e transformations was a key step in a 
rich development based on the work of Borchers and Wiesbrock 
in the 1990s \cite{Bo68, Bo92, Bo95, Bo97, Wi92, Wi93a, Wi93b, Wi98};
see also \cite{LRT78}, \cite{Lo82}, \cite{Fr85} and \cite{Ha96}.
They managed to distill the abstract essence from the Bisognano--Wichmann 
Theorem which led to a better understanding of the 
basic configurations of von Neumann algebras 
in terms of half-sided modular inclusions 
and modular intersections (see also \cite[\S 4.4]{NO17}). 
In his survey \cite{Bo00}, Borchers described how 
these concepts revolutionized quantum field theory. 
Subsequent developments can be found in
\cite{Ar99, Schr99, BGL02, Su96, Su05, Lo08, LMaR09, LW11, LL15, JM18, Mo18}. 

%

\vspace{4mm}

\nin {\bf Notation} 

\begin{itemize}
\item $\cH$ denotes a complex Hilbert space.
\item Strips in the complex plane:\
  \[ \cS_\beta = \{  z \in \C \: 0 < \Im z < \beta \} \quad \mbox{ and } \quad
    \cS_{\pm\beta} = \{  z \in \C \: |\Im z| < \beta \} \] 
\item The neutral element of a Lie group $G$ is denoted $e$, and
  $G_e$ is the identity component.
    \index{identity component $G_e$ of Lie group $G$ \scheiding }
\item The Lie algebra of a Lie group $G$ is denoted $\L(G)$ or $\g$.   
    \index{Lie algebra $\L(G),\g$ of Lie group $G$ \scheiding }
\item For an involutive automorphism $\sigma$ of $G$, we write
  $G^\sigma =  \{ g\in G \:  \sigma(g) = g \}$ for the subgroup of fixed points
  and $G_\sigma := G \rtimes \{\id_G, \sigma\}$ for the corresponding
 group extension. 
    \index{group $G_\sigma$ extended by involution $\sigma$  \scheiding } 
\item $\AU(\cH)$ is the group  of unitary or antiunitary operators
    on a complex Hilbert space~$\cH$. 
    \index{antiunitary group  $\AU(\cH)$  \scheiding} 
\item An antiunitary representation of $G_\sigma$ is a homomorphism
    $U \: G_\sigma \to \AU(\cH)$ with $U(G) \subeq \U(\cH)$ for which 
  the involution  $J := U(\sigma)$ is antiunitary, i.e., a {\it conjugation.}
    We denote representations as pairs $(U,\cH)$.
        \index{conjugation $J$ \scheiding } 
        \index{(anti-)unitary representation $(U,\cH)$ \scheiding } 
\item If $G$ is a group acting on a set $M$ and $W \subeq M$ a subset, then
  the stabilizer subgroup of~$W$ in $G$ is denoted 
  $G_W := \{ g \in G \:  g.W = W\}$, and the compression semigroup by
        \index{stabilizer group $G_W$ of subset $W$ \scheiding} 
        \index{compression semigroup $S_W$ of subset $W$ \scheiding} 
  $S_W := \{ g \in G \:  g.W\subeq W\}$.
  \item If $\g$ is a Lie algebra and $h \in \g$, then
    $\g_\lambda(h) = \ker(\ad h - \lambda \bone)$ is the $\lambda$-eigenspace of $\ad h$ and $\g^\lambda(h) = \bigcup_k \ker(\ad h - \lambda \bone)^k$ is the
    generalized $\lambda$-eigenspace.
\index{eigenspace $\g_\lambda(h)$ of $\ad h$ \scheiding}
g  \item An element $x$ of a Lie algebra $\g$ is called
    \begin{itemize}
    \item[$\diamond$] {\it hyperbolic}, if $\ad x$ is diagonalizable over $\R$. 
    \item[$\diamond$] {\it elliptic} or {\it compact}, if $\ad x$ is semisimple
      with purely imaginary spectrum, i.e., $\oline{e^{\R \ad x}}$ is a
      compact subgroup of $\Aut(\g)$. 
    \end{itemize}
    \index{hyperbolic element of a Lie algebra \scheiding} 
\index{elliptic element of a Lie algebra \scheiding} 
\item We write $\cE(\g)$ for the set of 
{\it Euler elements} $h \in \g$, i.e., $\ad h$ is non-zero and diagonalizable 
with $\Spec(\ad h) \subeq \{-1,0,1\}$. We call $h$ {\it symmetric} if 
$-h \in \cO_h := \Inn(\g)h$. We write $\tau_h := e^{\pi i \ad h} \in \Aut(\g)$ for the
involution of $\g$ specified by~$h$.
\index{Euler element!set $\cE(\g)$ of Euler elements \scheiding} 
\index{Euler element\scheiding} 
  \item {A {\it causal $G$-space} is a smooth $G$-space $M$,
      endowed with a $G$-invariant {\it causal structure}, i.e.,
    a field $(C_m)_{m \in M}$ of pointed generating closed convex cones
    $C_m \subeq T_m(M)$.}
\index{causal structure \scheiding} 
  \item For a unitary representation $(U,\cH)$ of $G$ we write:
  \begin{itemize}
  \item[$\diamond$] $\partial U(x) = \derat0 U(\exp tx)$ for the infinitesimal
    generator of the unitary one-parameter group $(U(\exp tx))_{t \in\R}$
    in the sense of Stone's Theorem. 
\index{infinitesimal generator $\partial U(x)$ \scheiding } 
  \item[$\diamond$] $\dd U \colon \g \to \End(\cH^\infty)$ for the representation of
    the Lie algebra $\g$ on the space $\cH^\infty$ of smooth vectors. Then
    $\partial U(x) = \oline{\dd U(x)}$ (operator closure) for $x \in \g$.
\index{derived representation $\dd U$ \scheiding } 
  \end{itemize}
  \item For a $*$-algebra $\cM \subeq B(\cH)$, we write
    $\cM_h := \{ A \in \cM \: A^* = A \}$ for the real subspace
    of hermitian elements, and for $\Omega \in \cH$, we put
    $\sV_{M,\Omega} := \oline{\cM_h\Omega}$. 
\index{real subspace! $\sV_{M,\Omega}$ associated to $(M,\Omega)$ \scheiding}
\end{itemize}

\nin{\bf Acknowledgment:} We thank
Michael Preeg, Jonas Schober and Tobias Simon for
reading preliminary versions of these notes, for spotting many typos and,
last but not least, for many useful comments. 

The author acknowledges support of the Institut Henri Poincar\'e
(UAR 839 CNRS-Sorbonne Universit\'e), and LabEx CARMIN (ANR-10-LABX-59-01).

\section{Nets of operator algebras and AQFT}
\label{sec:2}

\subsection{Standard subspaces  of Hilbert spaces}

In this subsection, we introduce the key concept
of a standard subspace $\sV$ of a complex Hilbert space~$\cH$. 
Standard subspaces are ``slanted'' real forms in the sense that
$\sV + i \sV$ is dense in $\cH$ and $\sV \cap i \sV = \{0\}$. 
As we shall see below, they are parametrized by
pairs $(\Delta, J)$, where $\Delta > 0$ is a selfadjoint operator
and $J$ is a {\it conjugation} 
(an antilinear isometric
involution) satisfying the modular relation $J \Delta  J = \Delta^{-1}.$
Standard subspaces appear naturally in the modular theory
of operator algebras (Tomita--Takesaki Theorem~\ref{thm:tom-tak})
and also in antiunitary representations of Lie groups, where they
correspond to antiunitary representations of the multiplicative
group $\R^\times \cong \R \times \{\pm 1\}$. This establishes an
important link between operator algebras and antiunitary
representations.

\begin{definition} \label{def:1.1}
  (a) A closed real subspace $\sV \subeq \cH$
  is called
\begin{itemize}
\item {\it separating} if $\sV \cap i \sV = \{0\}$, 
\item {\it cyclic} if $\sV + i \sV$ is dense in $\cH$,
\item {\it standard} if it is cyclic and separating.   
\end{itemize}
We write $\Stand(\cH)$ for the set of standard subspaces of~$\cH$.
\index{real subspace!separating\scheiding}
\index{real subspace!cyclic\scheiding}
\index{real subspace!standard\scheiding}

\nin (b) For a separating subspace $\sV$, we define the antilinear 
{\it Tomita involution}
\[ T_\sV \: \sV + i \sV \to \cH, \quad T_\sV(v + iw) = v - i w \quad
\mbox{ for } \quad  v,w \in \sV.\] 
\index{Tomita involution $T_{\sV}$ \scheiding } 

\nin (c) We write $\gamma(v,w) := \Im \la v,w \ra$ for the canonical symplectic
form on $\cH$. For a real subspace $\sV \subeq \cH$, we
define  its {\it symplectic orthogonal space} by 
\[ \sV' := \sV^{\bot_\gamma} = \{ w \in \cH \: \Im \la v, w \ra = 0\}
  = i \sV^{\bot_\R},\]
where $\sV^{\bot_\R}$ is the real orthogonal space of $\sV$ with respect to the 
real-valued scalar product $\Re \la v,w\ra$.
Note that $\la \sV, \sV' \ra \subeq \R$. 
\index{symplectic orthogonal space $\sV'$ \scheiding }
\end{definition}

\begin{lemma} \label{lem:tv-closed}
If $\sV$ is standard, then $T_\sV$ is closed and densely defined.
\end{lemma}

\begin{proof} As $\sV$ is cyclic, the operator $T_\sV$ is densely defined. 
  To see that the graph of $T_\sV$ is closed,
suppose that $\xi_n = a_n + i b_n$ is a sequence in
$\cD(T_\sV) =\sV + i \sV$ with $a_n, b_n \in \sV$, such that
$(\xi_n, T_{\sV} \xi_n) = (a_n + i b_n, a_n - i b_n) \to (\xi,\eta)$ in
$\cH \times \cH$.
As $\sV$ is closed, 
\[ a_n = \frac{1}{2}(a_n + i b_n + (a_n - i b_n))
  = \frac{1}{2}(\xi_n + T_\sV \xi_n) \to \frac{1}{2}(\xi+\eta)
=: a   \in \sV, \]
and 
\[ b_n = \frac{1}{2i}(a_n + i b_n - (a_n - i b_n))
  = \frac{1}{2i}(\xi_n - T_\sV \xi_n) \to \frac{1}{2i}(\xi-\eta)
=: b  \in \sV. \]
Therefore $\xi = a + i b \in\cD(T_\sV)$ satisfies
$T_\sV \xi = a - i b = \eta$. This means that $T_\sV$ is closed.
\end{proof}

\begin{definition}  \label{def:mod-obj}
  We have seen in Lemma~\ref{lem:tv-closed} that,
  for every standard subspace $\sV \subeq \cH$,
  the Tomita operator
\[ T_\sV  \: \cD(T_\sV) := \sV + i \sV \to \cH, \qquad 
T_\sV(v + i w) := v - i w \] 
is closed, hence has a polar decomposition
(\cite[Thm.~7.2]{Sch12}, \cite[Thm.~9.29]{SZ79}\footnote{To obtain the polar decomposition of a closed operator
    $T$, the main step is to show that the operator $T^*T$ is selfadjoint.
    Then the unique positive square root $|T| := \sqrt{T^*T}$ satisfies
    $\|\, |T|\xi\| = \|T\xi\|$ for all $\xi \in \cD(T)$, which easily
    leads to a partial isometry $U$ from $\oline{\cR(|T|)}
    =\cN(|T|)^\bot    =\cN(T)^\bot$
    to $\oline{\cR(T)}$ with $T = U |T|$.}), i.e., 
\[ \Delta_\sV := T_\sV^*T_\sV \] 
is a positive selfadjoint operator, and there exists an antilinear
isometry $J_\sV$ such that 
\[ T_\sV = J_\sV \Delta_\sV^{1/2}.\] 
The isometry $J_\sV$ is defined on all of $\cH$ because
$\Delta_\sV$ has dense range, which in turn follows from 
$\cR(\Delta_\sV)^\bot = \ker(\Delta_\sV) = \ker(T_\sV) = \{0\}$. 
The relation
\[ J_\sV \Delta_\sV^{1/2} = T_\sV = T_\sV^{-1}
  = \Delta_\sV^{-1/2} J_\sV^{-1} = J_\sV^{-1} (J_\sV \Delta_\sV^{-1/2} J_\sV^{-1}) \]
and the uniqueness of the polar decomposition 
now imply $J_\sV^2 = \bone$ and the {\it modular relation}
\begin{equation} \index{modular!relation \scheiding} 
  \label{eq:modrel}
J_\sV \Delta_\sV J_\sV = \Delta_\sV^{-1}.  
\end{equation}

The unitary one-parameter group $(\Delta_\sV^{it})_{t \in \R}$
is called the {\it modular group of $\sV$}.
\index{modular!group (of standard subspace) \scheiding} 
It has the important property that it preserves $\sV$
(Remark~\ref{rem:NO15-prop-3.1}(b)) and its true importance is revealed in 
the Tomita--Takesaki Theorem~\ref{thm:tom-tak}.  
\end{definition}

\begin{remark} \label{rem:NO15-prop-3.1}
(a)   The modular group $\Delta_\sV^{it}$
  commutes with the antiunitary conjugation $J_\sV$.
  In fact, the antilinearity of $J_\sV$ implies that
  \[ J_\sV \Delta_\sV^z J_\sV = \Delta_\sV^{-\oline z} \quad \mbox{ for } \quad
    z \in \C.\]   
  In view of \cite[Prop.~3.1]{NO15}, 
  a unitary one-parameter group $(U_t= e^{itH})_{t \in \R}$
  commutes with some conjugation~$J$ if and only if
  $H$ is {\it symmetric} in the sense that there exists a unitary involution
  $S$ satisfying $S H S^{-1} = -H$.

  \nin (b) The fact that the operators $\Delta_\sV^{it}$ commute with $J_\sV$
  implies that they also commute with $T_\sV$, hence leave $\sV$ invariant. 
\end{remark}

\begin{definition} We write $\Stand(\cH)$ for the set of standard subspaces
  of the complex Hilbert space $\cH$ and
  $\Mod(\cH)$ for the set of all pairs $(\Delta, J)$,
  where $J$ is a conjugation and 
  $\Delta > 0$ a positive selfadjoint operator satisfying the modular
  relation $J\Delta J =\Delta^{-1}.$ 
\end{definition}

\begin{proposition}\label{prop:11}
  The map
  \[ \Mod(\cH) \to \Stand(\cH), \quad
    (\Delta,J) \mapsto \Fix(J\Delta^{1/2}) \]
  is a bijection. Its inverse is given 
  by $\sV \mapsto (\Delta_\sV, J_\sV)$.
\end{proposition}

\begin{proof} (\cite[Prop.~3.2]{Lo08})
  To see that we obtain a bijection, suppose that
  $(\Delta, J)$ is a pair of modular objects,
i.e., a positive operator and a conjugation, satisfying the modular relation 
\eqref{eq:modrel}. Then 
$T := J \Delta^{1/2}$ is a closed, densely defined antilinear involution and 
\[ \sV := \Fix(T) := \{ \xi \in \cD(T) \: T\xi = \xi \} \] 
is a standard subspace with $J_\sV = J$ and $\Delta_\sV = \Delta$.
Here closedness of $T$ follows from the closedness of the selfadjoint
operator $\Delta^{1/2}$, and this implies the closedness of $\Fix(T)$.
\end{proof}

\subsection{More background on standard subspaces}

\begin{lemma} \label{lem:vv'} The map $\sV \mapsto \sV'$ has the following properties:
  \begin{description}
  \item[\rm(a)] $\sV'' = \sV$. 
  \item[\rm(b)] $\sV$ is cyclic if and only if $\sV'$ is separating. 
  \item[\rm(c)] $\sV$ is standard if and only if $\sV'$ is standard. 
  \item[\rm(d)] $T_\sV^* = T_{\sV'}$, i.e., $\cD(T_\sV^*) = \sV' + i \sV'$ and 
    \[ \la T_\sV \xi, \eta \ra = \oline{\la \xi, T_{\sV'} \eta \ra}
      \quad \mbox{ for }  \quad
      \xi \in \sV + i \sV, \ \ \eta \in \sV' + i \sV'.\] 
 \item[\rm(e)] $\Delta_{\sV'} = \Delta_{\sV}^{-1}$ and $J_{\sV'}= J_{\sV}$. 
 \item[\rm(f)] $J_\sV \sV = \sV'$. 
  \end{description}
\end{lemma}

\begin{proof} \begin{enumerate}
  \item[(a)]
follows immediately from the Hahn--Banach Theorem.
  Alternatively, we can use that $\sV' = i \sV^{\bot_\R}$ and that multiplication with $i$ is isometric, to obtain $\sV'' = i^2 (\sV^{\bot_\R})^{\bot_\R}
  = \sV$. 

  \item[(b)] The subspace $(\sV + i  \sV)' = \sV' \cap i \sV'$ vanishes
  if and only if $\sV'$ is separating if and only if $\sV$ is cyclic. 

\item[(c)] If $\sV$ is standard, then (b) implies that $\sV'$ is separating.
  That $\sV'$ is also  cyclic follows from (b) and $(\sV')' = \sV$
  being separating.
Hence $\sV'$ is standard if $\sV$ has this property. If $\sV'$ is
standard, then we now see with (a) that $\sV = \sV''$ is also standard. 

\item[(d)] First we show that $T_{\sV'} \subeq T_\sV^*$.
In fact, for $a,b \in \sV'$ and $v,w \in \sV$, we derive
from $\la \sV,\sV'\ra \subeq \R$ that 
\begin{align*}
&  \la T_{\sV'}(a + i b), v + i w \ra
  =  \la a- i b, v + i w \ra
  = \la a,v \ra - \la b,w \ra + i (\la b,v \ra + \la a, w\ra)\\
&  =  \oline{\la a+ i b, v - i w \ra}
  = \oline{\la a + i b, T_\sV(v + i w) \ra}
  = \la T_{\sV}^*(a + i b), v + i w \ra.
\end{align*}

Next we observe that, for $\xi \in \sV$ and $\eta \in \cD(T_\sV^*)$, we have
\[ \la \xi, T_\sV^* \eta \ra = \oline{\la T_{\sV} \xi, \eta \ra}
  = \oline{\la \xi,\eta \ra}.\]
From the equality of real and imaginary part, we derive that 
\[ T_\sV^* \eta - \eta \in \sV^{\bot_\R} = i \sV' \quad \mbox{ and } \quad 
  T_\sV^* \eta + \eta \in \sV'.\]
Therefore $\eta \in \sV' + i \sV' = \cD(T_{\sV'})$,
and hence that $T_{\sV'} = T_\sV^*$.

\item[(e)] From (d) we derive with Exercise~\ref{exer:abstar} that
\[ T_{\sV'} = (T_\sV)^*
  = (J_\sV \Delta_\sV^{1/2})^*
\ \ {\buildrel \ref{exer:abstar} \over  =}\ \  \Delta_\sV^{1/2} J_\sV^*  
  = \Delta_\sV^{1/2} J_\sV  
  = J_\sV  \Delta_\sV^{-1/2}.\]
Thus (e) follows from the uniqueness of the polar decomposition.

\item[(f)] If $v \in \sV$, then
\[ T_{\sV'} J_\sV v = J_\sV \Delta_\sV^{-1/2} J_\sV v
  = \Delta_\sV^{1/2} v = J_\sV v\]
shows that $J_\sV \sV \subeq \sV'$.
Likewise $J_{\sV} \sV' = J_{\sV'}\sV' \subeq \sV'' = \sV$, so that
$\sV' \subeq J_\sV \sV$, and thus $\sV' = J_\sV \sV$.
\end{enumerate}
\end{proof}

\begin{lemma} \label{lem:core} {\rm(\cite[Prop.~3.11]{Lo08})} 
  Let $U_t = e^{itA}$ be a unitary
  one-parameter group on $\cH$, and
  $f \: \R \to \C$ a locally bounded Borel measurable function.
  If $\cD \subeq \cD(f(A))$ is a $U$-invariant
  linear subspace dense in $\cH$, then it is a core for $f(A)$,
  i.e., the graph of $f(A)$ is the closure of its restriction
  to $\cD$.   
\end{lemma}

\begin{proof}   We factorize $f = f_0 f_1$ with $f_0(\R) \subeq \T$ and $f_1 \geq 0$,
  so that $f(A) = f_0(A) f_1(A)$. Then $f_0(A)$ is bounded and
  $\cD \subeq \cD(f(A))= \cD(f_1(A))$. It therefore suffices to show that
  $\cD$ is a core for $B := f_1(A)$, resp., that 
 the graph $\Gamma(B_0)$ of $B_0 := B\res_{\cD}$ is dense in the graph
  of $B$. This is equivalent to $B_0$ being essentially selfadjoint. 

  Replacing $B_0$ by its closure, whose domain is also $U$-invariant,
  we may assume that $B_0$ is closed and we have to show that
  $B_0 = B$. As $B$ is selfadjoint, it suffices to verify that
  $\cR(B_0 + i \bone)$ is dense in $\cH$. So let $v \in \cR(B_0 + i \bone)^\bot$.
  We have to show that $v = 0$.

  The closed subspace $\Gamma(B_0) \subeq \cH^2$ is invariant under the diagonal
  action of the operators $(U_t)_{t \in \R}$, hence also under the operators
  $U(\phi) = \int_\R \phi(t)U_t\, dt$ for $\phi \in L^1(\R)$.
  In view of the relation
  $U(\phi) = \hat\phi(A)$, these include the operators 
  $\psi(A), \psi \in \cS(\R,\R)$ (Schwartz functions). For all $w \in \cD$ and
  $\psi \in \cS(\R,\R)$, we thus have
  \[  v \bot (B_0 + i \bone) \psi(A) \cD = (f_1(A) + i \bone) \psi(A) \cD.\]
  If $\psi$ has compact support, then the operator
  $f_1(A)\psi(A)$ is bounded because $f_1$ is locally bounded.
  So the density of $\cD$ in $\cH$ implies that
$v \bot (B + i \bone) \psi(A) \cH.$   This in turn implies that
  \[ \psi(A)v \in  \cR(B + i \bone)^\bot = \{0\}.\]
  Choosing $\psi_n$ in such a way that
  $0 \leq \psi_n  \leq 1$ and $\psi_n\res_{[-n,n]} = 1$,
  then the relation $0 = \psi_n(A)v \to v$ entails that $v = 0$.
\end{proof}

\begin{lemma} \label{lem:lo08-3.10} 
  {\rm(Equality Lemma)}
  Let $\sH_1 \subeq \sV \subeq \sH_2$ be closed real subspaces such that
  $\sV$ is standard, $\sH_1$ is cyclic and $\sH_2$ separating.
    If $\Delta_\sV^{it}\sH_j =  \sH_j$ holds for all $t \in \R$,
  then $\sH_1 = \sV = \sH_2$.   
\end{lemma}

\begin{proof}  (\cite[Prop.~3.10]{Lo08}) 
  Our assumption implies that
  $\sH_1 + i \sH_1 = \cD(T_{\sH_1}) = \cD(\Delta_{\sH_1}^{1/2})$
  is a dense subspace of $\cH$, invariant under the modular group
  $U_t = \Delta_\sV^{it}$, $t \in \R$. This subspace is contained in 
  $\sV + i \sV = \cD(T_\sV) = \cD(\Delta_\sV^{1/2})$,
  hence a core of $\Delta_{\sV}^{1/2}$ by Lemma~\ref{lem:core},
  and therefore also a core of $T_\sV$.
  Since $T_\sV$ is an extension of~$T_{\sH_1}$,
  the closedness of $T_{\sH_1}$ implies that $T_{\sH_1} = T_\sV$,
  hence that $\sH_1 = \sV$.

  To deal with $\sH_2$, we note that $\sH_2' \subeq \sV'$
  is cyclic by Lemma~\ref{lem:vv'}(b).
  Our assumption now implies that the cyclic subspace
  $\sH_2'$ is invariant under the modular
  group of~$\sV'$, and the first part of the proof thus entails
  $\sH_2' = \sV'$. Finally, $\sH_2 = \sH_2'' = \sV'' = \sV$
  (Lemma~\ref{lem:vv'}).
\end{proof}

\begin{proposition} \label{prop:equality-prop}
  Let $\sV_1$ and $\sV_2$ be  two standard subspaces
    with $\Delta_{\sV_1} = \Delta_{\sV_2}$. If $\sV_1 \cap \sV_2$ is cyclic,
    then $\sV_1 = \sV_2$.
\end{proposition}

\begin{proof} The subspace $\sV := \sV_1 \cap \sV_2$ is invariant
  under the modular group $\Delta_{\sV_1}^{i\R} = \Delta_{\sV_2}^{i\R}$.
  Hence the assertion follows from the Equality Lemma.  
\end{proof}

\subsection{Standard subspaces and graphs}

Let $\sV \subeq \cH$ be a standard subspace and recall that
$\sV + i \sV = \cD(\Delta^{1/2}).$ 
The natural Hilbert space structure on this dense subspace of $\cH$
is obtained from the isomorphism with the graph
\[ \Gamma(\Delta^{1/2}) = \{(v,\Delta^{1/2}v) \: v \in \cD(\Delta^{1/2})\}
\subeq \cH \oplus \cH, \]
which is a closed subspace.

\begin{proposition} \label{prop:2.27} Let $\cH$ be a complex Hilbert space.
  Consider the complex structure on $\cH^{\oplus 2}$, defined by
  $I(v,w) := (iv,-iw)$ and a densely defined operator
  $A \: \cD(A) \to \cH$ with closed graph $\Gamma(A) \subeq \cH^{\oplus 2}$.
  Then the following assertions hold:
  \begin{description}
  \item[\rm(a)] The graph $\Gamma(A)$ is separating if and only if
    $A$ is injective. 
  \item[\rm(b)] The graph $\Gamma(A)$ is cyclic if and only if
    $A$ has dense range. 
  \item[\rm(c)] The graph $\Gamma(A)$ is standard if and only if
    $A$ is injective with dense range. 
  \item[\rm(d)] If $\sV  :=
    \Gamma(A)$ is standard, then its Tomita operator is given
    by $T_\sV(v,w) = (A^{-1}w,Av)$, and if $A = U |A|$ is the polar decomposition
    of $A$, then
    \[ \Delta_\sV= A^*A \oplus (A^{-1})^*A^{-1}, \quad
      J_\sV(v,w) = (U^{-1}w, Uv).\] 
  \item[\rm(e)] If $A > 0$ is strictly positive,
    then its graph $\Gamma(A) \subeq \cH^{\oplus 2}$ is a
    standard subspace with
  \[ \Delta_\sV= A^2 \oplus A^{-2} \quad \mbox{ and } \quad J_\sV(v,w) = (w,v).\]
  \end{description}
\end{proposition}

\begin{proof} (a) Let $\sV := \Gamma(A)$. We first observe that
  \[ I\sV
  = \{ (iv,-i Av) \: v \in \cD(A)\} 
  = \{ (v,-Av) \: v \in \cD(A)\} = \Gamma(-A),\]
and thus 
  \begin{equation}
    \label{eq:graph-intersect}
    \sV \cap I \sV = \Gamma(A) \cap \Gamma(-A)
    = \ker(A) \oplus \{0\}. 
  \end{equation}
  Therefore $\sV$ is separating if and only if $A$ is injective.

\nin (b)   Next we observe that
  \[ \Gamma(A)^{\bot_\R}
  = \{ (-A^*v,v) \: v \in \cD(A^*) \} =:\Gamma^{\rm flip}(-A^*).\]
So 
  \[ (\sV + I \sV)^{\bot_\R}
  = \sV^{\bot_\R} \cap  I \sV^{\bot_\R}
  = \Gamma^{\rm flip}(-A^*) \cap  I  \Gamma^{\rm flip}(-A^*)
  = \Gamma^{\rm flip}(-A^*) \cap  \Gamma^{\rm flip}(A^*) \]
is trivial if and only if $A^*$ is injective, which is equivalent
to $\cR(A) = A \cD(A)$ being dense.

\nin (c) follows from (a) and (b).

\nin (d)   To identify the corresponding modular objects, we claim that
  \[ \sV + I \sV = \cD(A) \oplus \cD(A^{-1}).\]
  Clearly, $\Gamma(\pm A) \subeq  \cD(A) \oplus \cR(A)
  = \cD(A) \oplus \cD(A^{-1})$,
  so that ``$\subeq$'' holds.
  For the converse, let $v \in \cD(A)$, $w \in \cD(A^{-1})$
  and put $u := A^{-1}w$. Then
  \[ (v,w) = (v,Au)
  = \Big(\frac{v + u}{2}, A\frac{v+u}{2}\Big)
  + \Big(\frac{v - u}{2}, -A\frac{v-u}{2}\Big)
  \in \sV + I \sV.\]

  The domain of the modular operator $T_\sV$ is
  $\sV + I \sV = \cD(A) \oplus \cD(A^{-1})$.
  On this domain the prescription
  \[ T(v,w) := (A^{-1}w, A v) \]
  defines an $I$-antilinear involution with $\Fix(T) = \Gamma(A) = \sV.$ 
  This implies that $T = T_\sV$ is the Tomita operator of the standard
  subspace~$\sV$.

  It is easy to see that the adjoint operator is given by
  \[ T^*(v,w) = (A^*w, (A^{-1})^*v) \quad \mbox{ with domain\ }\quad
  \cR(A^*) \oplus \cD(A^*).\]
  We thus obtain
  \[ \Delta_\sV(v,w) = (T^*T)(v,w) = T^*(A^{-1}w,Av)
    = (A^*Av, (A^{-1})^*A^{-1}w),\]
  and therefore $\Delta_\sV = |A|^2 \oplus |A^{-1}|^2$.
  Finally, we obtain
  \begin{align*}
 J_\sV(v,w)
&  = T_\sV \Delta_\sV^{-1/2}(v,w)
= T_\sV(|A|^{-1}v, |A^{-1}|^{-1}w) \\
    &= (A^{-1}|A^{-1}|^{-1}w, A|A|^{-1}v) = (U^{-1}w,Uv).
 \end{align*}

\nin (e) is a special case of (d). 
\end{proof}

 \begin{examplekh} 
Let $\sV \subeq \cH$ be a standard subspace.
  Then $\sV + i \sV = \cD(\Delta^{1/2})$ and the embedding
  \[ \sV_\C  \to \Gamma(\Delta^{1/2}), \quad
  (v,w) \mapsto (v+iw, \Delta^{1/2}(v+iw)) \]
  identifies $\sV_\C$ with a standard subspace of
  $\cH^{\oplus 2}$, endowed with the complex structure
  $I(v,w) = (iv,-iw)$. Its modular operator takes the
  form
  \[ \Delta_{\sV_\C} = \Delta_\sV \oplus \Delta_{\sV}^{-1}.\] 
\end{examplekh}

\begin{examplekh} \label{ex:5.12} 
  (Standard subspaces for the translation representation)
  We consider $\cH = L^2(\R)$, $\beta > 0$, and
  the standard subspace $\sV \subeq L^2(\R)$, specified by
  \[ Jf = \oline f \quad \mbox{ and }  \quad
    (\Delta^{-it/2\beta} f)(x) = f(x + t), \quad x,t \in \R.\]
  Then $\cD(\Delta^{1/2})$ consists of the space of boundary values of
  elements of the Hardy space
\[ H^2(\cS_\beta) := \Big\{ F \in \cO(\cS_\beta)
  \: \sup_{0 < y < \beta} \|F(\cdot + i y)\|_2 < \infty \Big\} \]
(cf.\ \cite[Prop~5.1]{Go69}). 
For $f \in \cD(\Delta^{1/2})$ we then have
(almost everywhere in the sense of $L^2$-functions) 
\[ (\Delta^{1/2}f)(x) = f(x + i \beta)  \]
(the upper boundary values on $\R + i \beta$), 
    so that $f$ is fixed by $J \Delta^{1/2}$ if and only if
    $f^\sharp = f$, where
    \[ f^\sharp(x) := \oline{f(x + i \beta )}
      \quad \mbox{ for }  \quad x \in \R.\]
    This shows that 
  \begin{equation}
    \label{eq:5.12}
    \sV = \{ f \in \cD(\Delta^{1/2}) \: f^\sharp = f\}. 
  \end{equation}
Endowed with the graph topology, we have
$\cD(\Delta^{1/2}) \cong \Gamma(\Delta^{1/2})$, and this further leads to 
\[ \Gamma(\Delta^{1/2}) \cong H^2(\cS_\beta) \subeq L^2(\R)^{\oplus 2},\]
where we identify $H^2(\cS_\beta)$
via the boundary value map $F \mapsto (F\res_{\R}, F\res_{\R + i \beta})$
 with a closed subspace of $L^2(\R)^{\oplus 2}$.

In this picture, the Tomita involution $T_\sV$ corresponds to the
involution on $H^2(\cS_\beta)$, given by
\begin{equation}
  \label{eq:tomita-hardy}
 f^\sharp(z) = \oline{f(\beta i + \oline z)}
 \quad \mbox{ for } \quad z \in \cS_\beta, 
\end{equation}
and the lower boundary value map thus induces an isometry
\begin{equation}
  \label{eq:h2sharp}
 H^2(\cS_\beta)^\sharp := \{ f \in H^2(\cS_\beta) \: f^\sharp = f\}
 \to \sV, \quad  f \mapsto f\res_\R
\end{equation}
 (cf.\ \cite[Ex.~3.16]{NO17}).
 On the pairs $(f_1, f_2) = (f, \Delta^{1/2}f)
 \in \Gamma(\Delta^{1/2}) \subeq L^2(\R)^{\oplus 2}$
 of boundary values of elements of $H^2(\cS_\beta)$, the
 involution $\sharp$ then takes the form 
    \[ (f_1, f_2)^\sharp = (\oline{f_2}, \oline{f_1}).\] 
  \end{examplekh}

\subsection{Modular theory and the Tomita--Takesaki Theorem}

The correspondence between modular objects and standard subspaces 
is the core of the modular theory of operator algebras.
It is the key structure in the Tomita--Takesaki Theorem discussed below.

\begin{definition} For a subset $S \subeq B(\cH)$, we write
\[ S' := \{ A \in B(\cH) \: (\forall M \in S)\,
    AM = MA \}\]
  for its {\it commutant}. It is a closed subalgebra
  and $*$-invariant if $S$ has this property.

  A {\it von Neumann algebra} is a $*$-invariant complex subalgebra
  $\cM \subeq B(\cH)$ satisfying $\cM = \cM''$.
For a von Neumann algebra $\cM$, a unit vector $\Omega \in \cH$ is called 
  \begin{itemize}
  \item {\it cyclic,} if $\cM \Omega$ is dense in $\cH$. 
  \item {\it separating,} if the orbit map $\cM \to \cH, M \mapsto M\Omega$
    is injective,
  \item {\it standard}, if it is cyclic and separating. 
  \end{itemize}
\index{vector!cyclic \scheiding} 
\index{vector!separating \scheiding} 
\index{vector!standard \scheiding} 
\end{definition}

\begin{lemma} \label{lem:1.1} 
  $\Omega \in \cH$ is cyclic for $\cM$ if and only if it is
  separating for $\cM'$. 
\end{lemma}

\begin{proof} Suppose first that $\Omega$ is cyclic for $\cM$.
  For $A \in \cM'$ with $A\Omega = 0$, we then obtain
  $A\cM\Omega = \cM A \Omega = \{0\}$, and since $\cM\Omega$ is
  dense in $\cH$, it follows that $A = 0$. So $\Omega$ is
  separating for $\cM'$.

  Suppose, conversely, that $\Omega$ is separating for $\cM'$.
  Let $P \: \cH \to \cH$ be the orthogonal projection
  onto $\oline{\cM\Omega}$. Then $P \in \cM'$ and
  $(\bone - P)\Omega = 0$ imply $\bone = P$, so that $\Omega$ is cyclic for $\cM$.
\end{proof}

\begin{definition} \label{def:vmomega}
For $\Omega \in \cH$ and $\cM \subeq B(\cH)$, we consider the closed
real subspace
\begin{equation}
  \label{eq:vmomegadisp}
 \sV := \sV_{\cM,\Omega} := \oline{\cM_h \Omega}
 \subeq \cH,
\end{equation}
where $\cM_h := \{ M \in \cM \: M^* = M\}$ is the real subspace of
hermitian elements in~$\cM$.
\end{definition}

\begin{lemma} \label{lem:1.2} The following assertions hold
  for a von Neumann algebra $\cM \subeq B(\cH)$ and
a unit vector  $\Omega \in \cH$. 
  \begin{description}
  \item[\rm(a)]   $\sV_{\cM,\Omega}$ is cyclic if and only if
    $\Omega$ is cyclic for $\cM$. 
  \item[\rm(b)] $\sV_{\cM,\Omega}$ is separating if and only if
    $\Omega$ is separating for the restriction of
    $\cM$ to  the cyclic subspace $\oline{\cM\Omega}$, i.e., for $A \in \cM$
    the relation 
    $A \Omega = 0$ implies $A\cM\Omega = \{0\}$. 
  \item[\rm(c)] $\sV_{\cM,\Omega}$ is standard if and only if $\Omega$ is
    a standard vector for~$\cM$.
  \end{description}
\end{lemma}

Note that $\sV_{\cM, \Omega}$ being separating only contains information
on the representation of $\cM$ on the cyclic subspace
$\cK := \oline{\cM\Omega} \subeq \cH$, but not on the representation
of $\cM$ on $\cK^\bot$. If $\cH = \C^2$,
$\cM \cong \C^2$ is the subalgebra of diagonal operators, 
and $\Omega = \be_1$, then $\sV_{\cM, \Omega} = \R \be_1$ 
is separating, but $\Omega$ is {\bf not} separating for $\cM$.
This subtlety does not play a role for (c) because we
also assume cyclicity. 

\begin{proof} (a) follows immediately from the definitions.

  \nin (b) Let $\cK := \oline{\cM\Omega}$.
  Suppose first that $\Omega$ is separating for the 
  von Neumann algebra $\cM_\cK := (\cM\res_\cK)''$.
    Then $\Omega$ is  cyclic for the commutant $\cM_\cK'
  \subeq B(\cK)$ of $\cM\res_{\cK}$ (Lemma~\ref{lem:1.1}).
  We have for $A \in \cM_{\cK,h}$ and $B \in \cM'_{\cK,h}$  the relation
  \[  \la A\Omega, B \Omega \ra
    = \la \Omega, AB \Omega \ra
    = \la \Omega, BA \Omega \ra
    = \la B\Omega, A \Omega \ra, \]
  so that
  \[ \la \sV_{\cM,\Omega}, \sV_{\cM_\cK', \Omega}\ra \subeq \R.\] 
  We conclude that
  \[ \sV_{\cM, \Omega} \cap i \sV_{\cM, \Omega}
    \subeq \sV_{\cM_\cK',\Omega}^\bot \cap \cK
    = (\cM_\cK'\Omega)^\bot \cap \cK = \{0\},\]
  i.e., $\sV_{\cM,\Omega}$ is separating. 
  
  Now we assume, conversely, that $\sV_{\cM, \Omega}$ is separating.
We show that $\Omega$ is separating for $\cM_\cK$.
  So let $A \in \cM_\cK$ with $A \Omega= 0$.
    For $B \in \cM_\cK$, we then have
  \[ A^*B\Omega = (A^*B + B^*A)\Omega \in \sV_{\cM,\Omega}, \]
  so that $A^*\cM_\cK\Omega \subeq \sV_{\cM, \Omega}$
  is a complex linear subspace,
  hence trivial because $\sV_{\cM,\Omega}$ is separating.
  Thus $A^*\cK = \{0\}$, and this implies that  $A = 0$.
  Therefore $\Omega$ is separating for $\cM_\cK$. 

\nin (c) follows from (a) and (b).   
\end{proof}

\begin{remark} \label{rem:1.10}
  (a) Cyclic vectors play an important role in representation
  theory because every $*$-representation on a Hilbert space
  is a direct sum of cyclic representations. Moreover, representations 
  with cyclic unit vector $\Omega$
  can be reconstructed completely from the corresponding state
  \[ \omega \: \cM \to \C, \quad \omega(M) := \la \Omega, M \Omega\ra.\]
  The map $\iota \: \cH \to \cM^*, \iota(v)(M) := \la \Omega, Mv \ra$
  is injective and intertwines the representation on $\cH$ with the
  right translation representation on $\cM^*$. The Hilbert space structure
  on $\iota(\cH)$, for which $\iota$ is isometric, is given by
  \[ \la \iota(M\Omega), \iota(N\Omega) \ra =  \omega(M^*N), \]
  exhibiting $\iota(\cH)$ as a reproducing kernel Hilbert space of linear
  functionals $f$, satisfying 
  \[ f(M) = \la M^*\Omega, f \ra \quad \mbox{ for } \quad
    M \in \cM, f \in \iota(\cH)\]
  (cf.~\cite[Ch.~I]{Ne99}). 

  \nin (b) If $\Omega$ is standard, then the orbit map
  $\pi^\Omega \: \cM \to \cH, M \mapsto M\Omega$ is a dense linear
  embedding of $\cM$ into $\cH$, so that we may consider $\cH$ as the
  completion of $\cM$ with respect to the scalar product
  $\la M, N \ra := \omega(M^*N)$.

  That $\Omega$ is separating corresponds to the property of the state
  $\omega$ that $\omega(M^*M) = 0$ implies $M = 0$
  ($\omega$ is {\it faithful}). One can show that all 
  normal  faithful states on a von Neumann algebra
  (cf.\ Appendix~\ref{app:wstar}) lead to equivalent
  GNS representations, called the {\it standard form representation}
  (\cite[Thm.~III.2.6.7]{Bla06}). For more details on these issues,
  see also the discussion of {\it symmetric form representations}
  in \cite{NO17}, \cite{Bla06}, \cite[\S 3.1]{BGN20} 
  and Remark~\ref{rem:weights-gns}. 
\end{remark}

\begin{theorem}  \label{thm:tom-tak} {\rm(Tomita--Takesaki Theorem; 
Tomita, 1967; Takesaki, 1970)}
Let $\cM \subeq B(\cH)$ be a von Neumann algebra and 
$\Omega \in\cH$ be a standard vector for $\cM$. 
Then $\sV := \sV_{\cM,\Omega} := \oline{\cM_h \Omega}$ is a standard subspace. 
The corresponding modular objects $(\Delta, J)$ satisfy
\begin{description}
\item[\rm(a)] $J \cM J = \cM'$ and $\Delta^{it} \cM \Delta^{-it} = \cM$
  for $t \in \R$.
\item[\rm(b)] $J \Omega = \Omega$, $\Delta \Omega = \Omega$ and 
$\Delta^{it}\Omega = \Omega$ for $t \in \R$.
\item[\rm(c)] For $M \in \cZ(\cM) := \cM \cap \cM'$, the
  center of $\cM$, we have 
$JMJ = M^*$ and $\Delta^{it} M \Delta^{-it} = M$ for $t \in \R$. 
\end{description}
\end{theorem}

It follows in particular, that
\begin{equation}
  \label{eq:alphat}
  \alpha_t(A) := \Delta^{it} A \Delta^{-it}
\end{equation}
defines a one-parameter group of automorphisms of $\cM$, called the
{\it modular automorphism group associated to $\Omega$}.
\index{modular automorphism group \scheiding}

\begin{proof} By Lemma~\ref{lem:1.2}(c),
$\sV$ is a standard subspace. We refer to 
\cite[Thm.~2.5.14]{BR87} for the other assertions,
whose proof is rather involved. The standard subspace
$\sV$ already provides~$\Delta$ and $J$. The main work consists in the
verification of (a).
\end{proof}

An approach to the Tomita--Takesaki Theorem
through bounded operators can be found in \cite{RvD77}.
For a rather  general approach to modular operators for
pairs of subspaces of real Hilbert spaces, we refer to \cite{NZ24}.

The passage to the commutant of an algebra translates easily
into the symplectic orthogonal space $\sV'$ (cf.~Definition~\ref{def:1.1}).  

\begin{lemma} \label{lem:2.10}  For a standard vector $\Omega$ of $\cM$, we have $(\sV_{\cM,\Omega})' = \sV_{\cM',\Omega}$.
\end{lemma}

\begin{proof} Let $J = J_{\cM,\Omega}$ and $\sV := \sV_{\cM,\Omega}$.
In view of  $J\Omega = \Omega$ and $J \cM J = \cM'$
  (Theorem~\ref{thm:tom-tak}), 
Lemma~\ref{lem:vv'}(f) yields
  \[ \sV'\ \ {\buildrel {\ref{lem:vv'}(f)}\over =}\ J \sV = J \oline{\cM_h \Omega}
    = \oline{ \cM'_h J\Omega} 
    = \oline{ \cM'_h \Omega}  = \sV_{\cM',\Omega}, 
  \]
and this implies the assertion. 
\end{proof}

\begin{examples} \label{ex:3.2} 
(a) Let $\cH = L^2(X,\fS,\mu)$ for a $\sigma$-finite measure space 
$(X,\fS,\mu)$ and $\cM = L^\infty(X,\fS,\mu)$, 
acting on $\cH$ by multiplication operators. 
Then the normal states of $\cM$ (cf.\ Appendix~\ref{app:wstar}) are of the form 
\[ \omega_h(f) = \int_X fh\, d\mu, \]
where $0 \leq h$ satisfies 
$\int_X h\, d\mu = 1$. Such a state is faithful if and only if 
$h\not=0$ holds $\mu$-almost everywhere. Then $\Omega := \sqrt{h} \in \cH$ 
is a corresponding standard unit vector. Let
$\sV = \sV_{\cM,\Omega}$ be the corresponding standard subspace.
As it consists of real-valued functions,  we obtain 
$T_\sV(f) = \oline f$, which is isometric and therefore 
$T_\sV = J$ and $\Delta_\sV = \bone$. 

\nin (b) Let $\cH = B_2(\cK)$ be the space of Hilbert--Schmidt operators on the complex 
separable Hilbert space $\cK$ and consider the von Neumann algebra 
$\cM = B(\cK)$, acting on $\cH$ by left multiplications. 
Then $\cM' \cong B(\cK)^{\rm op}$, the opposite algebra, 
acting by right multiplications. 
Normal states of $\cM$ are of the form 
\[ \omega_S(A) = \tr(AS), \quad \mbox{ where } \quad 0 \leq S
  \quad \mbox{ with } \quad \tr S = 1.\] 
Such a state is faithful if and only if 
$\ker S = \{0\}$ (which requires $\cK$ to be separable), 
and then $\Omega := \sqrt{S} \in \cH$ 
is a cyclic separating unit vector. 
Then $T_\sV(M\Omega) = M^*\Omega = (\Omega M)^*$ implies that 
\[ JA = A^* \quad \mbox{ and } \quad 
\Delta(A) = \Omega^{2} A \Omega^{-2}= S A S^{-1} 
\quad \mbox{ for } \quad A \in B_2(\cK).\] 

\nin (c) The prototypical pair $(\Delta, J)$ of a modular operator 
and a modular conjugation arises from the regular representation 
of a locally compact group $G$ on the Hilbert space $\cH = L^2(G, \mu_G)$ 
with respect to 
a left Haar measure~$\mu_G$. 
Here the modular operator is given by the multiplication 
\[ \Delta f = \Delta_G \cdot f,\] 
where $\Delta_G \: G \to \R^\times_+$ is the modular function of $G$
(cf.\ \eqref{eq:modfunc}),  
and the modular conjugation is given by 
\[ (Jf)(g) = \Delta_G(g)^{-\frac{1}{2}}  \oline{f(g^{-1})}.\] 
Accordingly, we have for $T = J \Delta^{1/2}$: 
\[ (Tf)(g) = \Delta_G(g)^{-1} \oline{f(g^{-1})} = f^*(g).\] 

The corresponding von Neumann algebra is the algebra
$\cM \subeq B(L^2(G,\mu_G))$, generated by the left regular representation. 
If $M_f h =f * h$ is the left convolution with $f \in C_c(G)$, then the value of 
the corresponding normal weight $\omega$ on $\cM$ (Remark~\ref{rem:weights-gns})
is given by 
$\omega(M_f) = f(e),$ so that $\omega$ corresponds to evaluation in~$e$, 
which is defined on the weakly dense subalgebra of~$\cM$
coming from the representation of the convolution algebra $(C_c(G),*)$.
\end{examples} 

\begin{remark} \label{rem:dirint} 
  Theorem~\ref{thm:tom-tak}(c) asserts that the modular group
  and $J$ commute with all central projections, and this entails
  that the whole structure adapts to the canonical central disintegration
  \[ \cM = \int_X^\oplus\, \cM_x\, d\mu(x) \]
  of $\cM$,   for which $\cZ(\cM) = L^\infty_{\rm loc}(X,\fS,\mu)$
  are the scalar decomposable operators on a locally finite measure space
  $(X,\fS,\mu)$, 
  and almost every von Neumann algebra $\cM_x$ is a factor, i.e.,
  $\cZ(\cM_x) = \C \bone$ 
(cf.\ Examples~\ref{exs:vN-exs}(b) and \cite[\S 5.4]{MN24} for more details).

  So the modular groups are ``direct integrals'' of modular groups
  of factors. For factors, the modular operators and their spectra 
  are a key tool in Connes'  classification of factors and in the characterization 
  of von Neumann algebras in terms of
  their natural cones by A.~Connes \cite{Co73, Co74}
  (see also \cite[\S 4.4]{NO17} and \cite{BR87}).  
\end{remark}

\subsection{Pairs of von Neumann algebras} 

\begin{proposition} \label{prop:vNincl} {\rm(\cite[Prop.~3.24]{Lo08})}
  Let $\cM \subeq B(\cH)$ be a von Neumann algebra with standard vector~$\Omega$.
  \begin{description}
  \item[\rm(a)] If $\cN_1, \cN_2 \subeq \cM$ are von Neumann algebras with
    $\sV_{\cN_1, \Omega} \subeq \sV_{\cN_2, \Omega}$, then $\cN_1 \subeq \cN_2$. 
  \item[\rm(b)] If $\cN$ is a von Neumann algebra commuting with $\cM$
    and $\sV_{\cN,\Omega} = \sV_{\cM,\Omega}'$, then $\cN = \cM'$. 
  \end{description}
\end{proposition}

\begin{proof} (a) Let $A \in \cN_1$ be selfadjoint. As $\cN_{1,h} \Omega
  \subeq \oline{\cN_{2,h}\Omega}$, there exists a sequence of hermitian
  elements $A_n \in \cN_2$ with $A_n\Omega \to A\Omega$.
  Then $A_n A' \Omega \to AA'\Omega$ for every $A' \in \cM'$.
  Thus $A_n \to A$ strongly on the dense subspace $\cM'\Omega$.
  Since the hermitian operators $A_n$ and $A$ are bounded,
  and $\Omega$ is separating, hence cyclic for~$\cM'$, 
the dense subspace $\cM'\Omega$ is a common core for all of them. 
With \cite[Thm.~VIII.25]{RS73} it now follows that $A_n \to A$ 
holds in the strong resolvent sense, i.e., that 
$(i\bone + A_n)^{-1} \to  (i\bone + A)^{-1}$ in the strong operator topology. 
This implies that $(i\bone + A)^{-1} \in \cN_2$, which entails $A \in \cN_2$.

  \nin (b) From $\cN \subeq \cM'$ and $\sV_{\cN,\Omega} = \sV_{\cM,\Omega}'
  = \sV_{\cM',\Omega}$ (Lemma~\ref{lem:2.10}) we derive with (a) that $\cN = \cM'$.   
\end{proof}

\begin{corollary} \label{cor:4.14} 
Let $\cM \subeq B(\cH)$ be a von Neumann algebra and 
$\Omega \in \cH$ separating for $\cM$.
To every von Neumann subalgebra $\cN \subeq \cM$ we associate
the closed real subspace 
$\sV_\cN := \oline{\cN_h \Omega}$. Then $\sV_{\cN_1} = \sV_{\cN_2}$
implies $\cN_1 = \cN_2$ for $\cN_1, \cN_2 \subeq \cM$. 
\end{corollary}

Note that the subspace $\sV_{\cN,\Omega}$ is standard if $\Omega$
is also cyclic for~$\cN$.

\subsection{The axioms for nets of local observables}

States of quantum mechanical systems are represented by 
one-dimensional subspaces $\C \Omega$ of a complex
Hilbert space $\cH$  and selfadjoint elements $A \in B(\cH)$
represent observables. The evaluation of an observable in a state
$[\Omega] := \C \Omega$ corresponds
to the evaluation of the corresponding state
\[ \omega(A) =   \frac{\la \Omega, A \Omega \ra}{\la \Omega, \Omega\ra}.\]
For some systems, the observables are restricted to selfadjoint
elements of a proper von Neumann subalgebra $\cM \subeq B(\cH)$.

In Algebraic Quantum Field Theory (AQFT) one
starts with a ``spacetime manifold'' $M$,
which, in the simplest case is Minkowski space $M = \R^{1,d-1}$.
We write its elements as pairs
\[ x = (x_0, \bx) = (x_0, x_1, \ldots, x_{d-1}) \]
and define the Lorentzian form by
\begin{equation}
  \label{eq:beta}
 \beta(x,y) = x_0 y_0 - \bx \by = x_0 y_0 - x_1 y_1 - \cdots
 - x_{d-1} y_{d-1}.
\end{equation}
We call $x \in \R^{1,d-1}$  {\it timelike} if $\beta(x,x) > 0$,
{\it lightlike} if $\beta(x,x) = 0$, and 
{\it spacelike} if $\beta(x,x) < 0$.
\index{vector!timelike \scheiding} 
\index{vector!lightlike\scheiding} 
\index{vector!spacelike \scheiding} 
The convex cone
\[ \jV_+ := \{ x \in \R^{1,d-1} \: x_0 > 0, \beta(x,x) > 0\} \]
is called the {\it positive lightcone}.
\index{lightcone \scheiding} 
Positive timelike vectors are possible tangent vectors
$\gamma'(t)$ to {\it worldlines}
$\gamma \: \R\to M$ of massive particles and positive lightlike vectors
are tangent vectors to light-rays (moving with the speed of light).
{\it Causal curves} are specified by $\gamma'(t) \in \oline{\jV_+}$ for
every~$t$, i.e., they correspond to movements not faster than light. 
\index{causal curve\scheiding} 

\begin{examples} \label{exs:ds-ads} \index{de Sitter space $\dS^d$ \scheiding } 
(a) There are also curved homogeneous spacetimes, such as {\it de Sitter space} 
\[ \dS^d = \{ (x_0, \bx) \in \R^{1,d} \: x_0^2 - \bx^2 = - 1\}.\]
It describes a model of a spherical (positively curved) expanding universe. 
It is a hypersurface in the $(d+1)$-dimensional Minkowski space~$\R^{1,d}$.
For $x\in \dS^d$, the
tangent space $T_x(\dS^d)$ can be identified with the hyperplane
\[ x^{\bot_\beta} = \{ y \in \R^{1,d} \: \beta(x,y) = 0\}.\]
Since $x$ is spacelike, the restriction of $\beta$ as  in
\eqref{eq:beta} to this hyperplane
is Lorentzian, and this specifies a Lorentzian metric on $\dS^d$.
The causal structure on $\dS^d$ is specified by
\[ C_x = \oline{\jV_+} \cap T_x(\dS^d).\]

\nin (b) \index{Anti-de Sitter space $\AdS^d$ \scheiding} 
{\it Anti-de Sitter space} is the hypersurface 
\[\AdS^d = \{ (x_1,x_2, \bx) \in \R^{2,d-1} \: x_1^2 + x_2^2 - \bx^2 =  1\} \]
in $\R^{2,d-1}$, endowed with the symmetric bilinear form 
\begin{equation}
  \label{eq:gamma-2-d}
 \gamma(x,y) = x_1 y_1 + x_2 y_2 - \bx \by \quad \mbox{ for } \quad
 x = (x_1, x_2, \bx) \in \R^{2,d-1}.
\end{equation}
Again, for $x \in \AdS^d$,
the tangent space $T_x(\AdS^d)$ can be identified 
with the hyperplane
\[ x^{\bot_\gamma} = \{ y \in \R^{2,d-1} \: \gamma(x,y) = 0\}.\]
Since $\gamma(x,x) = 1$, the restriction of $\gamma$ to this hyperplane 
is Lorentzian, and it is easy to verify that $\AdS^d$ is time-orientable
(there exists a continuous selection of  ``positive'' light cones) 
(cf.~\cite[\S 11]{NO23a}), so that there exists 
a causal structure on $\AdS^d$, for which
\[ C_{\be_1} = \{ (0,x_2, \bx) \: x_2 \geq \sqrt{\bx^2} \}.\] 
One can also argue by the connectedness
of the stabilizer group $(\SO_{2,d-1}(\R)_e)^{\be_1} \cong \SO_{1,d-1}(\R)_e$
to see that it leaves both light cones in $T_{\be_1}(\AdS^d)$ invariant. 
\end{examples}

For a family of von Neumann algebras $\cM(\cO) \subeq B(\cH)$ and a
unitary representation $(U, \cH)$ of a Lie group $G$ on $\cH$,
acting also on $M$, we  consider the following axioms:

\begin{enumerate}
\item[(Iso)] {\bf Isotony:}   
  $\cO_1 \subeq \cO_2$ implies $\cM(\cO_1) \subeq \cM(\cO_2)$.
\item[(RS)] {\bf Reeh--Schlieder property:} There exists a unit vector
  $\Omega\in \cH$ that is cyclic for $\cM(\cO),\cO \not=\eset$.
\item[(Cov)] {\bf Covariance:} $U_g \cM(\cO) U_g^{-1} = \cM(g\cO)$
  for $g \in G$. 
\item[(Vi)] {\bf Invariance of the  vacuum:} $U(g)\Omega = \Omega$ for $g \in G$.  
\item[(BW)] {\bf Bisognano--Wichmann property:}
  There exists a Lie algebra element $h \in \g$ and
  a connected open subset $W \subeq M$ (called a wedge region), such that
  $\Omega$ is cyclic and separating for $\cM(W)$, and the corresponding 
  modular operator $\Delta = \Delta_{\cM(W),\Omega}$ is given by 
  \[ \Delta = e^{2\pi i\cdot \partial U(h)}, \quad \mbox{ i.e., } \quad
    \Delta^{-it/2\pi} = U(\exp th) \quad \mbox{ for } \quad  t \in \R.\]
\item[(Loc)] {\bf Locality:} There exists an open non-empty $G$-invariant
  subset $\cD_{\rm loc} \subeq M \times M$ such that 
  $\cO_1 \times \cO_2 \subeq \cD_{\rm loc}$ implies
  $\cM(\cO_1) \subeq \cM(\cO_2)'$. 
\item[(Add)] {\bf Additivity:}
  The von Neumann algebra
  $\cM(\bigcup_j \cO_j)$ is generated by the algebras $\cM(\cO_j), j \in J$
  (cf.\ \cite{SW87}). 
\end{enumerate}

\begin{remark} \label{rem:poin} These axioms are an abstract form of the
  axioms imposed on nets of local algebras
  on Minkowski space $M = \R^{1,d}$
  \index{Poincar\'e group \scheiding} 
  and the {\it Poincar\'e group}
  \[ G = \R^{1,d-1} \rtimes \SO_{1,d-1}(\R)_e,\] 
  acting by affine isometries. We now explain the differences,
  resp., the specifics of the Minkowski case. 

\nin (a)   Here $h$ is a generator of a Lorentz
  boost:
  \begin{equation}
    \label{eq:l-boos}
    h.(x_0, x_1, \ldots, x_{d-1}) = (x_1, x_0, 0,\ldots, 0),
  \end{equation}
  and the corresponding wedge region is the {\it Rindler wedge}
\index{Rindler wedge $W_R$\scheiding} 
\begin{equation}
    \label{eq:WR}
    W_R := \{ x \in \R^{1,d-1} \: x_1 > |x_0| \},
  \end{equation}
  the set of all points $x$, where the linear vector field
  $x \mapsto h.x$ is positive timelike.
  The corresponding one-parameter group of $G$ consists of
  Lorentz boosts
  \[ e^{th} = \pmat{ \cosh t  & \sinh t  \\
      \sinh t  & \cosh t } \oplus \bone_{\R^{d-2}}.\] 

\nin  (b) The physical interpretation of the Reeh--Schlieder condition
  is that every state can be measured with arbitrary precision in any
  laboratory $\cO$.

\nin (c) In AQFT, one sometimes assumes, in addition to (Vi), the
  ``irreducibility condition''  that the fixed point space $\cH^G$ of $G$
  is one-dimensional, i.e., $\cH^G = \C \Omega$.
  We refer to \cite[\S 5.4]{MN24} for results concerning direct
  integral decompositions reducing to this case. 

  \nin (d) For Minkowski space, the subset $\cD_{\rm loc} \subeq M \times M$
  is the set  of spacelike pairs
  \[ \{ (x,y) \in \R^{1,d-1} \times \R^{1,d-1} \:
    \beta(x-y,x-y) < 0 \} \]
  for the Lorentzian form $\beta(x,y) = x_0 y_0 - \bx \by.$
  These are the pairs of spacetime events that cannot ``exchange'' information
  traveling not faster than light:
  \[ \cD_{\rm loc}
    = G.\{ (x,0) \: x \in \R^{1,d-1}, \beta(x,x) < 0\}
    = G.(\R_{> 0} \be_1 \times \{0\}).\]
  As a consequence, 
  observables in $\cO_1$ and $\cO_2$ are not subject
  to an uncertaintly relation if
  \[ \cO_1 \times \cO_2 \subeq~\cD_{\rm loc}.\] 
  To make this more precise, we recall that,
  for two selfadjoint operators $A_1$ and $A_2$, commuting is equivalent
   to the non-existence of uncertainties in common measurements
   (Exercise~\ref{exer:uncertain}). Then there exists a spectral measure
   $P$ on $\R^2$ with 
   \[ A_1 = \int_{\R^2} x_1\, dP(x) \quad \mbox{ and } \quad 
     A_2 = \int_{\R^2} x_2\, dP(x).\]
   As a consequence, states can be localized simultaneously with respect to $A_1$
   and $A_2$ with arbitrary precision.

   The monographs of Varadarajan \cite{Va85} and Mackey \cite{Ma78} 
   are excellent references for the connection between observables
   in Quantum Physics
   and selfadjoint operators.  We also recommend the recent 
paper \cite{Ba24} by J.~Baez on Jordan and Lie structures related to 
classical and quantum observables.

\end{remark}

We would like to understand the configurations
specified by the $G$-action on~$M$, the geometry of $M$,
the unitary representation $U \: G \to \U(\cH)$ and the
von Neumann algebras $\cM(\cO)$, satisfying
these axioms. As the algebra structure of the local algebras
$\cM(\cO)$ only enters through the modular groups, it makes sense to
strip it off to simplify the situation, with the hope that we arrive at 
more tractable structures. 

So we consider the family 
\begin{equation}
  \label{eq:ho1}
  \sH(\cO) = \sV_{\cM(\cO),\Omega} = \oline{\cM(\cO)_h\Omega} \subeq  \cH
\end{equation}
of closed real subspaces.
If $\Omega$ is standard for $\cM(\cO)$, then $\sH(\cO)$ is standard
(Lemma~\ref{lem:1.2}(c)),
and the corresponding modular
objects can be recovered from $\sH(\cO)$ (Definition~\ref{def:mod-obj}).
So we do not lose any
information on them.

The axioms for the algebras
$\cM(\cO)$ thus turn into the following axioms
for the net $\sH(\cO)$ of real subspaces:

\begin{enumerate}
\item[(Iso)] {\bf Isotony}: $\cO_1 \subeq \cO_2$ implies $\sH(\cO_1) \subeq \sH(\cO_2)$
\item[(RS)] {\bf Reeh--Schlieder property:} $\sH(\cO)$ is cyclic  if $\cO \not=\eset$. 
\item[(Cov)] {\bf Covariance:}
  $U_g \sH(\cO) = \sH(g\cO)$ for $g \in G$. 
\item[(BW)] {\bf Bisognano--Wichmann property:} 
  There exists a Lie algebra element $h \in \g$ and an open connected subset
   $W \subeq M$,  such that $\sH(W)$ is standard and the corresponding modular operator is
  \[ \Delta_{\sH(W)} = e^{2\pi i \cdot \partial U(h)},
    \quad \mbox{ i.e., } \quad
    \Delta_{\sH(W)}^{-it/2\pi} = U(\exp th), t \in \R.\]
\item[(Loc)] {\bf Locality:} There exists an open non-empty $G$-invariant
  subset $\cD_{\rm loc} \subeq M \times M$ such that 
  $\cO_1 \times \cO_2 \subeq \cD_{\rm loc}$ implies
  $\sH(\cO_1) \subeq \sH(\cO_2)'$.
\item[(Add)] {\bf Additivity:}
  $\sH(\bigcup_j \cO_j) = \oline{\sum_{j \in J} \sH(\cO_j)}$.
\end{enumerate} 

\begin{remark} (a) The covariance condition (Cov) for real subspaces
  follows from the $G$-invariance of $\Omega$
  and the covariance condition $U_g \cM(\cO) U_g^{-1} = \cM(g\cO)$.

  \nin (b) The subspace $\sH(M)$ is $G$-invariant by (Cov) and cyclic by (RS).
  If it is also separating, hence standard, then its modular operator
  $\Delta_{\sH(M)}$ and the conjugation $J_M := J_{\sH(M)}$ commute with $U(G)$.
  If (BW) holds, then the Equality Lemma~\ref{lem:lo08-3.10} below implies $\sH(W) = \sH(M)$,
  and thus $h$ is central in $\g$, provided $\ker U$ is discrete.
  This shows that $\sH(M)$ cannot be standard if the net is not very
  degenerate.
\end{remark}

\subsection{Operator algebras on the symmetric Fock spaces}

The passage from a net of algebras $\cM(\cO)$ 
to a net of real subspace $\sH(\cO)$ (which
is similar to a forgetful functor) can be ``inverted''
(in the spirit of an adjoint functor) by
procedures of second quantization assigning
operator algebras $\Gamma(\sH)$ to real subspaces $\sH \subeq \cH$
(see also \cite{Ar63} and \cite{NO17}).
Therefore any result on nets of real subspaces can be
transformed into a result on nets of local algebras
obtained by  second quantization (see also \cite[Rem.~4.10]{NO17}).
We note, however, that most second quantization procedures 
(such as the bosonic and fermionic one) 
are ``free'' in the sense that they do not take interaction
between particles into account.
For a recent systematic construction of twisted second quantization
functors, we refer to \cite{CSL23}. In any case,
{\bf the requirements on $G$, the $G$-action on $M$,
and the antiunitary representation of $G_{\tau_h}$
are the same for the existence of nets of real subspaces
and nets of von Neumann algebras. }

As far as the symmetries and the modular groups are concerned, 
the algebra axioms are faithfully represented by the axioms
for their associated real subspaces.
Even inclusions are rather well-behaved
(cf.\ Proposition~\ref{prop:vNincl}). 

\subsubsection{Weyl operators on the symmetric Fock space} 
\label{subsec:5.2}

In this subsection, we consider the symmetric (bosonic) Fock space 
\[ S(\cH) := \hat\bigoplus_{k \in \N_0} S^k(\cH) \]
of the complex Hilbert space $\cH$. We want to define natural 
unitary operators on this space, called the {\it Weyl operators}. 
They will form a unitary representation of the {\it Heisenberg group 
$\Heis(\cH)$}. 

We start by observing that, for every $v \in \cH$, the series 
\begin{equation}
  \label{eq:Exp}
  \Exp(v) := \sum_{n = 0}^\infty \frac{1}{n!} v^n\, ,
\end{equation}
defines an element in $S(\cH)$ and that by
\[ \la v^n, w^n \ra = n! \la v, w \ra^n \quad \mbox{ and } 
  \quad \|v^n\| = \sqrt{n!} \|v\|^{n}\]
(\cite[\S 6.1]{NO17}),  the scalar product of two 
exponentials is given by 
\[ \la \Exp(v), \Exp(w) \ra 
= \sum_{n = 0}^\infty \frac{n!}{(n!)^2} \la v, w\ra^n = e^{\la v, w \ra}.\] 

\begin{lemma} $\Exp(\cH)$ is total in $S(\cH)$,
  i.e., it spans a dense subspace. 
\end{lemma}

\begin{proof} Let $\cK \subeq S(\cH)$ be the closed subspace 
generated by $\Exp(\cH)$. We consider the unitary representation 
of the circle group $\T \subeq \C^\times$ on $S(\cH)$ by 
\[ U_z (v_1 \vee \cdots \vee v_n) := z^n (v_1 \vee \cdots \vee v_n)
  \quad \mbox{ for } \quad n \in \N_0, v_j \in \cH.\]
In particular we have
\begin{equation}
  \label{eq:usexp}
U_z \Exp(v) = \Exp(zv) \quad\mbox{ for } \quad   z\in \T, v \in \cH. 
\end{equation}
The decomposition $S(\cH) = \hat\bigoplus_{n \in \N_0} S^n(\cH)$ is the 
eigenspace decomposition with respect to the operators $U_z$ 
and it is easy to see that the action of $\T$ on $S(\cH)$ has 
continuous orbit maps (Exercise~\ref{exer:ker-cont}). 
For $\xi \in S(\cH)$ with $\xi = \sum_{n = 0}^\infty \xi_n$ and 
$\xi_n \in S^n(\cH)$, we have 
$U_z \xi = \sum_{n =0}^\infty z^n \xi_n$, so that 
\[ \xi_n = \frac{1}{2\pi} \int_0^{2\pi} e^{-2\pi i nt} U_{e^{it}} \xi\, dt \] 
(observe the analogy with Fourier coefficients). It follows that,
for $\xi \in \cK$, the existence of the above Riemann integral in the closed 
subspace $\cK$, which is invariant under $(U_z)_{z \in \T}$
by \eqref{eq:usexp},
implies $\xi_n \in \cK$. We conclude that 
$v^n \in \cK$ for $v \in \cH$ and $n \in \N_0$. Therefore it suffices
to observe that the subset $\{ v^n \: v \in \cH\}$ is total in $S^n(\cH)$
(Exercise~\ref{exer:3.2}).
\end{proof}

For $v,x \in \cH$ we have 
\[ \la \Exp(v+x), \Exp(w+x) \ra 
= e^{\la v+x, w + x\ra} 
= e^{\la v,w \ra} 
e^{ \la x, w\ra + \frac{\|x\|^2}{2}} 
e^{ \la v,x\ra + \frac{\|x\|^2}{2}},\] 
so that there exists a well-defined and uniquely determined 
unitary operator $U(x)$ on $S(\cH)$ satisfying  
\begin{equation}
  \label{eq:Ux-ops}
 U(x) \Exp(v) = e^{ -\la x, v\ra - \frac{\|x\|^2}{2}} \Exp(v+x) \quad 
\mbox{ for } \quad x,v \in \cH  
\end{equation}
(Exercise~\ref{exer:5.5}; the surjectivity of $U(x)$ follows from the totality of
$\Exp(\cH)$). 
A direct calculation then shows that 
\begin{equation}
  \label{eq:comm-rel-U}
U(x) U(y) = e^{-i\Im \la x, y \ra} U(x+y) \quad \mbox{ for  } \quad 
x, y \in \cH.
\end{equation}
In fact, for $v \in \cH$, we have 
\begin{align*}
U(x) U(y) \Exp(v) 
&= U(x) e^{ -\la y, v\ra - \frac{\|y\|^2}{2}} \Exp(v+y)\\
&= e^{ -\la y, v\ra - \frac{\|y\|^2}{2}} e^{-\la x, v + y \ra - \frac{\|x\|^2}{2}}\Exp(v+y+x)\\
&= e^{ -\la x+y, v\ra} e^{- \frac{\|y\|^2}{2} - \frac{\|x\|^2}{2} - \la x, y \ra}\Exp(v+y+x)
\end{align*}
and 
\begin{align*}
 U(x+y)  \Exp(v) 
&=  e^{ -\la x+y, v\ra - \frac{\|x+y\|^2}{2}} \Exp(v+y+x)\\
&=  e^{ -\la x+y, v\ra - \frac{\|x\|^2}{2} - \frac{\|y\|^2}{2} - \Re \la x,y \ra} \Exp(v+y+x)
\end{align*}
The relation \eqref{eq:comm-rel-U} shows that the map 
$U \: (\cH,+) \to \U(S(\cH))$ 
is not a group homomorphism. Instead, we have to replace the 
additive group $(\cH,+)$ by the {\it Heisenberg group}
\[ \Heis(\cH) := \T\times \cH \quad \mbox{ with } \quad 
(z,v)(z',v') := (zz' e^{-i\Im \la v,v' \ra}, v + v')  \] 
\index{Heisenberg group \scheiding} 
to obtain a unitary representation 
\[ \hat U \: \Heis(\cH) \to \U(S(\cH)) \quad \mbox{ by } \quad 
\hat U(z,v) := z U(v).\] 

The operators 
\[ W(v) := U\big(iv/\sqrt{2}\big), \qquad v \in \cH, \] 
are called {\it Weyl operators}. They satisfy the {\it Weyl relations} 
\begin{equation}
  \label{eq:weyl}
  W(v) W(w) = e^{-i \Im \la v,w \ra/2} W(v+w) \quad \mbox{ for } \quad v,w \in \cH. 
\end{equation}
They are an exponentiated form of the ``canonical commutation relations''
for the corresponding infinitesimal generators.

The {\it Weyl algebra} 
\[ W(\cH) := C^*(\{ W(v) \: v \in \cH\}) \subeq B(S(\cH)) \] 
is the $C^*$-subalgebra of $B(S(\cH))$ generated by the Weyl operators.  
It plays an important role in Quantum (Statistical) Mechanics 
and Quantum Field Theory. This is partly due to the fact that 
it is a simple $C^*$-algebra (all ideals are trivial), which implies 
in particular that all its representations are faithful. Closely related is its 
universal property: If $\cA$ is a unital $C^*$-algebra and 
$\phi \: \cH \to \U(\cA)$ a map satisfying the Weyl relations in the form 
\begin{equation}
  \label{eq:weyl2}
  \phi(v) \phi(w) = e^{-i \Im \la v,w \ra/2} \phi(v+w)\quad \mbox{ for }  
\quad v,w \in \cH,
\end{equation}
then there exists a unique homomorphism $\Phi \: W(\cH) \to \cA$ 
of unital $C^*$-algebras with $\Phi \circ W = \phi$. 
An excellent discussion of the Weyl algebra and its properties 
can be found in the monograph \cite{BR96} which also describes the 
physical applications in great detail. 

\subsubsection{From real subspaces to von Neumann algebras} 
\label{subsec:5.3}

In this subsection, we describe a mechanism that associates to real subspaces 
of a Hilbert space $\cH$ von Neumann algebras on the symmetric Fock space 
$S(\cH)$. This construction plays an important 
role in recent developments in  Algebraic Quantum Field Theory (AQFT) 
because it provides natural links between the geometric structure of 
spacetime and operator algebras (see in particular \cite{Ar99, Lo08, Le15}). 
It has also been of great interest for the classification 
of von Neumann factors because it provides very controlled constructions 
of factors whose type can be determined in some detail 
(\cite{AW63, AW68}). 

We write 
\[ \gamma(v,w) := \Im \la v, w \ra \quad \mbox{ for } \quad v, w \in \cH \] 
and observe that $\gamma$ is skew-symmetric and non-degenerate, so that 
the underlying real Hilbert space $\cH^\R$ carries the structure 
of a {\it symplectic vector space} $(\cH^\R, \gamma)$. 

Using the Weyl operators, we associate to every real 
linear subspace $V \subeq \cH$ a von Neumann subalgebra 
\begin{equation}
  \label{eq:R(V)}
  \cR(V) := W(V)'' = \{ W(v) \: v \in V \}''\subeq B(S(\cH)).
\end{equation}

\begin{lemma} \label{lem:6.3} We have 
  \begin{description}
  \item[\rm(i)] $\cR(V) \subeq \cR(W)'$ if and only if 
$V \subeq W'$. 
\item[\rm(ii)] $\cR(V)$ is commutative if and only if $V \subeq V'$. 
  \item[\rm(iii)] $\cR(\cH) = B(S(\cH))$, i.e., the representation 
of $\Heis(\cH)$ on $S(\cH)$ is irreducible. 
  \item[\rm(iv)] $\cR(V) = \cR(\oline V)$. 
\item[\rm(v)] $\Omega = \Exp(0) \in S(\cH)$ is cyclic for $\cR(V)$ if and only 
if $V + i V$ is dense in~$\cH$.
  \item[\rm(vi)] $\Omega = \Exp(0) \in S(\cH)$ is separating 
for $\cR(V)$ if and only if $\oline V \cap i \oline V = \{0\}$. 

  \end{description}
\end{lemma}

\begin{proof} (i) follows directly from the Weyl relations \eqref{eq:weyl}. 

\nin (ii) follows from (i). 

\nin (iii) follows from \cite[Prop.~5.2.4(3)]{BR96}. 

\nin (iv) follows from the fact that $\cH \to B(S(\cH)), v \mapsto W_v$ 
is strongly continuous and $\cR(V)$ is closed in the weak operator topology. 

\nin (v) Let $\cK := \oline{V + i V}$. Then $\cR(V)\Omega \subeq S(\cK)$, 
so that $\Omega$ cannot be cyclic if $\cK \not=\cH$. 

Suppose, conversely, that $\cK = \cH$ and that $f \in (\cR(V)\Omega)^\bot$. 
Then the holomorphic function $\hat f(v) := \la f, \Exp(v) \ra$ on 
$\cH$ vanishes on $iV$, hence also on $V + i V$, and since this subspace is 
dense in~$\cH$, we obtain $f = 0$ because $\Exp(\cH)$ is total in $S(\cH)$. 

\nin (vi) In view of (iv), we may assume that $V$ is closed. 
Let $0 \not= w \in \cK := V \cap i V$. To see that $\Omega$ is not separating 
for $\cR(V)$, it suffices to show that, for the one-dimensional Hilbert space 
$\cH_0 := \C w$, the vector $\Omega$ is not separating for 
$\cR(\C w) = B(S(\C w))$ (see (iii)). 
This is obviously the case because $\dim S(\C w) > 1$. 

Suppose, conversely, that $\cK = \{0\}$. As $\cK = V'' \cap (iV'') = (V' + i V')'$, 
it follows that $V' + i V'$ is dense in $\cH$. By (v), 
$\Omega$ is cyclic for $\cR(V')$ which commutes with $\cR(V)$. Therefore 
$\Omega$ is separating for $\cR(V)$ (Lemma~\ref{lem:1.1}). 
\end{proof}

\begin{theorem} \label{thm:araki-1} {\rm(\cite{Ar63})} 
{\rm(Araki's Duality Theorem)}  
For closed real subspaces $V,W, V_j$ of $\cH$, the following assertions hold: 
\begin{description}
\item[\rm(i)] $\cR(V) \subeq \cR(W)$ if and only if $V \subeq W$. 
\item[\rm(ii)] $\cR\big(\bigcap_{j \in J} V_j\big) = \bigcap_{j \in J} \cR(V_j)$. 
\item[\rm(iii)] $\cR(V)' = \cR(V')$ (Duality). 
\item[\rm(iv)] $\cZ(\cR(V)) = \cR(V \cap V')$
  and $\cR(V)$ is a factor if and only if $V \cap V' = \{0\}$. 
\end{description}
\end{theorem}

\begin{proof} We only comment on some of these statements: 

\nin (i) That $V \subeq W$ implies $\cR(V) \subeq \cR(W)$ is clear, but the converse 
is non-trivial. It can be derived from the duality property (iii), 
which is the main result of~\cite{Ar63}. 

\nin (ii) here ``$\subeq$'' is obvious. 

\nin (iii) is a deep theorem. 

\nin (iv) follows from (ii) and (iii) via
$\cR(V \cap V') =\cR(V) \cap \cR(V)' = \cZ(\cR(V))$. 
\end{proof}

The preceding theorem asserts in particular that 
\begin{itemize}
\item $\cR(V)$ is a factor if and only if $V \cap V' = \{0\}$. This means that the 
form $\gamma\res_{V \times V}$ is non-degenerate, i.e., that 
$(V,\gamma)$ is a symplectic vector space. 
\end{itemize}
Subspaces with this property are easy to construct. 
In \cite{Ar64b} many results on the types of the so-obtained factors
can be found.  In particular, 
it is shown that factors of type II do not occur,
and \cite{Ar64} provides an explicit criterion for 
$\cR(V)$ to be of type I. ``Generically'', the so-obtained 
factors are of type III. We refer to \cite{Sa71} for more
on types of von Neumann algebras.

\subsection{Appendices to Section~\ref{sec:2}}

  \subsubsection{Endomorphisms of standard subspaces and von Neumann algebras}

  Let $\Omega \in \cH$ be a standard vector for the von Neumann algebra
  $\cM$, let $\sV = \sV_{\cM,\Omega}$ be the corresponding standard subspace,
  and let $G \subeq \U(\cH)$ be a subgroup. 

The following example shows that the inclusion 
\[ S_{\cM,\Omega} = \{ g \in G \: g \cM g^{-1} \subeq \cM, g\Omega = \Omega\} 
\subeq S_{\sV,\Omega} = \{ g \in G\: g \sV\subeq \sV, g \Omega = \Omega \} \] 
may be proper. 

  \begin{examples}  \footnote{We thank Yoh Tanimoto for the discussion that led to this example.}
(a) We consider the Hilbert space 
$\cH := M_n(\C)$ of matrices, 
endowed with the Hilbert--Schmidt scalar product $\la A, B \ra := \tr(A^*B)$. 
By matrix multiplications from the left, we obtain a von Neumann 
subalgebra $\cM \subeq B(\cH)$, isomorphic to $M_n(\C)$, and its 
commutant $\cM'$ consists of right multiplications
(cf.\ Example~\ref{ex:3.2}(b)). 
The unit vector $\Omega := \frac{1}{\sqrt n}\bone_n$ is cyclic and separating, and 
the corresponding standard subspaces for $\cM$ and $\cM'$ coincide with 
\[ \sV_\cM = \sV_{\cM'} = \Herm_n(\C)\] 
of hermitian matrices. Now $\theta(A) := A^\top$ defines a unitary 
operator on $\cH$, preserving $\Omega$ and the standard subspace
$\sV_\cM = \sV_{\cM'}$, and satisfying $\theta \cM \theta^{-1} = \cM'$. 
For $G = \U(\cH)$, we therefore have $S_{\sV,\Omega} \not= S_{\cM, \Omega}$. 

\nin (b) In the situation above, when $\cM$ is given, 
the $G$-orbit of $\cM$ in the space of von Neumann subalgebras of $B(\cH)$ 
can be identified with the homogeneous space 
$G/G_\cM$, and similarly, $G/G_\sV \into \Stand(\cH), g G_\sV \mapsto g\sV$ is an 
embedding. The discrepancy between both spaces comes from the 
fact that the von Neumann algebra~$\cM$ need not be invariant 
under the stabilizer group $G_\sV$ of~$\sV$. 

Related questions have been analyzed by Y.~Tanimoto in \cite{Ta10}. 
He refines the picture by considering the closed convex cone 
\[ \sV_\cM^+ = \oline{\{ M\Omega \: 0 \leq M = M^* \in \cM\}} \subeq \sV_\cM,\] 
which leads to the inclusions 
\[ S_{\cM,\Omega} \into S_{\sV_\cM^+,\Omega}
  = \{ g \in G \: g \sV_\cM^+ \subeq \sV_{\cM}^+, g \Omega = \Omega\} 
\subeq S_{\sV_\cM}.\] 
The semigroup $S_{\sV_{\cM}^+,\Omega}$ appears to be much closer to $S_{\cM,\Omega}$ 
than $S_{\sV,\Omega}$. From \cite[Thm.~2.10]{Ta10} it follows in particular 
that, if $\cM$ is purely infinite, then 
$S_{\sV_{\cM}^+,\Omega} = S_{\cM,\Omega}$.
Note that the examples under (a) are of finite type. 

Let $\cM_*$ denote the predual of the von Neumann algebra $\cM$ 
(the space of normal linear functionals) and $\cM_*^+$ 
the convex cone of positive normal functionals.
In this context, it is also interesting to note that the map 
\[ \sV_\cM^+ \to \cM_*^+, \quad \xi \mapsto \omega_\xi, \qquad 
\omega_\xi(M) = \la \xi, M\xi \ra \] 
is bijective by \cite[Thm.~1.2]{Ko80}. Accordingly, 
every element $g \in S_{\sV_\cM^+}$ induces a continuous map on the
convex cone~$\cM_*^+$.
We refer to \cite{Ta10} and \cite{Co74} for more details. 
\end{examples}

\subsubsection{Positive definite functions on $\R$ satisfying a KMS condition}

This subsection has only illustrative character.
It explains how the KMS condition that classically
appears in the context of KMS states for $C^*$-algebraic 
dynamical systems (\cite{Ku57}, \cite{MS59}, \cite{HHW67}, \cite{BB94}, \cite{AA+20}, \cite{Ku02}, \cite{PW78},
\cite{Str08}, \cite{Ni23b}),\begin{footnote}
  {The fundamental KMS equilibrium condition originated in the context
    of analytic properties of Green's functions in
    \cite{Ku57} and \cite{MS59} and was formulated in the present form
    by Haag, Hugenholtz and Winnink in \cite{HHW67}.}
\end{footnote}can be formulated independently
of $C^*$-algebras as a condition 
for functions on $\R$ with values in spaces of bilinear forms.

\begin{definition} \label{def:6}
  Let $V$ be a real vector space 
and $\Bil(V)$ be the space of real bilinear maps $V \times V \to \C$.
A function $\psi \: \R \to \Bil(V)$ is said to be {\it positive 
definite} if the kernel \index{function!positive definite \scheiding} 
$\psi(t-s)(v,w)$ on $\R \times V$ is positive definite. 
We say that a positive definite 
function $\psi \: \R \to \Bil(V)$ satisfies the {\it KMS condition} 
for $\beta > 0$ if $\psi$  \index{KMS condition!for function \scheiding}
extends to a function $\oline{\cS_\beta} \to \Bil(V)$ which is pointwise 
continuous and pointwise holomorphic on the interior $\cS_\beta$, and satisfies 
\begin{equation}
  \label{eq:pd-kms}
  \psi(i \beta+t) = \oline{\psi(t)}\quad \mbox{ for } \quad t \in \R.
\end{equation}
\end{definition}

The central idea in the classification of positive definite functions 
satisfying a KMS condition is to relate them to
standard  subspaces.
A key result is the following characterization of the 
KMS condition in terms of standard subspaces (\cite[Thm.~2.6]{NO19}). 
Here we write $\Bil^+(V) \subeq \Bil(V)$ for the convex cone of all those 
bilinear forms $f$ for which the sesquilinear extension to $V_\C \times V_\C$ 
is positive semidefinite. 

\begin{theorem} \label{thm:1-intro}{\rm(Characterization of the KMS condition)} 
Let $V$ be a real vector space and $\psi \:  \R \to \Bil(V)$ be a pointwise continuous positive definite function. 
Then the following are equivalent: 
\begin{enumerate}
\item[\rm(i)] $\psi$  satisfies the 
KMS condition for $\beta > 0$. 
\item[\rm(ii)] There exists a standard real subspace 
$\sV$ in a Hilbert space $\cH$ and a linear map 
$j \: V \to \sV$ such that 
\begin{equation}
  \label{eq:pdform}
 \psi(t)(v,w) = \la  j(v), \Delta^{-it/\beta} j(w) \ra \quad \mbox{ for } \quad 
t \in \R,v,w \in V.
\end{equation}
\item[\rm(iii)] There exists a $\Bil^+(V)$-valued regular Borel measure $\mu$ 
on $\R$ satisfying 
\[ \psi(t) = \int_\R e^{it\lambda}\, d\mu(\lambda), 
\quad \mbox{ where  } \quad 
d\mu(-\lambda) = e^{-\beta\lambda}d\oline\mu(\lambda).\] 
\end{enumerate}
If these conditions are satisfied, then the function 
$\psi \: \oline{\cS_\beta} \to \Bil(V)$ is pointwise bounded. 
\end{theorem}

The equivalence of (i) and (ii) in the preceding theorem 
 describes the tight connection 
between the KMS condition and the modular objects associated to a standard 
real subspace. Part (iii) provides an integral representation that can be 
viewed as a classification result.

\begin{corollary} For a standard subspace $\sV \subeq \cH$ and the modular
  operator $\Delta_\sV$, the function
  \[ \psi \: \R \to \Bil(\sV), \quad
    \psi(t)(v,w) := \la v, \Delta_\sV^{-it/2\pi} w\ra \]
  satisfies the KMS condition for $\beta = 2\pi$.   
\end{corollary}

\begin{remark} (KMS states of $C^*$-algebras) Important special cases arise 
  from $C^*$-dynamical systems $(\cA,\R,\alpha)$,
  where $\cA$ is a $C^*$-algebra and
  $\alpha \:\R \to \Aut(\cA)$ defines a strongly continuous
  $\R$-action on $\cA$.
  Let
  \[ V := \cA_h := \{A  \in \cA \:  A^* = A\} \]  and 
consider an $\alpha$-invariant state $\omega$ on $\cA$. 
Such a state is a {\it $\beta$-KMS state} \index{KMS state\scheiding} 
if and only if 
\[ \psi \: \R \to \Bil(\cA_h), \quad 
\psi(t)(A,B) :=\omega(A\alpha_t(B)) \] 
satisfies the KMS condition for $\beta > 0$  
(cf.\  \cite[Prop.~5.2]{NO15}, \cite[Thm.~4.10]{RvD77}, \cite{BR96}). 
If $(\pi_\omega, U^\omega, \cH_\omega, \Omega)$ 
is the corresponding covariant GNS representation of $(\cA,\R)$, 
then 
\[ \omega(A) = \la \Omega, \pi_\omega(A) \Omega \ra \quad \mbox{ for }\quad 
A \in \cA \quad \mbox{ and }\quad 
U^\omega_t \Omega = \Omega \quad \mbox{ for }\quad t \in \R.\] 
For $A, B \in \cA_h$, we thus obtain 
\begin{align*}
  \psi(t)(A,B)
  &= \omega(A\alpha_t(B)) 
= \la \Omega, \pi_\omega(A\alpha_t(B)) \Omega \ra \\
&= \la \Omega, \pi_\omega(A) U^\omega_t \pi_\omega(B) U^\omega_{-t} \Omega \ra 
  = \la \pi_\omega(A)\Omega,  U^\omega_t \pi_\omega(B) \Omega \ra
\end{align*}
The corresponding standard real subspace of 
$\cH_\omega$ is $\sV_{\cA,\Omega} := \oline{\pi_\omega(\cA_h)\Omega}$.
Here we use that the KMS condition implies that
$\Omega$ is a separating vector for the
von Neumann algebra~$\pi_\omega(\cA)''$
(cf.\ \cite{Si23} and \cite{BR87}). 
\end{remark}

\subsubsection{KMS vectors for $1$-parameter groups} 
\label{subsec:4.1}

In this subsection, we collect some general tools
concerning holomorphic extensions of
orbit maps of one-parameter groups on locally convex spaces
to strips in the complex plane. They are instrumental
in formulating Kubo--Martin--Schwinger (KMS) boundary conditions
that are related to the construction of standard subspaces
(see Definition~\ref{def:6}).

\begin{definition} \label{def:linear-KMS} Let $(U_t)_{t \in \R}$
  be a one-parameter subgroup of $\GL(\cY)$ for a
  topological vector space~$\cY$
  and $J$ an antilinear operator on $\cY$, commuting with $(U_t)_{t \in \R}$. 

  We write $\cY_{\rm KMS}$ for the subspace
  of those $y \in \cY$, whose orbit map
\[ U^v \: \R \to \cY, t \mapsto U_t v \]  extends to a continuous map on  
  $\oline{\cS_\pi}:=\R+\ie[0,\pi]$,
  holomorphic on the interior $\cS_\pi$,  such that\footnote{By equivariance, it actually suffices that $U^v(\pi \ie) = Jv$.}    
\begin{equation}
  \label{eq:kms}
  U^v(\pi\ie + t) = J U^v(t) = J U_t v\quad \mbox{ for }\quad t \in \R.
\end{equation}
We call the elements of this space {\it KMS vectors} (with respect to $U$
and $J$). 
\index{KMS vectors!in space $\cY$, $\cY_{\rm KMS}$  \scheiding } 
\end{definition}

To connect with standard subspaces, we first
derive a characterization of the elements 
of a standard subspace $\sV$ specified by the pair $(\Delta, J)$
as the space $\cH_{\rm KMS}$ 
for the unitary one-parameter group $(\Delta^{\frac{-it}{2\pi}})_{t \in \R}$ 
and the conjugation~$J$. 

\begin{proposition} \label{prop:standchar} Let $\sV \subeq \cH$ be a standard subspace 
with modular objects $(\Delta, J)$. For 
$\xi \in \cH$, we consider the orbit map $\alpha^\xi \: \R \to \cH, t \mapsto 
\Delta^{-it/2\pi}\xi$. Then the following are equivalent:
\begin{enumerate} 
\item[\rm(a)] $\xi \in \sV$. 
\item[\rm(b)] $\xi \in \cD(\Delta^{1/2})$ with $\Delta^{1/2}\xi = J\xi$. 
\item[\rm(c)] The orbit map $\alpha^\xi \: \R \to \cH$ 
extends to a continuous map  $\oline{\cS_\pi} \to \cH$ which is 
holomorphic on $\cS_\pi$ and satisfies $\alpha^\xi(\pi i) = J\xi$. 
\item[\rm(d)] There exists an element $\eta \in \cH^J$ 
whose orbit map $\alpha^\eta$ 
extends to a continuous map  $\oline{\cS_{\pm\pi/2}} \to \cH$ which is 
holomorphic on the interior and satisfies $\alpha^\eta(-\pi i/2) =~\xi$. 
\end{enumerate}
\end{proposition}

\begin{proof} The equivalence of (a) and (b) follows from the definition 
of $\Delta$ and $J$. 

\nin (b) $\Rarrow$ (c): If $\xi \in \cD(\Delta^{1/2})$, then 
$\xi \in \cD(\Delta^z)$ for $0 \leq \Re z \leq 1/2$, so that the map 
\[ f \: \oline{\cS_\pi} \to \cH, \quad 
f(z) := \Delta^{-\frac{iz}{2\pi}} \xi \] 
is defined. 
Let $P$ denote the spectral measure of the selfadjoint operator  
\[ H := - \frac{1}{2\pi} \log\Delta \quad \mbox{ and let } \quad 
P^\xi = \la \xi, P(\cdot) \xi\ra \] 
denote the corresponding positive 
measure on $\R$ defined by $\xi \in \cH$. 
Then \cite[Lemma~A.2.5]{NO18}  shows that 
\[ \cL(P^\xi)(2\pi) = \int_\R e^{-2\pi\lambda}\, dP^\xi(\lambda)
  = \|e^{-\pi H} \xi\|^2  =  \|\Delta^{1/2} \xi\|^2 < \infty.\] 
This implies that the kernel 
\begin{align*}
 \la f(w), f(z) \ra 
&= \la \Delta^{-\frac{iw}{2\pi}} \xi, \Delta^{-\frac{iz}{2\pi}} \xi \ra 
= \la  \xi, \Delta^{-\frac{i(z-\oline w)}{2\pi}} \xi \ra \\
&= \la  \xi, e^{(z-\oline w)i H} \xi \ra 
     = \cL(P^\xi)\Big(\frac{z-\oline w}{i}\Big)
\end{align*}
is continuous on $\oline{\cS_{\pi}} \times \oline{\cS_{\pi}}$ 
by the Dominated Convergence Theorem,  
holomorphic in $z$, and antiholomorphic in $w$ on 
the interior (\cite[Prop.~V.4.6]{Ne99}). 
This implies (c) because it shows that 
$f$ is holomorphic on $\cS_\pi$ (\cite[Lemma~A.III.1]{Ne99}) 
and continuous on $\oline{\cS_\pi}$
(Exercise~\ref{exer:ker-cont}). 

\nin (c) $\Rarrow$ (d): For $\alpha^\xi \: \oline{\cS_\pi} \to \cH$ 
as in (c), we have 
\begin{equation}
  \label{eq:jequiv}
J \alpha^\xi(z) = \alpha^\xi(\pi i + \oline z) 
\end{equation}
by analytic continuation, so that 
\[ \eta := \alpha^\xi(\pi i/2) \in \cH^J \quad \mbox{ with } \quad 
\alpha^\eta(z) = \alpha^\xi\Big(z + \frac{\pi i}{2}\Big).\] 

\nin (d) $\Rarrow$ (b): We abbreviate $\cS := \cS_{\pm\pi/2}$. 
The kernel 
$K(z,w) := \la \alpha^\eta(w), \alpha^\eta(z) \ra$ 
is continuous on $\oline{\cS} \times \oline{\cS}$ and 
holomorphic in $z$ and antiholomorphic in $w$ on $\cS$. 
It also  satisfies 
$K(z + t, w) = K(z,w - t)$ for $t \in \R$. Hence there exists a 
continuous function $\phi$ on $\oline{\cS}$, holomorphic on $\cS$, such that 
\[ K(z,w) = \phi\Big(\frac{z - \oline w}{2}\Big).\] 
For $t \in \R$, we then have 
$\phi(t) = \la \eta, \alpha^\eta(2t) \ra = \int_\R e^{2it\lambda}\, dP^\eta(\lambda)$, 
so that \cite[Lemma~A.2.5]{NO18} yields 
$\cL(P^\eta)(\pm\pi) < \infty$ and 
$\eta \in \cD(\Delta^{\pm 1/4})$. This implies that 
$\alpha^\eta(z) = \Delta^{-iz/2\pi} \eta$ for $z \in \oline{\cS}$. 

From 
$\xi = \alpha^\eta(-\pi i/2) = \Delta^{-1/4} \eta$ we derive that 
\[  \alpha^\xi(z) = \alpha^\eta\Big(z- \frac{\pi i}{2}\Big) 
=  \Delta^{-iz/2\pi} \xi \quad \mbox{ for } \quad z \in \oline{\cS_\pi}.\] 
Further, $J\eta = \eta$ implies  
\[  J\alpha^\xi(z) 
= J\alpha^\eta\Big(z- \frac{\pi i}{2}\Big)
= \alpha^\eta\Big(\oline z+  \frac{\pi i}{2}\Big)
= \alpha^\xi(\pi i + \oline z).\] 
For $z = 0$, we obtain in particular 
$J\xi = \alpha^\xi(\pi i) = \Delta^{1/2}\xi$. 
\end{proof}

In \cite{BN24} we study for an antiunitary representation
$(U,\cH)$ of $G_{\tau_h}$ the space
$\cH^{-\infty}_{\rm KMS} := (\cH^{-\infty})_{\rm KMS}$
of KMS distribution vectors (see Appendix~\ref{app:c} for details),
on which we have the one-parameter group $U^{-\infty}(\exp th)$
generated by the Euler element $h$ and the action of the conjugation
$J^{-\infty} = U^{-\infty}(\tau_h)$. 

Our general results imply in particular 
the following theorem, which is a key tool to verify that
nets of real subspaces satisfy the Bisognano--Wichmann property
$\sH(W) \subeq \sV$. 
  \begin{theorem} \label{thm:BN24}
    Let $(U,\cH)$ be an antiunitary representation of
    $G_{\tau_h}$ and $\sV \subeq \cH$
    the standard subspace specified by $\Delta_\sV = e^{2\pi i \cdot \partial U(h)}$
    and $J_\sV = U(\tau_h)$.
Then the following assertions hold: 
\begin{description}
\item[\rm(a)]  $\cH^{-\infty}_{\rm KMS} := (\cH^{-\infty})_{\rm KMS}$
  is a weak-$*$ closed subspace of $\cH^{-\infty}$. 
\item[\rm(b)]  $\cH^{-\infty}_{\rm KMS}\cap\cH=\sV$. 
\item[\rm(c)] $\sV$ is weak-$*$ dense in $\cH^{-\infty}_{\rm KMS}$.
\end{description}  
  \end{theorem}

  \begin{proof} (a)-(c) follow from \cite[Thm.~6.2]{BN24},
    \cite[Thm.~6.4]{BN24}, and \cite[Thm.~6.5]{BN24}, respectively. 
  \end{proof}

  \subsubsection{Boundary values for one-parameter groups}
\label{subsubsec:bv-d=1} 

In this section, we collect some useful
facts on boundary values of analytically
extended orbit maps of unitary one-parameter groups~$(U_t)_{t \in \R}$
and a conjugation $J$, commuting with $U$.
The main point is to identify the subspace
$\cH^{-\infty}_{\rm KMS}$ of distribution vectors, satisfying the
KMS condition (cf.\ Definition~\ref{def:linear-KMS}) 
with elements of the real subspace $\cH^J_{\rm temp}$, specified in
terms of the spectral measure~$P$ of $U$.

Let $P$ be the uniquely determined spectral measure
on $\R$ for which
\[ U_t = \int_\R e^{itx}\, dP(x), \quad \mbox{ resp.} \quad
  U_t = e^{itA}, \ t \in \R,  \quad \mbox{ with } \quad
  A = \int_\R p\, dP(p).\] 
For $\xi \in \cH$, we thus obtain finite positive measures $P^\xi := \la \xi, P(\cdot) \xi\ra$,
and we define
\begin{equation}
  \label{eq:hjtemp}
 \cH^J_{\rm temp} := \{ \xi \in \cH^J \: e^{ \pi p}\, dP^\xi(p)\ 
  \mbox{ temp.}\}
  = \{ \xi \in \cH^J \: e^{-\pi p}\, dP^\xi(p)\
  \mbox{ temp.}\}.
\end{equation}
The equality of both spaces on the right follows from the symmetry
of the measures $P^\xi$, which is a consequence of $J\xi = \xi$.
For the positive selfadjoint operator $\Delta := e^{-2\pi A}$, we have
$J\Delta J = \Delta^{-1}$, so that
$J\cD(\Delta^{1/4}) = \cD(\Delta^{-1/4})$ implies that
\begin{align} 
  \label{eq:defdelta14}
  \cD(\Delta^{1/4}) \cap \cH^J
 & = \cD(\Delta^{-1/4}) \cap \cH^J 
= \Big\{ \xi \in \cH^J \: \int_\R e^{\pm \pi p}\, dP^\xi(p) < \infty\Big\}
   \notag\\
 &\subeq \cH^J_{\rm temp}.
\end{align}

\begin{theorem}   \label{thm:6-1-fno24} {\rm(\cite[Thm.~6.1]{FNO25b})}
  For $\xi \in \cH^J \cap \bigcap_{|t| < \pi/2} \cD(e^{tA})$, 
the following are equivalent:
\begin{description}
\item[\rm(a)] $\xi \in \cH^J_{\rm temp}$. 
\item[\rm(b)] The limits $\beta^{\pm}(\xi) := \lim_{t \to \pm\pi/2} e^{-tA} \xi$
  exist in $\cH^{-\infty}(U)$.
\item[\rm(c)] There exist $C,N > 0$ such that 
  $\|e^{\pm tA}\xi\|^2 \leq C \big(\frac{\pi}{2}-|t|\big)^{-N}$
  for $|t| < \pi/2$.   
\end{description}
\end{theorem}

\begin{proof} (a) $\Leftrightarrow$ (b): 
From \cite[Prop.~4]{FNO25a}, we recall that the temperedness of the
measure $\nu_\xi$, given by  $d\nu_\xi(p) := e^{\pi p}\, dP^\xi(p)$, 
is equivalent to the existence of $C, N > 0$ with 
\[ \int_\R e^{(\pi - t)p}\, dP^\xi(p) \leq C t^{-N}
  \quad \mbox{ for }  \quad 0 \leq t < \pi.\] 
Further, \cite[Lemma~10.7]{NO15} shows that this condition is equivalent
to the function $e^{\pi p/2}$ to define a distribution vector
for the canonical multiplication representation on $L^2(\R, P^\xi)$.
This representation is equivalent to the subrepresentation of $(U,\cH)$, 
generated by~$\xi$, where the constant function~$1$ corresponds to~$\xi$.

\nin (b) $\Rarrow$ (c): If $\lim_{t \to \pi/2} e^{tA} \xi$ exist in $\cH^{-\infty}(U)$, 
then   \cite[Lemma~10.7]{NO15}, applied to the cyclic subrepresentation 
generated by $\xi$, implies that the measure $\nu_\xi$ 
is tempered. Then the argument from above implies the existence
of $C, N > 0$ with 
\begin{equation}
  \label{eq:exp-esti}
 \|e^{tA}\xi\|^2 = \int_\R e^{2tx}\, dP^\xi(x)  
  \leq C \Big(\frac{\pi}{2}-t\Big)^{-N}
  \quad \mbox{ for }  \quad
  |t| < \pi/2.
\end{equation}
If $\lim_{t \to -\pi/2} e^{tA} \xi$ also exists in $\cH^{-\infty}(U)$,
then the same argument  applies again and we obtain (c). 

\nin (c) $\Rarrow$ (a): With the leftmost equality in
\eqref{eq:exp-esti}, we see that (c) implies that
the measures $d\nu_\xi(x) := e^{\pm \pi x}\, dP^\xi(x)$ are tempered
(\cite[Prop.~4]{FNO25a}). 
  Here we use that the measure $P^\xi$ is symmetric because $J\xi = \xi$.
\end{proof}

\begin{proposition} The map $\beta^+$ defines a bijection 
$\beta^+ \: \cH^J_{\rm temp} \to  \cH^{-\infty}_{\rm KMS}.$ 
\end{proposition}

\begin{proof} (a) Let $\xi \in \cH^J_{\rm temp}$.
  First we show that  $\beta^+(\xi) \in \cH^{-\infty}_{\rm KMS}$.
To this end, note that, for a real-valued test function
$\phi \in C^\infty_c(\R,\R)$, we have $J U(\phi) = U(\phi)J$.
For $\xi \in \cH^J_{\rm temp}$ we therefore have
$\eta := U(\phi)\xi \in \cH^J_{\rm temp}$.

For $\xi \in \cH$ we also note that the spectral integral
$U_t = \int_\R e^{itp}\, dP(p)$ representation yields with
\[   \hat\phi(x) = \int_\R e^{itx}\phi(t)\, dt\]
the relation 
\begin{align*}
 U(\phi)
 & = \int_\R \phi(t) U_t \, dt 
   = \int_\R \phi(t) \int_\R e^{itp}\, dP(p)\, dt
   = \int_\R \phi(t) \int_\R e^{itp}\, dt\, \, dP(p)\\
&  = \int_\R \Big(\int_\R \phi(t) e^{itp}\, dt\Big)\, dP(p) = \int_\R \hat\phi(p)\, dP(p).
\end{align*}
Therefore 
\[ dP^\eta(x) = |\hat\phi(x)|^2 dP^\xi(x),\]
where $\hat\phi$ is a Schwartz function.
The temperedness of $P^\xi$ hence implies that the measure
  \[ e^{\pi x}\, dP^\eta(x) = e^{\pi x} |\hat \phi(x)|^2\, dP^\xi(x) \]
  is finite, and thus $\eta \in \cD(\Delta^{1/4}) \cap \cH^J$. 
  This implies with Proposition~\ref{prop:standchar} that 
  \begin{equation}
    \label{eq:eta-prop}
 U^{-\infty}(\phi) \beta^+(\xi) = \beta^+(U(\phi)\xi)
 = \beta^+(\eta)  = \Delta^{1/4}\eta   \in \sV.
\end{equation}

Next we claim that 
\begin{equation}
  \label{eq:kmsrel1}
 \cH^{-\infty}_{\rm KMS}= \{ \alpha \in \cH^{-\infty} \:
 (\forall \phi \in C^\infty_c(\R,\R))\,   U^{-\infty}(\phi) \alpha \in \sV \}.
\end{equation}
In fact, if $\alpha \in \cH^{-\infty}$ satisfies 
$U^{-\infty}(\phi) \alpha \in \sV$ for all $\phi \in C^\infty_c(\R,\R)$,
we apply this to a $\delta$-sequence $\phi =  \delta_n$ in $0$ to obtain
\[ \alpha= \lim_{n \to \infty} U^{-\infty}(\delta_n) \alpha
  \in \oline{\sV}^{w-*} = \cH^{-\infty}_{\rm KMS}\] 
by the weak-$*$ closedness of $\cH^{-\infty}_{\rm KMS}$ (Theorem~\ref{thm:6-1-fno24}). 
Conversely, $\alpha \in \cH^{-\infty}_{\rm KMS}$ implies
\[ U^{-\infty}(\phi) \alpha \in \cH \cap \cH^{-\infty}_{\rm KMS} = \sV,\]
again by the closedness and $U(\R)$-invariance of $\cH^{-\infty}_{\rm KMS}$. 
With \eqref{eq:eta-prop} we thus obtain that
$\beta^+(\xi) \in \cH^{-\infty}_{\rm KMS}$.
 
 \nin (b) To see that $\beta^+$ is injective, we assume that $\beta^+(\xi) = 0$.
 Then the above argument implies that 
 $U(\phi)\xi \in \cH^J \cap \cD(\Delta^{1/4})$ vanishes for every
 $\phi \in C^\infty_c(\R,\R)$ because $\Delta^{1/4}$ is injective.
 Using an approximate identity in this space,
 $\xi = 0$ follows. 
 
 \nin (c) To see that $\beta^+$ is surjective, let
 $\gamma \in \cH^{-\infty}_{\rm KMS}$. Replacing $\cH$ by the
 cyclic subrepresentation generated by $\gamma$, resp., the subspace
 $U^{-\infty}(C^\infty_c(\R,\C))\gamma \subeq \cH$, we may w.l.o.g.\
 assume that $\cH = L^2(\R,\nu)$ for a positive Borel measure,
 where the constant function~$1$ corresponds to $\gamma$.
 Hence the measure $\nu$ on $\R$ is tempered (\cite[Lemma~10.7]{NO15}).
 Then, for $z = x + iy \in \cS_\pi$, the analytic continuation of the orbit
 map of $\gamma = 1$ takes the form 
 \[ U^\gamma \: \oline{\cS_\pi} \to L^2(\R,\nu)^{-\infty},
   \quad U^\gamma(z)(p) = e^{izp} = e^{ixp} e^{-yp}.\]
 Therefore all measures $e^{-yp}\, d\nu(p), 0 \leq y \leq \pi$, are tempered.
 It follows in particular  that they are actually finite for $0 < y < p$. 
 Hence $\xi(p) := e^{- \pi p/2}$ is an $L^2$-function, and
 $\xi = U^\gamma(\pi i/2)$ implies that $J\xi = \xi$. As a consequence, the measure
$dP^\xi(p) = e^{- \pi p}\, d\nu(p)$ 
 is finite and $e^{\pi p}\, dP^\xi(p) =\, d\nu(p)$ is tempered, so that
 $\xi \in \cH^J_{\rm temp}$. Therefore $\beta^+(\xi) = 1$ shows that
 $\beta^+$ is surjective.
\end{proof}

For $\xi, \eta\in \cH^J_{\rm temp}$, we consider the complex-valued measure
\[ P^{\xi,\eta}(E) := \la \xi, P(E) \eta \ra, \quad E \subeq \R.\]
Then
\begin{equation}
  \label{eq:pwv}
 \oline{P^{\xi,\eta}(E)} = \oline{\la \xi, P(E) \eta \ra}
 = \la \eta, P(E) \xi \ra = P^{\eta,\xi}(E)
\end{equation}
and  the relation 
  $J P(E) J = P(-E)$ implies that
  \begin{align} 
    \label{eq:pvw-symm}
    P^{\xi,\eta}(E)
    &= \la J\xi, P(E) J\eta \ra 
    = \la J\xi, J P(-E) \eta \ra \notag \\
& = \la P(-E) \eta, \xi \ra = P^{\eta,\xi}(-E) = \oline{P^{\xi,\eta}(-E)}.
  \end{align}
In particular, the measures $P^{\xi,\xi}$ are symmetric and positive.

We obtain on the strip $\cS_{\pm \pi}$ the
holomorphic function
\[ \phi^{\xi,\eta}(z) :=\hat{P^{\xi,\eta}}(z) = \int_\R e^{izp}\, dP^{\xi,\eta}(p),\]
and the temperedness of the measures $e^{\pm \pi p}\, dP^{\xi,\eta}(p)$
implies that this function has boundary values that are
tempered distributions on $\pm\pi i + \R$. For $t \in \R$, we have
$\phi^{\xi,\eta}(t) =\la \xi, U_t \eta \ra.$ 
Hence
\[ \phi^{\eta,\xi}(-t)  = \oline{\phi^{\xi,\eta}(t)}
  = \la U_t \eta, \xi \ra
  =  \la U_t J \eta, J \xi \ra 
  =  \la J U_t \eta, J \xi \ra 
  =  \la \xi, U_t \eta \ra = \phi^{\xi,\eta}(t),\]
and therefore
\begin{equation}
  \label{eq:phivwoline}
 \oline{\phi^{\xi,\eta}(z)} = \phi^{\eta,\xi}(-\oline z) = \phi^{\xi,\eta}(\oline z)
 \quad \mbox{ for }\quad   z \in \cS_{\pm \pi}.
\end{equation}

For $\alpha := \beta^+(\xi)$ and $\gamma:= \beta^+(\eta)$
the distribution
\[ D_{\alpha,\gamma}(\xi) := \gamma(U^{-\infty}(\xi)\alpha) \]
can be represented by the boundary values of a holomorphic function
\begin{align*}
 D_{\alpha,\gamma}(x)
&  = \lim_{t \to \pi/2}  \la U_x e^{tA}\xi,  e^{tA}\eta \ra
                        = \lim_{t \to \pi/2}  \int_\R e^{(2t-ix)p}\, dP^{\xi,\eta}(p)\\
& = \phi^{\xi,\eta}(-\pi i -x)  = \phi^{\eta,\xi}(\pi i +x).                                   
\end{align*}

\begin{small}
\subsection{Exercises for Section~\ref{sec:2}}

\begin{exercise} \label{exer:ker-cont} 
Let $X$ be a topological space, $\cH$ be a Hilbert space 
and $\gamma \: X \to \cH$ be a map. Show that $\gamma$ is continuous if and 
only if the corresponding kernel function 
\[ K \: X \times X \to \C, \quad K(x,y) := \la \gamma(x), \gamma(y) \ra \] 
is continuous.   
\end{exercise}

\begin{exercise} Let $(U_t= e^{itA})_{t \in \R}$ be a unitary one-parameter group
  on the complex Hilbert space~$\cH$ and consider on the complex Hilbert space
  $\tilde\cH := \cH \oplus \oline\cH$ the unitary one-parameter group
  \[ \tilde U_t := U_t \oplus U_t.\]
  Show that the flip involution $\tilde J(v,w) := (w,v)$
  and the positive operator $\tilde\Delta := e^A \oplus e^{-A}$ 
  form a modular pair of a standard subspace $\sV \subeq \tilde\cH$
  (cf.~Proposition~\ref{prop:2.27}).   
\end{exercise}

\begin{exercise} If $\sV \subeq \cH$ is a standard subspace, we consider
  the antiunitary representation of $\R^\times$, defined by
  \[ \gamma_\sV(e^t) := \Delta_\sV^{it}, \quad
    \gamma_\sV(-1) := J_\sV.\]
  Show that we thus obtain a bijection between 
  the set $\Stand(\cH)$ of standard subspaces of $\cH$ and
  the set of antiunitary (strongly continuous) representations
  $\gamma \: \R^\times \to \AU(\cH)$.
\end{exercise}

\begin{exercise} \label{holomext} (Holomorphic extensions of orbit maps) 
Let $\cH$ be a Hilbert space, let $A$ be a selfadjoint operator and let $U_t=e^{itA}$ be the corresponding one-parameter group.
\begin{enumerate} 
\item[\rm(a)] Show that the orbit map $U^v:\R\to \cH,\;t\mapsto U_t v$ extends to a holomorphic map
$$\cS_{\pm \beta}:=\{z\in \C\;|\;|\Im(z)|<\beta\}\to \cH
$$
if and only if $v\in \cD(e^{tA})$, for all $|t|<\beta$. \textit{Hint:} Consider the kernel
$$K^{v,v}:\R\times \R\to \C,\quad (s,t)\mapsto \langle U^v(s),U^v(t)\rangle.
$$
\item[\rm(b)] Suppose  $v\in \cD(e^{tA})$, for all $|t|<\beta$. Show that $v\in \cD(e^{\pm \beta A})$ if and only if the map \break
  $U^v:\cS_{\pm \beta}\to \cH$ has continuous boundary values on $\R\pm i\beta$.
\end{enumerate} 
\end{exercise}

\begin{exercise} Let $\cM \subeq B(\cH)$ be a von Neumann algebra. 
  For two unit vectors $\Omega_1,\Omega_2 \in \cH$,
  the states $\omega_{\Omega_1}$ and $\omega_{\Omega_2}$ coincide if and only
  if there exists an $\cM$-equivariant isometry
  \[ U \: \oline{\cM\Omega_1} \to \oline{\cM\Omega_2}
    \quad \mbox{ with } \quad U\Omega_1 = \Omega_2.\]

  Conclude further that, if $\cM \not=B(\cH)$, then there exist
  linearly independent unit vectors $\Omega_1$ and $\Omega_2$,
  defining the same state on $\cM$. Hint: $\cM\not= B(\cH)$ is equivalent to
$\cM'$ being non-trivial.
\end{exercise}

\begin{exercise} \label{exer:bgl} (The Brunetti--Guido--Longo (BGL)
  construction, \cite{BGL02}) Let $G$ be a Lie group,
  $\sigma \in \Aut(G)$ be an involution and
  $G_\sigma := G \rtimes \{ \bone,\sigma\}$ the corresponding semidirect
  product. We consider an antiunitary representation
  $U \: G_\sigma \to \AU(\cH)$, i.e., $U(G) \subeq \U(\cH)$ and
  $U(\sigma)$ antilinear.

  We consider the set
  \[ \cG(G_\sigma) := \{ (x,\tau) \in \g \times G\sigma \: \Ad(\tau)x = x,
    \tau^2 = e\}.\]
  Show that:
  \begin{description}
  \item[\rm(a)] Each $(x,\tau)$ defines a morphism
    \[ \gamma \: \R^\times \to G_\sigma, \quad
      \gamma(e^t) := \exp(tx), \quad \gamma(-1) := \tau.\]
  \item[\rm(b)] For each pair $(x,\tau)$ there exists a unique
    standard subspace $\sV  \subeq \cH$ with
    \[ J_\sV = U(\tau) \quad \mbox{ and }  \quad
      \Delta_\sV = e^{2\pi i\cdot \partial U(x)}.\]
  \end{description}
\end{exercise}

\begin{exercise} \label{exer:5.5} 
Let $\cH_1$ and $\cH_2$ be Hilbert space, 
$X$ be a set and $\gamma_j \: X \to \cH_j$, $j =1,2$, be maps with total range.
Then the following are equivalent: 
\begin{description}
\item[\rm(a)] There exists a unitary operator 
$U \: \cH_1 \to \cH_2$ with $U \circ \gamma_1 = \gamma_2$. 
\item[\rm(b)] 
$\la \gamma_2(x), \gamma_2(y) \ra = \la \gamma_1(x), \gamma_1(y) \ra$ 
for all $x,y \in X$. 
\end{description}
\end{exercise}

\begin{exercise}
  \label{exer:3.1} {\rm(Polarization)} 
Let $V$ and $W$ be $\K$-vector spaces, $\beta \: V^n \to W$ be a symmetric 
$n$-linear map and $\gamma(v) := \beta(v,\cdots, v)$. 
Show that $\beta$ is completely determined 
by the values on the diagonal 
$\beta(v,\ldots, v)$, $v\in V$. \\
Hint: Consider 
\[ \gamma(t_1 v_1 + \ldots + t_n v_n)
= \sum_{m_1 + \ldots + m_n = n} \frac{n!}{m_1!\cdots m_n!}  t_1^{m_1} \cdots t_n^{m_n} \beta(v_1^{m_1}, \ldots, v_n^{m_n})\]
and recover $\beta(v_1,\ldots, v_n)$ as a suitable partial derivative. 
Alternatively, one can verify the following explicit formula: 
\begin{equation}\label{polarformulvar}
\beta(v_1,\ldots,v_n) = 
\frac{1}{n!\, 2^n}\sum_{\eps_1,\ldots,\eps_n\in \{1,-1\}}
\eps_1\cdots\eps_n\, \gamma(\eps_1 v_1+\cdots+\eps_n v_n).
\end{equation}
\end{exercise}

\begin{exercise}
  \label{exer:3.2} 
Let $V$ be $\K$-vector space and $S^n(V) := (V^{\otimes n})^{S_n}$ be the 
$n$th symmetric power of~$V$. Show that 
\[ S^n(V) = \Spann \{ v^{\otimes n}\: v \in V \}.\] 
Hint: Use the same technique as in Exercise~\ref{exer:3.1}.
\end{exercise}

\begin{exercise} \label{exer:uncertain} (Abstract uncertainty principle)
  Let $A$ and $B$ be bounded selfadjoint operator on~$\cH$ and 
  $\Omega \in \cH$ a unit vector.
  Then $\Omega$ defines a state whose expectation
  values for the observable $A$ is given by
\[ c_A := \omega_\Omega(A) = \la \Omega, A \Omega \ra.\] 
The variance of the observable $A$ in the state $\omega_\Omega$ is
given by the expectation   value
\[ \sigma_A := \omega_\Omega( (A- c_A\bone)^2)^{1/2} 
  = \|(A - c_A\bone) \Omega \|.\]
  It vanishes if and only if $A \Omega  = c_A \Omega$,
  i.e., if $\Omega$ is an eigenvector of~$A$. 

  Verify the abstract uncertainty principle (\cite{Ro29}):
\begin{equation}
  \label{eq:uncert}
  \sigma_A \sigma_B  \geq \frac{1}{2}   |\la \Omega,[A, B] \Omega \ra|. 
\end{equation}
\end{exercise}

\begin{exercise}
  \label{exer:abstar}
  Let $A \: \cD(A) \to \cH$ and $B \: \cD(B) \to \cH$
  be densely defined unbounded operators on the real Hilbert space~$\cH$,
  so that their adjoints
  \[  A^* \: \cD(A^*) \to \cH, \quad B^* \: \cD(B^*) \to \cH \]
  are also defined by 
  \[ \la A^* v, w \ra = \la v, A w \ra \quad \mbox{ for } \quad
    w \in \cD(A), v \in \cD(A^*).\]
  The product $AB$ is defined on $\cD(AB) = B^{-1}\cD(A)
  \subeq \cD(B)$ by composition.
  Show that:
  \begin{description}
  \item[(a)] If $\cD(AB)$ is dense, then $(AB)^*$ is an extension of~$B^*A^*$. 
  \item[(b)] If $A$ is invertible, then $(AB)^* = B^*A^*$. 
  \end{description}
\end{exercise}

\begin{exercise} \label{exer:sym}  
Let $\sV\subset\cH$ be a standard subspace  and $U\in\AU(\cH)$ 
be a unitary or an antiunitary operator. 
Show that $U\sV$ is also standard and 
$U\Delta_\sV U^{-1}=\Delta_{U\sV}$ and $UJ_\sV U^{-1}=J_{U\sV}$. 
\end{exercise}

\end{small}

\section{Euler elements} 
\label{sec:3}

We have seen in Section~\ref{sec:2}
how nets of real subspaces arise from nets of algebras
of local observables. Eventually, one would like
to ``classify'' all these nets in a suitable sense,
but first one has to specify which structures
we are dealing with. Key points are 
\begin{description}
\item[(Q1)] Which elements $h \in \g$ appear in the
Bisognano--Wichmann (BW) condition?
\item[(Q2)] Which $G$-invariant structure do we need on
  $M$ as a fertile ground for nets of real subspaces? 
\item[(Q3)] How to find suitable domains $W \subeq M$ for which the
  (BW) condition may be satisfied?
\end{description}
As we shall see below, these questions are highly intertwined,
in particular when we discuss in Section~\ref{sec:4} 
which unitary representations $(U,\cH)$ of $G$
admit nets satisfying the axioms
(Iso), (Cov), (RS) and~(BW).

\subsection{The Euler Element Theorem}

\begin{definition}
  Let $\g$ be a finite-dimensional Lie algebra.
  We call $h \in \g$ an {\it Euler element} if \index{Euler element\scheiding}
  $\ad h$ is {\bf non-zero} and diagonalizable 
  with $\Spec(\ad h) \subeq \{-1,0,1\}$, i.e., if
  \[ \g = \g_1(h) \oplus \g_0(h) \oplus \g_{-1}(h).\] 
  Then $\tau_h := e^{\pi i \ad h} \in \Aut(\g)$ is an involutive automorphism
  of~$\g$. 
 We write $\cE(\g)$ for the set of Euler elements in $\g$.
An Euler element $h$ is called {\it symmetric} if $-h \in \cO_h := \Inn(\g)h$.
\index{Euler element!symmetric \scheiding}
\end{definition}

\begin{remark} We observe that $\cE(\g) + \fz(\g) = \cE(\g)$,
  where $\fz(\g) = \{ x \in \g \: [x,\g] = \{0\}\}$ is the center of $\g$.   
\end{remark}

The following theorem (\cite[Thm.~3.1]{MN24}) provides a very satisfying
answer to question (Q1). 

\begin{theorem} \label{thm:2.1} {\rm(Euler Element Theorem)} 
  Let $G$   be a connected finite-dimensional Lie group with 
  Lie algebra $\g$ and $h \in \g$. 
  Let $(U,\cH)$ be a unitary 
  representation of $G$ with discrete kernel. 
  Suppose that $\sV\subeq \cH$ is a standard subspace
    and $N \subeq G$ an identity neighborhood such that 
  \begin{description}
  \item[\rm(a)] $U(\exp(t h)) = \Delta_\sV^{-it/2\pi}$ for $t \in \R$,
    i.e., $\Delta_\sV = e^{2\pi i \, \partial U(h)}$, and 
  \item[\rm(b)] $\sV_N := \bigcap_{g \in N} U(g)\sV$ is cyclic.
  \end{description}
  Then $h$ is an Euler element or central, and the conjugation $J_\sV$ satisfies
  \begin{equation}
    \label{eq:J-rel}
 J_\sV U(\exp x) J_\sV = U(\exp \tau_h(x)) \quad \mbox{ for } \quad
 \tau_h = e^{\pi i \ad h}, x \in \g.
  \end{equation}
\end{theorem}

\begin{corollary} If $\sH(\cO)_{\cO \subeq M}$ is a
  net of real subspaces on open subsets of $M$
  satisfying {\rm(Iso), (Cov), (RS)} and {\rm(BW)}, 
  and $U$ has discrete kernel, then $h \in \g$ is an Euler element
  or central.
\end{corollary}

\begin{proof} Let $\cO \subeq W$ be a non-empty open, relatively compact subset.
  Then $\oline\cO$ is a compact subset of the open set $W$, so that
  \[ N := \{ g \in G \: g^{-1}.\oline\cO \subeq  W \} \]
  is an open $e$-neighborhood in $G$. For every $g \in N$ we
  have by (Cov) and (Iso), 
  \[  g^{-1}.\sH(\cO) = \sH(g^{-1}.\cO) \subeq
    \sH(W)\ {\buildrel {\rm(BW)}\over =}\ \sV.\]
  This implies that $\sH(\cO) \subeq \sV_N$,
  and (RS) entails that $\sH(\cO)$ is cyclic.
  Now the assertion follows from Theorem~\ref{thm:2.1}.    
\end{proof}

\begin{remark} \label{rem:antiuni-extend}   
(a) The relation \eqref{eq:J-rel} implies that, for the representations
we are dealing with, we may replace $G$
by its simply connected covering group
$\tilde G$ or by the quotient group $G/\ker(U)$ to ensure that the involution
$\tau_h^\g = e^{\pi i \ad h}$ on $\g$ integrates to an involution
$\tau_h$ on $G$, so that we can form the semidirect product
\[ G_{\tau_h} = G \rtimes \{\id_G, \tau_h \}. \]
Then \eqref{eq:J-rel} ensures that
$U$ extends to an antiunitary representation of $G_{\tau_h}$ by
$U(\tau_h) :=~J.$

\nin (b) If $\sV_N = \sV$ holds in the Euler Element Theorem,
then $U(g)\sV \supeq \sV$ for all $g \in N$, hence 
$U(g)\sV = \sV$ for all $g \in N \cap N^{-1}$. If $G$ is connected, this
implies that $U(G)$ fixes $\sV$ and hence that $h$ is central in $\g$.

\nin (c) Suppose that $h$ is central in $\g$ and that
$\Delta_\sV = e^{2\pi i \cdot \partial U(h)}$. Then all standard subspaces
$U_g \sV$ have the same modular group.
If there exists an $e$-neighborhood $N \subeq G$ for which
$\sV_N$ is cyclic, then we may assume that $N =N^{-1}$, and
this subspace is invariant under
the modular group $\Delta_\sV^{i\R} = U(\exp(\R h))$.
The Equality Lemma~\ref{lem:lo08-3.10}  thus shows that
$\sV_N = \sV$. Hence $U_g \sV \supeq \sV$ for all $g \in N$,
and the symmetry of $N$ now implies that $U_g\sV = \sV$, so that
$U_g$ commutes with $J_\sV$. If $G$ is connected, hence generated by $N$,
it follows that $J_\sV \in U(G)'$.
\end{remark}

 \begin{problem}
  In view of the preceding discussion, the following
  question is fundamental: Suppose that 
  $h \in \g$ is an Euler element,  $G$ is a corresponding
  connected Lie group, for which $G_{\tau_h}$ exists,
  and $M = G/H$ a homogeneous space. When does there exist
  an antiunitary   representation $(U,\cH)$ of~$G_{\tau_h}$,   
  a connected open subset $W \subeq M$
  and a net $\sH(\cO)_{\cO \subeq M}$
  on open subsets of $M$, satisfying
  (Iso), (Cov), (RS) and (BW)? 
\end{problem}

Below we shall see that this is  the case if $G$ is
reductive and $M$ is the non-compactly causal symmetric space
associated to $G$ and $h$ (cf.~Theorem~\ref{thm:4.9-semisimp}).
If $G$ is solvable, the corresponding question is open (cf.~\cite{BN25}).
In this context it is also interesting to consider
Theorem~\ref{thm:reg-net} that connects the existence of nets to a
regularity condition.

\subsection{First examples of Euler elements}

Before we descend deeper into structures related to Euler elements,
let us discuss some key examples.

\begin{examplekh} \label{ex:3.10}
 If $E$ is a finite-dimensional vector space
  and $0 \not=D \in \End(E)$ a diagonal endomorphism with eigenvalues
  contained in $\{1,0,-1\}$, then we form the solvable
  Lie algebra $\g := E \rtimes_D \R$. Here $h := (0,1)$ is an
  Euler element of $\g$.
\end{examplekh} 

  \begin{examples} \label{ex:3.10b}
(a) In $\g= \fsl_2(\R)$ the diagonal matrix 
\begin{equation}
  \label{eq:euler-h-sl2}
  h = \frac{1}{2}\pmat{1 & 0 \\ 0 & -1}
\end{equation}
is an Euler element.
  Conversely, every Euler element $h' \in \fsl_2(\R)$
  must be diagonalizable on $\R^2$ (Exercise~\ref{exer:diag}) 
  and the difference between its eigenvalues
  must be~$1$. In view of $\tr(h') = 0$, it is conjugate to~$h$.
  The set of Euler elements in $\fsl_2(\R)$ is
  \[\cE(\fsl_2(\R))
    = \big\{ x \in \fsl_2(\R) \: \det(x) = -{\textstyle\frac{1}{4}}\big\} = 
    \Big\{ \pmat{a & b \\ c & -a} \: a^2 + bc = \frac{1}{4}\Big\}\]
  and $\Inn(\g) \cong \SO_{1,2}(\R)_e$ acts transitively on this set. 
  In the following, we shall also use the Euler element 
  \begin{equation}
    \label{eq:def-sl2-k}
    k = \frac{1}{2} \pmat{ 0 & 1 \\ 1 & 0}.
  \end{equation}
The element 
\[ z_\fk := \frac{1}{2} \pmat{ 0 & 1 \\ -1 & 0} = [h, k]
  \quad \mbox{ satisfies  } \quad
  [z_\fk, h] = -k, \] 
    so that we have 
    \begin{equation}
      \label{eq:90deg-rot}
      e^{-\frac{\pi}{2} \ad z_\fk}h= - [z_\fk,h]  = k
      \quad \mbox{ and } \quad e^{\pm\pi\ad z_\fk}h= -h.
\end{equation}

\nin (b) If $\cA$ is a real unital associative algebra, then 
$h = \shalf\diag(1,-1)$ is also Euler in the Lie algebra 
$\gl_2(\cA)$. If $\sigma \in \Aut(\cA)$ is an involutive automorphism,
then $\sigma$ extends to a Lie algebra automorphism of $\gl_2(\cA)$
and $\g = \gl_2(\cA)^\sigma$ contains the Euler element $h$ with
$\g_1(h) \cong \cA^\sigma$. For the involution $\tau := \sigma \tau_h$, we
also find a Lie algebra with $\g_1(h) \cong \cA^{-\sigma}$.

This construction
provides a rich supply of Lie algebras with Euler elements. 
It even works for Jordan algebras~$\cA$
(cf.\ \cite{JvNW34}), hence in particular also for alternative algebras.
We refer to \cite{KSTT19}, \cite{dG17} and \cite{Be25}
for recent classification results in small dimensions. 
  \end{examples}

  \begin{examples} \label{ex:3.10cd}
(a) In the simple Lie algebra $\g := \fsl_n(\R)$, we write 
$n \times n$-matrices as block $2 \times 2$-matrices 
according to the partition $n = k + (n-k)$. Then 
\begin{equation}
  \label{eq:hk-sln}
  h_k := \frac{1}{n}\pmat{(n-k) \bone_k & 0 \\ 0 & -k\bone_{n-k}}
\end{equation}
is diagonalizable with the two eigenvalues $\frac{n-k}{n} = 1 - \frac{k}{n}$ 
and $-\frac{k}{n}$. 
Therefore $h_k$ is an Euler element (Exercise~\ref{exer:diag})
whose $3$-grading is given by 
\begin{align*}
\g_0(h) &= \Big\{ \pmat{a & 0 \\ 0 & d} \: 
a \in \gl_k(\R), d \in \gl_{n-k}(\R), \tr(a) + \tr(d) = 0\Big\},  \\ 
\g_1(h) &= \pmat{ 0 & M_{k,n-k}(\R)\\ 0 & 0}, \quad 
\g_{-1}(h) \cong \pmat{0 & 0 \\ M_{n-k,k}(\R) & 0}.
\end{align*}

It is easy to see that $h_1, \ldots, h_{n-1}$ represent all 
conjugacy classes of Euler elements in $\fsl_n(\R)$, whose restricted
root system is of type $A_{n-1}$, cf.~the general
Classification Theorem~\ref{thm:classif-symeuler} below.

The Euler element $h_k$ is symmetric, i.e., $-h_k \in \Inn(\g)h_k$,
if and only if $n = 2k$. In fact,
if $h_k$ is symmetric, then its eigenvalues have to be symmetric,
which is equivalent to $n = 2k$. That this condition is
sufficient follows by
embedding $h_k$ into an $\fsl_2(\R)$-subalgebra of block matrices with
entries in $M_k(\R)$ and using Example~\ref{ex:3.10b}. 

\nin (b) In the reductive Lie algebra $\gl_n(\R)$, we infer from
(a) that all conjugacy classes of Euler elements are represented by
elements of the form
\[ h  =\lambda \bone + h_k, \quad k = 1,\ldots, n-1.\]
Then $h$ is symmetric if and only if $\lambda = 0$ and $n = 2k$. 

These elements are also Euler in the semidirect sum
$\g := \R^n\rtimes \gl_n(\R)$ if and only if
$\lambda = \frac{k}{n}$ or $\lambda = \frac{k}{n}-1$,
which leads to the two Euler elements 
\[  h' = \pmat{\bone_k & 0 \\ 0 & 0} \quad \mbox{ and } \quad 
 h'' = \pmat{0 & 0 \\ 0 & -\bone_{n-k}}.\] 
In the first case,
\[ \g_1 \cong \R^k \oplus M_{k,n-k}(\R),\ \ 
 \g_0 = \gl_k(\R) \oplus (\R^{n-k} \rtimes \gl_{n-k}(\R))\ \ 
\mbox{ and } \ \  \g_{-1} =  M_{n-k,k}(\R), \]
whereas in the second case
\[ \g_1 \cong M_{k,n-k}(\R),\ \ 
 \g_0 = (\R^k \rtimes \gl_k(\R)) \oplus \gl_{n-k}(\R)\ \ 
 \mbox{ and } \ \ \g_{-1} =  \R^{n-k} \oplus M_{n-k,k}(\R). \]
None of these Euler elements is symmetric because
$\lambda \not=0$. 
\end{examples}
 
\begin{examples} \label{ex:poincare-1}
  (a)  In the Poincar\'e Lie algebra
  $\g = \R^{1,d} \rtimes \so_{1,d}(\R)$, every
  Euler element $h$ is conjugate
  to the generator $h \in \so_{1,d}(\R)$ of a Lorentz boost:
  \[ h.\be_0 = \be_1, \quad h.\be_1 = \be_0 \quad \mbox{ and } \quad
    h.\be_j = 0\quad \mbox{ for } \quad j > 1\]
 and $\fz(\g)= \{0\}$;
  see also Remark~\ref{rem:poin}.
  In fact,   Lemma~\ref{lem:levi-euler} below 
  reduces the assertion to the simple Lie algebra
  $\so_{1,d}(\R)$, and in this case the assertion
  follows from the Classification Theorem~\ref{thm:classif-symeuler}.

  \nin (b) (cf.\ \cite{BN25})  The $4$-dimensional
  split oscillator group is
  \[ G := \Heis(\R^2) \rtimes_\alpha \R
    \quad \mbox{ with } \quad
    \alpha_t = e^{th}, \quad h = \pmat{1 & 0 \\ 0 & -1},\]
so that
\[ \g = \heis(\R^2) \rtimes \R h\]
and $h$ is an Euler element in $\g$.
We choose a basis $p,q,z \in \heis(\R^2)$
with
\[ [q,p] = z, \quad [h,q] = 1, \quad [h,p] = -1, \quad [h,z] = 0.\]
The corresponding involution satisfies 
\[ \tau_h(z,q,p,t) = (z,-q,-p,t).\]

The Euler element $h$ is not symmetric, and all Euler elements
of $\g$ are, up to sign, conjugate to elements of the form
\[ h_\lambda = \lambda z + h.\] 

This Lie algebra can be realized as a subalgebra of $\fsl_3(\R)$, where 
\[ h = \frac{1}{3}
  \pmat{1 & 0 & 0 \\ 0 & -2 & 0 \\ 0& 0 & 1}, \quad
q = \pmat{0 & 1 & 0 \\ 0 & 0 & 0 \\ 0& 0 & 0}, \quad
p = \pmat{0 & 0 & 0 \\ 0 & 0 & 1 \\ 0& 0 & 0}, \quad
z = \pmat{0 & 0 & 1 \\ 0 & 0 & 0 \\ 0& 0 & 0}.\]
Note that $h$ is an Euler element of $\fsl_3(\R)$, 
i.e., $V := \R^3$ is a $2$-graded $\g$-module
(cf.\ Example~\ref{ex:3.10cd}(a)):
\[ V = V_{1/3} \oplus V_{-2/3}, \quad
 V_{1/3} = \R \be_1 + \R \be_3, \quad 
 V_{-2/3} = \R \be_2.\] 
\end{examples}

  \begin{remark}
  \nin (a) If $V$ is a non-trivial irreducible $\fsl_2(\R)$-module
  and
  \[ h \in \g := V \rtimes \fsl_2(\R) \]  an Euler element, then
the semisimple element $h$ is conjugate to an element of $\fsl_2(\R)$
(Lemma~\ref{lem:levi-euler} and $\fz(\g) = \{0\}$), 
so that we may assume that $h =  \shalf \diag(1,-1)$
(Example~\ref{ex:3.10b}(a)). 
This leaves only the possibility that $\dim V = 3$ is the adjoint module.

We obtain more freedom if we replace $\fsl_2(\R)$ by $\gl_2(\R)$.
Then Example~\ref{ex:3.10cd}(b)
also provides Euler elements in $\R^2 \rtimes \gl_2(\R)$
(cf.\ \cite[Lemma~2.15]{MN22} or Lemma~\ref{lem:slgl}).
We may also consider non-trivial $1$-dimensional representations of
$\gl_2(\R)$, for which
\[ \R \rtimes \gl_2(\R) \cong (\R \rtimes \R\bone) \oplus \fsl_2(\R).\] 
This example shows already how restrictive
the existence of a $3$-grading is for semidirect sums.

\nin (b) If $h \in \g$ is an Euler element contained in
a subalgebra $\fs \cong \fsl_2(\R)$, then all simple $\fs$-submodules
of $\g$ must be $1$ or $3$-dimensional.
If $h$ is contained in a subalgebra $\fl \cong \gl_2(\R)$,
then also $2$-dimensional irreducible submodules may occur
(cf.\ Lemma~\ref{lem:slgl} below). 
  \end{remark}

\subsection{Euler elements in simple Lie algebras}
\label{app:eul-simp}

We now present a classification
of Euler elements in simple real Lie algebras,
following \cite{MN21} (see also \cite{Mo25}). 
As they correspond to $3$-gradings,
it can also be derived from \cite{KA88}.
We also reproduce the list of the $18$ types 
from \cite[p.~600]{Kan98} and Kaneyuki's lecture notes \cite{Kan00}. 

\index{Cartan!involution \scheiding} 
Let $\g$ be a real semisimple Lie algebra. 
An involutive automorphism $\theta \in \Aut(\g)$ is called a {\it 
Cartan involution} if its eigenspaces 
\[ \fk := \g^\theta = \{ x \in\g \: \theta(x) = x \} \quad \mbox{ and }\quad 
\fp := \g^{-\theta} = \{ x \in\g \: \theta(x) = -x \} \] 
have the property that they are orthogonal with respect to the
Cartan--Killing form
\[ \kappa(x,y) = \tr(\ad x \ad y),\] 
which is  negative definite on $\fk$ and 
positive definite on $\fp$. This implies that
\begin{equation}
  \label{eq:kappa-pos}
  \kappa(x,\theta x) < 0 \quad \mbox{ for } \quad x \not=0.
\end{equation}
Then 
\begin{equation}
  \label{eq:cartandec} 
\g = \fk \oplus \fp 
\end{equation}
is called a {\it Cartan decomposition}. \index{Cartan!decomposition\scheiding} 
Cartan involutions always exist and two such involutions are conjugate 
under the group $\Inn(\g)$ of inner automorphism, so they produce 
isomorphic decompositions (\cite[Thm.~13.2.11]{HN12}). 
The subalgebra $\fk$ is maximal {\it compactly embedded}.
\begin{footnote}{We call a subalgebra $\fc\subeq \g$
    compactly embedded if the closure of the subgroup generated by
    $e^{\ad \fc}$ in $\Aut(\g)$ is compact.}  
\end{footnote}
An element $x \in \g$ is elliptic 
if and only if its adjoint orbit $\cO_x =\Inn(\g)x$ intersects $\fk$, 
and $x \in \g$ is hyperbolic if and only if 
$\cO_x \cap \fp \not=\eset$. 

For the finer structure theory, we start with a Cartan involution $\theta$ and 
fix a maximal abelian subspace $\fa \subeq \fp$. As $\fa$ is abelian, 
$\ad \fa$ is a commuting set of diagonalizable operators, hence 
simultaneously diagonalizable. 
For a linear functional $0 \not=\alpha \in \fa^*$, the simultaneous eigenspaces 
\[ \g_\alpha := \g_\alpha(\fa) := \{ y \in \g \: (\forall x \in \fa) \ [x,y] = \alpha(x)y\} \] 
are called {\it root spaces} and
\index{root system!restricted $\Sigma(\g,\fa)$ \scheiding }
\[ \Sigma := \Sigma(\g,\fa) := \{ \alpha \in \fa^* \setminus \{0\}  \: 
\g_\alpha \not=0\} \] 
is called the set of {\it restricted roots}. 
We pick a set 
\[ \Pi := \{ \alpha_1, \ldots, \alpha_n \} \subeq \Sigma \] 
of {\it simple roots}. This is a subset with the property that every 
root $\alpha \in \Sigma$ is a linear combination 
$\alpha = \sum_{j =1}^n n_j \alpha_j$, where the coefficients 
are either all in $\Z_{\geq 0}$ or in $\Z_{\leq 0}$. The convex cone 
\[ \Pi^\star := \{ x \in \fa \: (\forall \alpha \in \Pi) \ \alpha(x) \geq 0\} \] 
is called the {\it closed positive (Weyl) chamber corresponding to $\Pi$}. 
We have the {\it root space decomposition} 
\[ \g = \g_0 \oplus \bigoplus_{\alpha \in \Sigma} \g_\alpha 
\quad \mbox{ and }\quad 
\g_0 = \fm \oplus \fa, \quad \mbox{ where } \quad 
\fm = \g_0 \cap \fk.\] 
Now $\theta(\g_\alpha) = \g_{-\alpha}$, and for 
a non-zero element $x_\alpha \in \g_\alpha$, the 
$3$-dimensional subspace spanned by $x_\alpha, \theta(x_\alpha)$ and 
$[x_\alpha, \theta(x_\alpha)] \in \g_0 \cap\fp =  \fa$ is a Lie subalgebra 
isomorphic to $\fsl_2(\R)$. In particular, it contains 
a unique element $\alpha^\vee \in \fa$
with $\alpha(\alpha^\vee) = 2$.\begin{footnote}{
    To see this, we observe that $0 \not= x \in \g_\alpha$ implies
    $\theta(x) \in \g_{-\alpha}$, so that
    $[x,\theta(x)] \in \g_0 \cap \fp = \fa$.
    As $\kappa$ is positive definite on $\fa$, there exists a unique
    $h_\alpha \in \fa$ with $\kappa(\cdot,h_\alpha) = \alpha$.
    For $h \in \fa$, we then have
    \[ \kappa(h,[x,\theta(x)]) =  \kappa([h,x], \theta(x))
      = \alpha(h) \kappa(x,\theta(x))
      = \kappa(h, \kappa(x,\theta(x) h_\alpha),\]
    so that we must have $[x,\theta(x)] = \kappa(x,\theta(x)) h_\alpha$.
    As $\kappa(x,\theta(x)) < 0$ and
    $\alpha(h_\alpha) = \kappa(h_\alpha, h_\alpha) >~0$, we obtain 
 ${\alpha([x,\theta(x)]) < 0}$. 
}\end{footnote}
Then 
\[ r_\alpha \: \fa \to \fa, \quad r_\alpha(x) := x - \alpha(x) \alpha^\vee \] 
is a reflection in the hyperplane $\ker \alpha$, and the subgroup 
\[ \cW := \la r_\alpha \: \alpha \in \Sigma \ra \subeq \GL(\fa) \]
is called the {\it Weyl group}. 
\index{Weyl group $\cW$ !of root system\scheiding} 
Its action on $\fa$ provides a good description of the adjoint 
orbits of hyperbolic elements: Every hyperbolic 
element in $\g$ is conjugate to a unique element in the Weyl chamber
$\Pi^\star \subeq \fa$,
a fundamental domain for the $G$-action
on the subset of hyperbolic elements in $\g$
and a fundamental domain for the $\cW$-action on~$\fa$.
For $x \in \fa$, the intersection  
$\cO_x \cap \fa = \cW x$ is the Weyl group orbit 
(\cite[Thm.~III.10]{KN96}). 

{\bf From now on we assume that $\g$ is simple.} 
Then $\Sigma$ is an irreducible root system, hence of one of the following types: 
\begin{equation}
  \label{eq:dunkin}
 A_n,  \quad  
 B_n, \quad  
 C_n,  \quad  
 D_n, \quad  
 E_6, E_7, E_8,\quad  F_4,\quad  G_2
 \quad \mbox{ or } \quad BC_n, n \geq 1
\end{equation}
(cf.\ \cite{Bou90}).

The adjoint orbit of an Euler element in $\g$ contains  
a unique $h \in \Pi^\star$. For any Euler element 
$h \in \Pi^\star$, we have $\alpha(h) \in \{0,1\}$ for $\alpha \in \Pi$ 
because the values of the roots on $h$ are the eigenvalues of $\ad h$. 
If such an element exists, then the irreducible root system 
$\Sigma$ must be reduced. Otherwise, for any root $\alpha$ with 
$2\alpha \in \Sigma$, we must have $\alpha(h) = 0$ because 
$\ad x$ has only three eigenvalues. As the set of such roots 
generates the same linear space as $\Sigma$, this leads to~$h = 0$. 
This excludes the non-reduced irreducible root systems of type~$BC_n$.

To see how many possibilities we have for Euler elements in $\fa$, 
we recall that $\Pi$ is a linear basis of $\fa^*$, so that, for each $j \in \{1,\ldots, n\}$, there exists a uniquely determined element 
\begin{equation}
  \label{eq:hj}
  h_j \in \fa, \quad \mbox{  satisfying  } \quad \alpha_k(h_j) = \delta_{kj}. 
\end{equation}

The following theorem lists for each irreducible root system $\Sigma$ 
the possible Euler elements in the positive chamber $\Pi^\star$. 
Since every adjoint orbit in $\cE(\g)$ has a unique 
representative in $\Pi^\star$, this classifies the 
$\Inn(\g)$-orbits in $\cE(\g)$ for any non-compact simple real Lie algebra. 
For {\bf semisimple} Lie algebras $\g = \g_1 \oplus \cdots \oplus \g_k$, an 
element $x = (x_1, \ldots, x_k)$ is an Euler element if and only if its 
non-zero components $x_j \in \g_j$ are Euler elements, and its orbit is 
\[ \cO_x = \cO_{x_1} \times \cdots \times \cO_{x_k}.\] 
Therefore it suffices to consider simple Lie algebras, 
and for these the root system~$\Sigma$ is irreducible. 
As every complex simple Lie algebra $\g$ 
is also a real simple Lie algebra, our discussion also 
covers  complex Lie algebras.

\begin{theorem} \label{thm:classif-symeuler}
Suppose that $\g$ is a non-compact simple 
real Lie algebra, with restricted root system 
$\Sigma \subeq \fa^*$ of type $X_n$ as in \eqref{eq:dunkin}. 
We follow the conventions of the tables in {\rm\cite{Bou90}}
for the classification of irreducible root systems and the enumeration 
of the simple roots $\alpha_1, \ldots, \alpha_n$. 
Then every Euler element $h \in \fa$ on which 
$\Pi$ is non-negative is one of  $h_1, \ldots, h_n$, and for 
every irreducible root system, the Euler elements among the $h_j$ are 
 the following: 
\begin{align} 
&A_n: h_1, \ldots, h_n, \qquad 
\ \ B_n: h_1, \qquad 
\ \ C_n: h_n, \qquad \ \ \ D_n: h_1, h_{n-1}, h_n, \notag\\ 
&E_6: h_1, h_6, \qquad 
E_7: h_7.\label{eq:eulelts2}
\end{align}
For the root systems $BC_n$, $E_8$, $F_4$ and $G_2$ no Euler element exists 
(they have no $3$-grading). 
The Euler elements which are symmetric in the sense that
$-h \in \cO_h = \Inn(\g)h$, are 
\begin{align} 
  \label{eq:symmeuler}
& A_{2n-1}: h_n, \qquad 
B_n: h_1, \qquad C_n: h_n, \qquad 
D_{2n}: h_1, h_{2n-1},h_{2n}, \qquad 
D_{2n+1}: h_1, \notag \\ 
& E_7: h_7.  
\end{align}
\end{theorem}

\begin{proof} 
Writing the highest root in $\Sigma$ with respect to the simple system 
$\Pi$  as $\alpha_{\rm max} = \sum_{j = 1}^n c_j \alpha_j$, 
we have $c_j \in \Z_{>0}$ for each $j$. 
If $h \in \Pi^\star$ is an Euler element, then $\Pi(h) \subeq \{0,1\}$, 
and $1 = \alpha_{\rm max}(h) = \sum_{j = 1}^n c_j \alpha_j(h)$ implies
that at most one value $\alpha_j(h)$ can be $1$,  and then the others are~$0$, 
i.e., $h = h_j$ for some $j \in \{1,\ldots, n\}$. 
Conversely, $h_j$ is an Euler element if and only if $c_j = 1$. 
Consulting the tables on the irreducible root systems 
in \cite{Bou90}, we obtain the Euler elements listed in 
\eqref{eq:eulelts2}. 

To determine the symmetric ones, let $w_0 \in \cW$ be 
the element of the Weyl group, which is uniquely determined by 
$w_0^*\Pi = - \Pi$ for the dual action of $\cW$ on $\fa^*$. 
Then $h_j' := w_0(-h_j)$ is the Euler element in 
the positive chamber representing the orbit $\cO_{-h_j}$. 
Therefore $h_j$ is symmetric if and only if 
$-h_j \in \cW h_j$, which is equivalent to $h_j' = h_j$. 
Using the description of $w_0$ and the root systems in \cite{Bou90},  
now leads to 
\begin{align}
& A_{n-1}: h_j' = h_{n-j}, \quad 
B_n: h_1' = h_1, \quad C_n: h_n' = h_n, \\ 
& D_n: h_1' = h_1, h_n' =
\begin{cases}
  h_{n-1} & \text{ for } n \ \text{ odd},\\ 
  h_n & \text{for } n \ \text{ even}, 
\end{cases}\\
&E_6: h_1' = h_6, \quad E_7: h_7' = h_7. 
\end{align}
Hence the symmetric Euler elements are those
listed in~\eqref{eq:symmeuler}.
\end{proof}

There are many types of simple $3$-graded Lie algebras that are neither 
complex nor hermitian (cf.\ Proposition~\ref{prop:herm} below);
for instance the Lorentzian algebras 
$\so_{1,n}(\R)$. We refer to 
\cite[p.~600]{Kan98} or \cite{Kan00}. 
for the list of all $18$ types which is reproduced below
in a different order. We identify $\so^*(4n)$ with the Lie algebra
$\fu_{2r}(\bH,\Omega)$ of the isometry
group of the non-degenerate skew-hermitian form on~$\H^{2r}$
defined by the matrix $\Omega = \pmat{ 0 & \bone \\ -\bone & 0}$. 

\hspace{-4mm}
\begin{tabular}{||l|l|l|l|l||}\hline
& $\g$  & $\Sigma(\g,\fa)$  & $h$ & $\g_1(h)$  \\ 
\hline
& \text{Complex Lie algebras}&&& \\  \hline
1 & $\fsl_n(\C)$ & $A_{n-1}$ & $h_j, 1 \leq j \leq n-1$ & $M_{j,n-j}(\C)$  \\ 
2 & $\sp_{2n}(\C)$ & $C_{n}$ & $h_n$ & $\Sym_n(\C)$   \\ 
3a & $\so_{2n+1}(\C)$ & $ B_{n}$ & $h_1$ & $\C^{2n-1}$   \\ 
3b & $\so_{2n}(\C)$ & $ D_{n}$ & $h_1$ & $\C^{2n-2}$   \\ 
4 & $\so_{2n}(\C)$ & $ D_{n}$ & $h_{n-1}, h_n$ & $\Alt_n(\C)$   \\ 
5 & $\fe_6(\C)$ & $E_6$ & $h_1 = h_6' $ & $M_{1,2}(\bO)_\C$   \\ 
6 & $\fe_7(\C)$ & $E_7$ & $h_7 $ & $\Herm_3(\bO)_\C$   \\ 
\hline
& \text{Hermitian tube type Lie algebras}&&& \\  \hline
7 & $\su_{n,n}(\C)$ & $C_{n}$ & $h_n$ & $\Herm_n(\C)$  \\ 
8 & $\sp_{2n}(\R)$ & $C_{n}$ & $h_n$ & $\Sym_n(\R)$   \\ 
9a  & $\so_{2,d}(\R)$ & $ C_2\ (2<d)$ & $h_1$ & $\R^{1,d-1}$   \\ 
10  & $\so^*(4n) \cong \fu_{2r}(\H,\Omega)$ & $C_n$ & $h_n$ & $\Herm_n(\H)$   \\ 
  11 & $\fe_{7(-25)}$ & $C_3$ & $h_3$ & $\Herm_3(\bO)$   \\
  \hline
& \text{Non-hermitian split forms} &&& \\  \hline
  12 & $\fsl_n(\R)$ & $A_{n-1}$ & $h_j, 1 \leq j \leq n-1$ & $M_{j,n-j}(\R)$  \\ 
9b  & $\so_{n,n+1}(\R)$ & $B_n$ & $h_1$ & $\R^{2n-1}$   \\ 
13  & $\so_{n,n}(\R)$ & $D_n$ & $h_{n-1}, h_n$ & $\Alt_n(\R)$   \\ 
14 & $\fe_6(\R)$ & $E_6$ & $h_1=h_6' $ & $M_{1,2}(\bO_{\rm split})$   \\ 
15 & $\fe_7(\R)$ & $E_7$ & $h_7 $ & $\Herm_3(\bO_{\rm split})$   \\ 
  \hline
& \text{Non-hermitian non-split forms} &&& \\  \hline
  16 & $\fsl_n(\H)$ & $A_{n-1}$ & $h_j, 1 \leq j \leq n-1$ & $M_{j,n-j}(\H)$  \\ 
17 & $\fu_{n,n}(\H)$ & $C_{n}$ & $h_n$ & $\Aherm_n(\H)$  \\ 
9c  & $\so_{p,q}(\R), 2\not=p \not=q-1$ & $ B_p\ (p<q)$ & $h_1$ & $\R^{p+q-2}$   \\ 
  & & $D_p\ (p = q)$ &  &    \\ 
18 & $\fe_{6(-26)}$ & $A_2$ & $h_1$ & $M_{1,2}(\bO)$   \\ 
  \hline\hline
\end{tabular} \\[2mm] {\rm Table 1: Simple $3$-graded Lie algebras}\\

In our context, hermitian simple Lie algebras are of particular interest.
Recall that a simple real Lie algebra $\g$ is called {\it hermitian}
if the corresponding Riemannian symmetric space $G/K$
is a complex bounded symmetric domain, which is equivalent
to $\fz(\fk) \not=\{0\}$. It is said to be {\it of tube type}
if $G/K \cong \R^n + i \Omega$ for an open convex cone $\Omega \subeq \R^n$.
We collect some properties of hermitian Lie algebras 
in the following  proposition.

\begin{proposition} \label{prop:herm} 
For a simple real Lie algebra, the following assertions hold: 
  \begin{description}
  \item[\rm(a)] $\g$ is hermitian if and only if it contains a
    non-trivial closed convex $\Inn(\g)$-invariant cone $C_\g$. 
  \item[\rm(b)] A simple hermitian Lie algebra contains an Euler 
element if and only if it is of tube type, and in this case $\Inn(\g)$ 
acts transitively on $\cE(\g)$. 
  \end{description}
\end{proposition}

\begin{proof} (a) is a consequence of the Kostant--Vinberg Theorem 
(cf.\ \cite[Lem.~2.5.1]{HO97}). 

\nin (b) Since the restricted root system of a hermitian simple Lie algebra 
is of type $C_r$ or $BC_r$ (see \cite[\S 3.1]{MNO23} or
Table 2 below), and the first case characterizes the algebras of tube 
type, the assertion follows from 
Theorem~\ref{thm:classif-symeuler} because the root system 
$C_r$ only permits one class of Euler elements. 
\end{proof}

\begin{remark} As $h \in \fa$ implies $\theta(h) = -h$, the Cartan 
involution $\theta$ always maps $h$ into $-h$, but this only 
implies that $h$ is symmetric if $\theta \in \Inn(\g)$.
This is the case if $\g$ is hermitian, so that in these Lie algebras 
all Euler elements are symmetric
(cf.\ Proposition~\ref{prop:herm}).
\end{remark}

The classification of Euler elements requires
some interpretation. So let us first 
see what it says about complex simple Lie algebras~$\g$. In \eqref{eq:eulelts2} 
we see that, only if $\g$ is not of type $E_8, F_4$ or $G_2$, the Lie algebra $\g$ 
contains an Euler element. Euler elements correspond to $3$-gradings 
of the root system and these in turn to hermitian real forms $\g^\circ$, 
where $ih_j \in \fz(\fk^\circ)$ generates the center of a maximal 
compactly embedded subalgebra $\fk^\circ$ 
(\cite[Thm.~A.V.1]{Ne99}). We thus obtain the 
following possibilities. In Table 2, we write 
$\g^\circ$ for the hermitian real form, $\g$ for the complex Lie algebra, 
$\Sigma$~for its restricted root system, and $h_j$ for the corresponding Euler element. \\[5mm]
\begin{tabular}{||l|l|l|l|l||}\hline
{} $\g^\circ$ \mbox{(hermitian)} & $\Sigma(\g^\circ, \fa^\circ)$  & $\g = (\g^\circ)_\C$ & $\Sigma(\g,\fa)$ & 
{\rm Euler element} \\ 
\hline\hline 
$\su_{p,q}(\C), 1 \leq p\leq q$ & $BC_p (p < q)$, $C_p (p=q)$ & $\fsl_{p+q}(\C)$ & $A_{p+q-1}$ & 
$h_p$ \\ 
 $\so_{2,d}(\R), d > 2$ & $C_2$ & $\so_{2+d}(\C)$ & $B_{\frac{d+1}{2}}$, $d$ odd & $h_1$ \\ 
 &  & & $D_{1 + \frac{d}{2}}$, $d$ even &  \\ 
$\sp_{2n}(\R)$ & $C_n$ & $\fsp_{2n}(\C)$ & $C_n$ & $h_n$ \\ 
 $\so^*(2n)$ & $BC_{\lfloor\frac{n}{2}\rfloor} (n$ odd), $C_{\frac{n}{2}} (n$ even) & $\so_{2n}(\C)$  & $D_{n}$ 
& $h_{n-1}, h_n$ \\ 
$\fe_{6(-14)}$ & $BC_2$ & $\fe_6$ & $E_6$ & $h_1 = h_6'$ \\ 
$\fe_{7(-25)}$ & $C_3$ & $\fe_7$ & $E_7$ & $h_7$ \\ 
\hline
\end{tabular} \\[2mm] {\rm Table 2: Simple hermitian Lie algebras $\g^\circ$
($\g$ as in (1)-(6) in Table 1).}\\

Note that $\fsl_2(\R) \cong \so_{2,1}(\R) \cong \su_{1,1}(\C)$. More exceptional isomorphisms 
are discussed in some detail in \cite[\S 17]{HN12}.

In the correspondence of Euler elements in simplex complex Lie algebras
and their hermitian real forms, those real forms 
corresponding to symmetric Euler elements are of particular interest. 
Comparing with the list of hermitian simple Lie algebras 
of tube type (cf.~\cite[p.~213]{FK94}), we see that they 
correspond precisely to $3$-gradings specified by symmetric Euler elements, 
as listed in~\eqref{eq:symmeuler}. 
Since the Euler elements $h_{n-1}$ and $h_n$ for the root system of type 
$D_n$ are conjugate under a diagram automorphism, they correspond to 
isomorphic hermitian real forms. \\[2mm] 
\begin{tabular}{||l|l|l|l|l||}\hline
{} $\g^\circ$ \mbox{(hermitian)}  & $\Sigma(\g^\circ, \fa^\circ)$ & $\g = (\g^\circ)_\C$ & $\Sigma(\g,\fa)$ & {\rm symm.\ Euler element}\  $h$ \\ 
\hline\hline 
$\su_{n,n}(\C)$ & $C_n$ & $\fsl_{2n}(\C)$ & $A_{2n-1}$ & $h_n$ \\ 
  $\so_{2,d}(\R), d > 2$ \phantom{\ \ }& $C_2$ & $\so_{2+d}(\C)$
                                                                                         \phantom{\ }& $B_{\frac{d+1}{2}}$, $d$ odd & $h_1$ \\ 
 &  & & $D_{1 + \frac{d}{2}}$, $d$ even \phantom{\ \ }&  \\ 
$\sp_{2n}(\R)$ & $C_n$ & $\fsp_{2n}(\C)$ & $C_n$ & $h_n$ \\ 
 $\so^*(4n)$ & $C_n$ & $\so_{4n}(\C)$  & $D_{2n}$ & $h_{2n-1},  h_{2n}$ \\ 
$\fe_{7(-25)}$ & $C_3$ & $\fe_7$ & $E_7$ & $h_7$ \\ 
\hline
\end{tabular} \\[2mm] {\rm Table 3: Simple hermitian Lie algebras $\g^\circ$ 
of tube type ((7)-(11) in Table 1)}

\subsection{Conjugacy classes of Euler elements in general Lie algebras}

  To analyze Euler elements in general
  Lie algebra, it is instructive to consider abelian
  subalgebras $\fa \subeq \g$ which are maximal with respect to the
  property that $\ad \fa$ is diagonalizable. It follows from
  \cite[Thm.~III.3]{KN96}, applied to the symmetric Lie algebra
  $(\g^{\oplus 2}, \tau_{\rm flip})$, that they are conjugate
  under $\Inn(\g)$. Moreover, there always exists an
  $\ad\fa$-invariant Levi complement $\fs$
  (\cite[Prop.~I.2]{KN96}), so that
  \[ \fa = \fa_\fr \oplus \fa_\fs \quad \mbox{ for } \quad
    \g = \fr \rtimes \fs, \quad
    \fa_\fr = \fa \cap \fr, \ \ \fa_\fs = \fa \cap \fs.\]
  Then $[\fa_\fr,\fs] \subeq \fr \cap \fs = \{0\}$.
  As $\g$ is a nilpotent module of the ideal $[\g,\fr]$,
  it further follows that
  \[ \fa_\fr \cap [\g,\g] = \fa \cap [\g,\fr] \subeq \fz(\g) \cap [\g,\g],\]
  so that 
  \begin{equation}
    \label{eq:fa-dag}
    \fa = \fz(\g) \oplus \fa_\fr^c \oplus \fa_\fs,
  \end{equation}
where $\fa_\fr^c \subeq \fa_\fr$ is a complement of $\fz(\g)$ in $\fa_\fr$.

\begin{lemma} \label{lem:levi-euler}
For an Euler element $h \in \g$, the following assertions hold:
  \begin{enumerate}
  \item[\rm(a)] $\cO_h$ intersects $\fa$, hence also
    $\fz(\g) + \fl$, where $\fl = \fa_\fr^c \oplus \fs$ is a
    reductive subalgebra of $\g$.
    Moreover, $\cO_h \cap \fa = \cW.h$, where
    $\cW := \cW(\fs,\fa_\fs)$ is the Weyl group of the restricted
    root system $\Sigma(\fs,\fa_\fs)$. 
\item[\rm(b)] If $h \in [\g,\g]$ is an Euler element contained
  in the commutator algebra, then $\cO_h + \fz(\g)$ 
  intersects every Levi complement. 
  \end{enumerate}
  \end{lemma}

  \begin{proof} (a)  That $\cO_h \cap \fa = \cW.h$ 
    follows from \cite[Thms.~III.3, III.10]{KN96}, applied to the symmetric
    Lie algebra $(\g^{\oplus 2}, \tau_{\rm flip})$. 
    The rest of (a) now follows  from \eqref{eq:fa-dag}
    and the fact that $[\fa_\fr,\fs] =\{0\}$.

    \nin (b) In view of (a), we may assume that $h \in \fa$.
    Then $h \in \fa \cap [\g,\g] \subeq \fz(\g) + \fa_\fs
    \subeq \fz(\g) +  \fs$ implies~(b). Now the assertion follows 
from the fact that $\cO_h + \fz(\g)$ is invariant under 
$\Inn(\g)$ and that any two Levi complements are conjugate under this group. 
\end{proof}

For refinements of the following proposition we refer to \cite[\S 2.1]{MNO25}. 

\begin{proposition}   \label{prop:MN21:3.2} {\rm(\cite[Prop.~3.2]{MN21})} 
The following assertions hold: 
\begin{description}
\item[\rm(i)] An Euler element $h \in \g$ is symmetric if and only 
if $h$ is contained in a Levi complement $\fs$ and 
$h$ is a symmetric Euler element in~$\fs$.
\item[\rm(ii)] Let $\g = \fr \rtimes \fs$ be a Levi decomposition. 
\begin{description}
\item[\rm(a)]  If $h \in \g$ is a symmetric Euler element, then 
$\cO_h = \cO_{q(h)}= \Inn(\g)(\cO_h \cap \fs) $, 
where $q \: \g \to \fs$ is the projection along $\fr$. 
\item[\rm(b)] Two symmetric Euler elements are conjugate under $\Inn(\g)$ 
if and only if their images in $\fs$ are conjugate under~$\Inn(\fs)$. 
\end{description}
\end{description}
\end{proposition}

\begin{proof} (i) As $\cO_h \subeq h + [\g,\g]$ follows from the 
invariance of the affine subspace $h + [\g,\g]$ under 
$\Inn(\g)$, the relation 
$-h \in\cO_h$ implies $h \in [\g,\g]$.
In view of Lemma~\ref{lem:levi-euler}(b),
there exists a Levi complement $\fs$ with
$h \in \fz(\g) + \fs$.
Then $\fr$ and $\fs$ are $\ad h$-invariant, so that the 
$\ad h$-eigenspaces of the restrictions satisfy 
\[ \fr = \fr_1(h) + \fr_0(h) + \fr_{-1}(h) \quad \mbox{ and } \quad 
\fs = \fs_1(h) + \fs_0(h) + \fs_{-1}(h), \] 
and define $3$-gradings of $\fr$ and~$\fs$.
Further $\g_{\pm 1}(h) \subeq [h,\g] \subeq [\g,\g]$ and 
$\fs = [\fs,\fs] \subeq [\g,\g]$ imply that 
$\g =  \fr_0(h) + [\g,\g]$. The fact that $[\g,\g]$ is an ideal and $\fr_0(h)$ 
is a subalgebra of $\g$ entails that
the subgroup $\Inn_\g([\g,\g])$ of $\Inn(\g)$ is normal, 
and that  $\Inn(\g) = \Inn_\g([\g,\g]) \Inn(\fr_0(h))$. 
As $\Inn(\fr_0(h))$ fixes $h$, this in turn shows that 
$\cO_h = \Inn_\g([\g,\g])h = \Inn_\g([\g,\fr]) \Inn_\g(\fs)h$. 
Writing $h = h_z + h_s$ with $h_z \in \fz(\g)$ and $h_s \in \cE(\fs)$, we thus 
find $x \in [\g,\fr]$ and $s \in \Inn_\g(\fs)$ such that
\begin{footnote}{Here we use that the ideal $[\g,\fr]$ is nilpotent, 
so that the exponential function of the corresponding group 
$\Inn_\g([\g,\fr])$ is surjective, see \cite[Cor.~11.2.7]{HN12}.}  
\end{footnote}
\begin{equation}
  \label{eq:hz}
  -h_z- h_s = -h = e^{\ad x} s.h = h_z + e^{\ad x} s.h_s.
\end{equation}
Applying the Lie algebra 
homomorphism $q \: \g \to \fs$ to both sides, we derive from 
$q(h_z) = 0$ and $q \circ e^{\ad x} = q$ that $-h_s = s.h_s$, and therefore
by \eqref{eq:hz} 
\[e^{\ad x} h_s = h_s + 2h_z.\] 
We conclude that the unipotent linear map $e^{\ad x}$ preserves 
the plane $\R h_s + \R h_z$, and this implies that 
$\ad x = \log(e^{\ad x})$ also has this property. 
We thus arrive at 
\[   [h,x] = [h_s, x] \subeq \R h_s + \R h_z \subeq \g_0(h),\] 
so that $x \in \g_0(h) = \g_0(h_s)$, which in turn leads to 
$0 = e^{\ad x}h_s  - h_s =~2h_z$, i.e., $h = h_s\in \fs$. 

To prove the second assertion of (i), we observe that the 
projection $q \: \g \to \fs \cong \g/\fr$ satisfies 
\begin{equation}
  \label{eq:orbproj}
q(\cO_x) = \cO^\fs_{q(x)} \quad \mbox{ for } \quad x \in \g.
\end{equation}
Writing $\cE_{\rm sym}(\g)$ for the set of symmetric
Euler elements in $\g$, we obtain
$q(\cE_{\rm sym}(\g)) \subeq \cE_{\rm sym}(\fs)$. 
If, conversely, $h \in \cE_{\rm sym}(\fs)$, then we clearly 
have $-h \in \Inn_\g(\fs)h \subeq \Inn(\g)h$, so that $h \in \cE_{\rm sym}(\g)$.

\nin (ii)(a) As $\cO_h$ intersects $\fs$ by (i), 
$q(\cO_h) \cap \cO_h \not=\eset$, and since $\Inn(\fs)$ acts transitively 
on $q(\cO_h)$ by \eqref{eq:orbproj}, we obtain $q(\cO_h) \subeq \cO_h$ 
and thus $q(\cO_h) = \cO_h \cap \fs$. This further leads to 
\[ \cO_h = \Inn(\g)(\cO_h \cap \fs) = \Inn(\g)q(\cO_h) 
= \Inn(\g) \cO^\fs_{q(h)} = \cO_{q(h)}.\] 

\nin (ii)(b) follows immediately from (a). 
\end{proof}

Proposition~\ref{prop:MN21:3.2} reduces, for a given Lie algebra $\g$, 
the description of symmetric Euler elements 
up to conjugation by inner automorphisms to the case of simple Lie algebras.

It would be nice to have a classification of Euler elements
in any Lie algebra~$\g$, but, due to the complexity of Levi decompositions
$\g = \fr \rtimes \fs$, this is not a well-posed problem. 
If $\g$ is reductive, then the classification of Euler elements in~$\g$ follows immediately from the case of simple Lie algebras,
which is described in Theorem~\ref{thm:classif-symeuler}. 
For symmetric Euler elements $h$, Proposition~\ref{prop:MN21:3.2} 
reduces the classification to the semisimple case,
but then one has to describe the module structure of the radical.
\begin{footnote}{The role of the symmetry of $h$
    for the existence of nets of real subspaces
    is still not completely understood. It certainly plays an important
    role in specifying locality conditions (cf.~Section~\ref{subsec:locality}).
    If $h$ is not connected,
    one may be forced to also take non-connected causal manifolds $M$
    into consideration, resp., to replace $G$ by a suitable non-connected
    group.}
\end{footnote}

\begin{examplekh} \label{ex:hcsp}
  (An example from symplectic geometry)  A particularly interesting
  Lie algebra which is 
  neither semisimple nor solvable is the
  {\it conformal Jacobi--Lie algebra} 
\[ \g = \hcsp(V,\omega) := \heis(V,\omega) \rtimes \csp(V,\omega), \] 
where 
$(V,\omega)$ is a symplectic vector space, 
$\heis(V,\omega) = \R \oplus V$ is the corresponding Heisenberg algebra 
with the bracket $[(z,v),(z',v')] = (\omega(v,v'),0)$, and 
\[ \csp(V,\omega) := \sp(V,\omega) \oplus \R \id_V \] 
is the {\it conformal symplectic Lie algebra} of $(V,\omega)$.
\index{Lie algebra!conformal symplectic \scheiding}
The hyperplane ideal
\[ \fj := \heis(V,\omega) \rtimes \sp(V,\omega) \] 
(the {\it Jacobi--Lie algebra})  
\index{Lie algebra!Jacobi \scheiding}
can be identified by the linear isomorphism 
\[ \phi \colon \fj \to \Pol_{\leq 2}(V), \qquad 
\phi(z,v,x)(\xi) := z + \omega(v,\xi) + \frac{1}{2} \omega(x\xi,\xi), \quad 
\xi \in V \] 
with the Lie algebra of polynomials 
$\Pol_{\leq 2}(V)$ of degree $\leq 2$ on $V$, 
endowed with the Poisson bracket (\cite[Prop.~A.IV.15]{Ne99}).  
The set 
\[ C_\g := \{ f \in \Pol_{\leq 2}(V) \colon f \geq 0 \}  \] 
is a pointed generating invariant cone in~$\fj$. 
The element $h_0 := \id_V$ defines a derivation on $\fj$ by 
$(\ad h_0)(z,v,x) = (2z,v,0)$ for $z \in \R, v \in V, x \in \sp(V,\omega)$. 
Any involution $\tau_V$ on $V$ satisfying $\tau_V^*\omega = - \omega$ 
defines by 
\begin{equation}
  \label{eq:invtau}
\tilde\tau_V(z,v,x) := (-z,-\tau_V(v), \tau_V x \tau_V) 
\end{equation}
an involution on $\g$ with $\tilde\tau_V(h_0) = h_0$,  
and $-\tilde\tau_V(C_\g) = C_\g$ follows from 
\[ \phi(\tilde \tau_V(z,v,x)) = - \phi(z,v,x) \circ \tau_V.\] 
Considering $h_\fs :=\shalf\tau_V$ as an element of $\sp(V,\omega)$, 
the element
\begin{equation}
  \label{eq:hhs}
  h := h_\fs + \shalf\id_V\in \csp(V,\omega)
\end{equation}
is Euler  in $\g$. Writing $V = V_1 \oplus V_{-1}$ for 
the $\tau_V$-eigenspace decomposition, we have 
\begin{align*}
 \g_{-1} &= 0 \oplus 0 \oplus \sp(V,\omega)_{-1}, \quad 
\g_0 = 0 \oplus V_{-1}\oplus \sp(V,\omega)_0 
\cong V_{-1} \rtimes \gl(V_{-1}), \\ 
  \g_1 &= \R \oplus V_1\oplus \sp(V,\omega)_1.
\end{align*}
Note that 
\begin{equation}
  \label{eq:tauh=tauv}
  \tau_h =  e^{\pi i \ad h} = \tilde\tau_V.
\end{equation}
Here $\g_1$ can be identified with the space 
$\Pol_{\leq 2}(V_{-1})$ of polynomials of 
degree $\leq 2$ on $V_{-1}$ and 
\[ C_+ = C_\g \cap \g_1 = \{ f \in \Pol_{\leq 2}(V_{-1}) \colon f \geq 0\}.\] 
This  cone is invariant 
under the natural action of the 
affine group $G_0 \cong \Aff(V_{-1})_0 \cong V_{-1} \rtimes \GL(V_{-1})_0$ 
whose Lie algebra is $\g_0$. We also note that 
\[ \g_{-1} \cong \Pol_2(V_1) \quad \mbox{ and } \quad 
  C_- = - C_\g \cap \g_{-1} = \{ f \in \Pol_2(V_1) \colon f \leq 0\},\]
so that $C_-$ is also pointed and generating.

The Euler element $h$ is not symmetric because $\dim \g_1 \not=\dim \g_{-1}$. 

We also claim that the Lie algebra $\hsp(V,\omega)$ contains
{\bf no Euler element}. In fact, as it is perfect,
and $\heis(V,\omega) \rtimes \sp(V,\omega)$
is a Levi decomposition, it suffices by Lemma~\ref{lem:levi-euler}
to show that no  Euler element of $\g$ is contained in 
$\R \oplus \{0\} \oplus \sp(V,\omega)$. Since all Euler elements $h$ in
the hermitian Lie algebra $\sp(V,\omega)$ are conjugate
(Proposition~\ref{prop:herm}), it suffices to consider 
$h = h_\fs + (\lambda,0,0), \lambda \in \R$. As 
\[ \Spec(\ad h) = \Spec(\ad h_\fs) = \{ \pm 1, \pm \shalf, 0\},\]
$h$ is not Euler in $\heis(V,\omega) \rtimes \sp(V,\omega)$. 
\end{examplekh}

\subsection{Euler elements in low-dimensional subalgebras}

\begin{lemma} \label{lem:slgl}
  Let $\g$ be a finite-dimensional Lie algebra and 
$h \in \cE(\g)$ an Euler element. 
If $h$ is not contained in the solvable radical $\rad(\g)$, 
then there exists a 
Lie subalgebra $\fb \subeq \g$ containing $h$ such that 
\begin{description}
\item[\rm(a)] $\fb \cong \fsl_2(\R)$ if and only if $h$ is symmetric, and 
\item[\rm(b)] $\fb \cong \gl_2(\R)$ if $h$ is not symmetric. 
\item[\rm(c)] If $h$ is symmetric, 
then {$\Inn_\g(\fb) \cong \PSL_2(\R)$.}
\item[\rm(d)] If $h$ is not symmetric and $\g$ is simple, then 
$\Inn_\g([\fb,\fb]) \cong \SL_2(\R)$. 
\end{description}
\end{lemma}

\begin{proof} (a) If $h \in \fb \cong \fsl_2(\R)$,
  then $h$ is symmetric because all
  Euler elements in $\fsl_2(\R)$ are symmetric by Example~\ref{ex:3.10b}.
  If, conversely, $h$ is symmetric, then
  Proposition~\ref{prop:MN21:3.2} implies that
  $h$ is contained in a Levi complement~$\fs$.
  Therefore \cite[Thm.~3.13]{MN21}   
  implies that $h$ is contained in an $\fsl_2$-subalgebra.

\nin (b) Suppose that $h$ is not symmetric and 
pick a maximal abelian hyperbolic subspace $\fa \subeq \g$ containing~$h$. 
With \cite[Prop.~I.2]{KN96} we find an $\fa$-invariant 
Levi complement $\fs \subeq \g$. Then 
$\fa_\fs := \fa \cap \fs$ is maximal hyperbolic in $\fs$ 
and $\fa = \fa_\fs + \fz_\fa(\fs)$. As $h$ is not contained in
$\rad(\g)$, there exists a root 
$\alpha \in \Delta(\fs,\fa)$ with $\alpha(h) = 1$ and 
root vectors $x_\alpha \in \fs_\alpha$ and $y_\alpha \in \fs_{-\alpha}$ 
with $h_\alpha := [x_\alpha, y_\alpha] \not=0$.  
We stress that $x_\alpha\in\fs_1(h)$. We use that 
\[ [x_\alpha, y_\alpha] = \kappa(x_\alpha, y_\alpha) a_\alpha, \] 
where $a_\alpha \in \fa$ is the unique element with 
$\alpha(a) = \kappa(a_\alpha,a)$ for all $a \in \fa$,  
and that the Cartan--Killing form
 $\kappa$ induces a dual pairing $\fs_\alpha \times \fs_{-\alpha} \to \R$. 
Then 
\[ \fb_\alpha := \R x_\alpha + \R y_\alpha + \R h_\alpha \cong \fsl_2(\R)\] 
and $[h,\fb_\alpha] \subeq \fb_\alpha$. 
Hence $\fb := \R h +  \fb_\alpha$ is a Lie subalgebra of $\fg$. 
As $h$ is not symmetric, $h \not\in \fb_\alpha$, and therefore 
$\fb \cong \gl_2(\R)$. 

\nin (c) If $h$ is symmetric and $\fb = [\fb,\fb] \cong \fsl_2(\R)$
as in (a), then the fact that $\fb$ contains an Euler element of $\g$ 
implies that all simple $\fb$-submodules of $\g$ are either trivial 
of isomorphic to the adjoint representation of $\fsl_2(\R)$ 
(consider eigenspaces of $\ad h$). 
This implies that $\Inn_\g(\fb) \cong \PSL_2(\R)$. 

\nin (d) Suppose that $\g$ is simple. 
If $h$ is not symmetric, then the Weyl group reflection 
$s_\alpha$ corresponding to the root $\alpha$ from above satisfies 
\[ s_\alpha(h) 
= h - \alpha(h) \alpha^\vee = h - \alpha^\vee.\] 
As $h$ is not contained in $\R \alpha^\vee \subeq \fb_\alpha$, we have 
$s_\alpha(h) \not\in \R h$. 

The simplicity of $\g$ ensures that the root system $\Delta = \Delta(\g,\fa)$ 
is irreducible and $3$-graded by~$h \in \fa$. 
Therefore 
\[ \Delta_0 := \{\alpha \in \Delta \: \alpha(h)=0\} \] 
spans a hyperplane in $\fa^*$, which coincides with $h^\bot$, and 
thus $\R h = \Delta_0^\bot$ by duality. 
Since $s_\alpha(h)$ is not contained in $\R h$, there exists a 
$\beta \in \Delta_0$ with $\beta(s_\alpha(h)) \not=0$. 
Now $\beta(h) = 0$ implies 
\[ 0 \not= \beta(s_\alpha(h)) = -\beta(\alpha^\vee).\] 
As $s_\alpha(h)$ is an Euler element, we obtain 
$|\beta(\alpha^\vee)| = 1$. Therefore the central element
$e^{\pi i \ad \alpha^\vee}$ of $\Inn_\g(\fb_\alpha)$ acts non-trivially,
and this implies that 
$\Inn_\g(\fb_\alpha) \cong \SL_2(\R)$ because it is a linear Lie 
group with non-trivial center (\cite[Ex.~9.5.18]{HN12}). 
\end{proof}

\subsection{The Brunetti--Guido--Longo (BGL) net}

In this subsection we describe 
a construction that generalizes the algebraic
construction of free fields for AQFT models presented in \cite{BGL02}.
We refer to \cite{MN21} and \cite{MNO26}
for a detailed discussion of this construction;
see also Exercise~\ref{exer:bgl}.

\begin{definition} \label{def:euler}
For an involution $\sigma \in \Aut(G)$, we write
 $G_\sigma := G \rtimes \{\id_G, \sigma\}$ for the corresponding
 group extension. 

The set 
  \[ \cG {:= \cG({G_\sigma})} := \{ (h,\tau)\in\g \times G\sigma 
      \: \tau^2 = e, \Ad(\tau)h = h\}\] 
  is called the {\it abstract wedge space of $G_\sigma$}.
\index{wedge space!abstract, $\cG(G_\sigma)$ \scheiding }
  An element $(h,\tau) \in \cG$ is called an 
{\it Euler couple} if  $h\in\cE(\fg)$ and 
\begin{equation}\label{eq:eul}
  \Ad(\tau)=\tau_h. 
\end{equation} 
Then $\tau$ is called an {\it Euler involution} on~$G$. 
We write $\cG_E\subeq \cG$ for the subset of Euler couples.
\index{Euler!involution \scheiding} 

\nin (a) Consider the homomorphism $\eps \: G_\sigma \to \{\pm 1\}$, defined
by $\ker \eps = G$. On $\g$ we consider the 
{\it twisted adjoint action} of $G_\sigma$ which changes the sign on odd group elements:  \index{adjoint action!twisted \scheiding} 
\begin{equation}
  \label{eq:adeps}
  \Ad^\eps \: G_\sigma \to \Aut(\g), \qquad 
\Ad^\eps(g) := \eps(g) \Ad(g).
\end{equation}
It extends to an action of $G_\sigma$ on $\cG$  by 
\begin{equation}
  \label{eq:cG-act}
 g.(h,\tau) := (\Ad^\eps(g)h, g\tau g^{-1}).
\end{equation}

\nin (b) (Duality operation) 
The notion of a  
``causal complement'' is defined on the abstract wedge space as follows: 
For $W = (h,\tau) \in \cG$, we define the {\it dual wedge} by   
\[ W' := (-h,\tau) {= \tau.W}.\] 
\index{wedge!dual $W'$ \scheiding}
Note that $(W')' = W$ and $(gW)' = gW'$ for $g \in G$ 
by \eqref{eq:cG-act}.  
This relation fits the geometric interpretation in the context
of wedge domains in spacetime manifolds (see also Subsection~\ref{subsec:locality}). 
\end{definition}

\begin{definition} \label{def:bgl-net}
If $(U,\cH)$ is an antiunitary representation of $G_\sigma$,   
then we obtain a standard subspace $\sH_U(W)$, determined for $W
  = (h, \tau) \in \cG$ 
by the couple of operators (cf.~Proposition \ref{prop:11}):
\begin{equation}
  \label{eq:bgl}
J_{\sH_U(W)} = 
U(\tau) \quad \mbox{ and } \quad \Delta_{\sH_U(W)} = e^{2\pi i \cdot\partial U(h)}, 
\end{equation}
and thus a $G$-equivariant map $\sH_U \:  \cG \to \Stand(\cH)$
(cf.\ Exercise~\ref{exer:sym}). 
This is the so-called {\it Brunetti--Guido--Longo (BGL) net} 
  \[ \sH_U^{\rm BGL} \: \cG(G_\sigma) \to \Stand(\cH).\]
\index{BGL net \scheiding} 
\end{definition}
For a detailed discussion of the properties of this
net and the structures on $\cG$,
we refer to \cite{MN21} and \cite{MNO25}.

\begin{small}
  \subsection{Exercises for Section~\ref{sec:3}}

\begin{exercise}
  \label{exer:diag}
  Let $h \in \fsl_n(\R)$. Show that $h$ is an Euler element
  if and only if $h$ is diagonalizable with $2$ eigenvalues
  $\lambda$, $\mu$ satisfying $\lambda - \mu = 1$.
\end{exercise}

\begin{exercise} Describe the conjugacy classes of
  Euler elements in the Lie algebras
  $\g = \fsl_n(\R),$ $\gl_n(\R)$ and $\so_{1,n}(\R)$ up to
  conjugation.   
\end{exercise}
\end{small}

\section{Causal homogeneous spaces and wedge regions} 
\label{sec:3b}

The Euler Element Theorem~\ref{thm:2.1} provides
us with the information that Euler elements
are the natural candidates for the elements $h$ in
the Bisognano--Wichmann property (BW),
but it provides no information on
how to find appropriate regions  $W \subeq M$?

Motivated by the Bisognano--Wichmann property (BW) in AQFT,
the modular flow on $W \subeq M$, given by
$\alpha^W_t(m) = \exp(th).m$
should, in a suitable sense, correspond to the ``flow of time''
on the spacetime region~$W$. This is based on the interpretation
of the modular group
in the context of the Tomita--Takesaki Theorem as
the dynamics of the corresponding quantum system,
the {\bf thermal time hypothesis}, 
a point of view advocated by A.~Connes and C.~Rovelli 
(cf.\ \cite{CR94}). References for the AQFT perspective
on this issue are \cite{BB99, BY99, BMS01, Bo98, SW03, Bo09}, \cite[\S 3]{CLRR22}.
For a perspective from non-commutative geometry,
see \cite{KG09}, \cite{Kot19} and \cite{He25}. 

To formulate what it means that a vector field generates
on an open domain $W \subeq M$ a flow that qualifies as a ``flow of time''
requires a {\it causal structure on the manifold $M$}, i.e.,
in each tangent space $T_m(M)$, we specify a
pointed, generating, closed convex cone $C_m \subeq T_m(M)$. 
\begin{footnote}{A closed convex cone $C$ in a finite-dimensional
    vector space $V$ is called {\it pointed} if $C \cap - C = \{0\}$,
    and {\it generating} if $C - C = V$, i.e., if $C$ has
    interior points.}
  \index{convex cone!pointed \scheiding}
  \index{convex cone!generating\scheiding}
\end{footnote}
We think of elements in the interior $C_m^\circ$ as
{\it timelike}, i.e., tangent vectors to curves
describing the dynamics on a region in~$M$
(following the ``flow of time''). \\

\nin {\bf Assumption:} For simplicity, we also assume that
$M$ is a homogeneous space, i.e., $M \cong  G/H$ for a closed subgroup
$H \subeq G$ with Lie algebra $\fh$. Then the tangent space
$T_{eH}(M)$ in the base point identifies naturally with
the quotient space $\fq := \g/\fh$. Hence the existence
of a $G$-invariant causal structure on $M$ is equivalent to
the existence of an $\Ad_\fq(H)$-invariant pointed
generating cone $C_\fq \subeq \fq$ (cf.~\cite{HN93}, \cite{HO97}, \cite{Se71, Se76}). Then
\[ C_{gH} := g.C_{eH} = g.C_\fq \quad \mbox{ for  } \quad g \in G, \]
is the corresponding causal structure on~$M = G/H$.
Here we write
\[ G \times TM \to TM, \quad (g,v) \mapsto g.v \]
for the induced
action of $G$ on the tangent bundle $TM$ of~$M$. 

\subsection{Causal structures and wedge regions}
\label{sec:3.2}

Coming back to the question of how to find~$W$,
let us fix an Euler element $h \in \g$.
Then we call 
\begin{equation}
  \label{eq:xhdef}
 X_h^M(m) 
 :=  \frac{d}{dt}\Big|_{t = 0} \exp(th).m 
\end{equation}
the corresponding {\it modular vector field}.
  \index{vector field!modular \scheiding}
In view of the ``flow of time''-philosophy, the subset 
$W$ should be contained in the {\it positivity region}
  \index{vector field!positivity region of \scheiding}
\begin{equation}\label{def:WM}
  W_M^+(h) := \{ m \in M \: X^M_h(m) \in C_m^\circ \},
\end{equation}
  which is the largest open subset on which the flow is ``future-directed''.
  For $m = gH \in M = G/H$ and the projection $p_\fq \: \g \to \fq= \g/\fh
  \cong T_{eH}(M)$,
  we have
\begin{equation}
  \label{eq:xmh}
 X^M_h(gH)
  =  \frac{d}{dt}\Big|_{t = 0} \exp(th).gH  
=  \frac{d}{dt}\Big|_{t = 0} gg^{-1}\exp(th).gH 
=  g.p_\fq(\Ad(g)^{-1}h).
\end{equation}
By $G$-invariance of the causal structure, this calculation shows that
$X^M_h(gH) \in C_{gH}^\circ$ is equivalent to
$p_\fq(\Ad(g)^{-1}h) \in C^\circ$, 
so that we obtain the Lie algebraic description
\begin{equation}\label{def:WM2}
  W_M^+(h) = \{ g H \in G/H \: \Ad(g)^{-1} h \in p_\fq^{-1}(C^\circ) \}
\end{equation}
of the positivity region. 

\begin{definition} \label{def:wedge}
  \index{wedge region \scheiding}
  A {\it wedge region} for $h$ on the causal homogeneous
  space $M$ is a connected component $W$ of the positivity region 
  $W_M^+(h)$.  
\end{definition}

At this point it is not clear why to focus on connected
components and not the whole positivity region. As the concrete examples, where
$W_M^+(h)$ is not connected, show, the inclusions
$\sH(W) \subeq \sH(W_M^+(h))$ are often proper.
If this is the case and $\sH(W) = \sV$, then the subspace 
$\sH(W_M^+(h))$ can not be separating
by the Equality Lemma~\ref{lem:lo08-3.10}. Therefore the
connected components turn out to be the natural choice for wedge regions.
In this context, Theorem~\ref{thm:reg-net}
shows that small open $\exp(\R h)$-invariant subsets may already
satisfy~(BW).

\begin{examplekh}
  In Minkowski space $M = \R^{1,d-1}$ (Remark~\ref{rem:poin}), the causal
  structure is given by the constant cone field $C_x = C$ for $x \in M$ and 
  \[ C = \oline{\jV_+} = \{ x \in \R^{1,d-1} \: x_0 \geq  \sqrt{\bx^2}\}. \]

$M$ is a homogeneous space of the Poincar\'e group
  $G = \R^{1,d-1} \rtimes \SO_{1,d-1}(\R)_e$ with base point $0$, whose
  stabilizer is the Lorentz group $\SO_{1,d-1}(\R)_e$.

  For the Lorentz boost $h(x) = (x_1, x_0, 0,\cdots, 0)$,
  the corresponding vector field is linear, i.e.,
  \[ X^M_h(x) = h(x),\]
  and its values are positive timelike, i.e., contained in
  $C^\circ = \jV_+$ if and only if $x_1 > |x_0|$, which specifies
  the Rindler wedge $W_R = \{ (x_0, \bx) \: x_1 > |x_0|\}$. 
\end{examplekh}

\begin{lemma} Any wedge region $W \subeq W_M^+(h)$ is invariant under
   the identity component $G^h_e$ of the centralizer
   \[ G^h := \{ g \in G \: \Ad(g)h = h\} \]
   of the Euler element~$h$, hence in particular
   under $\exp(\R h)$.
 \end{lemma}

 The following proposition provides a sufficient criterion
 for the positivity region on $M$ being non-empty.
The condition $h \in \fh$ is equivalent  to the
 base point being fixed under the modular flow. 

 \begin{proposition} \label{prop:2.13} {\rm(Sufficient conditions for
   the existence of wedge regions)} 
   Suppose that $M = G/H$, that
   $h \in \fh$ is an Euler element and that 
   $\tau_h \in \Aut(G)$ fixes $H$ and induces an anti-causal 
   map, i.e.,  $\tau_h^M(C_m) = - C_{\tau_h^M(m)}$ for $m \in M$.
   Then $W_M^+(h) \not=\eset$.    
 \end{proposition}

 \begin{proof} For the action of the one-parameter group
   $e^{\R \ad h}$ on $\fq := \g/\fh$,
   we write $\fq_j$, $j =1,0,-1$, for the corresponding
  $\ad h$-eigenspace and   \begin{footnote}{
   For the linear vector field defined by $h$ on $\fq$, the positivity region is 
$W_\fq^+(h) = C_+^\circ + \fq_0 + C_-^\circ$ 
   (cf.~\eqref{eq:w-decomp0}). This is why we consider these two cones.     
}   \end{footnote}
   \[ C_\pm := \pm C \cap \fq_{\pm 1}.\]
   In view of \eqref{def:WM2},
   it suffices to show that, for $x_{\pm 1} \in C_\pm^\circ$, there exists
   $t > 0$ such that
   \[ g_t := \exp(t x_{-1}) \exp(t x_1)\]
   satisfies $\Ad(g_t)^{-1}h \in p_\fq^{-1}(C^\circ)$.
   With Lemma~\ref{lem:Project} below, we see that
   $-\tau_h(C) = C$ implies that
   \[ C_+^\circ - C_-^\circ = (C_+ - C_-)^\circ \subeq C^\circ.\]

   For $t > 0$ we then have
   $e^{-t \ad x_{-1}} h = h - t [x_{-1}, h] = h -t x_{-1}$ because
   $(\ad x_{-1})^2 h \in \g_{-2}(h) = \{0\}$. We thus obtain 
\begin{align*}
 \Ad(g_t)^{-1}h
&     = e^{-t \ad x_1} e^{-t \ad x_{-1}} h
     = e^{-t \ad x_1} (h - t x_{-1}) 
     = h + t x_1 -t e^{-t \ad x_1} x_{-1} \\
  &     = h + t (x_1 - x_{-1}) - t(e^{-t \ad x_1} -\bone) x_{-1}.
\end{align*}
   As $p_\fq(h) = 0$,
   this element is contained in $p_\fq^{-1}(C^\circ)$ if and only if this is the case 
   for 
\[   x_1 - x_{-1} - (e^{t \ad x_1} -\bone) x_{-1}.\]
For $t \to 0$, this expression tends to $x_1 - x_{-1} \in  C^\circ$,
so that for some $t > 0$, we have $g_t H \in W_M^+(h)$
by \eqref{def:WM2}.
 \end{proof}
 
 \begin{remark} For a homogeneous space $M = G/H$, the positivity
   region $W_M^+(h)$ is non-empty if there exists an open subset
   $\cO \subeq G$ such that
   $p_\fq(\Ad(H\cO)h) \subeq \fq$ is contained in a pointed open
   convex cone. This depends very much on the geometry of the adjoint
   orbit $\cO_h$, the $H$-action on this orbit and its  position 
   with respect to $\fh= \ker p_\fq$.
 \end{remark}

To understand how wedge regions look like, we first discuss some simple classes
of examples.

\subsubsection{One-parameter groups on affine causal spaces}
\label{subsubsec:one-par-affine}

To develop the key facts on modular flows on causal homogeneous
spaces, we start in this subsection with the case of causal
affine spaces, i.e., pairs $(E,C)$, where $E$ is a finite-dimensional
vector space and $C \subeq E$ a pointed generating closed convex cone.

Specifically, we consider the following data (cf.~\cite{NOO21}):
\begin{description}
\item[\rm(A1)] $E$ is a finite-dimensional real vector space.
\item[\rm(A2)] $h \in \End(E)$ is 
diagonalizable with eigenvalues $\{-1,0,1\}$ and 
$\tau_h := e^{\pi i h}$.
\item[\rm(A3)] $C \subeq E$ is a pointed, generating 
  closed convex cone invariant under the one-parameter group $e^{\R h}$ and
  the involution $-\tau_h$.
\end{description}

Writing $E_\lambda = E_\lambda(h) := \ker(h - \lambda\bone)$ for the $h$-eigenspaces 
and $E^\pm := \ker(\tau_h \mp \bone)$ for the $\tau_h$-eigenspaces, 
(A2) implies 
\begin{equation}\label{eq:decomp}
E=E_{1}\oplus E_0 \oplus E_{-1}, \quad
E^- = E_1\oplus E_{-1}, \quad \mbox{ and } \quad
E^+ =E_0.
\end{equation}
We  put $C_\pm := C \cap E_{\pm 1}$.
For $x \in E$, we write $x= x_1 + x_0+ x_{-1}$ for the decomposition
into $h$-eigenvectors. 

\begin{lemma}\label{lem:Project}
For the projections 
\[ p_{\pm 1}: E\to E_{\pm 1}, x \mapsto x_{\pm 1}, \quad \mbox{ and } \quad 
  p^-:E\to E_1\oplus E_{-1}=E^-, x \mapsto x_1 + x_{-1},\]
the  following assertions hold:
\begin{description}
\item[\rm (i)] $p_{\pm 1}(C)=\pm C_\pm$ 
and $p_{\pm 1}(C^\circ)=\pm C_\pm^\circ\not=\eset$. 
\item[\rm (ii)] $p^- (C)=C\cap E^-=C_+\oplus -C_-$ and 
$p^- (C^\circ)=C^\circ\cap E^-=C^\circ_+\oplus -C_-^\circ$. 
\item[\rm (iii)] $C\subeq C_+\oplus E_0 \oplus -C_-$. 
\end{description}
\end{lemma}

\begin{proof}  (i) From $\pm C_\pm \subset C$, we get 
$\pm C_\pm \subset p_{\pm 1}(C)$. Using the $e^{th}$-invariance of $C$
and writing $x=x_1+x_0+x_{-1}$ as before, 
$e^{th}x=e^{t}x_1 + x_0 + e^{-t} x_{-1}.$ 
Now take the limit $t\to \infty$ to see that 
\[ C \ni e^{-t}e^{th}x = x_1 +e^{-t}x_0+e^{-2t}x_{-1}\to x_1  \quad \mbox{as } \quad t\to \infty.\]
We likewise get 
$x_{-1} = \lim_{t \to -\infty}   e^{t}e^{th}x \in C$. 
It follows that $x_\pm \in \pm C_\pm$, so that 
$p_{\pm 1}(C) = \pm C_\pm$. 
As $p_{\pm 1}$ are projections and $C^\circ \not=\eset$, it follows that 
$p_{\pm 1}(C^\circ) \subeq \pm C_{\pm}^\circ$. To obtain equality, it suffices to observe that 
$C_+^\circ \oplus - C_-^\circ \subeq (E^- \cap C)^\circ \subeq C^\circ$ follows from 
$-\tau_h(C) = C$. 

\nin (ii) The two leftmost equalities follow from $-\tau_h(C) = C$, 
and the second two rightmost equalities from (i) and $p^- = p_1 + p_{-1}$. 

\nin (iii) follows from (ii). 
\end{proof}

As the linear vector field on $E$ corresponding to $h$ is given by
$X^E_h(x) = x_1- x_{-1}$, Lemma~\ref{lem:Project}(ii) implies that 
its  positivity domain is the wedge region  
\begin{equation}
  \label{eq:w-decomp0}
W_E^+(h) =  C_+^\circ \oplus E_0 \oplus C_-^\circ \quad \mbox{ for } \quad 
C_\pm = \pm C \cap E_{\pm 1}.\end{equation}
In particular, it is not empty.
Here (A3) ensures that $C^\circ$ intersects $E^- = \im(h)$.
Otherwise we would include cones
of the form $C = C_1 + C_0 + C_{-1}$ with $C_j \subeq E_j$. Any such cone
is invariant under $e^{\R h}$, but for such cones
${C^\circ \cap (E_{+1} + E_{-1})=\eset}$ implies that $W_E^+(h) = \eset$.

\begin{examplekh} \label{ex:causal1b} 
 (The affine group on $\R$) 
We endow $M = \R$ with the canonical causal structure given by
$C_x = \R_{\geq 0}$ for $x \in \R$. Then the connected {\it affine group} 
  \index{affine group  \scheiding}
$G = \Aff(\R)_e = \R \rtimes \R_+$ is $2$-dimensional. 
Its elements are denoted $(b,a)$, and they act by the affine,
orientation preserving maps $(b,a)x = ax + b$ on the real line.

Here $h = (0,1) \in \g$ is an Euler element whose flow is
given by $\alpha_t(x) = e^t x$. Therefore its positivity region is 
\[ W_\R^+(h) = \{x\in\R: x > 0 \} = \R_+ \]
and the corresponding reflection is $\tau_h(x)= -x$.

All other Euler elements in $\g$ are of the form $h' = (x,\pm 1)$,
where $\cO_h = \R \times \{1\}$ and $\cO_{-h} = \R \times \{-1\}$.
The corresponding positivity regions are the proper unbounded open intervals
in $\R$. 
\end{examplekh}

\subsubsection{More examples of  wedge regions}

The first example refers also to an affine causal space, but now the 
linear part of the automorphism group is larger. 

\begin{examplekh} \label{ex:causal1}
  (Poincar\'e group and Rindler wedges)
The example arising most prominently in physics is the 
connected {\it Poincar\'e group}   \index{Poincar\'e group \scheiding}
\[ G := \cP_+^\up := \R^{1,d-1} \rtimes \SO_{1,d-1}(\R)_e. \]
It acts on $d$-dimensional Minkowski space $\R^{1,d-1}$ 
as an isometry group of the Lorentzian metric given by 
$(x,y)= x_0y_0-\bx\by$ for $x = (x_0, \bx) \in \R^{1,d-1}$.
The $G$-action preserves the constant cone field defined by the closure 
$C = \oline{\jV_+}$ of the open future light cone
\[ \jV_+ = \{ (x_0, \bx) \in \R^{1,d-1} \: x_0 > \sqrt{\bx^2}\}.\] 

The generator
\[  h(x_0,x_1,x_2, \ldots, x_{d-1}) = (x_1, x_0, 0, \ldots, 0)\]
of the Lorentz boost in the  $(x_0,x_1)$-plane 
is an Euler element in $\so_{1,d-1}(\R)$, and $e^{\pi i h}$ acts by
the reflection
\[ \tau_h(x) =(-x_0,-x_1,x_2,\ldots,x_{d-1}),\]  for which
$-\tau_h(C) = C$.
In view of the Classification Theorem~\ref{thm:classif-symeuler},
the fact that the restricted root system of $\so_{1,d-1}(\R)$
is of type $A_1$ implies 
that there exists only one conjugacy class of Euler elements
in $\so_{1,d-1}(\R)$. With Lemma~\ref{lem:levi-euler} it follows
that the same holds for the Poincar\'e algebra because its center
is trivial.
So Euler elements in this Lie algebra are precisely the Lorentz boosts
in different affine coordinate systems.

By \eqref{eq:w-decomp0}, the positivity region of $h$ is
\[ W_M^+(h) = \R_+ (\be_0 + \be_1) - \R_+ (\be_0 - \be_1) + \Spann\{\be_2,
  \ldots, \be_{d-1}\} =  \{x\in\R^{1,d-1}: |x_0|<x_1\}. \]
It is called the {\it standard right wedge} or {\it Rindler wedge $W_R$}
  \index{Rindler wedge $W_R$\scheiding}
  and plays a key role in AQFT as a localization region
  for a uniformly accelerated observer, represented
  by an orbit of the modular flow in $W_R$ 
  (\cite{BGL02, LL15}; see also Remark~\ref{rem:poin}).
\end{examplekh}

The following example is the smallest compact one.
It is a causal flag manifold. We refer to Subsection~\ref{subsec:causal-flag-man}
for more on this class of examples. 

\begin{examplekh} \label{ex:causal1c} 
  (The action of $\PSL_2(\R)$ on $\bS^1 \cong \R_\infty$)
The group $G := \SL_2(\R)$ acts on the one-point compactification
$M = \R_\infty = \R \cup \{\infty\} \cong \bS^1$ of $\R$  
by 
\[ g.x := \frac{a x + b}{cx + d} \quad 
  \mbox{ for }\quad g = \pmat{a & b \\ c & d}\in \SL_2(\R), \quad
  x \in \R.\]
The subgroup $\SO_2(\R)$ acts transitively by 
  \[ \rho(t).z:=
    \pmat{ \cos(t/2) & \sin(t/2) \\ - \sin(t/2)
      & \cos(t/2)}.x
    = \frac{\cos(t/2) \cdot x + \sin(t/2)}{- \sin(t/2) \cdot x + \cos(t/2)}
    = \frac{x + \tan(t/2)}{1- \tan(t/2) \cdot x},\]
  generated by the vector field $X^M(x) = \frac{1}{2}(1 + x^2)$.
  As this flow is $2\pi$-periodic, it induces a diffeomorphism
  $\R/2\pi \Z \to \R_\infty$. This shows that the natural causal
  structure on $\R$ extends to $M$ in a $G$-invariant fashion. 

  In $\g= \fsl_2(\R)$ we consider the Euler element 
  $h = \frac{1}{2}\diag(1,-1)$   (cf.\ Example~\ref{ex:3.10b}). 
  The flow it generates on $\R_\infty$ is given by
$\alpha_t(x) = e^t x$, where $0$ and $\infty$ are fixed. Accordingly,
\[ W_M^+(h) = \R_+ \subeq \R_\infty.\]
As $G$ acts transitively of the set $\cO_h = \cE(\fsl_2(\R))$
of Euler elements in $\fsl_2(\R)$  (Example~\ref{ex:3.10b}),  their positivity regions
in $\bS^1$ are precisely the non-dense open intervals.

The Cayley transform 
\[ C: \R_\infty \to \bS^1 := \{z\in\C: |z|=1\}, \quad 
  C(x) := \frac{i-x}{i+x} = \frac{1+ ix}{1-ix},
  \qquad C(\infty) := -1, \] 
is a homeomorphism, identifying $\R_\infty$ with the circle. 
Its inverse is 
\[ C^{-1}:\bS^1 \to \R_\infty, \quad C^{-1}(z) = i\frac{1-z}{1+z}\]
(cf.\ Exercise~\ref{exer:stereo}). 
It maps the upper semicircle $\{ z \in \bS^1 \: \Im z > 0\}$ 
to the positive half line $\R_+$. 
The Cayley transform intertwines the action of
$\SL_2(\R)$ with the action of $\SU_{1,1}(\C)$ on
the circle $\bS^1 \subeq \C$ by fractional linear transformations.
This action preserves the causal structure on $\bS^1$, specified by
$C_z = \R_{\geq 0} i z \subeq T_z(\bS^1) = i \R z$ for $z \in \bS^1$.
 \end{examplekh}

 \begin{examplekh} \label{ex:causal2}
   The Lie group $G := \SL_2(\R)$
   has three classes of causal homogeneous spaces.
      In Example~\ref{ex:causal1c} we have already
   seen its action on the $1$-dimensional circle~$\bS^1$,
   a flag manifold of~$\SL_2(\R)$.

   Observing that $\Ad(\SL_2(\R)) \cong \SO_{1,2}(\R)_e$
   (Exercise~\ref{exer:sl2-Lorentz}), we obtain two
   other examples:
   \begin{itemize}
   \item Two-dimensional de Sitter space 
     \[ \dS^2 = \{ (x_0, x_1, x_2)\in \R^{1,2} \: x_0^2 - x_1^2 - x_2^2= - 1\}\]
     carries an $\SO_{1,2}(\R)_e$-invariant causal structure
     with the positive cone in the base point $\be_1$ given by 
     \[ C_{\be_1} := \{ (x_0, 0,x_2) \: x_0 \geq |x_2| \}
       \subeq T_{\be_1}(\dS^2) = \R \be_0 + \R \be_2. \]
The inversion $-\bone$ on $\dS^2$ is an anti-causal map. 

\begin{multicols}{2}
  \vbox{
         For the Euler element defined by
     $h(x_0, x_1, x_2) = (x_1, x_0,0)$, we obtain the connected
     wedge region
     \[W_{\dS^2}^+(h) = \{ (x_0, x_1, x_2) \in \dS^2 \: x_1 > |x_0|\}.\]
The wedge region $W := W_{\dS^2}^+(h)$ and the
  orbits of the modular flow in $W$ are marked in the picture
  on the right. }
\columnbreak
  \includegraphics[width=0.4\textwidth]{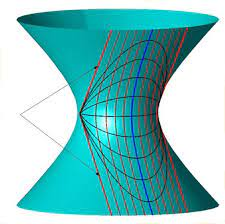}\\
\end{multicols}

   \item Two-dimensional anti-de Sitter space 
     \[\AdS^2 = \{ (x_1,x_2, x_3) \in \R^{2,1} \: x_1^2 + x_2^2 - x_3^2=  1\} \]
         carries an $\SO_{2,1}(\R)_e$-invariant causal structure
     with the positive cone in the base point $\be_2$ given by 
     \[ C_{\be_2} := \{ (x_1, 0,x_3) \: x_3 \geq |x_1| \}
       \subeq T_{\be_2}(\AdS^2) = \R \be_1 + \R \be_3. \]
The inversion $-\bone$ on $\AdS^2$ is a causal map. 
     For the Euler element defined by
     $h(x_1, x_2, x_3) = (0, x_3,x_2)$, we obtain the positivity region 
     \[ W_{\AdS^2}^+(h) = \{ (x_1, x_2, x_3) \in \AdS^2 \:
       x_1 x_3 > 0, |x_2| < |x_3| \}. \]
     It has two connected components, specified by the sign of~$x_1$
     (\cite[Lem.~11.3]{NO23a}). Note that $|x_2| < |x_3|$ specifies
     the region on which $h$, as a vector field on $\AdS^2$, is timelike.
     This region has four connected components, and $x_1 x_3 > 0$ selects
     the two on which it is positive. They are exchanged by
     the inversion~$-\bone$.
\end{itemize}
As homogeneous spaces, $\AdS^2$ and $\dS^2$ can be identified
with the adjoint orbit 
$\cO_h \cong G/G^h \cong \cE(\fsl_2(\R))$,  where $h := \shalf\diag(1,-1)$
is an Euler element in $\fsl_2(\R)$
(cf.~Example~\ref{ex:3.10b}). However, both carry natural causal
structures, and these are non-isomorphic
because $\AdS^2$ admits closed causal curves and
$\dS^2$ does not.
\end{examplekh}

\subsection{The compression semigroup of a wedge region}

Let $M = G/H$ be a causal homogeneous space
with causal structure given by the cone field $(C_m)_{m \in M}$.
For the vector field $X^M_y(m)
=  \frac{d}{dt}\big|_{t = 0} \exp(ty).m$,
the set
\begin{equation}
  \label{eq:cm}
 C_M
  := \{ y \in \g \:  (\forall m \in M) X_y^M(m) \in C_m \} 
\ {\buildrel \eqref{eq:xmh} \over  =}\  \bigcap_{g \in G} \Ad(g) p_\fq^{-1}(C)
\end{equation}
of those Lie algebra elements whose vector fields on $M$ are everywhere
positive is a closed convex $\Ad(G)$-invariant cone in $\g$.
It consists of all $y \in \g$ corresponding
to everywhere  ``positive'' vector fields on~$M$. If $G$ acts
effectively on~$M$, then it is also pointed
because elements in $C_M \cap - C_M$ correspond to vanishing vector fields 
on $M$. This cone is a geometric analog of the positive cone $C_U$
of a unitary representation of~$G$ (see \eqref{eq:CU}).\begin{footnote}
  {Note that the existence of a pointed generating invariant cone in
    a Lie algebra $\g$ has strong structural implications (cf.\ \cite{Ne99}). If, f.i.,
    $\g$ is simple, then it must be hermitian.}\end{footnote}
The following observation shows that it behaves in many
respects similarly (cf.\  \cite{Ne22}). 

As any connected component $W \subeq W_M^+(h) \subeq M$ is invariant
under $G^h_e \supeq \exp(\R h)$,\begin{footnote}{Recall that
    $G^h =\{ g\in G \: \Ad(g)h = h\}$.}\end{footnote}
the same holds for the closed convex cone 
\begin{equation}
  \label{eq:def-cw}
  C_W := \{ y \in \g \:  (\forall m \in W)\ X_y^M(m) \in C_m \} \supeq  C_M.
\end{equation}
Below we show that this cone determines the tangent wedge of the compression
semigroup of~$W$.

\begin{proposition} \label{prop:LSW}
  For a connected component $W \subeq W_M^+(h)$, its compression semigroup
  \[ S_W := \{ g \in G\: g.W \subeq W \} \]
  is a closed subsemigroup of $G$ with
  $G_W := S_W \cap S_W^{-1} \supeq G^h_e$ and
\[  \L(S_W) := \{  x \in \g \:  \exp(\R_+ x) \subeq S_W \}
  =  \g_0(h) + C_{W,+} + C_{W,-}, \]
with
\[ C_{W,\pm} := \pm C_W \cap \g_{\pm 1}(h).\] 
In particular, the convex cone $\L(S_W)$ has interior points if $C_M$ does.   
\end{proposition}

\begin{proof} As $W \subeq M$ is an open subset, its complement $W^c := M \setminus W$
  is closed, and thus 
  \[ S_W = \{ g \in G \: g^{-1}.W^c \subeq W^c \} \]
  is a closed subsemigroup of $G$, so that its tangent wedge $\L(S_W)$ is a closed
  convex cone in $\g$ (\cite{HHL89}, \cite[\S 1.4]{HN93}).

  Let $m = gH \in W$, so that $p_\fq(\Ad(g)^{-1}h) \in C^\circ$.
  For $x \in \g_{\pm 1}(h)$ we then derive from $\g_{\pm 2}(h) = \{0\}$ that 
  \[ e^{\ad x} h = h + [x,h] = h \mp x.\] 
  This leads to 
  \begin{align*}
 p_\fq(\Ad(\exp(x)g)^{-1}h)
&    = p_\fq(\Ad(g)^{-1} e^{- \ad x}h)
    = p_\fq(\Ad(g)^{-1} ( h \pm x))\\
&      = p_\fq(\Ad(g)^{-1}h) \pm  p_\fq(\Ad(g)^{-1} x).
  \end{align*}
  For $x \in C_{W,\pm}$, we have $p_\fq(\pm\Ad(g)^{-1}x) \in C$, so that
  $p_\fq(\Ad(\exp(x)g)^{-1}h) \in C^\circ$, 
 which in turn implies that
    $\exp(x).m \in W$ for $m \in W$. 
    So $\exp(C_{W,\pm}) \subeq S_W$, and thus $C_{W,\pm} \subeq  \L(S_W)$.
 The invariance of $W$ under the identity component
  $G^h_e$ of the centralizer of $h$ further entails
  $\g_0(h) \subeq \L(S_W)$, so that
  \begin{equation}
    \label{eq:incl1}
    C_{W,+} + \g_0(h) + C_{W,-} \subeq \L(S_W).
  \end{equation}
  We now prove the converse inclusion. Let $x \in \g_1(h)$.  
  If $X^M_x(m) \not \in C_m$, i.e., $p_\fq(\Ad(g)^{-1}x) \not\in C$,
  then there exists a $t_0 > 0$ with 
\[ p_\fq(\Ad(g)^{-1}h) + t_0 \cdot p_\fq(\Ad(g)^{-1} x)\not\in C \]  (\cite[Prop.~V.1.6]{Ne99}),
  so that $\exp(t_0 x).m \not \in W$. We conclude that
\[     \L(S_W) \cap \g_1(h) = C_{W,+}.\] 
Further, the  invariance of the closed convex cone
  $\L(S_W)$ under $e^{\R \ad h}$ implies that, for
  \[ x = x_{-1} + x_0 + x_1 \in \L(S_W) \quad \mbox{ and } \quad
    x_j \in \g_j(h),\] we have
  \[ x_{\pm 1} = \lim_{t \to \infty} e^{-t} e^{\pm t \ad h} x \in \L(S_W) \cap
    \g_{\pm 1}(h) = C_{W,\pm}, \]
  which implies the other inclusion 
$\L(S_W) \subeq C_{W,+} + \g_0(h)+ C_{W,-},$ 
  hence equality by \eqref{eq:incl1}.
  
  Let $p_\pm \:  \g \to \g_{\pm 1}(h)$ denote the projection along the other
  eigenspaces of $\ad h$. Then
  \[ C_{W,\pm} \supeq C_{M,\pm} := \pm C_M \cap \g_{\pm 1}(h) = \pm p_{\pm}(C_M) \]
  also follows from Lemma~\ref{lem:Project}. 
  Therefore $C_M^\circ \not=\eset$
  implies $C_{W,\pm}^\circ \not=\eset$, and this is equivalent to
  $\L(S_W)^\circ \not=\eset$.   
\end{proof}

\begin{remark} In many situations, such as the action
    of $\PSL_2(\R)$ on the circle $\bS^1 \cong \bP_1(\R)$,
  the cones $C_{W,\pm} \supeq C_{M,\pm}$ 
    coincide, and we believe that this is probably always the case.
If $x \in C_{W,+}$, then the positivity region 
    \[ \Omega_x := \{ m \in M \: X^M_x(m) \in C_m \} \]
    contains $W$ (by definition), and it is also invariant under
    the identity component $N_x = \la \exp \fn_\g(\R x) \ra$
    of the normalizer or $\R x$, to that
    \begin{equation}
      \label{eq:omegax}
      \Omega_x \supeq N_x.W \supeq \bigcup_{t > 0} \exp(-tx).W.
    \end{equation}
    Clearly, $\Omega_x = M$ follows if $N_x.W$     is dense in $M$,
    and we are not aware of examples, for which this is not the case. 

    If $G$ is the connected Poincar\'e group acting on Minkowski
    space $M = \R^{1,d}$ and
    \[ W = W_R = \{ (x_0, \bx) \: x_1 > |x_0| \},  \]
    then
    \[ S_{W} = \oline{W} \rtimes
      \big(\SO_{d-1}(\R) \times \SO_{1,1}(\R)^\uparrow\big)\]
    (Lemma~\ref{lem:4.17})     implies that
    \[ C_{W,\pm} = \L(S_{W})\cap \g_1(h)
      = \R_+ (\pm\be_0 + \be_1) \]
    consists of constant vector fields, so that
    $C_{W,\pm} = C_{M,\pm}$ in this case.
    For $x = \be_0 + \be_1 \in C_{W,+}$, the domain 
    $W - \R_+ x$ is an open half space,
    hence in particular not dense in~$M$,
    but $N_x.W \supeq W + \R^{1,d} = \R^{1,d}$ is. 
\end{remark}

\subsubsection{The Rindler wedge in Minkowski space}

Let $G  = P(d)_e$ be the identity component of the
{\it Poincar\'e group}  \index{Poincar\'e group \scheiding} 
\[ P(d) := \R^{1,d-1} \rtimes \OO_{1,d-1}(\R) \] 
and $h \in\g$ the Euler element corresponding  
to the Lorentz boost in the $(\be_0,\be_1)$-plane
with wedge region
\[     W_R = \{ x \in \R^{1,d-1} \: x_1 > |x_0| \}\] 
(Example~\ref{ex:causal1}). The corresponding reflection is
$\tau_h= \diag(-1,-1,1,\ldots, 1)$. 

\begin{lemma}
  \label{lem:4.17} 
  The stabilizer group of $W_R$ is
\begin{equation}
  \label{eq:stabgrp2}
  G_{W_R} \cong {\rm SE}(d-2) \times \SO_{1,1}(\R)_e
  \cong (E_R \rtimes \SO_{d-2}(\R)) \times \SO_{1,1}(\R)_e,
\end{equation}
where ${\rm SE}(d-2)$ denotes the connected group of proper euclidean motions on
\[ E_R := \Spann\{\be_2, \ldots, \be_{d-1}\} \cong \R^{d-2} \] 
and $\SO_{1,1}(\R)$ acts on $\Spann\{ \be_0, \be_1\}$.
The compression semigroup of $W_R$ is 
\[ S_{W_R} := \{ g \in P(d) \: gW_R \subeq W_R\}
=  \oline{W_R} \rtimes \OO_{1,d-1}(\R)_{W_R}.\] 
\end{lemma} 

\begin{proof} The stabilizer group $P(d)_{W_R}$ 
contains the translation group corresponding to the edge $E_R$, 
and $gW_R = W_R$ implies $g(0) \in E_R$, so that 
\[ P(d)_{W_R} \cong E_R \rtimes \OO_{1,d-1}(\R)_{W_R}.\] 
Further, each $g \in \OO_{1,d-1}(\R)$ preserving $E_R$ also preserves its
orthogonal complement, so that 
\[ \OO_{1,d-1}(\R)_{W_R} 
= \OO_{d-2}(\R) \times \OO_{1,1}(\R)_{W_R} 
= \OO_{d-2}(\R) \times (\SO_{1,1}(\R)_e \{\bone, r_1\}), \] 
where $r_1 = \diag(1,-1,1,\ldots, 1)$.

Next we use Lemma~\ref{lem:cone-endos} in Appendix~\ref{app:E} to see that 
\[ S_{W_R} = \oline{W_R} \rtimes
  \{ g \in \SO_{1,d-1}(\R)_e \: gW_R \subeq W_R\}.\]
Any $g \in \SO_{1,d-1}(\R)_e$ with $gW_R \subeq W_R$
satisfies $gE_R = E_R$ because $g$ is injective 
and $\dim E_R < \infty$. This in turn implies that $g$ commutes with 
$\tau_h= \diag(-1,-1,1,\ldots, 1)$, so that $g = g_1 \oplus g_2$  
with $g_1 \in \OO_{1,1}(\R)$ preserving the wedge region
$W_R^2 \subeq \R^{1,1} = \Spann \{\be_0, \be_1\}$.
As $g_1 W_R^2$ is a quarter plane bounded by light rays,
it cannot be strictly smaller than $W_R^2$, hence $g_1 W_R^2 = W_R^2$,
and finally $g W_R = W_R$. This completes the proof. 
\end{proof}

\subsection{Causal Lie groups}
\label{subsec:group-type}

The most structured examples of causal homogeneous spaces
are causal groups with a biinvariant causal structure.

   Let $G$ be a connected Lie group and $C_\g \subeq \g$
  be a pointed generating closed convex cone.
  Then $C_g := g.C_\g \subeq T_g(G)$ defines on $G$ a left-invariant
  causal structure. These structures become more interesting
  if $C_\g$ is also $\Ad(G)$-invariant, so that
  the action of $G \times G$ by $(g_1, g_2).g = g_1 g g_2^{-1}$
  preserves the causal structure.  \begin{footnote}{That a $G$ action
     on  $M$ preserves the causal structure $(C_m)_{m \in M}$  means that
      $g.C_m = C_{g.m}$ for $g \in G, m \in M$.}  \end{footnote}
  If $h_0 \in \g$ is an Euler element,
  then $h := (h_0, h_0) \in \g^{\oplus 2}$ is Euler as well.
  It generates the flow
  \[ \alpha_t(g) = \exp(th_0) g \exp(-th_0).\]
  The corresponding vector field is
  \[  X_h^G(g)
    =   \frac{d}{dt}\Big|_{t = 0} \exp(th_0)g \exp(-t h_0) 
    = h_0.g - g.h_0 = g.(\Ad(g)^{-1}h_0 - h_0),\]
so that 
  \begin{equation}
    \label{eq:wg+}
    W_G^+(h) = \{ g \in G \: \Ad(g)^{-1}h_0 - h_0 \in C_\g^\circ \} 
   =  \{ g \in G \: \Ad(g) h_0 - h_0 \in - C_\g^\circ \}. 
  \end{equation}
  It is easy to see that this is an open subsemigroup of $G$, contained in
  the closed subsemigroup 
  \begin{equation}
    \label{eq:sh0g}
    S(h_0, C_\g)
    :=  \{ g \in G \: h_0 - \Ad(g) h_0 \in C_\g \}.
  \end{equation}
For the $G$-invariant order structure on $\g$, defined by 
  \[ x \leq_{C_\g} y \quad \mbox{ if } \quad y -x \in C_\g, \]
  this means that
  \[ S(h_0, C_\g)
    =  \{ g \in G \: \Ad(g) h_0 \leq_{C_\g} h_0 \}.\]
  We likewise have for the strict order, defined by 
  \[ x <_{C_\g} y \quad \mbox{ if } \quad y -x \in C_\g^\circ\]
that
\[ W_G^+(h)  =  \{ g \in G \: \Ad(g) h_0 <_{C_\g} h_0 \}.\]
To describe this domain, we need the two pointed generating
$\Ad(G^{h_0})$-invariant cones 
\begin{equation}
  \label{eq:cpm}
  C_\pm = \pm C_\g \cap \g_{\pm 1}
\end{equation}
(cf.\ Lemma~\ref{lem:Project}). 

  We claim that
  \begin{equation}
    \label{eq:sw-incl}
 \exp(C_+^\circ) G^{h_0} \exp(C_-^\circ) \subeq W_G^+(h_0), 
  \end{equation}
  which, by passing to the closure, implies
  \begin{equation}
    \label{eq:sw-incl2}
 \exp(C_+) G^{h_0} \exp(C_-) \subeq S(h_0, C_\g).    
  \end{equation}
  As the centralizer $G^{h_0}$ of $h_0$ is obviously contained in
  $S(h_0,C_\g)$ and $\exp(C_+) G^h = G^h \exp(C_+)$, it suffices to  show that
  $\exp(C_+^\circ) \exp(C_-^\circ) \subeq W_G^+(h_0)$.
For $x_\pm \in C_\pm^\circ$, this follows from 
\begin{align*}
& e^{\ad x_+} e^{\ad x_-} {h_0} - {h_0}
    = e^{\ad x_+} ({h_0} + [x_-,{h_0}])-{h_0}
    = e^{\ad x_+} ({h_0} + x_-)-{h_0}\\
&    = [x_+, {h_0}] + e^{\ad x_+} x_-  
    =  - x_+ + e^{\ad x_+} x_-
    =   e^{\ad x_+} (x_- - x_+) \\
&    \in  - e^{\ad x_+} (C_+^\circ - C_-^\circ) 
  \subeq   - C_\g^\circ\notag
\end{align*}
(cf.\ the proof of Proposition~\ref{prop:2.13}). 
Here we used $-\tau_h(C_\g) = C_\g$ for the inclusion
$C_+^\circ - C_-^\circ  \subeq  C_\g^\circ$
(Lemma~\ref{lem:Project}(ii)).

\begin{definition} \label{def:olsh-sgrp} (The semigroups of KMS points)
  Assume that $\tau_h^\g = e^{\pi i \ad h}$ integrates
  to an automorphism $\tau_h$ of $G$. 
  Using the complex Olshanski semigroup 
  $S(iC_\g) := G \Exp(iC_\g)$
  (see \cite[\S IX.1]{Ne99}, \cite[3.20]{HN93} and
  also   \cite[\S 2.4]{Ne22} for a detailed
  discussion),  \begin{footnote}{If $G$ is simply connected
      and $\eta_G \: G \to G_\C$ its universal complexification,
      then $S(iC_\g)$ is the simply connected covering of the subsemigroup 
      $\eta_G(G) \exp(i C_\g) \subeq G_\C$, and if
      $G$ is not simply connected, it is the quotient of
      $S_{\tilde G}(i C_\g)$ by the kernel
      of the covering map $q_G \: \tilde G \to G$.
      The map $\Exp \: iC_\g \to S(iC_\g)$ is the map corresponding
      in this context to the exponential function
      $iC_\g\to G_\C$ and $G \times C_\g \to S(iC_\g), (g,x) \mapsto
      g \Exp(ix)$ is a homeomorphism. } \end{footnote}
  we define the  subsemigroup
  $G_{\rm KMS} \subeq G$ as the set of those elements $g \in G$ 
for which the orbit map 
\[ \alpha^g \colon \R \to G, \quad \alpha^g(t) = \alpha_t(g) \] 
extends analytically to a map $\oline{\cS_{\pi}} \to S(iC_\g)$
with $\alpha^g(\cS_\pi) \subeq  S(iC_\g^\circ)$, such that
$\alpha^g(\pi i) = \tau_h(g)$. 
\end{definition}

\begin{theorem} \label{thm:semigroups-equal}
  If $G$ is simply connected, $h \in \g$ an Euler element,
  and $C_\g \subeq \g$ a pointed closed convex invariant cone
  with $-\tau_h^\g(C_\g) = C_\g$, then 
  \begin{equation}
    \label{eq:allsemigroupsequl}
 S(h,C_\g) = \exp(C_+) G^h \exp(C_-) = G^h \exp(C_+ + C_-), 
  \end{equation}
  the positivity domain
  \[ W_G^+(h) = \exp(C_+^\circ) G^h \exp(C_-^\circ)\]
  is a subsemigroup, 
  and
  \[ G_{\rm KMS} = \exp(C_+^\circ) G^h_e \exp(C_-^\circ)
    =  G^h_e\exp(C_+^\circ + C_-^\circ) = S(h,C_\g)^\circ_e\]
  is a connected component of $W_G^+(h)$. 
\end{theorem}

\begin{proof} The first two equalities in \eqref{eq:allsemigroupsequl}
  are the Decomposition Theorem \cite[Thm.~2.16]{Ne22}.
  Further \cite[Thm.~2.21]{Ne22} shows that
  $S(h,C_\g)$ coincides with the set of all $g \in G$ for which
  $\alpha^g$ extends to a map $\oline{\cS_\pi} \to S(iC_\g)$.

Next we show that the additional requirement that
  $\alpha^g(\cS_\pi) \subeq S(iC_\g^\circ)$ specifies 
the open subset  $G^h \exp(C_+^\circ + C_-^\circ) = S(h,C_\g)^\circ$.
  For $g = g_0 \exp(x_1 + x_{-1})$ with
  $x_{\pm 1} \in C_\pm$, we have
  \[ \alpha^g(z) = g_0 \Exp(e^z x_1 + e^{-z} x_{-1}).\]
  For $z = a + ib$ with $0 < b < \pi$, we have for
$x_{\pm 1} \in C_\pm^\circ$ 
  \[ \Im(e^z x_1 + e^{-z} x_{-1}) = \sin(y)( x_1 - x_{-1})
    \in (C_+ + C_-)^\circ. \]
  This shows that
  \[   G^h_e\exp(C_+^\circ + C_-^\circ) = S(h,C_\g)^\circ_e
    = \exp(C_+^\circ) G^h_e\exp(C_-^\circ)  \subeq G_{\rm KMS}.\]
Here we used that, for $g = g_0 \exp(x_1 + x_{-1})$ we have
$\tau_h(g) = \tau_h(g_0) \exp(-x_1 - x_{-1})$,
so that we find for $G_{\rm KMS}$ the additional condition
that $g_0 \in G^{\tau_h}= G^h_e$ (cf.\ \cite[Cor.~2.22]{Ne22}). 

  If, conversely, $x_{\pm 1} \in C_\pm$ and
  $\alpha^g(\pi i/2) = g_0 \Exp(i(x_1 - x_{-1})) \in S(iC_\g^\circ),$ 
then
\[ x_1 - x_{-1} \in C_\g^\circ \cap \g^{-\tau_h}
  = C_+^\circ - C_-^\circ\]
(Lemma~\ref{lem:Project}). 
\end{proof}

\begin{remark} For the antiholomorphic extension
  $\oline\tau_h$ of $\tau_h$ to the complex semigroup $S(iC_\g)$, 
  the fixed point set
\[ S(iC_\g)^{\oline \tau_h}
  = G^{\tau_h} \Exp(i C_\g^{-\tau_h})
  = G^{\tau_h} \Exp(i (C_+ - C_-))\]
is a real Olshanski semigroup
in the c-dual group  $G^c$ (with respect to $\tau_h$)
with Lie algebra $\g^c = \g_0 + i \g_0^{-\tau_h}$.
The invariance condition $- \tau_h^\g(C_\g) = C_\g$ implies that
$C_\g^{-\tau_h} = C_+ - C_-$
has interior points (cf.~Lemma~\ref{lem:Project}).
\end{remark}

\begin{remark}
In the context of causal Lie groups, specified by a pair
$(G,C_\g)$ as above, $\g$ may not contain an
Euler element, but there may be an {\it Euler derivation}
  \index{Euler!deriviation  \scheiding}
$D \in \der(\g)$, i.e., $D$ is diagonalizable with eigenvalues contained
in $\{-1,0,1\}$ (see Example~\ref{ex:hcsp} below, and
Example~\ref{ex:3.10} for the case where $G = E$ is a vector space).
Then $\tau_D := e^{\pi i D}$ defines an involutive
automorphism of $\g$, and compatibility with the causal structure corresponds
to the requirements
\begin{equation}
  \label{eq:euler-der-cond}
  e^{\R D} C_\g = C_\g \quad \mbox{ and } \quad
  -\tau_D C_\g = C_\g.
\end{equation}
To implement a modular flow on $G$, we assume that
all automorphisms $\alpha_t^\g := e^{tD}$ of $\g$ integrate to
automorphisms $\alpha_t$ of~$G$. Then
$G^\flat := G \rtimes_\alpha \R$ is a Lie group
acting by causal automorphisms on $M := G$, where $(g,0)\in G^\flat$ acts
by left translation and $(0,t)$ by $\alpha_t$. This action
leaves the biinvariant cone field invariant, and the involution
$\tau_D^G$ on $G$ is {\it anti-causal}, i.e., flips the cone field
  \index{anticausal diffeomorphism \scheiding}
into its negative.
Now $h^\flat := (0,1) \in \g^\flat$ is an Euler element,
and for every $g = (g,0) \in G \subeq G^\flat$, we have
$\Ad(g)h^\flat - h^\flat \in \g$. We may therefore consider the closed
subsemigroup
\[ S(h^\flat, C_\g) := \{ g \in G \: h^\flat - \Ad(g)h^\flat \in C_\g \}\]
and find the positivity domain 
\[ W_G^+(h^\flat) = \{ g \in G \: h^\flat - \Ad(g)h^\flat \in C_\g^\circ \}.\]
With the same arguments as above, we also obtain with \cite{Ne22}
\begin{equation}
  \label{eq:G-inclusions}
  W_G^+(h^\flat) =  \exp(C_+^\circ) G^{h^\flat} \exp(C_-^\circ)
  = G^{h^\flat} \exp(C_+^\circ + C_-^\circ) = S(h^\flat, C_\g)^\circ.
\end{equation}
\end{remark}

\begin{examples} (a) Not every Euler element has a non-trivial positivity region.
  If $M = G$ is a causal Lie group with biinvariant cone
  field corresponding to $C_\g \subeq \g$, on which $G \times G$-acts,
  then every Euler element $h_0 \in \g$ specifies an Euler element
  $h := (h_0,0) \in \g^{\oplus 2}$, but the corresponding modular vector
  field is $X^G_h(g) = g.h$, and this is never contained in
  $C_g = g.C_\g$ because $h\not\in C_\g$. This follows from
  the fact that $h$ is hyperbolic and the semisimple Jordan components
  of elements in $C_\g$ are elliptic (\cite[Cor.~B.2]{NOe22}).
  We also note that $\tau_h = \tau_{h_0} \oplus \id_\g$ does not commute
  with the flip, hence cannot be implemented on the symmetric space~$G$
  in a natural way. 

  \nin (b) For {\bf left} invariant causal structure on a Lie group $G$,
  the cone $C \subeq \g \cong T_e(G)$ can be any pointed generating closed
  convex cone. Then $W_G^+(h) \not=\eset$ is equivalent to $h \in C^\circ$,
  and in this case $W_G^+(h) = G$, so that the situation is quite degenerate.
\end{examples}

\subsection{Causal flag manifolds}
\label{subsec:causal-flag-man}
  
We have seen in Section~\ref{sec:3} that Euler elements $h \in \g$
play a key role in AQFT, and that we have to understand causal homogeneous
spaces $M = G/H$ for which the positivity region
$W_M^+(h)$ is non-empty. Otherwise we have no
wedge regions for the Bisognano--Wichmann property. 
As the most well-behaved homogeneous
spaces are symmetric spaces and flag manifolds,
this is the class of manifolds for which we investigate this
question first. In Physics, the most prominent example
is the conformal compactification $(\bS^1 \times \bS^{d-1})/\{ \pm \bone\}$
of $d$-dimensional Minkowski space (Example~\ref{ex:qmink}). 

\begin{definition}
To define flag manifolds for a connected semisimple Lie group,
consider $x \in \g$ such that $\ad x$ is diagonalizable, put  
    \[ \fq_x = \sum_{\lambda \leq 0} \g_\lambda(x)
      \quad \mbox{ and } \quad
      Q_x := \{ g \in G \:  \Ad(g)\fq_x = \fq_x\}.\]
    Then $Q_x$ is called a {\it parabolic subgroup} of $G$ and
    $G/Q_x$ the corresponding {\rm flag manifold}. 
  \index{parabolic sugroup \scheiding}
  \index{flag manifold \scheiding}
\end{definition}

For the description of the causal flag manifolds, we
also need hermitian Lie algebras. 

\begin{definition} A simple Lie algebra
  $\g$ with Cartan decomposition $\g = \fk \oplus \fp$ is called
  \index{Lie algebra!hermitian \scheiding}
  {\it hermitian} 
  if the center $\fz(\fk)$ 
of a maximal compactly embedded subalgebra~$\fk$ is 
non-zero. For hermitian Lie algebras, 
the restricted root system $\Sigma = \Sigma(\g,\fa)$, 
with respect to a maximal abelian subspace $\fa \subeq \fp$, 
is either of type $C_r$ or $BC_r$ (cf.\ Harish Chandra's Theorem 
\cite[Thm.~XII.1.14]{Ne99}), 
and we say that $\g$ is 
{\it of tube type} if the restricted root system is of type $C_r$.
\index{Lie algebra!hermitian of tube type \scheiding} 
The terminology comes from the fact that the corresponding hermitian
symmetric space $G/K$ is a tube domain, i.e., biholomorphic
to $\jV_+ + i \jV \subeq \jV_\C$ for a real vector space~$\jV$
and an open convex cone $\jV_+ \subeq \jV$.
\end{definition}

\begin{theorem} \label{thm:causal-flagman}
  {\rm(Classification of causal flag manifolds, \cite{Ne25})}
  Let $G$ be a connected semisimple Lie group
  and $Q \subeq G$ be a parabolic subgroup such that $\fq$ contains no
  non-zero ideals of $\g$.   Suppose that the corresponding flag manifold  $G/Q$
  carries a $G$-invariant causal structure.
  Then $\g$ is a direct sum of hermitian simple ideals
  and   there exists an Euler element $h \in \g$ such that
  \[ \fq = \fq_h = \g_0(h) + \g_{-1}(h).\]
  If, conversely, this is the case, then $G/Q_h$ is a causal flag manifold. 
\end{theorem}

If $\g$ is simple hermitian, then an Euler element $h$ exists in $\g$
if and only if $\g$ is of tube type, and then
they are all conjugate and $h$ is symmetric
(Proposition~\ref{prop:herm}). 
We fix one and consider the corresponding
causal flag manifold $M = G/Q_h$. The tangent space in the base point
is
\[ \g/\fq_h \cong \g_1(h),\]
and the causal structure on $M$ is specified by
the cone
\[ C_+ = C_\g \cap \g_1(h),\]
where $C_\g$ is a pointed generating closed convex $\Ad(G)$-invariant
cone in $\g$. We thus obtain an
(up to sign unique) causal structure on $M$, i.e., any other
cone $C_\g'$ satisfies
$C_\g' \cap \g_1(h) = C_+$ or $C_\g' \cap \g_1(h) = -C_+$ 
(\cite[\S 3.5]{MNO23}).
This also follows from the fact that $\g_1(h)$
only contains two $e^{\ad \g_0}$-invariant non-trivial closed convex
cones (\cite[Prop.~A.I.5]{HNO96}). 

On the open dense subset of $M$ obtained by embedding
$\g_1$ via $\eta(x) := \exp(x) Q_h$, the vector field $X^M_h$ is the
{\bf Euler vector field} on $\g_1$, so that
$\eta(C_+^\circ)  \subeq W_M^+(h)$, and we actually have that 
\begin{equation}
  \label{eq:pos-reg-flag}
W := W_M^+(h) = \eta(C_+^\circ)   
\end{equation}
(\cite[Lemma~2.7]{MN25}).

\begin{proposition} \label{prop:sw-conformal}
The  compression semigroup of $W \subeq M = G/Q_h$ is 
\begin{equation}
  \label{eq:Sw-flag}
  S_W = \{ g \in G \: g.W \subeq W \}
  = \exp(C_+) G^h \exp(C_-), 
\end{equation}
where $C_\pm  = \pm  C_\g \cap \g_{\pm 1}(h).$
\end{proposition}

\begin{proof} As $G_1 = \exp(\g_1)$ is abelian,
  the inclusions    $\exp(C_+) \subeq S_W$ and $G^h \subeq S_W$
  are obvious. 

  All elements $x \in C_+$ correspond to constant vector fields on
  the open subset $\eta(\g_1) \subeq M$, 
  and since this subset is dense (Bruhat decomposition), 
  we obtain
  \[ C_+\subeq   C_M = \{ y \in \g \:  (\forall m \in M)\ X_y^M(m) \in C_m \} \]
  (cf.\ \eqref{eq:cm}). As $G$ acts on $M$ in such a way that
  the associated   homomorphism $G \to \Diff(M)$ has discrete kernel, 
  the closed convex $\Ad(G)$-invariant cone $C_M \subeq \g$ is pointed,
  and the preceding argument yields
  \[ C_{M,+} := C_M \cap \g_1 = C_+.\]
  As $C_M - C_M$ is an ideal of $\g$ and $\g$ is simple,
  the cone $C_M$ is also generating, so that we also obtain 
  $C_- = - C_M \cap \g_{-1}$ by the discussion preceding the proposition.
  Thus $\exp(C_-) \subeq S_W$ follows from Proposition~\ref{prop:LSW}.
Putting everything together, we get
\begin{equation}
  \label{eq:swsuper}
  S_W \supeq \exp(C_+) G^h \exp(C_-).
\end{equation}
The hard part is to verify equality in \eqref{eq:Sw-flag}. This involves
showing that the product set on the right is a subsemigroup
(which is not easy to see) and that it actually coincides with $S_W$, by
showing that it is maximal, hence equal to $S_W$.
We refer to \cite[Lemma~3.7, Thm.~3.8]{Ne18} for more details
and references.
\end{proof}

\begin{problem} Theorem~\ref{thm:causal-flagman}
  describes all causal flag manifolds
  $M = G/Q_h$ for semisimple Lie groups, but it makes good sense to ask
  for a result on non-semisimple groups: 
  \begin{enumerate}
  \item[(C1)] Let $x \in \g$ be such that $\ad x$ is diagonalizable, put  
    \[ \fq_x = \sum_{\lambda \leq 0} \g_\lambda(x)
      \quad \mbox{ and } \quad
      Q_x := \{ g \in G \:  \Ad(g)\fq_x = \fq_x\}.\] 
    Show that, if $M = G/Q_x$ is causal, then $x$ must be an Euler element
    (cf.\ \cite{Ne25} for similar arguments). Note that $x \in\fq_x$ implies that
    $\fq_x$ is self-normalizing, so that $\L(Q_x) = \fq_x$. 
  \item[(C2)] Assume that $h \in \g$ is an Euler element.
    Determine those manifolds $M = Q/Q_h$ with an invariant causal structure
    on which $G$ acts effectively.
  \end{enumerate}
\end{problem}

\begin{remark} \label{rem:affine-case}  (The affine case) 
 Particular examples arise for Euler elements with $\g_{-1} = \{0\}$.
 Then $M = G/Q_h = \eta(\g_1) \cong \g_1$ and we may assume that
 $G \cong \g_1 \rtimes G_0$.
 
This covers the action of $\Aff(\R)_e$ on $\R$
and of the Poincar\'e group on Minkowski space. 
More generally, we may start with a finite-dimensional real linear space
$E$ and a pointed generating convex cone $C \subeq E$.
We write $\Aut(C) \subeq \GL(E)$ for its linear automorphism
group, which is a closed subgroup. Then $G :=E \rtimes \Aut(C)$ acts
transitively on the affine causal manifold~$M := E$, endowed
with the constant cone field $C_m = C$ for $m \in M$.
Further, $h := (0,\id_E)$ is an Euler element with $\g_{-1} = \{0\}$,
$\Aut(C) = G^h$ and $\g_1 \cong E$.
The corresponding positivity region is 
\[ W := W_M^+(h) = C^\circ, \]
and its compression semigroup is readily identified with
\[     S_W = C \rtimes \Aut(C) \]
because $\Aut(C) \subeq S_W$ (cf.~also Lemma~\ref{lem:cone-endos}).

Lie algebra elements $(b,a) \in \g = \g_1 \rtimes \g_0$
correspond to affine vector fields
  $X(x) = b + ax$, and such a vector field is positive on all of $E$ if and
  only if $b + a E \subeq C$, which is equivalent to $a = 0$ and $b \in C$.
  Therefore the invariant cone $C_M \subeq \g$ coincides with
  $C \subeq \g_1$. 
\end{remark}

\nin {\bf Appendix:  Euclidean Jordan algebras.}

  The causal flag manifolds of simple Lie groups are precisely
  the conformal compactifications of simple euclidean Jordan
  algebras (\cite{Ne25}, \cite{Be96}).

    The following table lists the simple hermitian Lie algebras of tube type,
  the only non-simple Lie algebra listed is 
  $\so_{2,2}(\R) \cong \so_{1,2}(\R)^{\oplus 2}$, 
  corresponding to the non-simple Jordan algebra 
  $\jV = \R^{1,1} \cong \R \oplus \R$
  (the Minkowski plane, decomposing in lightray coordinates). 

\begin{center} 
  \begin{tabular}{|l|l|l|l|l|l|l|}\hline
\text{Hermitian Lie algebra} & $\g$ &  $\fsp_{2r}(\R)$ & $\su_{r,r}(\C)$  &  $\so^*(4r)$ 
& $\fe_{7(-25)}$  &   $\so_{2,d}(\R)$\\ \hline
\text{Euclidean Jordan algebra} &  $\jV$ &$\Sym_r(\R)$& $\Herm_r(\C)$  & $\Herm_r(\H)$
  & $\Herm_3(\bO)$ &$\R^{1,d-1}\phantom{\Big)}$  \\ \hline
Rank fo  $\jV$  &  &$r$  &$r$ &$r$  & $3$ & $2$   \\
\hline 
\end{tabular}\\[2mm]
{Table 4: Hermitian Lie algebras 
  of tube type and euclidean Jordan algebras}  
\end{center}

The corresponding flag manifolds $M$ have interesting
geometric interpretations.
For $\g = \so_{2,d}(\R)$, the manifold $M$  is the isotropic
quadric
\begin{equation}
  \label{eq:def-Q}
 Q = Q(\R^{2,d})
 = \{ [(x_1, x_2, \bx)] \in \R^{2,d} \:
 x_1^2 + x_2^2 - \bx^2 = 0\} 
\end{equation}
in the real projective space $\bP(\R^{2,d})$, and for 
\[ \Omega := \Omega_{2r} := \pmat{ 0 & \bone_r \\ -\bone_r & 0} \in M_{2r}(\K), \quad \K = \R, \C, \H,\]
we obtain a uniform realization of the Lie algebras 
$\sp_{2r}(\R), \fu_{r,r}(\C)$ and $\so^*(4r)$ as
\begin{equation}
  \label{eq:fuomega}
  \fu(\Omega,\K^{2r}) := \{ x \in \gl_{2r}(\K) \: x^* \Omega + \Omega x =0\}.
\end{equation}
Then $M$ is the space of
maximal isotropic subspaces $L \subeq \K^{2r}$ with respect to the
skew-hermitian form $\beta(z,w) := z^* \Omega w$ on $\K^{2r}$

\begin{examplekh} \label{ex:qmink}
  (The Lorentzian case)
  For $d$-dimensional Minkowski space
  $\jV = \R^{1,d-1}$ and $\tilde\jV := \R \oplus \jV \oplus \R$, we
  realize the conformal completion $M$ of $\jV$ as the quadric
\begin{equation}
  \label{eq:quadtildev}
 Q := Q(\R^{2,d}) := \{ [\tilde v] \in \bP(\tilde\jV) \:
 \tilde\beta(\tilde v, \tilde v) = 0\},
\end{equation}
where $\tilde\beta$ is the symmetric bilinear form on
$\tilde\jV \cong \R^{2,d}$, given
by
\[ \tilde\beta(x,y) = x_1 y_1 + x_2 y_2 - x_3 y_3 - \cdots - x_{d+2} y_{d+2}.\]
The natural dense open embedding $\R^{1,d-1} \to Q$ is given by
\begin{equation}
  \label{eq:eta-flow}
  \eta \: \jV \to Q,
  \quad  \eta(v) := \Big[ \frac{1 - \beta(v,v)}{2} : v : - \frac{1 + \beta(v,v)}{2}\Big]  \in Q \subeq \bP(\R^{2,d}), 
\end{equation}
corresponding to the action of the translation group
$(\jV,+) \cong \g_1(h)$ on $Q$ (cf.\ \cite[\S 17.4]{HN12}, \cite{Ne25}).
\end{examplekh}

\subsection{Causal symmetric spaces}
\label{subsec:css}

In the preceding subsection we discussed causal flag manifolds,
the ``most symmetric'' class of causal homogeneous space.
They represent ``conformal geometries''. Metric geometries,
which are more rigid, correspond in this context to causal symmetric spaces. 

We start with some
terminology and observations concerning symmetric spaces and
symmetric Lie algebras (cf.\ \cite{HO97}, \cite{CW70}, \cite{KO08}): 

\begin{itemize}
\item A  {\it symmetric Lie algebra}
\index{symmetric Lie algebra! $(\g, \tau)$  \scheiding} 
\index{Lie algebra!symmetric, $(\g,\tau), \g = \fh \oplus \fq$ \scheiding } 
is a pair $(\g,\tau)$, where $\g$ is a finite-dimensional real Lie algebra 
and $\tau$ is an involutive automorphism of~$\g$. 
We write 
\begin{equation}
  \label{eq:hq-decomp}
 \g = \fh \oplus \fq \quad \mbox{ with } \quad 
\fh = \g^\tau= \ker(\tau -\bone) \quad \mbox{ and } \quad 
\fq = \g^{-\tau}= \ker(\tau +\bone).
\end{equation}
\index{symmetric space $M = G/H$ \scheiding} 
\item A {\it symmetric space} is a homogeneous space
  of the form $M = G/H$, where $H \subeq G^\tau$ is an open subgroup
  and $\tau \in \Aut(G)$ an involution.
Then $H$ contains the identity component 
$G^\tau_e := (G^\tau)_e$. We call the triple 
$(G,\tau,H)$ a {\it symmetric Lie group} because this triple specifies 
the symmetric space~$M$. For a more intrinsic approach to symmetric
spaces as ``reflection spaces'', we refer to \cite{Lo69}. 
\index{Lie group!symmetric \scheiding} 
\index{symmetric Lie algebra!causal,  $(\g, \tau, C)$  \scheiding} 
\item A {\it causal symmetric Lie algebra}
is a triple $(\g,\tau,C)$, where $(\g,\tau)$ is a symmetric Lie algebra 
and $C \subeq \fq$ is a pointed generating 
closed convex cone, invariant under the group $\Inn_\g(\fh) := \la e^{\ad \fh}\ra
\subeq \Aut(\g)$. 
We call $(\g,\tau,C)$
\begin{itemize}
\index{symmetric Lie algebra!compactly causal\scheiding } 
\index{symmetric Lie algebra!non-compactly causal \scheiding} 
\index{symmetric Lie algebra!modular causal $(\g, \tau, C,h)$ \scheiding } 
\item {\it compactly causal~(cc)} if 
$C$ is {\it elliptic} in the sense that, for $x \in C^\circ$
(the interior of $C$), the 
operator $\ad x$ is semisimple with purely imaginary spectrum. 
\item {\it non-compactly causal (ncc)} if 
$C$ is {\it hyperbolic} in the sense that, for $x \in C^\circ$, the 
operator $\ad x$ is diagonalizable. 
\end{itemize}
\item For a symmetric Lie algebra $(\g,\tau)$, the
  pair $(\g^c, \tau^c)$ with   
  $\g^c := \fh + i \fq$ and $\tau^c(x + iy)  = x- iy$
  is called the {\it c-dual symmetric Lie algebra}.
\index{symmetric Lie algebra!c-dual, $\g^c = \fh + i \fq$\scheiding } 
\item   A~{\it modular causal symmetric Lie algebra} is a 
  quadruple $(\g,\tau, C, h)$, where
  $(\g,\tau,C)$ is a causal symmetric Lie algebra,
  $h \in \g^\tau$ is an Euler element,  
  and the involution $\tau_h$ satisfies $\tau_h(C) = - C$.
\end{itemize}

\begin{remark} \label{rem:modular-duality}
(a) $(\g,\tau,C)$ is non-compactly causal
if and only if $(\g^c,\tau^c,iC)$ is compactly causal. 

\nin (b)   $(\g,\tau, C, h)$ is modular if and only if the
$c$-dual quadruple $(\g^c, \tau^c, i C, h)$ is modular.
\end{remark}

\begin{remark} If $C_\g \subeq \g$ is a pointed generating invariant
  cone in $\g$ and $h \in \g$ an Euler element satisfying
  $-\tau_h(C_\g) = C_\g$, then there is a variety of associated
  causal symmetric Lie algebras:
  \begin{enumerate}
  \item[\rm(a)] $(\g^{\oplus 2}, \tau_{\rm flip}, C, (h,h))$ 
with $C = \{ (x,-x) \: x \in C_\g\}$ is a modular 
causal symmetric Lie algebra of group type
(cf.\ Subsection~\ref{subsec:group-type}). 
\item[\rm(b)] $(\g_\C, \sigma, i C_\g, h)$
  is a modular non-compactly causal symmetric Lie algebra of complex type.
  \item[\rm(c)] $(\g, \tau_h, C_+ - C_-, h)$ 
    is a modular compactly causal symmetric Lie algebra,
    called if {\it Cayley type} if $\g$ is simple. 
Note that $C_\g^{-\tau_h} = C_+ - C_-$ by Lemma~\ref{lem:Project}(ii).
  \item[\rm(d)] $(\g, \tau_h, C_+ + C_-, h)$ 
  is a modular non-compactly causal symmetric Lie algebra of Cayley type.
  \end{enumerate}
Note that
\begin{equation}
  \label{eq:wickrot}
  \kappa_h = e^{\frac{\pi i}{2} h} \: \g_\C \to \g_\C
  \quad \mbox{ satisfies} \quad
  \kappa_h(\g) = \g^c,    
\end{equation}
so that $(\g, \tau_h) \cong (\g^c, \tau_h)$ as symmetric Lie algebras. 
Moreover, 
\[ \kappa_h(C_\g^{-\tau_h}) = \kappa_h(C_+ - C_-)
  = i (C_+ + C_-) \]
implies 
\begin{align} 
  \label{eq:wick-dual}
(\g, \tau_h, C_+ - C_-,h)   \cong & (\g^c, \tau_h, i (C_+ + C_-),h)
  \cong  (\g, \tau_h, C_+ + C_-,h)^c
\end{align} 
as modular causal symmetric Lie algebras. 
\end{remark}

\begin{remark} (Cartan motion groups)
  If $(G,\tau, H)$ is a connected symmetric Lie group
  corresponding to the causal symmetric Lie algebra
  $(\g,\tau,C)$, then we obtain on $\fq$ a constant causal
  structure defined by $C$ that is invariant under the action
  of the semidirect product group $\fq \rtimes H$
  (a so-called {\it Cartan motion group}), where~$H := G^\tau_e$
  (cf.\ Remark~\ref{rem:affine-case}).
  If, in addition, $h \in \fh$ is an Euler element with
  $\tau_h(C) = - C$, then the pair $(\fq,C)$ satisfies
  the assumptions (A1)-(A3) from Section~\ref{subsubsec:one-par-affine}.
  So $\fq$ is an affine causal symmetric space, and
  \eqref{eq:w-decomp0} in Subsection~\ref{subsubsec:one-par-affine}
  implies that
  \[ W_\fq^+(h) =  C_+^\circ \oplus \fq_0(h) \oplus C_-^\circ
    \quad \mbox{ for } \quad C_\pm := \pm C \cap \fq_{\pm 1}(h).\] 
\end{remark}

\begin{remark} (Lorentzian symmetric spaces)
  Important examples of causal symmetric spaces
  are those where causal structure comes from a Lorentzian form,
  for instance Minkowski space,
  de Sitter space $\dS^d$ and anti-de Sitter space $\AdS^d$
  (see Examples~\ref{exs:ds-ads}).

 If $M_1 = G_1/H_1$ is a Lorentzian symmetric space
  and $M_2 = G_2/H_2$ is a Riemannian symmetric space,
  then the product $M = M_1 \times M_2$ is also Lorentzian.
  Natural examples are
  \[  \AdS^p \times \bS^q \quad \mbox{ and } \quad
    \dS^p\times \bHy{q}. \]
  The compact group $\U_n(\C)$ carries a one-parameter family of
  biinvariant Lorentzian
  structures. We refer to \cite{Ne25} for more details
  and conformal embeddings of these spaces for $p + q= d$ into $Q(\R^{2,d})$. 
\end{remark}

\subsubsection{Causal symmetric spaces of group type}

We assume first that $\g$ is
simple hermitian and that $h_0 \in \g$ is an Euler element.
Then any $\Ad(G)$-invariant closed convex pointed generating cone
$C_\g \subeq\g$ specifies a biinvariant causal structure
on the group $G$, considered as a symmetric space, on which
$G \times G$ acts transitively.
Then the Euler elements $h \in \g^{\oplus 2}$ for which
$W_G^+(h) \not=\eset$ are conjugate to $h = (h_0, h_0)$ for some
Euler element  $h_0 \in \g$, and in this case
\begin{equation}
  \label{eq:WG+}
  W_G^+(h) = \exp(C_+^\circ) G^h \exp(C_-^\circ) = S(h, C_\g)^\circ 
\end{equation}
follows from Theorem~\ref{thm:semigroups-equal},
cf.~also \eqref{eq:wg+} and~\eqref{eq:sh0g}. Note that $W_G^+(h)$ only depends
on the cones $C_\pm$, hence is unique up to sign
if $\g$ is simple (\cite[\S 3.5]{MNO23}).

\subsubsection{Modular compactly causal  symmetric spaces} 

If $(\g, \tau, C)$ is an irreducible compactly causal symmetric Lie algebra
which is not of group type, then $\g$ is simple
hermitian (\cite[Prop.~2.13]{NO23b} and 
$c$-duality\begin{footnote}{The dual symmetric Lie algebra $(\g^c,\tau^c,iC)$ is irreducible,
    non-complex and non-compactly causal. Hence $\g^c$ is simple.
    Moreover $\tau^c = \tau_h \theta^c$ for a causal Euler element
    $h \in i\fq = \fq^c$ and a Cartan involution
    $\theta^c$ of $\g^c$.
    Then $\g^c_0 = \fz_{\g^c}(h) = \fh^c_\fk \oplus \fq^c_\fp
    = \fh_\fk \oplus i \fq_\fk$ implies that $\fz_\g(ih) = \fk$. So
  $ih \in \fz(\fk)$ implies that $\g$ is hermitian.}\end{footnote}).
If $\g$ contains an Euler element,
then $\g$ is of tube type, $\Ad(G)$ acts transitively on~$\cE(\g)$
(Proposition~\ref{prop:herm}) 
and there exist $\tau$-fixed Euler elements
(Corollary~\ref{cor:mod-struc} in Appendix~\ref{subsubsec:2.8.6}). Now the embedding 
\begin{equation}
  \label{eq:cc-grp-emb}
 (\g,\tau,C) \into (\g^{\oplus 2}, \tau_{\rm flip}, \tilde C), \quad
 x \mapsto (x, \tau(x))
\end{equation}
can be used to determine the positivity region $W_M^+(h)$
by using the results for spaces of group type.

On the global side, we consider the action of $G$ on $G$
by $g.x := g x \tau(g)^{-1}$, corresponding to the embedding
\eqref{eq:cc-grp-emb}. Then
$M := G.e$ is the identity component in the
fixed point set of the involution $g^\sharp := \tau(g)^{-1}$
and a symmetric space with symmetric Lie algebra $(\g,\tau)$.
If $C = C_\g \cap \fq$, then we even have an embedding of causal
symmetric spaces which is equivariant for the modular flow.
This easily implies with \eqref{eq:WG+} that 
\begin{equation}
  \label{eq:conmg}
  W_M^+(h) = W_G^+(h) \cap M \ {\buildrel \eqref{eq:WG+} \over =}\  S(C_\g, h)^\circ \cap M,
\end{equation}
and 
\[ W = G^h_e.\exp(C_+^\circ + C_-^\circ)
  \quad \mbox{ for } \quad  C_\pm = \pm C_\g^{-\tau} \cap \g_{\pm 1}(h).\] 
The compression semigroup of $W$ is 
\begin{equation}
  \label{eq:SW-cc}
  S_W = G_W \exp(C_+ + C_-) \quad \mbox{  with  } \quad G_W = G^h_e H^h.
\end{equation}
Furthermore, $G_W$ is open in $G^h$
(\cite[Thm.~9.1]{NO23a}). We refer to \cite{NO23a} for details.

\subsubsection{Non-compactly causal  symmetric spaces}
\label{subsec:ncc-spaces}

Irreducible non-compactly causal
symmetric Lie algebras $(\g, \tau, C)$ are $c$-dual to
irreducible compactly causal ones.
The dual $(\g^c, \tau^c)$ is of group type if and only if
$\g$ is a complex simple Lie algebra
(considered as a real one) and $\tau$ is antilinear, so that
$\fh =  \g^\tau$ is a real form and $\g \cong \fh_\C$.
Then $(\g^c, \tau^c) \cong (\fh^{\oplus 2}, \tau_{\rm flip})$.
The existence of the causal structure implies that $\fh$
is hermitian, but these real forms are precisely
those for which the corresponding conjugation
$\tau$ is of the form $\theta \tau_h$, where $h \in \g$
is an Euler element (\cite[Thm.~4.21]{MNO23}).
So Euler elements in complex simple Lie algebras
automatically determine causal symmetric Lie algebras
of complex type.

This picture prevails for general simple Lie algebras~$\g$.
Whenever $h\in \g$ is an Euler element and
$\theta$ a Cartan involution with $\theta(h) = -h$, then 
$\tau := \theta \tau_h$ is an involution of $\g$.
Further $h$ is also Euler in the complexification 
$\g_\C$, on which the antilinear extension $\oline\theta$ of
$\theta$ to $\g_\C$ defines a Cartan involution.
Then $\oline\tau := \oline\theta \tau_h$ is an antilinear extension of the 
involution $\tau = \theta\tau_h$ on $\g$, 
and $\g^c := (\g_\C)^{\oline\tau} = \fh + i \fq$ is a hermitian real form
of $\g_\C$  with $\fz(\fk^c) = \R i h$. 
For any invariant cone $C_{\g^c} \subeq \g^c$
  containing $-ih$, we then obtain by 
  \[ C := i C_{\g^c} \cap \fq \]
  an $e^{\ad \fh}$-invariant cone in $\fq$ with $h \in C^\circ$.
This specifies an embedding 
\[ (\g, \tau, C) \into (\g_\C, \oline\tau, i C_{\g^c}) \]
of causal symmetric Lie algebras of non-compact type,
and we thus obtain a parametrization of irreducible
non-compactly causal symmetric Lie algebras
in terms of Euler elements (\cite[Thm.~4.21]{MNO23}):

\begin{theorem}   \label{thm:ncc-classif}
{\rm(Classification of irreducible non-compactly
    causal symmetric Lie algebras)}
  Let $\g$ be a simple real Lie algebra
  and pick a Cartan involution $\theta$ with
  $\theta(h) = -h$. Then the assignment
  \[ h \mapsto (\g, \tau_h \theta, C) \]
  described above defines a bijection from the set
  $\cE(\g)/\Inn(\g)$ of conjugacy classes of Euler elements
  to the isomorphism classes of irreducible non-compactly
  causal symmetric Lie algebras with maximal $\Inn(\fh)$-invariant cone.
\end{theorem}
 
\begin{theorem} \label{thm:unique-euler-ncc} {\rm(\cite[Cor.~6.3]{MNO24})} 
  For an irreducible non-compactly causal symmetric space   
$M = G/H$ there exists a unique conjugacy class of Euler elements
$\cO_h \subeq \g$ for which $W_M^+(h)\not=\eset$.
In particular $W_M^+(-h) = \eset$ if $h$ is not symmetric. 
\end{theorem}

Let us assume for simplicity that $M = G/H$ is {\it minimal},
i.e., that all other causal symmetric spaces with the
same triple $(\g,\tau,C)$ are coverings of $M$
(this is $M_{\rm ad}$ in the notation of Appendix~\ref{app:wedgeinncc}).
In addition, we assume that the causal structure
is maximal, i.e., that
$C \subeq \fq$ is a maximal proper $\Inn(\fh)$-invariant convex cone in~$\fq$.
We choose a Cartan involution $\theta$ commuting with $\tau$.
  Let $\fq_\fk= \fq \cap \fk$ for
  a Cartan decomposition $\g = \fk \oplus \fp$ with~$h \in \fq_\fp$ and
  consider the domain
\[   \Omega_{\fq_\fk}
    = \Big\{ x \in \fq_\fk \: r_{\rm Spec}(\ad x) < \frac{\pi}{2}\Big\},\]
  where $r_{\Spec}(\ad x)$ is the spectral radius of $\ad x$.
Then the connected component 
$W := W_M^+(h)_{eH}$ of the base point $eH$ in the positivity domain
$W_M^+(h)$ is the region 
\begin{equation}
  \label{eq:wm+ss}
  W = G^h_e \exp(\Omega_{\fq_\fk}).eH
\end{equation}
(\cite[Thm.~3.6]{MNO24}  and \eqref{eq:wm+min} in
Appendix~\ref{app:wedgeinncc}). The semigroup
$S_W$ actually is a group, as we shall see in
Theorem~\ref{thm:SW-ncc-spaces} below. 

\begin{definition} \label{def:ncc-reductive}
  (The canonical ncc symmetric space associated
  to a reductive Lie group)  
Assume that $\g$ is reductive
and that $G$ is a corresponding connected Lie group.
We choose an involution $\theta$ on $\g$ in such a way that 
  it fixes the center pointwise and restricts to a Cartan involution
  on the semisimple Lie algebra $[\g,\g]$. Then
  the corresponding Cartan decomposition $\g = \fk \oplus \fp$ satisfies
$\fz(\g) \subeq \fk$. We write $K := G^\theta$ for the subgroup of
$\theta$-fixed points in~$G$.

For an Euler element $h \in \g$, we
write $\g = \g_1 \oplus \g_2$, where $\fz(\g) \subeq \g_1$, 
$h = h_z + h_2 \in \fz(\g) \oplus \g_2$, and $\g_2$ is minimal, i.e.,
$\g_2$ is the ideal generated by
the projection $h_2$ of $h$ to the commutator
algebra. We consider the involution $\tau$ on $\g$ with
\[ \tau\res_{\g_1} = \id_{\g_1}
  \quad \mbox{ and } \quad  
 \tau\res_{\g_2} = \tau_h \theta.\] 
We {\bf assume} that $\tau$ integrates to an involutive automorphism
$\tau^G$ of $G$.
  We write $\fh := \g^\tau$ and $\fq := \g^{-\tau} \subeq \g_2$
  for the $\tau$-eigenspaces in $\g$.
    Then there exists in $\fq$ 
    a unique maximal pointed generating $e^{\ad \fh}$-invariant
    cone $C$ containing $h_2$ in its interior
    (\cite[Thm.~4.21]{MNO23} deals with minimal cones, but
    the minimal and the maximal cone     determine each other by duality). 
    We choose an open $\theta$-invariant
    subgroup $H \subeq G^\tau$, satisfying $\Ad(H)C = C$.
    This  is always the case for $H_{\rm min} = G^\tau_e$ (the minimal choice of
    an open subgroup~$H\subeq G^\tau$).
    By \cite[Cor.~4.6]{MNO23}, $\Ad(H)C = C$ is 
    equivalent to $H_K= H \cap K$ fixing~$h$, so that
    we also have a maximal choice $H_{\rm max} = K^{\tau_h,h} \exp(\fh_\fp)$,
    which leads to a minimal causal symmetric space $G/H_{\rm max}$.
    We call 
    \begin{equation}
      \label{eq:ncc}
      M = G/H \cong G_2/H_{2, {\rm max}}
    \end{equation}
    the {\it (minimal) non-compactly causal symmetric space} corresponding
to $\g$, resp., to~$G$. 
\end{definition}

\subsubsection{Non-triviality of  wedge regions}

Wedge regions have been studied in detail for compactly and
non-compactly causal symmetric spaces 
in \cite{NO23a} and \cite{NO23b, MNO24}, respectively.
For causal flag manifolds (Section~\ref{subsec:causal-flag-man}), 
we refer to \cite{MN25} and \cite{Ne25}
for a discussion of wedge regions and
to \cite{Be96, Be98, Be00, BN04} for Jordan theoretic aspects.
The case of general Lie groups is still poorly understood;
but see \cite{BN25} and \cite{Oeh22a, Oeh23}. We shall return to this topic 
below.

\begin{problem} Let $h \in \g$ be an Euler element and
  $M = G/H$ a causal homogeneous space.
  \begin{enumerate}
  \item[(a)]   How can we determine effectively
    if $W_M^+(h) \not=\eset$? A sufficient condition is
    given in Proposition~\ref{prop:2.13}. 
  \item[(b)] If $h$ is symmetric, i.e., $-h = \Ad(g)h$ for some $g \in G$, 
then $W_M^+(-h) = g.W_M^+(h)$ is nonempty 
    if $W_M^+(h) \not=\eset$. The converse is not true
    by Example~\ref{ex:hcsp}, where
    $W_M^+(\pm h) \not=\eset$ but $h$ is not symmetric. 
    However, for irreducible non-compactly causal symmetric spaces 
    it is true  (Theorem~\ref{thm:unique-euler-ncc}).
    Is there a natural characterization of
    those cases where $W_M^+(\pm h) \not=\eset$? 
  \item[(c)] How are these conditions related to the existence of fixed
    points of the modular
    vector field $X_h^M$, i.e., to $\cO_h \cap \fh \not=\eset$? 
  \end{enumerate}
\end{problem}

  \begin{examplekh} \label{ex:sl3} 
    In this context, the Euler element
    \[   h_1 := \frac{1}{3}\pmat{2  & 0 & 0 \\ 0 & -1 & 0\\
        0 & 0 & -1}       \in \fsl_3(\R) \] 
from Example~\ref{ex:3.10cd} is instructive. It is not
  symmetric; note that $-h_1 \in \cO_{h_2} \not=\cO_{h_1}$.
  The corresponding non-compactly
  causal symmetric space is
  \[ M =G.I_{1,2} = \{ gI_{1,2} g^\top \: g \in \SL_3(\R) \}
    \subeq \Sym_3(\R), \qquad I_{1,2} = \diag(1,-1,-1).\]
  Then $I_{1,2} \in W_M^+(h_1) \not=\eset$, but 
  $W_M^+(-h_1) =\eset$ and the vector field
  \[ X_{h_1}^M(x) = h_1 x + x h_1 \]
  has no zeros on~$M \subeq \Sym_3(\R)$. In fact, if $X_{h_1}^M(x) = 0$,
  then $x$ anticommutes with~$h_1$.
  If $v \in \R^3$ is an $h_1$-eigenvector with
  $h_1 v = \lambda v$, it follows that 
  $h_1 xv = - \lambda xv$; contradicting the fact that
  the eigenvalues of $h_1$ are $\frac{2}{3}$ and $-\frac{1}{3}$. 
  We refer to \cite[Prop.~5.7]{Ne25} for a detailed discussion
  of this class of spaces and their modular flows. 
  \end{examplekh}

\subsection{Wedge regions in non-compactly causal symmetric spaces}
\label{app:wedgeinncc}

Having introduced several classes of causal homogeneous spaces,
we now turn to wedge regions in irreducible non-compactly causal
symmetric spaces, such as de Sitter space. 
We mainly put some of the results from \cite{MNO24} into the
context in which they are used below.  

Here $G$ denotes a connected simple Lie group, $h \in \g$ is an Euler
element, $\tau =  \theta \tau_h$ for a Cartan involution $\theta$
satisfying $\theta(h) = - h$ and $M = G/H$ is a corresponding non-compactly
causal symmetric space, where the causal structure is specified by a
maximal $\Ad(H)$-invariant closed convex cone $C \subeq \fq$ satisfying
$h \in C^\circ$ (Theorem~\ref{thm:ncc-classif}). 

First we consider the ``minimal'' space associated to the triple
$(\g,\tau,C)$. It is obtained as
\[ M_{\rm ad} := G_{\ad}/H_{\ad},\]
where
\[ G_{\rm ad} := \Ad(G) = \Inn(\g) \quad \mbox{ and }\quad
  H_{\rm ad} := K_{\rm ad}^h \exp(\fh_\fp) \subeq G_{\ad}^\tau\]
(see \cite[Rem.~4.20(b)]{MNO23} for more details).
In this space, the positivity domain $W_{M_{\rm ad}}^+(h)$ is connected
by  \cite[Thm.~7.1]{MNO24}, and
\cite[Thm.~8.2, Prop.~8.3]{MNO24} imply that the positivity domain
is given by 
\begin{equation}
  \label{eq:wm+min}
  W_{M_{\rm ad}}^+(h) = G^h_e\exp(\Omega_{\fq_\fk}).eH_{\rm ad}.
\end{equation}
By \cite[Rem.~4.20(a)]{MNO23}, 
we have
$H = H_K \exp(\fh_\fp)$ with $H_K \subeq K^h$, so that
$\Ad(H) \subeq H_{\rm ad}$. Therefore
\[  q \colon M \to M_{\rm ad}, \quad gH \mapsto \Ad(g) H_{\rm ad} \]
defines a covering of causal symmetric spaces. The stabilizer in $G$
of the base point in $M_{\rm ad}$ is the subgroup
\[ H^\sharp := \Ad^{-1}(H_{\rm ad})
  = K^h \exp(\fh_\fp) \]
because $Z(G) = \ker(\Ad) \subeq K^h$.
Note that $H^\sharp$ need not be contained in $G^\tau$ because
$\tau$ may act non-trivially on $K^h$
(cf.\ Remark~\ref{rem:sl2-cover}).
So we may consider $M_{\rm ad}$ as the homogeneous $G$-space 
\[ M_{\rm ad} \cong G/H^\sharp.\]
As $q$ is a $G$-equivariant covering of causal manifolds,
\begin{align*}  
  W_M^+(h)
  &= q^{-1}(W_{M_{\rm ad}}^+(h))
  = q^{-1}(G^h_e\exp(\Omega_{\fq_\fk}).eH_{\rm ad}) 
  = G^h_e\exp(\Omega_{\fq_\fk})H^\sharp.eH \\
&  = G^h_e\exp(\Omega_{\fq_\fk})K^h.eH
  = G^h_eK^h\exp(\Omega_{\fq_\fk}).eH
  = G^h\exp(\Omega_{\fq_\fk}).eH,
\end{align*}
and the inverse image under the map $q_M \colon G \to G/H = M$ is
therefore given by
\begin{align*}
 q_M^{-1}(W_M^+(h))
&  = G^h \exp(\Omega_{\fq_\fk}) H^\sharp
  = G^h \exp(\Omega_{\fq_\fk}) K^h \exp(\fh_\fp)\\
  &  = G^h  K^h\exp(\Omega_{\fq_\fk}) \exp(\fh_\fp)
= G^h\exp(\Omega_{\fq_\fk}) \exp(\fh_\fp).
\end{align*}

Next we recall from \cite[Cor.~8.4]{MNO24} that the map
\begin{equation}
  \label{eq:ad-fiber-diffeo}
 G^h_e \times_{K^h_e} \Omega_{\fq_\fk} \to W^+_{M_{\rm ad}}(h), \quad
 [g,x] \mapsto g\exp(x)H_{\ad}
\end{equation}
is a diffeomorphism. Therefore $W^+_{M_{\rm ad}}(h)$ 
is an affine bundle over the Riemannian symmetric space $G^h_e/K^h_e$,
hence contractible and therefore simply connected.
So its inverse image $W_M^+(h)$ in $M$ is a union of open connected
components, all of which are mapped diffeomorphically onto
$W^+_{M_{\rm ad}}(h)$ by~$q_M$, and the group
$\pi_0(K^h) \cong K^h/K^h_e$ acts transitively on the set of connected
 components. 
It follows in particular that the diffeomorphism \eqref{eq:ad-fiber-diffeo}
lifts to a diffeomorphism 
\begin{equation}
  \label{eq:fiber-diffeo}
 G^h_e \times_{K^h_e} \Omega_{\fq_\fk} \to W := W^+_M(h)_{eH}, \quad
 [g,x] \mapsto g\exp(x)H.
\end{equation}

\begin{remark} \label{rem:sl2-cover} (The possibilities for $H$) 
 For $m \in \N \cup \{\infty\}$, 
    let $G_m$ be a connected Lie group with Lie algebra $\g = \fsl_2(\R)$
    and $|Z(G_m)| = m$. For $m \in \N$ this means that
    $Z(G_m) \cong \Z/m\Z$ and
    $G_m$ is an $m$-fold covering of $\Ad(G_m) \cong \PSL_2(\R) \cong G_1$.
    Note that $G_2 \cong \SL_2(\R)$. 
    Further $G_\infty \cong\tilde\SL_2(\R)$ is simply connected
    with $Z(G_\infty) \cong \Z$. 

    We consider the Cartan involution $\theta(x) = - x^\top$, 
    the Euler element
    \[ h = \frac{1}{2}{\pmat{1 & 0\\ 0 & -1}} \quad \mbox{   and } \quad
      z_\fk = \frac{1}{2}\pmat{ 0 &  \\ -1 & 0}\in \fk = \so_2(\R), \]
    which satisfies
    $e^{2\pi z_\fk} = -\bone$. Then
  \[ K = \exp(\R z), \quad Z(G_m) = \exp(2\pi\Z z_\fk), \quad \mbox{ and }  \quad 
    \tau_h(\exp tz_\fk) = \tau(\exp tz_\fk) = \exp(-tz_\fk)\] 
  because $\tau = \theta \tau_h$.   We conclude that
  \[ K^\tau = \{e\} \quad \mbox{  if } \quad m = \infty
    \quad \mbox{ and } \quad K^\tau = \{e, \exp(m\pi z_\fk)\}\
    \ \mbox{ otherwise}.\]
For $m = \infty$, $H = G_m^\tau$ is connected. For $m \in \N$, the group
$G_m^\tau= K^\tau \exp(\fh)$ has two connected components,
but if $m$ is odd, then $K^\tau$ does not fix the Euler element~$h\in C^\circ$.
Therefore only $H := \exp(\fh)$ leads to a causal symmetric space $G_m/H$.
If $m$ is even, then $H$ can be either $(G_m)^\tau_e$ or $G_m^\tau$.

In $G_1 \cong \PSL_2(\R)$, the subgroup $H$ corresponds to $\SO_{1,1}(\R)_e$
and the non-compactly causal symmetric space 
$G_1/H \cong \dS^2$ is the $2$-dimensional de Sitter space.

The universal covering $\tilde\dS^2$ is obtained  for $m = \infty$,
$G_\infty = \tilde\SL_2(\R)$, and then $H = \exp(\fh)$ is connected.
All other coverings of $\dS^2$ are obtained as $G_m/H$ for $H = \exp(\fh)$. 
\end{remark}

\subsection{Modular reductive compactly causal symmetric spaces}
\label{subsubsec:2.8.6}

In \cite{NO23a} positivity regions of modular flows have been studied in modular compactly causal symmetric spaces (cf.\ Subsection~\ref{subsec:css}),
because the existence of an Euler element in $\g$ already
implies the existence of a modular structure
(Corollary~\ref{cor:mod-struc}), and this is needed for wedge regions
and positivity regions to be defined.

In this subsection we present a quite general result on the ``automatic''
existence of modular structures for reductive compactly causal
symmetric spaces (Proposition~\ref{prop:mod-struc-gen}). 

We shall need the
following observation, which is a consequence of \cite[Prop.~3.12]{Oeh22b}.

\begin{proposition} \label{prop:mod-struc}
  {\rm(Modular structures via Jordan involutions)}
  Let $\g$ be simple hermitian, 
$h \in \g$ an Euler element, and $\jV := \g_1(h)$ the corresponding euclidean 
Jordan algebra. For every involutive automorphism $\alpha \in \Aut(\jV)$,  
there exists a unique 
automorphism $\sigma_\alpha \in \Aut(\g)$ with $\sigma_\alpha\res_\jV = \alpha$,
and then $(\g, \tau_h \sigma_\alpha, C_\g^{-\tau_h \sigma_\alpha}, h)$ is modular
compactly causal. Conversely,  every simple
modular compactly causal Lie algebra  is of this form.
\end{proposition}

Since the modular flow on $G/H$ generated by an Euler element
$h \in\g$ has a fixed pixed if and only if $\cO_h$ intersects $\fh$,
the following corollary implies the existence of fixed points
for irreducible compactly causal symmetric spaces. 

\begin{corollary} \label{cor:mod-struc} {\rm(Fixed points of modular flows)} 
  Let  $(\g,\tau,C)$ be simple compactly causal
  and $h \in \g$ an Euler element. Then $\cO_h \cap \fh \not=\eset$.
\end{corollary}

\begin{proof} Since $\cE(\g) = \cO_h$, the assertion follows from
  Proposition~\ref{prop:mod-struc}, which asserts that
  $\tau$ fixes some Euler element $k$ with
  $\tau = \tau_k \sigma_\alpha$.
\end{proof}
  
\begin{proposition} \label{prop:mod-struc-gen}
{\rm (Modular structures on reductive compactly causal symmetric Lie algebras)} 
Let $(\g,\tau,C)$ be an effective reductive compactly causal
symmetric Lie algebra with $C^\circ \cap [\g,\g] \not=\eset$.
If $\g$ contains a non-central Euler element, then
there exist an Euler element $h' \in \fq = \g^\tau$
and a cone $C' \subeq C$ such that
$(\g,\tau,C',h')$ is a modular causal symmetric Lie algebra.
\end{proposition}

\begin{proof} (cf.\ \cite[Lemma~3.3]{Ne25}) (a) First we use the Extension Theorem
    \cite[Thm.~2.4]{NO23a} to find a pointed generating
    $\Inn(\g)$-invariant
  cone $C_\g$ in $\g$ with $-\tau(C_\g) =C_\g$ and $C = C_\g \cap \fq$.
  It follows in particular that $\g$ is quasihermitian, i.e.,
  its simple ideals are either compact or hermitian (\cite[Def.~VII.2.15]{Ne99}).
  We write
  $\g = \fz(\g) \oplus \g_h \oplus \fu$ with $\fu$ compact semisimple
  and $\g_h$ a sum of hermitian simple ideals.
  Projecting along the compact semisimple ideal $p_\fu \: \g \to \fz(\g) + \g_h$
  (the fixed point projection
  of the compact group $\Inn(\fu)$), it follows that
  \[ C_\g^\circ \cap (\fz(\g) + \g_h) = p_{\fu}(C_\g^\circ) \not=\eset\]
  (cf.\ Lemma~\ref{lem:coneint} in Appendix~\ref{app:E})   and likewise 
  \[ C_\g^\circ \cap \g_h = p_{\fu}(C_\g^\circ \cap [\g,\g]) \not=\eset.\]
  Here we use that our assumption implies that
  \begin{equation}
    \label{eq:cintersect}
 \eset \not= C^\circ \cap [\g,\g]
 = C_\g^\circ \cap \fq \cap [\g,\g].
  \end{equation}
  
  \nin (b) Let $h_1 \in\g$ be an Euler element.
  Then the ideal $\g_1 \trile \g$ generated by $[h_1,\g]$ has trivial
  center and contains no compact ideal, hence only simple hermitian
  ideals with an Euler element appear, so that they are  of tube type.
  The $\tau$-invariant ideal $\g_2 := \g_1 +\tau(\g_1)$ also has
  only simple hermitian tube type ideals. We may thus replace
  $h_1$ by an Euler element $h_2 \in [\g,\g]$ generating the ideal $\g_2$.

\nin (c)  Let $\fj \trile \g_2$ be a minimal $\tau$-invariant ideal.
  Then either $\fj$ is simple or a sum of two simple ideals exchanged
  by $\tau$. In the latter case, $\fj \cong \fb \oplus \fb$ with
  $\tau$ acting by $\tau(a,b) = (b,a)$. 
  Any generating Euler element
  in $\fj$ has non-zero components, and all these are conjugate under
  inner automorphisms (Proposition~\ref{prop:MN21:3.2}). 
  So the projection of $h_2$ to $\fj$ is conjugate
  to an element of the form $(x,x) \in \fj^\tau$.
  If $\fj$ is simple, then $\fh = \g^\tau$ contains an Euler element 
  by Proposition~\ref{prop:mod-struc}. 
  Putting these results on minimal
  invariant ideals together, we see that $h_2$ is conjugate to an
  element of $\g^\tau$, i.e., $\g^\tau$ contains an Euler element $h_3$
  generating~$\g_2$. 

  \nin (d) The involution $\tau_3 := \tau_{h_3}$ commutes with $\tau$. 
  Next we observe that $\g^{-\tau_3} \subeq \g_2$ is contained in
  a sum of hermitian simple
  ideals. Therefore \cite[Prop.~2.7(d)]{NO23a} implies that
the cones $C_\g^{\rm min}$ and $C_\g^{\rm max}$ are $-\tau_{3}$-invariant and
\[ (C_\g^{\rm max})^{-\tau_{3}} = (C_\g^{\rm min})^{-\tau_{3}} = C_\g^{-\tau_{3}}.\]
As $\g_2$ intersects the interior of $C_\g$ by \eqref{eq:cintersect},
and the cone 
$C_\g^{\rm min} \subeq \g_2$ is generating, it
follows with Lemma~\ref{lem:coneint} in Appendix~\ref{app:E} that
\[ \eset \not = (C \cap \g_2^{-\tau_3})^\circ
  =(C_\g \cap \g_2^{-\tau_3})^\circ
  =C_\g^\circ \cap \g_2^{-\tau_3}. \] 
Now 
  \[ C' := C \cap (-\tau_{3}(C))  \subeq \fq \]
  is an $\Inn(\fh)$-invariant pointed cone in $\fq$.
  As it contains $C_\g \cap \g_2^{-\tau_3} \cap \fq
  = C \cap \g_2^{-\tau_3}$, hence interior points of $C_\g$,
  it has non-trivial interior.
    Therefore $(\g,\tau, C', h_3)$ is modular.
\end{proof}

\subsection{The geometric KMS condition}
\label{subsec:1.2}

On the geometric side, KMS conditions as in Definition~\ref{def:linear-KMS}
can be modeled as follows.
We consider a connected complex manifold $\Xi$, endowed with a smooth
$\R$-action $(\sigma_t)_{t \in \R}$ by holomorphic maps
and an antiholomorphic involution
$\tau_\Xi$ commuting with each $\sigma_t$.
We further assume that $\Xi$ is an open domain in a larger complex manifold 
and that the boundary $\partial \Xi$ contains a real submanifold $M$ with 
the property that, for every $\tau_{\Xi}$-fixed point 
$m$ in an open subset of $\Xi^{\tau_{\Xi}}$, the orbit map 
$\R \to \Xi, t \mapsto \sigma^m(t)$ extends to a holomorphic 
map $\sigma^m \: \cS_{\pm\pi/2} \to \Xi$, which further 
extends to a continuous map 
\begin{equation}
  \label{eq:extcond}
\sigma^m \: \oline{\cS_{\pm\pi/2}} \to \Xi \cup M \quad \mbox{ with } \quad 
\sigma^m(\pm i \pi/2) \in M.
\end{equation}
We have already seen these structures for $M = G$
in Theorem~\ref{thm:semigroups-equal}.

\begin{definition}
We then write
\[  W_{\rm KMS} \subeq M \]
for the set of all elements whose orbit map
$\sigma^m \: \R \to M$ extends analytically to a continuous map $\oline{\cS_\pi} \to \Xi \cup M$,
analytic on $\cS_\pi$, such that
\[ \sigma^m(\pi i) = \tau_\Xi(m).\]
Note that this implies that  $\sigma^m(t + \pi i) = \tau_\Xi(\sigma^m(t))$ 
for $t \in \R$, and hence, by antiholomorphic continuation, that
$\sigma^m(\oline z + \pi i) = \tau_\Xi(\sigma^m(z))$ for $z \in \cS_\pi,$ 
so that
\begin{equation}
  \label{eq:pq-3-38}
  p := \sigma^m\Big(\frac{\pi i}{2}\Big) \in \Xi^{\tau_\Xi}
  \quad \mbox{ and } \quad
  m = \sigma^p\Big(-\frac{\pi i}{2}\Big).  
\end{equation}
\end{definition}

\begin{examples} 
    (Domains in $\C$) \label{ex:g.1} 
In one-dimension we have the following standard examples 
of simply connected proper domains in $\C$ with 
their natural actions of $\R \times \{\pm 1\}$. 

\nin (a) (Strips) On the  strip 
$\cS_\pi = \{ z \in \C \: 0 < \Im z < \pi\}$ 
we have the antiholomorphic involution 
$\tau_{\cS_\pi}(z) = \pi i+ \oline z$ with fixed point set 
\[ \cS_\pi^{\tau_{\cS_\pi}} = \Big\{ z \in \cS_\pi \: \Im z = \frac{\pi}{2}\Big\}.\] 
The group $\R$ acts by translations commuting with
$\tau_{\cS_\pi}$ via $\sigma_{t}(z) = z + t$, 
$M := \R \cup (\pi i + \R) = \partial \cS_\pi$ is a real submanifold, and 
for $\Im z = \pi/2$, the orbit map $\sigma^z(t)$ extends to the 
closure of the strip $\cS_{\pm\frac{\pi}{2}}$ with 
$\sigma^z\big(\pm \frac{\pi i}{2}\big) = z \pm \frac{\pi i}{2} \in M$. 
For the strip we have $W_{\rm KMS} = \R$.

\nin (b) (Upper half-plane) On the upper half-plane 
$\C_+  = \{ z \in \C \: \Im z > 0\}$, we have the 
antiholomorphic involution $\tau_{\C_+}(z) = -\oline z$ 
and the action of $\R$ by dilations $\sigma_t(z) = e^t z$. 
Here $M := \R = \partial \C_+$ is a real submanifold, and 
for $z = i y$, $y > 0$, the orbit map $\sigma^z(t) = e^t  z$ extends to the 
closure of the strip 
$\cS_{\pm\frac{\pi}{2}}$ with 
$\sigma^z\big(\pm \frac{\pi i}{2}\big) = \pm i (i y) = \mp y$.
In this case $W_{\rm KMS} = \R_+$. 

\nin (c) (Unit disc) On the unit disc 
$\bD = \{ z \in \C \: |z| < 1\}$ we have the antiholomorphic 
involution $\tau_\bD(z) = \oline z$ and the 
action of $\R\cong \SO_{1,1}(\R)_e$ by the fractional linear maps 
\begin{equation}
  \label{eq:discact}
 \sigma_t(z) = \frac{ \cosh(t/2) z + \sinh(t/2)}{\sinh(t/2)z + \cosh(t/2)}.
\end{equation} 
Here $M := \bS^1 = \partial \bD$ is a real submanifold, and 
for $z \in \bD \cap \R = \bD^{\tau_\bD}$,
the orbit map $\sigma^z(t)$ extends to the 
closure of the strip $\cS_{\pm\pi/2}$ with 
\[ \sigma^z(\pm \pi i/2) 
= \frac{ \cos(\pi/4) z \pm i \sin(\pi/4)}{\pm i \sin(\pi/4)z + \cos(\pi/4)}
= \frac{z \pm i}{\pm iz + 1} 
= \mp i \cdot \frac{z \pm i}{z \mp i}\]
and
\[ W_{\rm KMS} = \big\{ x - i \sqrt{1 - x^2} \: -1 < x < 1\big\}\]
is the lower half circle.

The biholomorphic maps 
\begin{equation}
  \label{eq:cayley}
 \Exp \: \cS_\pi \to \C_+, \ \   z \mapsto e^z \quad \mbox{ and } \quad 
\Cay \: \C_+ \to \bD, \ \  \Cay(z) := \frac{z-i}{z+ i} 
\end{equation}
are equivariant for the described actions of $\R \times \{ \pm 1\}$
on the respective domains.
\end{examples}

\begin{lemma} \label{lem:b.0}
  For a proper simply connected domain
  $\Omega \subeq \C$,
two antiholomorphic involutions on $\Omega$ 
are conjugate under the group  $\Aut(\Omega)$ of biholomorphic automorphisms.
In particular, they have fixed points. 
\end{lemma}

\begin{proof} (\cite[Lemma~B.1]{ANS22})
  By  the Riemann Mapping Theorem, we may assume that 
$\Omega = \bD$ is the unit disc. Let $\sigma \: \bD \to\bD$ be an antiholomorphic 
involution. Then $\sigma$ is an isometry for the hyperbolic metric. 
Therefore the unique
midpoint of $0$ and $\sigma(0)$ is fixed by $\sigma$. Conjugating 
by a suitable automorphism of~$\bD$, we may therefore assume that 
$\sigma(0) = 0$. Then $\psi(z) := \sigma(\oline z)$ is a holomorphic automorphism 
fixing $0$, hence of the form $\psi_\theta(z) = e^{i\theta} z$
for some $\theta \in \R$,
so that $\sigma(z) = e^{i\theta} \oline z = \gamma(\oline{\gamma^{-1}(z)})$ 
for $\gamma(z) = e^{i\theta/2}z$. Thus $\sigma$ is conjugate to
complex conjugation. 
\end{proof}

\begin{proposition} Up to automorphisms of $\R \times \{\pm 1\}$, any 
  antiholomorphic action of this group on a proper simply connected 
domain $\Omega \subeq \C$ is equivalent to the one in 
{\rm Examples~\ref{ex:g.1}(a)-(c)}.
\end{proposition}

\begin{proof} Up to conjugation with biholomorphic maps, we may assume that 
  $\tau(z) = \oline z$ on $\Omega = \bD$
  (Lemma~\ref{lem:b.0}). 
Now we simply observe that the centralizer 
of $\sigma_{-1}$ in the group $\PSU_{1,1}(\C) \cong \Aut(\bD)$ is 
$\PSO_{1,1}(\R)$, and, up to automorphisms of $\R \times \{\pm 1\}$, 
this leads to the action in \eqref{eq:discact}.
\begin{footnote}
  {The automorphisms of the group $\R^\times$ have the form
    $\phi(x) = \sgn(x)|x|^{\lambda}$, $\lambda\in \R$. }
\end{footnote}
\end{proof}

\begin{examples} \label{exs:kms-dom}
  (Examples of KMS domains) \\
  \nin (a) If $G = E \rtimes_\alpha \R$ as in Example~\ref{ex:3.10},
  then $\Xi := E + i C^\circ$ is a tube domain in the complex
  vector space $E_\C$ with $E \subeq \partial\Xi$,
  and Theorem~\ref{thm:semigroups-equal}   implies that
  \[ E_{\rm KMS} = C_+^\circ \oplus E_0 \oplus C_-^\circ,\]
  which in this concrete case can be verified easily. 

  \nin (b) For a causal Lie group $G$ 
  and the complex semigroup $\Xi = S(i C_\g^\circ)$, we  obtain
  from   Theorem~\ref{thm:semigroups-equal} that 
  \[ G_{\rm KMS} = \exp(C_+^\circ) G^h_e \exp(C_-^\circ)
    \quad \mbox{ for } \quad C_\pm = \pm C_\g \cap \g_{\pm 1}(h). \] 

  \nin (c) For a non-compactly causal symmetric space
  $M = G/H$, realized in the boundary of a complex crown domain
  $\Xi \subeq G_\C/K_\C$ as the $G$-orbit of
  $o_M := \exp\big(-\frac{\pi i}{2}h\big).K_\C$
  (\cite{GK02, GKO03, NO25}) we also have
  \[ M_{\rm KMS} = W_M^+(h)_{eH} = G^h_e \exp(\Omega_{\fq_\fk}).eH \]
(cf.~Subsection~\ref{subsec:semisim} and \cite[Thm.~8.2]{MNO24}). 
\end{examples}

\begin{small}
\subsection{Exercises for Section~\ref{sec:3b}}

\begin{exercise}
  \label{exer:stereo}
  The Cayley transform $C \:\R \to \bS^1, C(x) = \frac{i-x}{i+x}$
  has a natural interpretation in terms of the stereographic projection.
  Show that, projecting the point $1 + 2 i x$ on the tangent line
  through $1 \in \bS^1 \subeq \C$ with the center $-1 \in \bS^1$ onto
  the circle yields $C(x)$.
\end{exercise}

\begin{exercise}\label{ex21} (Positivity region of a Lorentz boost) 
Consider the two-dimensional Minkowski space $\R^{1,1}=\{(x_0,x_1)\;|\;x_0,x_1\in \R\}$ and the $2d$-Poincar\'e group  $G:=\R^{1,1}\rtimes \SO_{1,1}(\R)_e$. 
\begin{enumerate} 
\item[\rm(i)] Show that $$h=\begin{pmatrix}
0&1\\
1&0
\end{pmatrix}\in \so_{1,1}(\R)$$
is an Euler element in $\g=\R^{1,1}\rtimes \so_{1,1}(\R)$. 
\item[\rm(ii)] Show that $M=\R^{1,1}\cong G/H$, for $H=\SO_{1,1}(\R)_e$, and determine the positivity region $W_M^+(h)$
  for the causal structure on $M$, specified by the cone
  \[  C:=\{(x_0,x_1)\;|\; |x_1|\leq x_0\}\subseteq \R^{1,1}\cong T_0 M.\] 
\end{enumerate} 
\end{exercise}

\begin{exercise} \label{exer:mink-deSitter} (Minkowski space) 
  For the Minkowski space $E := \R^{1,d}$, let $\eta(x,y):=x_0y_0-\langle \bx,\by\rangle$ and consider the closed positive light cone
  \[ C = \{ x\in E \;|\; x_0 \geq 0, \quad \eta (x,x) \geq 0 \},\]
so that $E$ is a causal manifold, for $C_m:=C$ and $m\in E$. Show that:
  \begin{enumerate}
  \item[\rm(i)] $G := \SO_{1,d}(\R)_e$ acts by causal automorphisms on~$E$ and classify $G$-orbits in $E$. Hint: Witt's Theorem, which asserts that any
    $\eta$-isometry between subspaces of $E$ extends to an isometry of the whole space~$(E,\eta)$.
  \item[\rm(ii)] Show that, for $0 \not=m \in E$, the orbit $\cO_m := G.m$ is a causal
    manifold if $T_m(\cO_m) = m^{\bot_\eta}$ intersects $C^\circ$, i.e.,
    $T_m(\cO_m)$ contains timelike vectors. 
  \item[\rm(iii)] Show that conditon (ii) is satisfied for $\be_d = (0,\ldots,0,1)$. If $d>1$, then its orbit is the de Sitter space
    \[  \dS^d = \{ (x_0,\bx) \:  x_0^2 - \bx^2 =- 1\}\]
    and its stabilizer is $H:=\SO_{1,d-1}(\R)_e$, i.e. $\dS^d\cong \SO_{1,d}(\R)_e/\SO_{1,d-1}(\R)_e$.
  \item[(iv)] Show that $\g=\so_{1,d}(\R)$ is a direct sum $\g=\fh\oplus \fq$ for the stabilizer subalgebra
    \[ \fh=\g_{\be_d} =
      \{A\in \so_{1,d}(\R)\;|\;A\be_d=0\}\cong \so_{1,d-1}(\R), \]  and
\[ \fq=\left\{L(x_0,\bx) \;|\; (x_0,x)\in \R^{1,d-1}\right\}\qfor L(x_0,\bx):=\begin{pmatrix}
    0&0&x_0\\
    0&0&-\bx\\
    x_0&\bx^T&0
\end{pmatrix}.
\] 
\item[(v)] For $x \in \be_d^\bot$, we have
  $\Exp_{\be_d}(x):=\exp(L(x))\be_d=C(\eta(x,x))\be_d+S(\eta(x,x))x$, for 
\[ C(z)=\sum_{k=0}^\infty \frac{z^k}{(2k)!}\qand S(z)=\sum_{k=0}^\infty \frac{z^k}{(2k+1)!}.\] 
Show that $\Exp_{\be_d}(x)\in \dS^d$ and $\difftev \Exp_{\be_d}(tx)\in C$
    if and only if $x \in C$. 
\item[(vi)] Let $\cE(\g)\subseteq \g$ be the subset of Euler elements. Show that 
$$h=\begin{pmatrix}
0&0&1\\
0&0_{d-1}&0\\
1&0&0
\end{pmatrix}\in \cE(\g)\quad\text{and}\quad W_M^+(h)=\dS^d\cap W_R,
$$
for $M=\dS^d$ and the \text{right Rindler wedge} $W_R:=\{(x_0,\bx)\in E\;|\; |x_0|<x_d\}$.
\item[(vii)] For $G_e^h:=\langle \exp(\g_0(h))\rangle$, show that $$W_M^+(h)=G_e^h \exp(\Omega_{\fq_\fk}).\be_d,$$
where $\Omega_{\fq_\fk}:=\left\{R(x)\ \: \|x\|<\pitwo\right\}$,
$\fq_\fk:=\left\{R(x)\ \: x\in \R^{d-1}\right\}$ and
$R(x)=\begin{pmatrix}
0&0&0\\
0&0&-x\\
0&x^T&0
\end{pmatrix}.$ 
  \end{enumerate}
\end{exercise}

\begin{exercise} (Causal symmetric submanifolds of $\Sym_n(\R)$) 
In the linear space $E := \Sym_n(\R)$, we consider the closed convex cone
  $C$ of positive semidefinite matrices, so that $E$, endowed with the constant cone
  field $C_m = C$, for $m \in E$, is a causal manifold. Show that:
  \begin{enumerate}
  \item[\rm(i)] $\GL_n(\R)$ acts via $g.A = gAg^\top$ by causal automorphisms on~$E$. 
  \item[\rm(ii)] For $G \subeq \GL_n(\R)$, any orbit
    $M := \{ g A g^\top \: g \in \SL_n(\R)\}$ for which
\[ T_A(M) \cap C^\circ \not=\eset \]  inherits the structure of a causal manifold. 
  \item[\rm(iii)] Let $I_{p,q} = \bone_p \oplus - \bone_q$. When is the orbit $M$ of
    $I_{p,q}$ under $\SL_n(\R)$ a causal manifold?
    If this is the case, find an Euler element $h \in \fsl_n(\R)$ for
    which $W_M^+(h) \not=\eset$. 
  \end{enumerate}
\end{exercise}

\begin{exercise} \label{exer:sl2-Lorentz}
  We consider the following linear bijection
  \[ \phi \: \R^3  \to \fsl_2(\R), \quad
      x=(x_0,x_1,x_2) \mapsto 
      \tilde x := \frac12
      \left(\begin{array}{cc}x_1&-x_0-x_2\\x_0-x_2&-x_1\end{array}\right), \] 
  and
\[ \sigma_0=\frac12\left(\begin{array}{cc}0&-1 \\ 1&0
\end{array}\right),\; 
\sigma_1=\frac12\left(\begin{array}{cc}0&1\\ 1&0
\end{array}\right),\;
\sigma_2=\frac12\left(\begin{array}{cc}1&0\\0&-1
\end{array}\right).
\]
Show that
\begin{enumerate}
\item[\rm(a)] $\phi^{-1}(X) = ({-2}\Tr(X\sigma_0),{2}\Tr(X\sigma_2),-
{2}\Tr(X\sigma_1))$. 
\item[\rm(b)] The Lorentz form $x^2 = x_0^2 - x_1^2 - x_2^2$ on $\R^3$ 
  corresponds to the determinant by $x^2=4\det \tilde x$.
  In particular, $x\in \dS^2$ if and only if $\det\tilde x=-\frac14$. 
\item[\rm(c)] Show that
  \[ \Lambda \: \SL_2(\R) \to \SO_{1,2}(\R)_e, \quad
    \Lambda(g) = \phi^{-1} \circ \Ad(g) \circ \phi  \]
  defines a $2$-fold covering with kernel $Z(\SL_2(\R)) = \{\pm \bone\}$.
\item[\rm(d)]  The one-parameter groups 
  $\lambda_{\sigma_i}(t)=\exp(\sigma_it)\in\SL_2(\RR), i =1,2,$
are lifts of Lorentz boosts
and $r(\theta)=\exp(-\sigma_0\theta)$ is the one-parameter group lifting the 
 space rotations 
\begin{equation}
  \label{eq:R}
 \Lambda(r(\theta)) = R(\theta)=\pmat{ 1 & 0 & 0 \\ 
0 & \cos \theta&-\sin\theta\\ 
0 & \sin\theta&\cos\theta}. 
\end{equation}
\end{enumerate}
\end{exercise}

\end{small}

\section{Analytic continuation of orbit maps and crown domains}
\label{sec:4}

In this section, we turn to constructions of nets
for a given antiunitary representation $(U,\cH)$
of $G_{\tau_h} = G \rtimes \{\id_G, \tau_h \}$, where
$h \in \g$ is an Euler element for which $\tau_h$ exists on $G$ 
(cf.\ Section~\ref{sec:3}).
The representation $U$ specifies in particular a standard subspace $\sV = \sV(h,U)$ by 
\[  \Delta_\sV := e^{2 \pi i \cdot \partial U(h)} \quad \mbox{ and } \quad
   J_\sV := U(\tau_h).\]
 In view of Proposition~\ref{prop:standchar},
 this standard subspace also has a description
in terms of a condition resembling the KMS (Kubo--Martin--Schwinger)
condition for states of operator algebras (cf.\ \cite{BR96}). 
It consists of all vectors $\xi \in \cH$
for which the orbit map $U^\xi_h(t) := U_h(t)\xi$ 
extends analytically 
to a continuous map $U^\xi \: \oline{\cS_\pi} \to \cH$
satisfying
$U^\xi_h(\pi \ie) = J_\sV \xi$. Then $U^\xi(\pi \ie/2) \in \cH^{J_\sV}$ is a
$J_\sV$-fixed vector whose orbit map extends to the strip
$\cS_{\pm \pi/2}$ (Proposition~\ref{prop:standchar} and
Subsection~\ref{subsec:1.2}). 

To extend this one-dimensional picture to
higher dimensional Lie groups $G$, one has to specify complex
manifolds $\Xi$ (crown domains), containing $G$ as a totally real submanifold,
to which orbit maps of $J_\sV$-fixed analytic vectors extend.
These domains $\Xi$ generalize the strip 
$\cS_{\pm \pi/2}$, corresponding to the one-dimensional Lie group $G = \R$.
For semisimple Lie groups complex crown domains
are obtained from crowns of Riemannian symmetric spaces $G/K$ 
of non-compact type  as their inverse images
$\Xi_{G_\C}$ in the complexified group $G_\C$.
These are particular well-known examples that have been used in
harmonic analysis for some time (cf.\ \cite{AG90}, \cite{KSt04}, \cite{FNO25a}),
but they have never  been studied systematically for general Lie groups, 
as outlined in this section (cf.~\cite{BN25}). 

The key idea for constructing nets of real subspaces 
is to extend the description of the standard subspace $\sV(h,U)$, 
attached to the antiunitary representation $(U,\cH)$ in
terms of a KMS condition on the complex strip, to general Lie groups.
This leads us to natural requirements on $\Xi$ and~$h$ described
in Subsection~\ref{sec:crownlie}. 
From the Euler Element Theorem~\ref{thm:2.1} we know that a necessary condition
for the existence of a net of real subspaces
satisfying (Iso), (Cov), (RS) and
(BW), is that $h \in \g$ is an Euler element
and that we may assume that  $\tau_h^\g := \ee^{\pi \ie \ad h}$
integrates to an involutive automorphism $\tau_h$ on $G$, so that we can form
the group 
\[ G_{\tau_h} = G \rtimes \{\bone,\tau_h\} :=  G \rtimes_{\tau_h} \{\pm 1\}.\]

For a crown domain $\Xi$, containing $G$, and an antiunitary
representation $(U,\cH)$ of $G_{\tau_h}$, we write 
\[ \cH^\omega(\Xi) \subeq \cH \]
for the subspace of those analytic vectors $v$ whose orbit map
$U^v \: G \to \cH$ extends analytically to $\Xi$.
To specify the necessary boundary behavior of the 
extended orbit maps on $\Xi$, we put $J := U(\tau_h)$ and
reall from Subsection~\ref{subsubsec:bv-d=1}  the dense subspace 
\[   \cH^J_{\rm temp} \subeq \cH^J = \Fix(J) \] 
of those $J$-fixed vectors
$v$, for which the orbit map $U^v_h(t) = U(\exp th)v$
extends to a holomorphic map $\cS_{\pm\pi/2} \to \cH$,  and the limit
\[ \beta^+(v) := \lim_{t \to -\pi/2} U^v_h(\ie t) \]
exists in the subspace $\cH^{-\infty}_{U_h}$ 
of distribution vectors of the one-parameter group~$U_h$
in the weak-$*$ topology. For any real linear subspace
$\sF \subeq \cH^\omega(\Xi) \cap \cH^J_{\rm temp},$
we then obtain a real subspace
\begin{equation}
  \label{eq:def3}
  \sE := \beta^+(\sF) \subeq \cH^{-\infty}_{U_h} \subeq \cH^{-\infty},
\end{equation}
and from this space we construct a net of real subspaces indexed
by open subsets $\cO \subeq G$ via 
\begin{equation}
  \label{eq:HE}
  \sH_\sE^G(\cO) := \oline{\spann_\R U^{-\infty}(C^\infty_c(\cO,\R))\sE}.
\end{equation}
The operators $U^{-\infty}(\phi)$, $\phi \in C_c^\infty(G,\R)$,  map 
$\cH^{-\infty}$ into $\cH$ because they are adjoints of
continuous operators $U(\phi) = \int_G \phi(g) U(g)\, dg
\colon \cH \to \cH^\infty$
  (Appendix~\ref{app:c}). 
Accordingly, the closure in \eqref{eq:HE} is taken with respect
to the topology of $\cH$. 
It is easy to see that the net $\sH^G_\sE$ satisfies (Iso) and (Cov),
and it is a key result  that, 
  if $\sF$ is $G$-cyclic in the sense that $U(G)\sF$
  spans a dense subspace,
  then the net $\sH^G_\sE$ also satisfies (RS) and (BW)  
  (Theorem~\ref{thm:4.9}).

\subsection{Crown domains in Lie groups} 
\label{sec:crownlie}

We consider the following setting:
\begin{itemize}
\item $G$ is a connected Lie group whose universal complexification
  $\eta_G \: G \to G_\C$ has discrete kernel.
  If $G$ is simply connected, then $G_\C$ is the simply connected
  group with Lie algebra $\g_\C$, and this condition is satisfied
  (cf.\ \cite[Thm.~15.1.4(i)]{HN12}). 
\item $h \in \g$ is an Euler element, 
i.e., $(\ad h)^3=\ad h \not=0$, 
for which the associated involution $\tau_h^\g = \ee^{\pi \ie \ad h}$
of $\g$ integrates to an involutive automorphism
$\tau_h$ of $G$. This is always the case if  $G$ is simply connected.
We write
  \begin{equation}
    \label{eq:gtauh}
G_{\tau_h} = G \rtimes \{\bone,\tau_h\} :=  G \rtimes_{\tau_h} \{\pm 1\}
\end{equation}
for the corresponding semidirect product
and abbreviate $\tau_h = (e,-1)$ for the corresponding element of
  $G_{\tau_h}$. 
The universality of $G_\C$ implies the existence of a unique antiholomorphic involution $\oline\tau_h$ on $G_\C$ satisfying $\oline\tau_h \circ \eta_G = \eta_G \circ \tau_h$.
\end{itemize}

We now present an axiomatic specification of crown domains for $G$
to which orbit maps of $J$-fixed vectors in antiunitary
representations may extend in such a way that boundary values
lead to nets of real subspaces on $G$ as in \eqref{eq:HE}.

\index{crown domain!of Lie group $G$, $\Xi$  \scheiding } 
\index{crowned Lie group \scheiding } 

\begin{definition} \label{def:1.1b}
  (a) A {\it $(G,h)$-crown domain} is a connected
  complex manifold $\Xi$ containing
$G$ as a closed totally real submanifold, such that the following conditions
are satisfied:
\begin{enumerate}
\item[\rm(Cr1)] 
The natural action of $G_{\tau_h}$ on $G$
by  $(g,1).x = gx$ and $(e,-1).x = \tau_h(x)$ for $g,x \in G$, 
  extends to an
  action on $\Xi$, such that $G$ acts by holomorphic maps  
  and $\tau_h$ by an antiholomorphic involution, denoted $\oline\tau_h$. 
  These extensions are unique because $G$ is totally real in $\Xi$.
\item[\rm(Cr2)] 
There exists an $e$-neighborhood 
$W^c$ in $\Xi^{\oline\tau_h}$  
(the set of $\oline\tau_h$-fixed points in $\Xi$) such that, for every $p \in W^c$, 
the orbit map
  \[ \alpha^p \: \R \to \Xi, \quad  \alpha^p(t) := \exp(th).p \]
  extends  to a holomorphic map
  $\cS_{\pm \pi/2} \to \Xi$. 
\item[\rm(Cr3)] 
The map $\eta_G \: G \to G_\C$ 
  extends to a holomorphic ($G_{\tau_h}$-equivariant) 
  map $\eta_\Xi \: \Xi \to G_\C$ which is a covering of the open subset
  $\Xi_{G_\C} := \eta_\Xi(\Xi)$ 
so that we have the commutative diagram 
  \[\xymatrix{
  G\ \ar@{^{(}->}[r] \ar[dr]_{\eta_G} & \Xi \ar[d]^{\eta_\Xi} \\
 & \Xi_{G_\C}  }\]
\end{enumerate} 

\nin (b) We call the 
triple $(G,h,\Xi)$   a {\it crowned Lie group}. 
  For crowned Lie groups 
$(G_j, h_j, \Xi_j)$, 
  $j = 1, 2$, a
  {\it morphism of crowned Lie groups}
  is a holomorphic map $\phi \: \Xi_1 \to \Xi_2$,
  restricting to a Lie group morphism
  $\phi_G \: G_1 \to G_2$ such that
  \[\L(\phi_G) h_1 = h_2    .\]
  This implies that $\phi_G \circ \tau_{h_1} = \tau_{h_2} \circ \phi_G$,
  and, by analytic continuation, 
  $\phi$ intertwines the $G_{1,\tau_{h_1}}$-action
  on $\Xi_1$ with the $G_{2,\tau_{h_2}}$-action on $\Xi_2$.
\end{definition}

\begin{remark}
	\label{remW}
	 (On  the condition (Cr2)) 
Let us denote by $\cW$ the set of all points $p\in \Xi^{\oline\tau_h}$ 
whose orbit map $\alpha^p$ has the holomorphic extension property 
referred to in the axiom (Cr2). 
 The fixed point set $\Xi^{\oline\tau_h}$ is invariant under the
  connected subgroup $G^h_e$ which commutes with the involution
  $\oline\tau_h$ on $\Xi$. As it also commutes with $\exp(\R h)$, 
 the set $\cW$ is also $G^h_e$-invariant. 
  This shows that, if $e$ belongs to the interior of $\cW$, then 
  there exists a $G^h_e$-invariant connected open subset $W^c\subseteq\cW$ with $e\in W^c$.

  If $\Omega' \subeq \g^{-\tau_h^\g}$ is a convex open $0$-neighborhood 
  with $\exp(\ie\Omega') \subeq \eta_\Xi(\Xi)$, then the exponential
  function $\ie\Omega' \to \eta_\Xi(\Xi)$ lifts uniquely to an analytic map
  \begin{equation}
    \label{eq:exp-omega'}
 \exp \: \ie\Omega' \to \Xi \quad \mbox{ with } \quad 
 \exp(0) = e, \ \mbox{ the unit element in } G \subeq \Xi.
  \end{equation}
 Since $\ie\g^{-\tau_h^\g}\subseteq\g_\C^{\oline\tau_h^\g}$, we have 
  $\exp(\ie\Omega') \subeq \Xi^{\oline\tau_h}$, so that
  any $G^h_e$-invariant domain $W^c$ contains an open subset of the form
  $G^h_e.\exp(\ie\Omega')$. 
  For this reason, we may assume that $W^c$ is of this form. 

  For any element $x = x_1 + x_{-1} \in \g^{-\tau_h^\g}$ with
  $x_{\pm 1} \in \g_{\pm 1}(h)$, we have
  \[ \zeta(\ie x) = x_1 - x_{-1} \in \g^{-\tau_h^\g}\quad \mbox{ for } \quad
    \zeta := e^{-\frac{\pi \ie}{2} \ad h} \in \Aut(\g_\C).\]
  We thus associate to $W^c = G^h_e.\exp(\ie \Omega')$, the open subset
  \begin{equation}
    \label{eq:defwg}
 W^G := G^h_e.\exp(\Omega)  \subeq G\quad \mbox{ for }\quad
 \Omega := \zeta(\ie \Omega') \subeq \g^{-\tau_h^\g}.
  \end{equation}
  \end{remark}

  \begin{remark} (On the existence of crown domains)
    If $\eta_G$ is injective, we may identify $G$ with a closed subgroup
  of $G_\C$. In this case $\Xi$ is an open subset of $G_\C$, invariant
  under the $G_{\tau_h}$-action and 
  (Cr3) 
  is trivially satisfied. 
  Typical domains with this property are easily constructed
  as products $\Xi := G \exp(\ie \Omega)$, where $\Omega \subeq \g$
  is an open convex 
  $0$-neighborhood invariant under $-\tau_h^\g$, for which the
  polar map
  \[ G \times \Omega \to G_\C, \quad (g,x) \mapsto g \exp x \]
  is a
  diffeomorphism onto an open subset. 
  Then
\[ \Xi^{\oline\tau_h}
  = G^{\tau_h}.\exp(\ie \Omega^{-\tau_h^\g}) \supeq G^h_e.\exp(\ie \Omega^{-\tau_h^\g}),\] 
  so that any sufficiently small open convex $0$-neighborhood
  $\Omega' \subeq \Omega^{-\tau_h^\g}$ specifies an open neighborhood
  $W^c := G^h_e.\exp(\ie\Omega')$ of $e\in \Xi^{\oline\tau_h}$.
  Since $\exp(\ie t h)\in \Xi$ for $|t| \leq \pi/2$,
  a compactness argument shows that, if $\Omega'$ is small enough, 
  then   $\exp(\ie t h) \exp(\ie\Omega') \subeq \Xi$ for $|t| \leq \pi/2$,
  and thus (Cr2) is satisfied.
  This shows that crown domains $\Xi$ satisfying (Cr1--3) exist in abundance. 
\end{remark}

\begin{examplekh} For $G = \R\subeq \C = G_\C$
  and $h = 1$ (a basis element in $\g = \R$),
any strip  
  \[ \Xi = \cS_{\pm r} = \{ z \in \C \: |\Im z| < r \} \subeq \C = G_\C, \qquad
    r  \geq \pi/2, \]
  is a crown domain for $G = \R$ with $W^c = G$.
  In this case $\tau_h = \id$, $\oline \tau_h(z) = \oline z$
    and $G_{\tau_h} \cong (\R^\times, \cdot)$.  
\end{examplekh}

Below we shall encounter various kinds of non-abelian examples.

\subsection{Constructions of crown domains}

The following lemma is useful tool to construct crown domains.
We shall use in various contexts.

\begin{lemma} \label{lem:Mxi} Suppose that $\eta_G$ is injective
  and that $G_{\C, \oline\tau_h} = G_\C \rtimes \{\bone,\oline\tau_h\}$
  acts on the complex manifold $M$
  in such a way that $G_\C$ acts by holomorphic maps and
  $\oline\tau_h$  by an antiholomorphic involution~$\tau_{M}$.
Let $\Xi_{M} \subeq M$ be a $G_{\tau_h}$-invariant connected open subset,   
for which there exists an open subset $W^{M,c}\subeq \Xi_M^{\tau_M}$ satisfying 
\[     \exp(\cS_{\pm \pi/2}h).W^{M,c} \subeq  \Xi_M. \]
Then, for every $m_0 \in W^{M,c}$,  the open subset 
  \[ \Xi := \{ g \in G_\C \: g.m_0 \in \Xi_{M}\}\]
  satisfies {\rm(Cr1-3)} with the open subset 
$W^c := \{ g \in G_\C^{\oline\tau_h}  \: g.m_0 \in W^{M,c}\}
  \subseteq\Xi^{\oline\tau_h}.$ 
\end{lemma}

\begin{proof} (Cr1): Since $\oline\tau_h(m_0) = \tau_M(m_0) = m_0$, $\oline\tau_h(\Xi) = \Xi$
  follows from the antiholomorphic action of $G_{\tau_h}$ on~$\Xi_{M}$.

  \nin (Cr2): The inclusion $\exp(\cS_{\pm \pi/2} h) W^c \subeq \Xi$
  follows from $\exp(\cS_{\pm \pi/2} h) W^{M,c} \subeq \Xi_{M}$.

  \nin (Cr3) is redundant because $G \subeq G_\C$. 
\end{proof}

\begin{examplekh} \label{ex:affine-group}  (The affine group of the real line) 
  We consider the $2$-dimensional affine group of the real line 
  \[ G = \Aff(\R)_e \cong \R \rtimes \R_+\quad \mbox{ with } \quad
    \g = \R  x \rtimes \R h, \quad x = (1,0), \quad h = (0,1),\]
  so that 
$[h,x] = x$ and $\tau_h(b,a) = (-b,a)$.
A pair $(b,a) \in G$ acts on $\R$ by the affine map
$(b,a).x = b + a x$ and so does the complex affine group 
$\Aff(\C) \cong \C \rtimes \C^\times$ on the complex line~$\C$.
The antiholomorphic involution $\sigma(b,a) = (\oline b, \oline a)$
satisfies $\Aff(\C)^\sigma = \Aff(\R) \cong \R \rtimes \R^\times$,
and $G$ is the identity component of this group. Note that
$G_\C \cong \tilde\Aff(\C) \cong \C \rtimes \C$ with the universal map 
\[ \eta_G \: G \to G_\C, \quad \eta_G(b,a) = (b, \log a).\]
The antiholomorphic extension of $\tau_h$ to $\Aff(\C)$, is given by 
  \[  \oline\tau_h(b,a) = (-\oline b, \oline a)
    \quad \mbox{ with } \quad \Aff(\C)^{\oline\tau_h} = \ie \R \rtimes \R^\times.\]
  The subset
  \[ \exp(\cS_{\pm \pi/2}h) = \C_r = \{ z \in \C \: \Re z > 0\} 
    \subeq \C^\times \]
  is the right half-plane in the complex dilation group.

First, we consider in $\Aff(\C)$ the domain
\begin{equation}
  \label{eq:xi1-axb}
 \Xi_1 := \C \times \C_r 
 \quad \mbox{  with } \quad \Xi_1^{\oline\tau_h} = \ie \R \times \R_+
 = (\Aff(\C)^{\oline\tau_h})_e.
\end{equation}
Then $W_1^c := \Xi_1^{\oline \tau_h}$ satisfies 
$\exp(\cS_{\pm \pi/2} h) W_1^c \subeq  \Xi_1,$ 
so that (Cr2) is satisfied. Further $\eta_G$
extends to a biholomorphic map
\[ \eta_{\Xi_1} \: \Xi_1 \to
\C\times\cS_{\pm\pi/2} \subeq  G_\C \cong \C \rtimes \C, \quad
  \eta_G(b,a) \mapsto (b, \log a). \]
In \cite[Prop.~3.2]{BN25} it is shown that $\Xi_1$ is too large for our purposes
because $\cH^\omega(\Xi_1) \cap \cH^J_{\rm temp} = \{0\}$ holds for
infinite-dimensional irreducible representations. 
A natural strategy to find better crown domains
is inspired by Riemannian symmetric spaces (cf.~Theorem~\ref{thm:gss} below).

The group $\Aff(\C)_{\oline\tau_h}$
  acts naturally  on $M := \C$, where $\oline\tau_h.z = - \oline z$.
  The two domains~$\C_{\pm}$ (upper and lower half plane)
  are invariant under the real group $G_{\tau_h}$, and
  \[ (\C_\pm)^{\oline\tau_h} =  \pm \ie \R_+ \]
  satisfy $\exp(\cS_{\pm \pi/2} h).(\pm \ie \R_+) = \C_\pm$.
  For $m_0 = \pm r \ie$, $r > 0$, we thus obtain with Lemma~\ref{lem:Mxi}
  a crown domain 
  \begin{align*}
 \Xi_{\pm, r} 
    &  := \{ (b,a) \in \Aff(\C) \: b \pm r a\ie \in  \C_{\pm}\}
= \{ (b,a) \in \Aff(\C) \: \pm r^{-1} b + a\ie \in \C_+\}\\
& = \{ (b,a) \in \Aff(\C) \: \pm r^{-1} \Im b + \Re a > 0 \}.
  \end{align*}
Conjugation with $G$ changes the parameter $r$, so that
  it suffices to consider the domains $\Xi_{\pm, 1}$.

To connect with crowns of symmetric spaces,
we consider $\C_+$ as a real \break {$2$-dimensional}
  homogeneous space of $G$ via the orbit map
  $(b,a) \mapsto (b,a).\ie = b + a \ie$. 
It has a ``complexification''
  \[ \eta_{\C_+} \:  \C_+ \to \C_+ \times \C_- \subeq \C^2,
    \quad \eta_{\C_+}(z) = (z,\oline z).\]
 The complex Lie group $\Aff(\C)$  acts naturally on
  $\C \times \C$ by the diagonal action with respect to the canonical
  action on $\C$ by affine maps.
  We consider the complex manifold $\Xi_{\C \times \C} := \C_+ \times \C_-$ 
  as a 
  crown domain of the upper half plane
  $\C_+ \cong \eta_{\C_+}(\C_+)$, which is a Riemannian
  symmetric space of $\SL_2(\R)$, 
  acting by M\"obius transformations
  (see \cite{Kr09}  and \cite[\S 2.1]{Kr08}).
  It carries the antiholomorphic involution
  $\tau(z,w) := (-\oline z, -\oline w)$
  and it is invariant under the real affine group
  $G = \R \rtimes \R_+$, so that we obtain with $\tau$ an action
  of $G_{\tau_h}$ on $\Xi_{\C \times \C}$. 
  As $\C_+ = G.\ie$, the corresponding crown domain in $\Aff(\C)$,
  in the sense of Lemma~\ref{lem:Mxi} with $m_0 = \ie$, is 
  \begin{align*}
    \Xi_2 
& := \{ g \in \Aff(\C) \:  g.\eta_{\C_+}(\ie) \in \C_+ \times \C_-\}
    = \{ (b,a) \in \Aff(\C) \: b \pm a\ie \in \C_\pm \} \\
&= \Xi_{+,1} \cap \Xi_{-,1}
  =  \{ (b,a) \in \Aff(\C) \: |\Im b| < \Re a \}.
  \end{align*}
  For this domain
\[     \Xi_2^{\oline\tau_h}  
  = \{ (\ie c,a) \in \ie \R \times \R_+ \: |c| < a \}, \]
and, for $|t| < \pi/2$ and $(\ie c,a) \in \Xi_2^{\oline\tau_h}$, we have
$e^{\ie th}(\ie c,a) = (e^{\ie t}\ie c, e^{\ie t} a)$ with
\[ |\Im(e^{\ie t}\ie c)| = \cos(t) |c|
  < \cos(t) a = \Re(e^{\ie t}a).\]
This proves (Cr2) for $\Xi_2$ and $W^c_2 :=  \Xi_2^{\oline\tau_h}$.

As shown in \cite[Thm.~3.1]{BN25},
there are irreducible unitary representations $(U,\cH)$ of $G$
  for which $\cH^\omega(\Xi_1)$ is dense but
  \[ \cH^\omega(\Xi_1) \cap \cH^J_{\rm temp} =\{0\}.\]
  Therefore this domain $\Xi_1$ is too large.
  The smaller domain $\Xi_2$ behaves much better than~$\Xi_1$.
On the other hand, the even larger domains $\Xi_{\pm, 1}$ works 
for representations satisfying the spectral condition
$\mp \ie \partial U(x) \geq 0$ for the translation group.
\end{examplekh}

\begin{examplekh}  \label{ex:olshanski}
  {\rm(Complex Olshanski semigroups)} 
  Let $G$ be a connected Lie group for which
  $\eta_G$ is injective and $G_\C$ is simply connected,
  so that we may assume that $G \subeq G_\C$, 
  and let $h \in \g$ be an Euler element.  We assume that 
  $C \subeq \g$ is a pointed closed convex $\Ad(G)$-invariant cone,
  satisfying  $-\tau_h(C) = C$, and
  that the ideal $\fm := C - C \trile \g$ satisfies
  $\g = \fm + \R h$. 
  If $\fm = \g$, we replace $\g$ by $\fm \rtimes_{\ad h} \R$, so that we
  may assume that $\g = \fm \rtimes \R h$. 
  
  Let $M \trile G$ 
  and $M_\C \trile G_\C$ be the normal integral subgroups corresponding to $\fm$ and $\fm_\C$, respectively. 
  Since $G_\C$ is simply connected, 
  it follows from  \cite[Thm. 11.1.21]{HN12} that the subgroup
  $M_\C\subseteq G_\C$ is closed and simply connected,
  and this implies that $M \subeq M_\C$ is closed in $G$,
    as the identity component of the group of fixed points of the
    complex conjugation on $M_\C$ with respect to~$M$. 
Then the construction in Step~1 of the proof of \cite[Thm. 15.1.4]{HN12} 
shows that the inclusion map $M\into M_\C$ is the universal complexification of $M$. 
Using again that 
$G_\C$ is simply connected, it follows from  
\cite[Prop.~11.1.19]{HN12} that   $G_\C \cong M_\C \rtimes_\alpha \C$ with
  $\alpha_z(g) = \exp(zh) g \exp(-zh)$, and thus
  \[ G \cong  M \rtimes_\alpha \R \quad \mbox{ with }  \quad
    \alpha_t(g) = \exp(th) g \exp(-th).\]
We consider the   complex Olshanski semigroup 
  \[ S := M \exp(\ie C^\circ) \subeq M_\C,\]
for which the multiplication map
  $M \times C^\circ  \mapsto S, (g,x) \mapsto g \exp(\ie x)$ is a
  diffeomorphism
(\cite[\S XI.1]{Ne99}). 
Recall from Lemma~\ref{lem:Project} 
that
  \[ C \subeq C_+ + \g_0(h)  - C_-
    \quad \mbox{ and } \quad C_+^\circ - C_-^\circ = C^\circ \cap \g^{-\tau_h^\g} \quad \mbox{ for }  \quad C_\pm := \pm C \cap \g_{\pm 1}(h)\]
follow from the invariance of $C$ under $e^{\R \ad h}$ and
$-\tau_h$.
  On the complex manifold $S$, the group $G_{\tau_h}$ acts  by
  \[ (g,t).m = g\alpha_t(m) \quad \mbox{ and } \quad
    \tau_h.m = \oline\tau_h(m)\]
where we use the abuse of notation $\tau_h=(e,-1)\in G_{\tau_h}$ introduced after \eqref{eq:gtauh}. 
 The fixed points of $\oline\tau_h$ form the subsemigroup 
$S^{\oline\tau_h} = M^{\tau_h} \exp(\ie(C_+^\circ -C_-^\circ)).$ 
  For $x_{\pm 1} \in C_\pm^\circ$, we consider the element 
$s_0 := \exp(\ie(x_1 - x_{-1})) \in S^{\oline\tau_h}$, 
  and in $G_\C$ the crown domain
  \[ \Xi = \{ (g,z) \in M_\C \times \cS_{{\pm \pi/2}} \:
    (g,z).s_0 = g \alpha_z(s_0) \in S \}\]
  (Lemma~\ref{lem:Mxi}). 
  We now verify (Cr1-2); (Cr3) is redundant because $\Xi \subeq G_\C$.

  \nin (Cr1): First we observe that $\oline\tau_h(s_0) = s_0$,
  and that   $\oline\tau_h(S) = S$
  follows from ${-\tau_h^\g(C) = C}$. Therefore 
  \[ \oline\tau_h((g,z).s_0) = \oline\tau_h(g,z).s_0\]
implies that the domain $\Xi$ is invariant under
  $\oline\tau_h(g,z) = (\oline\tau_h(g), \oline z)$.

  \nin (Cr2): We recall from \cite[Thms.~2.16, 2.21]{Ne22}
  that, for $y = y_1 - y_{-1} \in C_+^\circ - C_-^\circ$, 
  we have
  \[ \alpha_{\ie t}(\exp(\ie y)) = \exp(\ie(e^{\ie t} y_1 - e^{-\ie t} y_{-1})) \in S
\quad \mbox{ for }  \quad |t| < \pi/2. \]
So we find with 
  \[ W_M^c := M^h_e\exp(\ie(C_+^\circ + C_-^\circ)) s_0^{-1}
    \subeq M_\C^{\oline\tau_h} \]
  that
  \begin{equation}
    \label{eq:Wc-olshan}
W^c := G^h_e (W_M^c \times \{0\}) \subeq \Xi^{\oline\tau_h}
    \quad \mbox{ satisfies }\quad
    \alpha_{\ie t}W^c \subeq \Xi \quad \mbox{ for } \quad |t| < \pi/2.
  \end{equation}
  This proves (Cr2).
\end{examplekh}

\begin{remark} The construction in Example~\ref{ex:affine-group}
  can also be viewed as a special  case of Example~\ref{ex:olshanski}.
  To see this, write
\[ G = \Aff(\R)_e = \R \rtimes \R_+ \cong M \rtimes \R_+
  \subeq M_\C \rtimes \C^\times
  \quad \mbox{ with } \quad M = \R\times \{1\}.\] 
The invariant  cone $C := \R_+ x$ generates the ideal $\fm = \R x$,
and $C = C_+$.   Then $S = \C_+$ with $S^{\oline\tau_h} = \ie \R_+$.
  For $s_0 = \ie r$, $r > 0$, we obtain
  \[ \Xi = \{ (b,a) \in M \times \C^\times  \: (b,a).s_0
    = b + r a \ie  \in \C_+ \} = \Xi_{+,r}.\] 
\end{remark}

\subsection{Nets of real subspaces} 

Given a domain $\Xi \subeq G_\C$ satisfying (Cr1-3), and an antiunitary
representation $(U,\cH)$ of $G_{\tau_h}$, we write
\[ \cH^\omega(\Xi) \subeq \cH \]
for the subspace of those analytic vectors, whose orbit map extends
to $\Xi$. That the non-triviality of this space
imposes serious restrictions on $\Xi$
follows in particular from the discussion in
the last section of \cite{BN24}.
For the group $G = \Aff(\R)_e \cong \R \rtimes \R_+$, the domain
must be contained in $\C \rtimes \C_r$, where $\C_r$ is the open right
half-plane. So one has to understand the boundary behavior of the
extended orbit maps on the domain $\Xi$. Let
\[ \cH^J_{\rm temp} \subeq \cH^J = \Fix(J) \quad \mbox{ for } \quad
J = U(\tau_h) \]
be the dense real linear subspace of $\cH^J$, consisting of those vectors $v$ 
for which the orbit map $U^v_h(t) = U(\exp th)v$
extends to the open strip $\cS_{\pm \pi/2}$, and the limit
\begin{equation}
  \label{eq:beta+}
  \beta^+(v) := \lim_{t \to \pi/2} U^v_h(-\ie t)
\end{equation}
exists in the subspace $\cH^{-\infty}(U_h) \subeq \cH^{-\infty}$
of distribution vectors of the one-parameter group~$U_h$. 
\begin{footnote}{The notation $\cH_{\rm temp}$ refers to the ``temperedness''
    of the boundary values, which in the classical context corresponds
    to tempered distributions.  
  }\end{footnote}
In view of Theorem~\ref{thm:6-1-fno24}, this is equivalent to
  the existence of $C, N > 0$ such that
  \begin{equation}
    \label{eq:growthcond}
 \| U^v_h(\ie t)\|^2 \leq C \Big(\frac{\pi}{2}-|t|\Big)^{-N}
 \quad  \mbox{  for } \quad |t| < \pi/2.
  \end{equation}
By \cite[\S 2.2, Lemma~2]{FNO25a}, the limit $\beta^+(v)$ always
exists in the larger space $\cH^{-\omega}_{U_h, {\rm KMS}}$,
so that the main
point is the additional regularity (temperedness), related to the
weak convergence after pairing with elements of~$\cH^\infty_{U_h}$.

These boundary values are actually contained in the space
$\cH^{-\infty}_{\rm KMS}$ (see Appendix~\ref{subsec:4.1}), consisting
of those distribution vectors $\alpha$ whose orbit map
\[ U_h^{-\infty,\alpha} \: \R \to \cH^{-\infty}, \quad
  t \mapsto U_h^{-\infty}(t)\alpha = \alpha \circ U(\exp -th) \]
extends analytically to the
closed strip $\oline{\cS_\pi}$, such that
\[ U_h^{-\infty,\alpha}(\pi i) = J \alpha.\]
Using Theorem~\ref{thm:BN24}, it then follows that smearing
with test functions on $\R$ maps $\alpha$ into $\sV = \sV(h,U)$.
Therefore any real linear subspace
\[ \sF \subeq \cH^\omega(\Xi) \cap \cH^J_{\rm temp} \]
which is $G$-cyclic in the sense that $U(G)\sF$ spans a dense subspace
of~$\cH$, leads to a real subspace
\[   \sE := \beta^+(\sF) \subeq \cH^{-\infty}, \] 
and from this space we construct a net of real subspaces on
$G$ as follows.

\begin{definition} Let $\sE \subeq \cH^{-\infty}$ be a real linear subspace.
  Then, for each $\phi \in C^\infty_c(G,\C)$, the operator
  \[ U^{-\infty}(\phi) = \int_G \phi(g) U^{-\infty}(g)\, dg \]
maps $\cH^{-\infty}$ into $\cH$,  because it is an adjoint of 
the operator $U(\phi^*) \colon \cH \to \cH^\infty$. 
To an open subset $\cO \subeq G$, we associate the closed real subspace 
\begin{equation}
  \label{eq:HE2}
  \sH_\sE^G(\cO) := \oline{\spann_\R U^{-\infty}(C^\infty_c(\cO,\R))\sE},
\end{equation}
where the closure is taken with respect to the topology of $\cH$.
\end{definition}

\begin{remark} \label{rem:iso-cov-net}
It is obvious that the net $\sH^G_\sE$ satisfies (Iso).
To see that (Cov) also holds, observe that the left-invariance
of the Haar measure $dg$ on $G$ yields 
\[ U^{-\infty}(g) U^{-\infty}(\phi)  = U^{-\infty}(\delta_g * \phi),\]
where $(\delta_g * \phi)(x) = \phi(g^{-1}x)$ is the left translate
of $\phi$.
\end{remark}

\begin{remark} One may also consider subspaces 
  $\sE \subeq \cH$, but the key advantage of working with the larger
  space $\cH^{-\infty}$ of distribution vectors is that it contains
  finite-dimensional subspaces invariant under
 $\ad$-diagonalizable elements and
  non-compact subgroups. 
  For  finite-dimensional subspaces of~$\cH$, this is excluded by 
  Moore's Theorem if $\ker U$ is discrete (\cite{Mo80}).
\end{remark}

The following proposition is useful to verify
the inclusion $\sH_\sE^G(\cO) \subeq \sV$ for an
open subset $\cO \subeq G$. 

  \begin{proposition} \label{prop:4.8}
    Let $(U,\cH)$ be an antiunitary representation of
    $G_{\tau_h}$ and $\sV = \sV(h,U)$ the corresponding
    standard subspace.     For an open subset
  $\cO \subeq G$ and a real subspace $\sE \subeq \cH^{-\infty}$,
  the following are equivalent:
  \begin{description}
  \item[\rm(a)]   $\sH_\sE^G(\cO) \subeq \sV$. 
  \item[\rm(b)] For all $\phi \in C^\infty_c(\cO,\R)$ we have
    $U^{-\infty}(\phi)\sE \subeq \sV$. 
  \item[\rm(c)] For all $\phi \in C^\infty_c(\cO,\R)$ we have
    $U^{-\infty}(\phi)\sE \subeq \cH^{-\infty}_{\rm KMS}.$ 
  \item[\rm(d)]  $U^{-\infty}(g) \sE \subeq \cH^{-\infty}_{\rm KMS}$
    for every $g \in \cO$.
  \end{description}
\end{proposition}

To show that $\sH_\sE^G(W^G) \subeq \sV$, we thus need to show that
$U^{-\infty}(W^G) \sE \subeq \cH^{-\infty}_{\rm KMS}$. 

\begin{proof}  (\cite[Prop.~9]{FNO25a})
By the   definition of $\sH^G_\sE(\cO)$, it is clear that (a) 
  is equivalent to (b). Further, (b) implies (c)  because
  $\sV \subeq \cH^{-\infty}_{\rm KMS}$
  (Theorem~\ref{thm:BN24}(b)). 
  
  For the implication (c) $\Rightarrow $ (d), let $(\delta_n)_{n \in \N}$ be a $\delta$-sequence in $C^\infty_c(G,\R)$.
  Then $U(\delta_n)\xi \to \xi$ in $\cH^\infty$ and hence also in $\cH^{-\infty}$.
  It follows in particular that
  \[U^{-\infty}(\delta_n * \delta_g) \eta
  =  U^{-\infty}(\delta_n) U^{-\infty}(g) \eta \to  U^{-\infty}(g) \eta
  \quad \mbox{ for } \quad \eta\in \cH^{-\infty}.\] 
  Hence the closedness of
  $\cH^{-\infty}_{\rm KMS}$
    (Theorem~\ref{thm:BN24}(a)), 
  shows that (c) implies (d).
  Here we use that $\delta_n * \delta_g \in C^\infty_c(\cO,\R)$
  for $g \in \cO$ if $n$ is sufficiently large. 

As the $G$-orbit maps in $\cH^{-\infty}$ are continuous
and $\cH^{-\infty}_{\rm KMS}$ is closed, hence stable under integrals over
compact subsets and $U^{-\infty}(C_c^\infty (\cO,\R))\cH^{-\infty} \subset \cH^\infty$, we see that (d) implies (b).
\end{proof}

The following corollary provides a
generalization to homogeneous spaces.

  \begin{corollary} \label{cor:4.8}
    Let $(U,\cH)$ be an antiunitary representation of
    $G_{\tau_h}$ and $\sV = \sV(h,U)$ the corresponding
    standard subspace.     For an open subset
    $\cO \subeq M = G/H$, the projection $q_M \: G \to G/H$,
    and a real subspace $\sE \subeq \cH^{-\infty}$,
  the following are equivalent:
  \begin{description}
  \item[\rm(a)]   $\sH_\sE^M(\cO) \subeq \sV$. 
  \item[\rm(b)] For all $\phi \in C^\infty_c(q_M^{-1}(\cO),\R)$ we have
    $U^{-\infty}(\phi)\sE \subeq \sV$. 
  \item[\rm(c)] For all $\phi \in C^\infty_c(q_M^{-1}(\cO),\R)$ we have
    $U^{-\infty}(\phi)\sE \subeq \cH^{-\infty}_{\rm KMS}.$ 
  \item[\rm(d)]  $U^{-\infty}(g) \sE \subeq \cH^{-\infty}_{\rm KMS}$
    for every $g \in q_M^{-1}(\cO)$.
  \end{description}
\end{corollary}

\begin{proof} Apply the preceding proposition
  to the open subset $q_M^{-1}(\cO)\subeq G$ and use that 
  $\sH^M_\sE(\cO) = \sH^G_\sE(q_M^{-1}(\cO))$.    
\end{proof}

\begin{lemma}   \label{invar_left_lem}
The following assertions hold: 
\begin{enumerate} 
\item[\rm(a)] 
	 For arbitrary $y\in\g$, the corresponding Lie derivative operator
\[ L_y\colon C^\infty(\Xi,\cH)\to C^\infty(\Xi,\cH), \quad 
	 (L_yf)(g):=\frac{d}{dt}\Big|_{t = 0}\ f(\exp(ty)g) \] 
	 satisfies $L_y(\cO(\Xi,\cH))\subseteq\cO(\Xi,\cH)$. 
\item[\rm(b)] 
  We have $\dd U(\g)\cH^\omega(\Xi)\subseteq\cH^\omega(\Xi)$ 
and, for $x \in \g_\C,\ p\in\Xi,\ v\in\cH^\omega(\Xi)$, 
\begin{equation}\label{invar_left_rem_eq1} 
U^{\dd U(x)v}(p)=\dd U(\Ad(\eta_\Xi(p))x) U^v(p)
\end{equation}
where $U^v\in\cO(\Xi,\cH)$ is the holomorphic extensioon of the analytic
orbit map $U^v\colon G\to\cH$. 
\item[\rm(c)] 
The closure of $\cH^\omega(\Xi)$ in $\cH$ is $U(G)$-invariant. 
If, in particular, $U$ is irreducible and $\cH^\omega(\Xi)$ is
non-zero,  then $\cH^\omega(\Xi)$ is dense in~$\cH$. 
\end{enumerate} 
\end{lemma}

\begin{proof} (a) 
  	The operator $L_y$ is the Lie derivative with respect to a 
    fundamental vector field defined by the action of $G$ by left-translations 
on $\Xi$. 
    As this vector field is holomorphic, it preserves on each open subset
    the subspace of holomorphic functions.

\nin (b) 
Let $v\in\cH^\omega(\Xi)$ and $x\in\g$. 
	For arbitrary $g\in G$ we have 
	$U(g)\partial U(x)v=\partial U(\Ad(g)x)U(g)v$, so that
        \eqref{invar_left_rem_eq1}  follows 
        for $p=g\in G$. 
        The general case $p\in\Xi$ is then obtained by analytic extension.
	Since the mapping $G_\C\to \g_\C$, $g\mapsto \Ad(g)x$, is holomorphic, 
for any basis $y_1,\dots,y_m$ of $\g$, 
there exist $\chi_1,\dots,\chi_m\in\cO(G_\C)$ with
\[ \Ad(g)x=\chi_1(g)y_1+\cdots+\chi_m(g)y_m
\quad \mbox{ for all } \quad g\in G_\C.\]
By plugging this in \eqref{invar_left_rem_eq1}, it
suffices to prove that, for every $y\in\g$,
the function $w\mapsto  \partial U(y)U^v(w)$       
is holomorphic on $\Xi$. 
 
We now check this last fact. 
From the $G$-action on $\Xi$ it follows that, for every $w\in\Xi$,  
there exists $t_w\in\R_+$, such that for all $t\in (-t_w,t_w)$ we have $\exp(ty)w\in \Xi$, and then $U^v(\exp(ty)w)=U(\exp(ty))U^v(w)$. 
Taking the derivative at $t=0$ in this equality, we obtain 
\[ L_y(U^v)(w)=\frac{d}{dt}\Big|_{t = 0} U^v(\exp(ty)w)=\partial U(y)U^v(w)\]  
for $w\in\Xi$, where $L_y\colon C^\infty(\Xi,\cH)\to C^\infty(\Xi,\cH)$ is the Lie derivative operator in the direction $y\in\g$. 
	Since  $L_y(\cO(\Xi,\cH))\subseteq \cO(\Xi,\cH)$ 
    by (a),we are done.
    
\nin (c) In view of (b), 
    the closure of $\cH^\omega(\Xi)$ is $U(G)$-invariant by 
	\cite[Cor. to Thm.~2, pp.~210--211]{HC53} or
        \cite[Prop.~4.4.5.6]{Wa72}.
	Therefore, if the representation $U$ is irreducible
        and $\cH^\omega(\Xi)\ne\{0\}$, then
        the linear subspace $\cH^\omega(\Xi)$ is dense in~$\cH$. 
\end{proof}


As before, $(U,\cH)$ is  an antiunitary representation
of $G_{\tau_h}$ and $J = U(\tau_h)$.
We will establish in Lemma~\ref{lem:beta+-equiv}  a natural 
	$\dd U(\g)$-equivariance property of the mapping 
	$\beta^+\colon\cH^\omega_{U_h}(\cS_{\pm\pi/2})\to\cH^{-\omega}_{U_h}$. 
The domain and the range of $\beta^+$ may  not be $\dd U(\g_\C)$-invariant, 
so that such an equivariance property does not make sense on
$\cH^\omega_{U_h}(\cS_{\pm\pi/2})$. However, (Cr2) ensures that
$\cH^\omega(\Xi)\subeq\cH^\omega_{U_h}(\cS_{\pm\pi/2})$, 
and $\cH^\omega(\Xi)$ carries the action of $\dd U(\g_\C)$
from Lemma~\ref{invar_left_lem}(b), so that we may
study $\dd U(\g_\C)$-equivariance for the restriction of $\beta^+$ to  $\cH^\omega(\Xi)$.

\begin{lemma} \label{lem:beta+-equiv}
  The map
  \[ \beta^+ \: \cH^\omega(\Xi) \to \cH^{-\omega}, \quad 
    \beta^+(v) := \lim_{t \to \pi/2} U^v(\exp(-\ie th)) 
= \lim_{t \to \pi/2} e^{-\ie t \partial U(h)}v  \]
  satisfies the following equivariance relation with respect to the
  action of $\g_\C$ on both sides:
  \begin{equation}
    \label{eq:beta+rel1}
 \beta^+ \circ \dd U(x) = \dd U^{-\omega}(\zeta(x)) \circ \beta^+
    \quad \mbox{ for } \quad
    \zeta := e^{-\frac{\pi \ie}{2} \ad h} \in \Aut(\g_\C), x \in \g_\C.
  \end{equation}
\end{lemma}

\begin{proof}
  (cf.\ \cite{BN25}, and
  \cite[\S 3, Prop.~7(d)]{FNO25a} for the semisimple case)
  Let $v \in \cH^\omega(\Xi)$ and $x \in \g_\C$. Then
  the continuity of the operators $\dd U^{-\omega}(z)$, $z \in \g_\C$,
    implies 
\begin{equation*}
   \dd U^{-\omega}(\zeta(x))\beta^+(v) 
  =  \lim_{t \to \pi/2}  \dd U^{-\omega}(\zeta(x)) U^v(\exp(-\ie t h)).
\end{equation*}
From \eqref{invar_left_rem_eq1} in Lemma~\ref{invar_left_lem}, we then obtain
\begin{equation}
  \label{eq:du-uv}
\dd U(x) U^v(p) = U^{\dd U(\Ad(\eta_\Xi(p))^{-1} x)v}(p)
 \quad \mbox{ for } \quad p \in \Xi.
\end{equation}
This formula holds obviously for $p \in G$, and for
general $p$ it follows by analytic continuation, using that
\[ p \mapsto \dd U(\Ad(\eta_\Xi(p))^{-1} x)v \in \dd U(\g_\C) v \]
is a holomorphic function with values in a finite-dimensional space.
The relation~\eqref{eq:du-uv} implies 
\begin{align*}
 \dd U(\zeta(x)) U^v(\exp(-\ie t h))
&  = U^{\dd U(e^{\ie t \ad h} \zeta(x))v}(\exp(-\ie th)) \\
&  = e^{-\ie t \partial U(h)} \dd U\big(e^{\ie t \ad h}\zeta(x)\big)v.
\end{align*}
We now observe that
$t \mapsto   \dd U\big(e^{\ie t \ad h}\zeta(x)\big)v$ is a continuous curve
in the finite-dimensional subspace $\dd U(\g_\C)v$, so that
we obtain the limit 
\[\lim_{t \to \pi/2}  \dd U(\zeta(x)) U^v(\exp(-\ie t h))
  =\beta^+(\dd U\big(\zeta^{-1}\zeta(x)\big)v)
  =\beta^+(\dd U(x)v).
\]
This proves the lemma. 
\end{proof}

\begin{theorem} \label{thm:4.9}
  {\rm(Construction Theorem for nets of real subspaces)} 
   Let $(U,\cH)$ be an antiunitary
  representation of $G_{\tau_h} = G \rtimes \{\bone,\tau_h\}$ and
 \[ \sF \subeq \cH^J_{\rm temp} \cap \cH^\omega(\Xi) \] 
 be a $G$-cyclic subspace of $\cH$, i.e., $U(G)\sF$ is total in $\cH$.
 We consider the linear subspace
 \[ \sE = \beta^+(\sF)
   \subeq \cH^{-\infty}.\]
 Then the net $\sH^G_{\sE}$ on $G$ 
 satisfies {\rm (Iso), (Cov), (RS)} and
 {\rm (BW)}, in the sense that
$\sH^G_\sE(W^G) = \sV$  holds for $W^G$ as in \eqref{eq:defwg}. 
\end{theorem}

\begin{proof} We refer to \cite{BN25} for a detailed proof.
  Here we only give an outline. We have already argued above that
  (Iso) and (Cov) are satisfied.
To verify the  Reeh--Schlieder property (RS), one has to show that, for 
  $\eset\not=\cO$ (w.l.o.g.\ $\cO \subeq W^G$),
  we have $\sH_\sE^G(\cO)^\bot = \{0\}$.
  This is derived from the fact that, if, for $\alpha
  = \beta^+(\xi) \in \sE$,   the orbit map $U^\xi \: G \to \cH^{-\infty}$
  extends to a holomorphic map $\Xi \to \cH^{-\infty}$, then 
  it  is determined by the values of the orbit map
  $U^{-\infty, \alpha} \: G \to \cH^{-\infty}$ 
  on every open subset of~$G$ (\cite[Thm.~2.9]{BN25}). 

  For the Bisognano--Wichmann property (BW), it suffices to show that
\[ U^{-\infty}(W^G) \sE \subeq \cH^{-\infty}_{\rm KMS}.\] 
  Then Proposition~\ref{prop:4.8} yields   $\sH^G_\sE(W^G) \subeq \sV$,
  and by (RS), $\sH^G_\sE(W^G)$ is cyclic, so that (Cov) and
  $\exp(\R h)W^G = W^G$ lead to equality with the
  Equality Lemma~\ref{lem:lo08-3.10}.   
\end{proof}

\subsection{Push-forwards to homogeneous spaces}

\begin{definition} \label{def:push-forward-net}
On a homogeneous space $M = G/H$ with the projection map
$q_M \colon G \to M$, we obtain from every net $\sH^G$
on open subsets of $G$ a ``push-forward net'' 
\begin{equation}
  \label{eq:pushforward}
  \sH^M(\cO) := ((q_M)_*\sH^G)(\cO) = \sH^G(q_M^{-1}(\cO)).
\end{equation}
The so-obtained net on $M$ thus corresponds to the restriction
  of the net $\sH^G_\sE$ indexed by open subsets of $G$, to those open subsets
  $\cO\subeq G$  which are $H$-right invariant in the
  sense that $\cO = \cO H$;
  these are the inverse images of open subsets of $M$ under~$q_M$.
\end{definition}

\begin{remark} \label{rem:push-forward}
(a) If
  $\sE$ is invariant under $U^{-\infty}(H)$, then Lemma~\ref{lem:sw64-genb}(b)
  in Appendix~\ref{app:4.1} 
  implies that $\sH^G_\sE(\cO)= \sH^G_\sE(\cO H)$ for any open subsets 
  $\cO \subeq G$, so that $\sH^G_\sE$ can be recovered from the
  net $\sH^M_\sE$ on $M$ by
  $\sH^G_\sE(\cO) = \sH^G_\sE(\cO H) = \sH^M_\sE(q_M(\cO))$.

  \nin (b) We have already seen in Remark~\ref{rem:iso-cov-net}
  that the net $\sH^G_\sE$, and hence also $\sH^M_\sE$,  satisfy (Iso) and (Cov).
  Further, the net $\sH^M_\sE$ inherits (RS) from $\sH^G_\sE$.
  If (BW) holds for $\sH^G_\sE$ and the wedge region $W^G \subeq G$
  in the sense that $\sH^G_\sE(W^G) = \sV$, 
  then it holds for its image in $G/H$ if {\bf $\sE$ is $H$-invariant},
  which implies with $W^M = q_M(W^G)$     that
  \[ \sH^G_\sE(W^G) = \sH^G_\sE(W^G H)
    = \sH_\sE^G(q_M^{-1}(W^M))= \sH_\sE^M(W^M) \]
  (Lemma~\ref{lem:fragment} in Appendix~\ref{app:4.1}). 

\nin (c) If $\sE$ is not $H$-invariant,
  then the situation is more complicated.
  We may enlarge $\sE$ to the closed subspace $\hat\sE$ of $\cH^{-\infty}$
  generated by $U^{-\infty}(H)\sE$, but then it is not clear if this
  still has the form $\beta^+(\hat\sF)$ for some
  $\hat\sF \subeq \cH^\omega(\Xi) \cap \cH^J_{\rm temp}$.
\end{remark}

\subsection{Complex Olshanski semigroups}

We return to Example~\ref{ex:olshanski},
  where 
  \[ S = M \exp(\ie C^\circ) \subeq M_\C \]
  is a complex Olshanski semigroup,
  $s_0 := \exp(\ie(x_1 - x_{-1})) \in S^{\oline\tau_h}$
  with $x_{\pm 1} \in C_\pm^\circ$, 
 and assume that $(U,\cH)$ is an antiunitary representation
 of $G_{\tau_h}$,  for which
\[ C \subeq C_U \cap \fm, \quad \mbox{ where } \quad
   C_U = \{ x \in \g \: -\ie \cdot \partial U(x) \geq 0\}\] 
is the {\it positive cone of $U$}. 
 This implies that $U$ extends to a strongly continuous contraction representation
 of $\oline S := M \exp(\ie C)$ by
 \[ U(m \exp(\ie x)) = U(m) \ee^{\ie \partial U(x)},\quad  m \in M, x \in C,\]
 and $U\res_{S} \: S \to B(\cH)$ is holomorphic 
with respect to the operator norm topology 
 (cf.\ \cite[Thm.~XI.2.5]{Ne99}). 
 For any $v \in \cH^J$, we thus obtain a holomorphic orbit map
 \[ U^v \: S \to \cH, \quad U^v(s) = U(s)v.\]
 It follows in particular, that the map
 \begin{align} 
   \label{eq:holo1}
   \Xi &= \{ (g,z) \in M_\C \times \cS_{{\pm \pi/2}} \:
    (g,z).s_0 = g \alpha_z(s_0) \in S \} 
    \to B(\cH), \notag \\
    \quad (g,z) &\mapsto U^v((g,z).s_0)
 \end{align}
 is holomorphic.  We claim that 
 \begin{equation}
   \label{eq:temp-xi-1}
U(s_0)\cH^J_{\rm temp} \subeq \cH^\omega(\Xi) \cap \cH^J_{\rm temp},  
\end{equation}
in particular $\cH^\omega(\Xi)\cap \cH_{\rm temp}^J\neq \{0\}$.
 So let $w \in \cH^J_{\rm temp}$ and
 $v := U(s_0)w$. Then
 \[  U^v(g,t)
   = U(g) U_h(t) U(s_0) w 
   = U(g \alpha_t(s_0)) U_h(t) w 
   = U((g,t).s_0) U_h^w(t).\]
The evaluation map $B(\cH) \times \cH \to \cH$ is holomorphic,
so that the representation $U \: S \to B(\cH)$, and
thus the orbit map $U^w \: S \to \cH$, are holomorphic.
For $w \in \cH^J_{\rm temp} \subeq \cH^\omega_{U_h}(\cS_{\pm \pi/2})$, 
the prescription
\[ U^v(g,z)   := U((g,z).s_0) U_h^w(z) \] 
 defines a holomorphic map $\Xi \to \cH$, extending
 the orbit map $U^v$ on~$G$.
 This proves~$v \in \cH^\omega(\Xi)$.
 To see that also $v \in \cH^J_{\rm temp}$, note that 
\[ U^v(e,t)  = U(\exp th) U(s_0)w = U(\alpha_t(s_0)) U_h^w(t) 
  \quad \mbox{ for }  \quad t \in \R \]
implies by analytic continuation 
\[ e^{\ie t \partial U(h)} v =  U^v(e,\ie t) = U(\alpha_{\ie t}(s_0)) U_h^w(\ie t)
  \quad \mbox{ for } \quad |t| < \pi/2.\]
As $U(S)= U(M) e^{\ie \partial U(C)}$ 
consists of contractions, we obtain
\[ \|e^{\ie t \partial U(h)} v \| \leq \|U_h^w(\ie t)\|. \] 
It thus follows with \eqref{eq:growthcond} 
that $w \in \cH^J_{\rm temp}$ implies that $v \in \cH^J_{\rm temp}$.

We conclude that the dense subspace $U(s_0) \cH^J_{\rm temp}$
is contained $\cH^\omega(\Xi) \cap \cH^J_{\rm temp}$, so that this
subspace is dense as well, and now Theorem~\ref{thm:4.9} applies
to all antiunitary representations $(U,\cH)$ of
$G_{\tau_h}$ with $C \subeq C_U$. So we obtain: 

\begin{theorem} Let $G$ be a connected Lie group for which
  $\eta_G$ is injective and $G_\C$ is simply connected,
  let $h \in \g$ be an Euler element for which $\tau_h$ exists on $G$,
  and  let   $C \subeq \g$ a pointed closed convex $\Ad(G)$-invariant cone,
  satisfying  $-\tau_h(C) = C$ and $\g = C - C + \R h$.
  Further, for $x_{\pm 1} \in C_\pm^{\circ}$,
  let $s_0 := \exp(\ie(x_1 - x_{-1}))$. 
Let $(U,\cH)$ be an antiunitary
representation of $G_{\tau_h} = G \rtimes \{\bone,\tau_h\}$
and
\[ \sE = \beta^+(U(s_0) \cH^J_{\rm temp}).\] 
Then the net $\sH^G_{\sE}$ on $G$ 
 satisfies {\rm (Iso), (Cov), (RS)} and
 {\rm (BW)}, in the sense that
 $\sH^G_\sE(W^G) = \sV$  holds for $W^G= G^h_e.\exp(\Omega_1 + \Omega_{-1})
 \subeq G$
 as in \eqref{eq:defwg}, i.e.,
 $\Omega' = \Omega_1 - \Omega_{-1} \subeq \g^{-\tau_h^\g}$
 is an open connected
 $0$-neighborhood with $\exp(\ie\Omega').s_0 \subeq S$. 
\end{theorem}

\subsection{Semisimple Lie groups}
\label{subsec:semisim}

The following theorem builds on the analytic extension
results from \cite{KSt04}, and their generalization
to non-linear groups by T.~Simon in~\cite{Si24},
which also contains the result on the temperedness, resp.,
the growth condition~\eqref{eq:growthcond}.
These results are used in \cite{FNO25a} to construct
nets of real subspaces on non-compactly causal symmetric spaces.

  \begin{theorem} \label{thm:gss} {\rm(The crown of a semisimple group)}
    Suppose that $\g$ is semisimple, 
  $\g = \fk \oplus \fp$ is a Cartan decomposition, and that 
  $h \in \fp$ an Euler element. We consider three case for a
  connected Lie group $G$ with Lie algebra on which $\tau_h$ exists: 
  \begin{enumerate}
  \item[\rm(a)] If $G \subeq G_\C$,
    $K = \exp_G \fk$ and $K_\C = K \exp(\ie \fk)\subeq G_\C$, 
    then we consider the domain 
\[ \Xi_{G_\C} = G \exp(\ie\Omega_\fp) K_\C \subeq G_\C,
    \qquad 
    \Omega_\fp = \{ x \in \fp \: \Spec(\ad x) \subeq
    (-\pi/2,\pi/2)\},\]
  and
  \begin{equation}
    \label{eq:wc-ss}
    W^c = G^h_e \exp(\ie \Omega_\fp^{-\tau_h^\g}) K_\C^{-\oline\tau_h}.
  \end{equation}
\item[\rm(b)] If $G$ is simply connected,
  $K = \exp_G(\fk)$, and $K_\C$ is the universal complexification of $K$,   then we consider 
    the  simply connected covering manifold $\Xi := \tilde\Xi_{G_\C}$
    of the complex manifold $\Xi_{G_\C}$ and
\begin{equation}
    \label{eq:wc-ssb}
    W^c = G^h_e.\exp(\ie \Omega_\fp^{-\tau_h^\g}). K_\C^{-\oline\tau_h},
  \end{equation}
  with respect to the $G$-action from the left and
  the $K_\C$-action from the right on~$\Xi$. 
\item[\rm(c)] If $G \cong \tilde G/\Gamma$ is a connected Lie group
  with Lie algebra $\g$ and $\Gamma$ is $\tau_h$-invariant, so  that it can also
  be implemented on $G$, then
  we put
  \[ \Xi := \tilde\Xi_{G_\C}/\Gamma \]
  for the simply connected crown domain $\tilde\Xi_{G_\C}$ from {\rm(b)}.
  Then $G$ acts on $\Xi$ from the left, and,
  for $K := \exp_G \fk$ and its universal complexification $K_\C$, the group
  $K_\C$ acts naturally on $\Xi$ from the right, so that we obtain a domain
    \begin{equation}
    \label{eq:wc-ssc}
    W^c = G^h_e.\exp(\ie \Omega_\fp^{-\tau_h^\g}).K_\C^{-\oline\tau_h}.
  \end{equation}
\end{enumerate}
Then, in all cases {\rm(a), (b), (c)}, the conditions
{\rm(Cr1-3)} are satisfied.
\end{theorem}

\begin{proof} (a) As $\Xi_{G_\C}$ is the inverse image of the open crown domain
  \[ \Xi_{G_\C/K_\C} = G \exp(\ie\Omega_\fp)K_\C
    \cong G \times_K \ie \Omega_\fp \]
  in $G_\C/K_\C$,   it is an
  open subset of $G_\C$ which is a holomorphic $K_\C$-principal
  bundle over the contractible space $\Xi_{G_\C/K_\C}$.
  Condition (Cr1) follows from the invariance of~$G$,
  $\exp(\ie \Omega_\fp)$ and $K_\C$ under the antiholomorphic
  involution $\oline\tau_h$.
  To verify (Cr2), we first observe that $W^c
  \subeq \Xi_{G_\C}^{\oline \tau_h}$ follows from the fact that all
  $3$ factors in \eqref{eq:wc-ss} consist of fixed points of $\oline\tau_h$.
  The set of $\oline\tau_h$-fixed points in $\Xi_{G_\C}$
  is a fiber bundle over
  \[ \Xi_{G_\C/K_\C}^{\oline\tau_h}
    \cong G^h_e \times_{K^h_e} \ie \Omega_\fp^{-\tau_h^\g}
\into  G_\C^{\oline\tau_h}/K_\C^{\oline\tau_h}\]
  (\cite[Thm.~6.1]{NO23b}),
  so that the $K_\C$-principal bundle structure of $\Xi_{G_\C}$
  implies that
  $\Xi_{G_\C}^{\oline \tau_h}$ is a $K_\C^{\oline\tau_h}$-principal bundle.
  We conclude that $W^c = \Xi_{G_\C}^{\oline \tau_h}$.
The inclusion $\exp(\cS_{\pm \pi/2}h) W^c \subeq \Xi_{G_\C}$ is equivalent to 
$\exp(\cS_{\pm \pi/2}h) \Xi_{G_\C/K_\C}^{\oline\tau_h} \subeq \Xi_{G_\C/K_\C}$,
which is shown in \cite[\S 8]{MNO24}.

\nin (b) We now assume that $G$ is simply connected, 
  so that $\eta_G \: G \to G_\C$ has discrete kernel and
  $G_\C$ is simply connected. 
  The discussion in the proof of \cite[\S 3, Prop.~5]{FNO25a} shows
  that the simply connected covering $\Xi$ of $\Xi_{G_\C}$ is a complex manifold which is a
  $ K_\C$-principal bundle over the contractible space
  $\Xi_{G_\C/K_\C}$. 
  This implies that 
  $\eta_G \: G \into \Xi_{G_\C}$ lifts to an embedding 
  $G \into \Xi$ and
  $\pi_1(\eta_G(G)) \cong \ker(\eta_G) \cong \pi_1(\Xi_{G_\C})$
  acts as a group  of deck transformations on $\Xi$.
  We thus obtain a free action of~$G$ on $\Xi$ from the left, 
  and a free holomorphic action of the simply connected
 universal complexification $ K_\C$ of the integral subgroup
 $ K = \exp_G \fk \subeq G$ from the right. 
The exponential map $\ie \Omega_\fp \to \Xi_{G_\C}$ also lifts
  to a map $\exp \: \ie\Omega_\fp \to \tilde\Xi_{G_\C}$. 
In this sense, we obtain a factorization 
  \[ \Xi = G \exp(\ie \Omega_\fp)  K_\C. \]
The involution $\tau_h$ of $G$ extends to an antiholomorphic
involution $\oline\tau_h$ on $\Xi$ (by the Lifting Theorem for Coverings),
and we obtain a connected open subset 
  \[ W^c := G^h_e \exp(\ie \Omega_\fp^{-\tau_h^\g})  K_\C^{\oline\tau_h}
    \subeq \Xi^{\oline\tau_h}. \]
Here $ K_\C^{\oline\tau_h}$ is connected because
  $ K_\C$ is simply connected (\cite[Thm.~IV.3.4]{Lo69}).  
  We thus obtain crowned Lie groups in both situations.

  \nin (c) follows easily from (b) by factorization of the discrete
  central subgroup $\pi_1(G) \cong \Gamma \subeq \tilde G$. 
\end{proof}

\begin{theorem} \label{thm:4.9-semisimp}
  Let $(U,\cH)$ be an irreducible antiunitary representation of 
  $G_{\tau_h} = G \rtimes \{\bone,\tau_h\},$
  where $G$ is a connected semisimple Lie group.
  Let $\cF \subeq \cH$ be a finite-dimensional subspace invariant under
  $K$ and $J$.   Then
  \[ \sF := \cF^J \subeq \cH^\omega(\Xi) \cap \cH^J_{\rm temp}
    \quad \mbox{ and }\quad \sE = \beta^+(\sF) \subeq \cH^{-\infty},\]
  for
  $\beta^+$ from \eqref{eq:beta+}.
  The net push-forward net $\sH^M_{\sE}$ from \eqref{eq:pushforward}
  on the non-compactly causal 
  symmetric space $M = G/H$ for $H = K^{\tau_h,h} \exp(\fh_\fp)$ 
  from {\rm Definition~\ref{def:ncc-reductive}}
    satisfies {\rm (Iso), (Cov), (RS)} and
  {\rm (BW)}, where $W = W_M^+(h)_{eH}$ is the connected component
  of the positivity domain of $h$ on $M$, containing the base point.
\end{theorem}

Note that $\tau_h(K) = K$ implies that $J$ leaves the dense subspace
$\cH^{[K]}$ of $K$-finite vectors invariant. Therefore $J$-invariant
finite-dimensional $K$-invariant subspaces exist in abundance.

\begin{proof} (Sketch)  
  The Kr\"otz--Stanton Extension Theorem and Simon's generalization \cite{Si24}
  for non-linear groups imply that
  the space $\cH^{[K]}$ of $K$-finite vectors is contained in
  $\cH^\omega(\Xi)$. By \cite[Thm. 3.2.6]{Si24},
   the space $\cH^{[K]} \cap \cH^J$
   of $J$-fixed $K$-finite vectors, which is dense in $\cH^J$,
   is contained  in $\cH^J_{\rm temp}$, so that
\[ \sF \subeq \cH^\omega(\Xi) \cap \cH^J_{\rm temp} \] 
  (cf.~Theorem~\ref{thm:6-1-fno24}).
  Irreducibility of $U$ further implies that
  $U(G)\sF$ is total in $\cH$.
  Therefore the assumptions of Theorem~\ref{thm:4.9} are satisfied
  for the finite-dimensional subspace
  $\sE = \beta^+(\sF) \subeq \cH^{-\infty}$,
  and natural equivariance properties,
  such as Lemma~\ref{lem:beta+-equiv} then imply that
  it is {\bf $H$-invariant}. Therefore
  the net $\sH^M_{\sE} = (q_M)_* \sH^G_{\sE}$
on $M = G/H$, defined  as in \eqref{eq:pushforward}, 
also satisfies (RS) and (BW) (Remark~\ref{rem:push-forward}). 
\end{proof}

\begin{examplekh} (cf.\ Exercise~\ref{exer:mink-deSitter})
  For de Sitter space $M = \dS^d \subeq \jV := \R^{1,d}$
  and the Lorentzian forms $x^2 = x_0^2 - \bx^2$ on $\R^{1,d}$,
  a natural complexification is the complex sphere
  \[ M_\C :=\{ z = (z_0, \bz) \in \C^{1 + d} \:
    z_0^2 - \bz^2 = -1\}.\]
  It contains $M = M_\C \cap \R^{1,d}$ and also the Riemannian symmetric spaces
  \[ \bH_\pm := \{ (iy_0, i \by) \: y_0^2 - \by^2 = 1, \pm y_0 > 0 \}
    \cong \SO_{1,d}(\R)_e/\SO_d(\R).\]
  Here $G = \SO_{1,d}(\R)_e \subeq G_\C = \SO_{1,d}(\C) \cong\SO_{1+d}(\C)$
  and   $K = \SO_d(\R) \subeq K_\C = \SO_d(\C)$.
  The crown domains of the hyperbolic spaces $\bH_\pm \cong G/K$ are the
intersections with tube domains   $\jV \pm i \jV_+$:  
  \[ \Xi_\pm := M_\C \cap (\jV \pm i \jV_+). \]
For both domains,
\[ \dS^d = \{ (x_0, \bx) \in \R^{1,d} \: x_0^2 - \bx^2 = - 1\}
  \subeq \partial_{M_\C} \Xi_\pm.\] 
 For the Euler element $h$ given by the Lorentz boost 
 \[     h.(x_0, x_1, \ldots, x_{d-1}) = (x_1, x_0, 0,\ldots, 0),\]
$\zeta := \exp(-\frac{\pi \ie}{2} h)$  acts by 
 \[ \zeta.x =  (-i x_1, -i x_0, x_2, \ldots, x_d),\]
 so that $\zeta.i\be_0 = \be_1 \in \dS^d$.

 We also note that, for $\jV \subeq \partial(\jV + i \jV_+)$
 and $C := \oline{\jV_+}$, the set of KMS-points in $\jV$  is
 \[ \jV_{\rm KMS} = C_+^\circ + \jV_0 + C_-^\circ = W_\jV^+(h),
   \quad \mbox{ where } \quad
   C_\pm = \R_{\geq 0}(\be_1 \pm \be_0)\]
 (cf.~Examples~\ref{exs:kms-dom}).  Accordingly,
 \[ \dS^d_{\rm KMS} = \jV_{\rm KMS} \cap \dS^d
   = W_{\dS^d}^+(h).\]
\end{examplekh}

\begin{examplekh} \label{ex:Ipq}
  (cf.\ Exercise~\ref{exer:satellites}) 
  For $G = \SL_n(\R)$, $n = p + q$, and the Euler element
  \[ h := h_q := \frac{1}{n}\pmat{q \bone_p & 0 \\ 0 & -p\bone_q} \in \fsl_n(\R) \]
  from \eqref{eq:hk-sln}, the corresponding non-compactly causal
  involution   $\tau = \tau_h \theta$ (cf.~Subsection~\ref{subsec:ncc-spaces}) is
  \[ \tau(x) = - I_{p,q} x^\top I_{p,q}.  \] 
  Therefore $G^\tau = \SO_{p,q}(\R)$ and, for the action of
  $G$ on $\Sym_n(\R)$, we have 
  \[ M := G.I_{p,q} = \{ g I_{p,q} g^\top \: g \in \SL_n(\R)\}.\]
  This space carries a causal structure for which
  $M \into (\Sym_n(\R), \Sym_n(\R)_+)$ becomes an embedding of
  causal manifolds.

  Here $M_r := G.I_n = G.\bone \cong G/K$ is the corresponding Riemannian symmetric
  space. For
  \[ \zeta := \exp\Big(-\frac{\pi i}{2}h_q\Big) \]
  we have
  \[     \zeta.I_n
    = \exp(-\pi i h_q)  
    = e^{-\pi i q/n}\bone_p \oplus e^{\pi i p/n} \bone_q
  = e^{-\pi i q/n} I_{p,q},\]
so that $G.(\zeta.I_n) \cong G.I_{p,q} \cong M$.
\end{examplekh}

\subsection{The Poincar\'e group}

We consider the Poincar\'e group
\[ G := \R^{1,d} \rtimes \SO_{1,d}(\R)_e
  \subeq G_\C := \C^{1+d} \rtimes \SO_{1+d}(\C)\]
and the Euler element $h \in \so_{1,d}(\R) \subeq \g$, generating
a Lorentz boost:
\[ h(x_0, x_1, \ldots, x_d) := (x_1, x_0,0,\cdots, 0).\]
The corresponding involution
\[  e^{\pi \ie h} = \diag(-1,-1,1,\ldots, 1) \in \SO_{1,d}(\R) \]
acts by conjugation on $G_\C$ (denoted $\tau_h)$ and 
we also obtain an antiholomorphic involution on $G_\C$ 
by $\oline\tau(g) := \tau_h(\oline g)$. 

We consider the action of $G_{\C,\oline\tau_h} = G_\C \rtimes \{\bone,\oline\tau_h\}$
on $M = \C^{1+d}$ by real affine maps, and
\[ \tau_M(z_0, \cdots, z_d) = (-\oline z_0, -\oline z_1,
  \oline z_2, \cdots, \oline z_d).\]
Write 
\[  \jV_+ := \{ x = (x_0, \bx) \in \R^{1,d} \: x_0 > \sqrt{\bx^2} \} \]
for the open future light cone in Minkowski space $\jV := \R^{1,d}$.
Then 
\[ \Xi_M := \R^{1,d} + \ie \jV_+
= \{z \in \C^{1+d} \: \Im z \in \jV_+ \} \]
is an open tube domain in $M$, invariant under $G_{\tau_h}$,
and the subset of $\tau_M$-fixed points is 
\[ \Xi_M^{\tau_M} = \{ (\ie y_0, \ie y_1, y_2, \ldots, y_d) \:
  y_j \in \R, y_0 > |y_1| \} \subeq i \R^2 \oplus \R^{d-1}.\]
For $z := (\ie y_0, \ie y_1, y_2, \ldots, y_d) \in \Xi_M^{\tau_M}$, we have 
\[ \Im(e^{\ie th}z) = (\cos(t)y_0, \cos(t)y_1, 0,\ldots,0), \]  
so that
\[ e^{\ie t h} \Xi_M^{\tau_M}  \subeq \Xi_M \quad \mbox{ for } \quad
  |t| < \pi/2.\]
We thus put $W^{M,c} :=  \Xi_M^{\tau_M}$. 

For $m_0 := \ie \be_0 \in \Xi_M^{\tau_M}$ we now obtain a crown domain
\[ \Xi := \{ g \in G_\C \: g.m_0 \in \Xi_M \},\]
and we put 
\[ W^c := \{ g \in G_\C^{\oline\tau_h} \: g.m_0 \in \Xi_M \} 
  = \{ g \in G_\C^{\oline\tau_h} \: g.m_0 \in W^{M,c} \}
  = \Xi^{\oline\tau_h}.\]
By Lemma~\ref{lem:Mxi},
  this data defines a crowned Lie group $(G,h, \Xi)$. 

The unitary representations
of $G$ that are most relevant in Physics can be realized 
on a Hilbert space $\cH$ of holomorphic functions $f \: \Xi_M \to \cK$,
where $\cK$ is a finite-dimensional Hilbert space, and $\cH$
is specified by a reproducing kernel of the form
\begin{equation}
  \label{eq:K(z,w)}
 K(z,w) = \tilde\mu(z - \oline w) =
 \int_{\jV_+^\star} e^{\alpha(z - \oline w)}\, d\mu(\alpha),
\end{equation}
where $\mu$ is a tempered $\Herm(\cK)_+$-valued measure on the dual cone 
\[ \jV_+^\star = \{ \lambda \in  \jV^* \: \lambda(\jV_+)
\subeq [0,\infty[\}.\]  
The Fourier transform of $\mu$  
is considered as a holomorphic function $\tilde\mu \: \Xi_M \to B(\cK)$, 
whose boundary values define an element in $\cS'(\R^{1+d}, B(\cK))$. 
More concretely, there exists a
representation ${\rho \: G \to \GL(\cK)}$ such that
\[ (U(x,g)f)(z)  = \rho(g) f(g^{-1}.(z-x)),
  \quad  (x,g) \in G, z \in \Xi_M.\] 
We refer to \cite{NOO21} for a detailed discussion of the analytic aspects
of such Hilbert spaces and the standard subspaces associated to~$h$.
We extend $U$ to an antiunitary representation of $G_{\tau_h}$ by
\[ (Jf)(z) := J_\cK f(\tau_M(z)),\]
where $J_\cK$ is a conjugation on $\cK$.
Then $\cH^J$ consists of those functions
with $f(\Xi_M^{\tau_M}) \subeq \cK^{J_\cK}$
(\cite[Lemma~2.5]{NOO21}).

Let
\[ K_z \: \cH \to \cK, \quad K_z(f) := f(z) \]
denote the evaluation operator in $z \in \Xi_M$ and
$K_z^* \: \cK \to \cH$ its adjoint. Then the functions 
\begin{equation}
  \label{eq:kwxi}
  K_w^* \xi, \quad w \in \Xi_M^{\tau_M}, \xi \in \cK^{J_\cK}
\end{equation}
are contained in $\cH^J$ and span a dense subspace thereof
(\cite[Lemma~3.11]{NOO21}). A~straightforward calculation shows that
\[ U(g) K_w^*\xi = K_{g.w}^* \rho(g^{-1})^* \xi.\]
As the representation $\rho$ extends to a holomorphic representation
of $G_\C$, it follows that $K_w^*\xi \in \cH^\omega(\Xi)$ if 
$w \in W^{M,c} = \Xi_M^{\tau_M}$, and thus
all functions \eqref{eq:kwxi} are contained in $\cH^\omega(\Xi)^J$.
To see that they are actually contained in $\cH^J_{\rm temp}$,
we need to estimate the norms
$\|e^{\ie t \partial U(h)} K_w^*\xi\|$ for $|t| \to \pi/2$
(cf.\ \eqref{eq:growthcond}). 
As $\rho(\exp(\ie th))$ is bounded for $|t| \leq \pi/2$, we have to verify
that
\[ \| K(e^{\ie th}w, e^{\ie th}w)\| \leq C \Big(\frac{\pi}{2} - |t|\Big)^{-N}
\quad \mbox{ for some } \quad C, N > 0.\] The operator
$K(e^{\ie th}w, e^{\ie th}w)$ is the Fourier transform of $\mu$, evaluated
in
\[ e^{\ie th} w - e^{-\ie th} \oline w, \quad
  w = (\ie  y_0, \ie  y_1, y_2, \ldots, y_d),\]
so that
\begin{equation}
  \label{eq:expws}
  e^{\ie th} w - e^{-\ie th} \oline w
  = 2\cos(th) (\ie y_0 \be_0 + \ie y_1 \be_1) \in \R \ie \be_0 + \R \ie \be_1.
\end{equation}
In view of \cite[Prop.~4.11, \S 2.3]{FNO25a}, the
temperedness of the measure $\mu$ yields an estimate
\[ \|\tilde\mu(x + \ie y)\| \leq C \|y\|^{-N}\quad \mbox{ for } \quad
  x + \ie y \in \R^{1,d} + \ie \jV_+, \]
and we conclude from \eqref{eq:expws}
that the functions $K_w^*\xi$ are contained in $\cH^J_{\rm temp}$. 
Therefore all our assumptions are satisfied for the finite-dimensional space
\[ \sF := \{ K_{\ie \be_0}^* \xi \:  \xi \in \cK^{J_\cK}\}.\]
We refer to \cite{NOO21} for detailed descriptions of the corresponding
standard subspaces $\sV \subeq\cH$. 

\subsection{Appendices to Section~\ref{sec:4}}

\subsubsection{Tools for nets of real subspaces}
\label{app:4.1}

\begin{lemma} \label{lem:fraglem}
Let $\cO \subeq G$ be open and $\phi \in C^\infty_c(\cO)$. 
We further assume that $(\cO_j)_{j \in J}$ is an open cover of~$\cO$.
Then there exist $j_1, \ldots, j_k \in J$ and $\phi_\ell \in C^\infty_c(\cO_{j_\ell})$ such that $\phi  = \phi_1 + \cdots + \phi_k$.
\end{lemma}

\begin{proof} The family $(\cO_j)_{j \in J}$ is
  an open cover of $\supp(\phi)$, and there
  exist $j_1, \ldots, j_k \in J$ such that
  \[ \supp(\phi) \subeq \cO_{j_1} \cup \cdots \cup \cO_{j_k}.\]  
  Then
  \[ G \setminus \supp(\phi), \quad  \cO_{j_1}, \ldots, \cO_{j_k}\]
  is an open cover of $G$. Let $\chi_0, \ldots, \chi_k$ be a subordinated
  partition of unity. Then   $\phi = \sum_{j = 1}^k \phi_j$,
  where $\phi_j := \chi_j \phi$ satisfies $\supp(\phi_j) \subeq \cO_j$.
\end{proof}

\begin{lemma} \label{lem:fragment} {\rm(Fragmentation Lemma)} 
For $\eset\not=\cO \subeq G$ open,  the following assertions hold: 
\begin{enumerate}
\item[\rm(a)] If $H \subeq G$ is a closed subgroup, then 
  \begin{enumerate} 
  \item[\rm(i)] every test function $\phi \in C^\infty_c(\cO H,\R)$ is a finite sum of 
test functions of the form 
\[ \psi \circ \rho_p \: G \to \C, \quad g \mapsto \psi(gp),
\quad 
\psi \in C^\infty_c(\cO,\R), p \in H.\] 
  \item[\rm(ii)] every test function $\phi \in C^\infty_c(H\cO,\R)$ is a finite sum of 
test functions of the form 
\[ \psi \circ \lambda_p \: G \to \C, \quad g \mapsto \psi(pg),
\quad 
\psi \in C^\infty_c(\cO,\R), p \in H.\] 
\end{enumerate} 
\item[\rm(b)] Every $\phi \in C^\infty_c(G,\R)$ 
is a finite sum $\sum_{j = 1}^n \phi_j \circ \lambda_{g_j}$ 
with $\phi_j \in C^\infty_c(\cO,\R)$ and $g_j \in G$. 
\end{enumerate}
\end{lemma}

\begin{proof} (a)(i) The family $(\cO p)_{p \in H}$ is an open cover of the 
  compact subset $\supp(\phi)$, so that
  Lemma~\ref{lem:fraglem} implies that
$\phi = \sum_{j = 1}^n \phi_j$ with $\supp(\phi_j) \subeq \cO p_j$. 
Then $\psi_j := \phi_j \circ \rho_{p_j} \in C^\infty_c(\cO,\R)$ and 
$\phi = \sum_{j= 1}^n \psi_j \circ \rho_{p_j^{-1}}$. 

\nin (a)(ii) and (b) are proved along the same lines. For (b), we use 
the open cover $(g\cO)_{g \in G}$ of the group $G$. 
\end{proof}

\begin{lemma}  \label{lem:sw64-genb} 
Let $(U, \cH)$ be a unitary representation of $G$, let 
$\sE \subeq \cH^{-\infty}$ be a real linear subspace, 
$H \subeq G$ a closed subgroup 
and $\eset\not=\cO \subeq G$. 
Then the following assertions hold: 
\begin{description}
\item[\rm(a)] $\sH_\sE^G(\cO H) = \sH_{\hat\sE}^G(\cO)$ for
  $\hat\sE := \Spann(U^{-\infty}(H)\sE)$.
\item[\rm(b)] $\sH_\sE^G(\cO H) = \sH_\sE^G(\cO)$ if $\sE$ is $H$-invariant. 
\item[\rm(c)] $\sH_\sE^G(H\cO)$ is the closed real span of $U(H)\sH_\sE^G(\cO)$. 
\item[\rm(d)] The real subspace 
spanned by $U(G)\sH_\sE^G(\cO)$ is dense in $\sH_\sE^G(G)$. 
\end{description}
\end{lemma}

\begin{proof} (a)
  For $\phi = \psi \circ \rho_g$, $\psi \in C^\infty_c(\cO)$ and $g \in H$, we 
obtain with \eqref{eq:rightrel} in Appendix~\ref{app:c1} 
\[ U^{-\infty}(\phi)\sE 
= U^{-\infty}(\psi \circ \rho_g) \sE
= \Delta_G(g)^{-1} U^{-\infty}(\psi) U^{-\infty}(g^{-1})\sE
\subeq  U^{-\infty}(\psi)\hat\sE.\] 
Hence Lemma~\ref{lem:fragment}(a)
implies that $\sH_\sE^G(\cO H) \subeq \sH_{\hat\sE}^G(\cO)$.

Conversely, for $g \in H$ and $\psi \in C^\infty_c(\cO,\R)$, we have
\[ U^{-\infty}(\psi) U^{-\infty}(g)\sE 
=  \Delta_G(g)^{-1} U^{-\infty}(\psi \circ \rho_g) \sE 
\subeq \sH^G_\sE(\cO H),\]
hence also $U^{-\infty}(\psi) \hat\sE \subeq \sH^G_\sE(\cO H),$ 
and this implies that $\sH_{\hat\sE}^G(\cO) \subeq \sH_\sE^G(\cO H)$.

\nin (b) follows from (a).

\nin (c)  From Remark~\ref{rem:iso-cov-net} we know  that 
$U(g) \sH_\sE^G(\cO) = \sH_\sE^G(g\cO) \subeq \sH_\sE^G(H\cO)$ 
for $g \in H$. Now the assertion follows from Lemma~\ref{lem:fragment}(b). 

\nin (d) is an immediate consequence of (c), applied with $H = G$. 
\end{proof}

\begin{remark} For a net $\sH^M_\sE$ on the homogeneous space $M = G/H$, the preceding
  lemma implies that, for all open subset $\cO \subeq M$, we have
  \[ \sH^M_\sE(\cO)
    = \sH^G_\sE(q_M^{-1}(\cO))
    = \sH^G_\sE(q_M^{-1}(\cO)H) = \sH^G_{\hat\sE}(q_M^{-1}(\cO)) = \sH^M_{\hat\sE}(\cO).\]
  Therefore the nets $\sH^M_\sE$ on $M$ can always be constructed from
  $U(H)$-invariant subspaces $\sE \subeq \cH^{-\infty}$.

  If (BW) holds for some wedge region $W \subeq M$ containing the base point,
  if follows in particular from $\sV = \sH^M_\sE(W)$ that
  \[ U^{-\infty}(H)\sE \subeq \cH^{-\infty}_{\rm KMS} \]
  (cf.\ Corollary~\ref{cor:4.8}(d)). This requirement on $\sE$ makes it
  harder to construct nets of the form $\sH^M_\sE$ on homogeneous spaces satisfying
  the (BW) condition; see in particular Problem~\ref{prob:cc-rs}. 
\end{remark}

\begin{problem} Let $(U,\cH)$ be an irreducible antiunitary
  representation of $G_{\tau_h}$ and $H \subeq G$ an integral subgroup.
  When does $\cH^{-\infty}_{\rm KMS}$ contain a non-trivial
  $U^{-\infty}(H)$-invariant subspace.

  We know from \cite{FNO25a} that this is always the case if $G$ is semisimple
  and $G/H$ is a non-compactly causal symmetric space associated
  to the Euler element $h$ as in Theorem~\ref{thm:ncc-classif},
  but in general the existence of such subspaces is not clear.
  The case where $G/H$ is a modular compactly causal symmetric space
  is relevant for an answer to Problem~\ref{prob:cc-rs} below. 
\end{problem}

\subsubsection{Simon's Growth Theorem}

The following result is \cite[Thm.~3.2.6]{Si24},
where, in addition, we use \cite[Thm.~3]{FNO25a} for the existence of the limit
in the smaller subspace $\cH^{-\infty}(\partial U(x))\subeq \cH^{-\infty}$.
This result generalizes the extension results
by Kr\"otz and Stanton \cite{KSt04} by
removing the condition on the group~$G$ that its universal
complexification is injective.

\begin{theorem} \label{thm:simon-gro} {\rm(Simon's Growth Theorem)} 
  Let $G$ be a connected semisimple Lie group with Cartan decomposition
  $G = K \exp \fp$ and   $(\pi, \cH)$ be an irreducible unitary representation
  of~$G$. Then there exist for every $K$-finite vector $v \in \cH$
  constants $C,n > 0$ such that, for every $x \in \fp$ with
spectral radius  $r_{\Spec}(\ad x) < \pi/2$, we have 
\[  \|e^{\ie \partial U(x)}v\| \leq C \Big( \frac{\pi}{2} - r_{\Spec}(\ad x)\Big)^{-n}.\]
  In particular,   $\lim_{t \to \frac{\pi}{2}-} e^{it\partial U(h)}v$ exists in
  $\cH^{-\infty}(U_h)$ for $h \in \fp$ with   $r_{\Spec}(\ad h) = 1$
  and $U_h(t)= U(\exp th)$.
\end{theorem}

The last statement uses the equivalence of (b) and (c) in
Theorem~\ref{thm:6-1-fno24}.

\begin{remark} The conclusion of the preceding theorem does not
  hold for reducible unitary representations without suitable restrictions.
  It prevails for finite sum or irreducible representations, but in general
  not for infinite sums or direct integrals. 
  
  Consider, for instance, the Lie subalgebra
  $\g = \so_{1,2}(\R) \subeq \g^\sharp := \su_{1,2}(\C)$
  and an Euler element $h \in \g$ (a Lorentz boost).
  Then $h$ is {\bf not} an Euler element in $\g^\sharp$
  because the eigenvalues of $\ad_{\g^\sharp}(h)$ are
  $\{0,\pm 1, \pm 2\}$.

  Now consider a unitary representation $(U,\cH)$ of
  $G$ that extends to a unitary representation of the larger group
  $G^\sharp$. Then  $K^\sharp = G^\sharp \cap \U_3(\C)$
  is a maximal compact subgroup. We consider a
  $K^\sharp$-finite vector $v \in \cH$.
  Then Simon's Growth Theorem provides estimates
  for $e^{i t \partial U(h)}v$ for $t \to \pm \frac{\pi}{4}$,
  and in general the limit does not exist in $\cH$,
  so that $e^{i t \partial U(h)}v$ is not defined for
  $\frac{\pi}{4} \leq t < \frac{\pi}{2}$.
\end{remark}

\begin{small}
\subsection{Exercises for Section~\ref{sec:4}}

\begin{exercise} (Crown domains of  convex cones)
  Let $C \subeq E$ be a generating closed convex cone in the
  finite-dimensional real linear space $E$
  and $h \in \End(E)$ be diagonalizable with eigenvalues
  $\{-1,0,1\}$, such that
   \[ e^{\R h} C = C \quad \mbox{ and } \quad
     -\tau_h(C) = C \quad \mbox{ for } \quad \tau_h = e^{\pi i h}.\]
   Show that 
   \begin{enumerate}
   \item[(a)] If $C_\pm := \pm C \cap E_{\pm 1}(h)$, then
     $\pm C_\pm = p_{\pm 1}(C)$, where $p_{\pm 1} \: E \to E_{\pm 1}(h)$
     is the projection along the other eigenspaces of $h$. 
   \item[(b)] $\oline\tau_h(z) := \oline{\tau_h(z)}$ defines an 
     antilinear involution on $E_\C$, preserving the tube domain
     $\Xi := E + i C^\circ$, and $\Xi^{\oline\tau_h} = i C^\circ$.
   \item[(c)] The set $E_{\rm KMS}$ of those elements $v$ for which
     the orbit map $\alpha^v(z) := e^{zh}v, \C \to E_\C$ satisfies
     \[ \alpha^v(\cS_\pi) \subeq \Xi \quad \mbox{ and } \quad
       \alpha^v(\pi i) = \tau_h v \]
     coincides with
     \[ E_{\rm KMS} = C_+^\circ \oplus E_0(h) \oplus C_-^\circ.\] 
   \end{enumerate}
\end{exercise}

\begin{exercise} (The crown of real hyperbolic space)
  Let $E := \R^{1,d}$ be $(d+1)$-dimensional Minkowski space and
  \[ C := \{ (x_0, \bx) \: x_0 \geq 0, x_0^2 - \bx^2 \geq 0 \} \]
  the closed positive light cone.
  We consider the action of the group $G := \SO_{1,d}(\R)_e$ on $E$
  and its complexification $E_\C = E + i E$.  Show that:
  \begin{enumerate}
  \item[\rm(a)] $\bH^d := \{ (ix_0, i \bx) \: x_0^2 - \bx^2 = 1 \}$ 
  is the orbit $G.i\be_0 \cong G/K$ for $K \cong \SO_d(\R)$.
\item[\rm(b)] $\dS^d := \{ (x_0, \bx) \: x_0^2 - \bx^2 = -1\}
  = G.\be_1$ and both lie in the orbit  
  \[ M_\C := G_\C.\be_1 = G_\C.i\be_0
    = \{ (z_0, \bz) \in \C^{1+d} \: z_0^2 - \bz^2 = - 1\} \] 
(the complex sphere) of the complex orthogonal group $G_\C = \SO_{1,d}(\C)$. 
\item[\rm(c)] Let $h(x_0,\bx) = (x_1, x_0, 0, \ldots, 0)$
    denote the Lorentz boost and
    \[ \oline\tau_h(z) = (-\oline z_0, - \oline z_1, \oline z_2,
      \cdots, \oline z_d) \]
    the corresponding antilinear involution. Then
    \[ E_{\rm KMS} := \{ x \in E \: 
\alpha^x(\cS_\pi) \subeq E + i C^\circ, 
\tau_h(x) = \alpha^x(\pi i)\} 
= W_R  := \{ (x_0, \bx) \: x_1 > |x_0| \} \]  
    is the Rindler wedge.
\item[\rm(d)] The open domain
  \[ \Xi := (E + i C^\circ) \cap M_\C\]
  has the following properties:
  \begin{enumerate}
  \item[\rm(i)] It is a $G$-invariant open neighborhood of $\bH^d$. 
  \item[\rm(ii)] $\dS^d \subeq \partial \Xi$. 
  \item[\rm(iii)]  $\Xi^{\oline\tau_h}
    = \{ (i x_0, i x_1, x_2, \ldots, x_d) \in M_\C \:
      x_j \in \R, x_0 > |x_1| \}$ and
\[ e^{ -\frac{\pi i}{4}h} \Xi^{\oline \tau_h}= W_R \cap \dS^d.\]  
  \item[\rm(iv)] $\dS^d_{\rm KMS}
    = \{ x \in \dS^d \: \alpha^x(\cS_\pi) \subeq \Xi, \tau_h(x)
    = \alpha^x(\pi i)\} =  W_R \cap \dS^d$.
  \end{enumerate}
  \end{enumerate}
\end{exercise}

\begin{exercise} \label{exer:satellites}
  (Crowns of the ncc spaces $\GL_{p+q}(\R)/\OO_{p,q}(\R)$ from
  Example~\ref{ex:Ipq}) 
For $n = p + q, 0 < p < n$, 
we consider in $\Sym_n(\R)$ the causal symmetric space 
\[ M = \GL_n(\R).I_{p,q} 
= \{ gI_{p,q} g^\top \: g \in \SL_n(\R) \} 
\subeq E := \Sym_n(\R), \quad 
I_{p,q} = \bone_p \oplus -\bone_q\]
and the Euler element 
\[ h_p := \bone_p \oplus 0 = \diag(1,\ldots, 1,0,\ldots, 0).\] 
Show that: 
\begin{description}
\item[\rm(a)] $M$ is the set of all symmetric matrices of signature~$(p,q)$. 
In particular, $M$ is open. 
\item[\rm(b)] $M_\C := \Sym_n(\C) \cap \GL_n(\C)$ 
is a complex homogeneous space of $\GL_n(\C)$. 
\item[\rm(c)] For $\Xi := M_\C \cap (E + i E_+)$ 
($E_+ = \{ x \: x \> 0\}$ is 
the open cone of positive definite matrices), we have 
\[ M_{\rm KMS} = M \cap 
\Big\{ \pmat{ a & b \\ b^\top & d} \in \Sym_{p+q}(\R) \: 
a \> 0, d \< 0\Big\}.\] 
\end{description}
\end{exercise}

\end{small}

\section{Minimal and maximal nets for unitary representations} 
\label{sec:5}

Let $G$ be a connected Lie group,
$h \in \g$ an Euler element, and suppose that
the  involution $\tau_h^\g = e^{\pi i \ad h}$ on $\g$ integrates to an involution
$\tau_h$ on $G$, so that we can form the semidirect product
$G_{\tau_h} = G \rtimes \{\bone, \tau_h\}$.

We also fix a homogeneous space $M = G/H$, in which we 
consider an open subset $W\not=\eset$, invariant under the one-parameter group
$\exp(\R h)$. We call the translates $(gW)_{g \in G}$ of $W$
{\it wedge regions} in~$M$.  \index{wedge region!in $M = G/H$ \scheiding} 
At the outset,  we do not assume any specific properties
of~$W$ or $M$, but Lemma~\ref{lem:direct-net} below will indicate which
properties good choices of $W$ should have.

We consider an antiunitary representation $(U,\cH)$ of $G_{\tau_h}$
and the canonical standard subspace $\sV = \sV(h,U) \subeq \cH$,
specified by 
$\Delta_\sV = e^{2 \pi i \cdot\partial U(h)}$ and $J_\sV = U(\tau_h)$
(cf.\ The Euler Element Theorem~\ref{thm:2.1}). 

\subsection{Minimal and maximal nets}

We associate to the open $\exp(\R h)$-invariant
subset $W \subeq M = G/H$ and
the antiunitary representation $(U,\cH)$ of $G_{\tau_h}$ the two nets
$\sH_M^{\rm min}$ and $\sH^{\rm  max}_M$,
defined on open subsets of $M$ by 
\begin{equation}
  \label{eq:def-ho}
  \sH_M^{\max}(\cO) := \bigcap_{g\in G, \cO \subeq gW} U(g)\sV
  \quad \mbox{ and } \quad 
  \sH^{\rm min}_M(\cO) := \oline{\sum_{g\in G, gW \subeq \cO} U(g)\sV}.
\end{equation}
We call $\sH_M^{\rm max}$ the {\it maximal net} and 
$\sH_M^{\rm min}$ the {\it minimal net} associated to $U,M,W$. This
is justified
by Lemma~\ref{lem:maxnet-larger} below.
By construction, these nets are isotone and covariant,
and we shall see in Lemma~\ref{lem:direct-net} below that
they assign $\sV$ to $W \subeq M$ if and only if
\index{net!maximal on $M$, $\sH^{\rm max}_M$ \scheiding } 
\index{net!minimal on $M$, $\sH^{\rm min}_M$ \scheiding } 
\begin{equation}
  \label{eq:semi-incl}
  S_W = \{ g \in G \: g.W \subeq W \} \subeq S_\sV
  = \{ g \in G \:  U(g)\sV \subeq \sV\}.
\end{equation}
Any other properties of these nets require a more detailed
analysis.

\begin{remark} (a)  If there exists 
no $g \in G$ with $\cO \subeq gW$, i.e., $\cO$ is not contained in
any wedge region, then $\sH_M^{\max}(\cO) = \cH$ (the empty intersection).
Otherwise $\sH_M^{\max}(\cO)$ is always separating because it is
contained in a standard subspace.

We likewise get
$\sH_M^{\rm min}(\cO) := \{0\}$ (the empty sum) if there exists
no $g \in G$ with $gW \subeq \cO$, i.e., $\cO$ contains no
wedge region. Otherwise $\sH_M^{\rm min}(\cO)$ is always cyclic because
it contains a standard subspace.

\nin (b) If $\eset \not= W \not= M$, then we have in particular
 \[  \sH_M^{\rm min}(\eset) = \{0\} \subeq
   \sH_M^{\rm max}(\eset) = \bigcap_{g \in G} U(g)\sV \]
 and
 \[    \sH_M^{\rm min}(M) = \oline{\sum_{g \in G} U(g)\sV} \subeq
   \sH_M^{\rm max}(M) = \cH.\] 
 \end{remark}
As for locality issues, we note that
\[ \sV' = \sV(-h,U) = J\sV \]
need not be contained in $U(G)\sV$, and even if this is the case,
then locality properties of $\sH^{\rm max}_M$ are not immediate.
We refer to \cite{NO25} for a discussion of non-compactly causal symmetric
spaces in this context. 

The following lemma is elementary. It only
uses the Equality Lemma~\ref{lem:lo08-3.10} to verify the equality of
standard subspaces.

\begin{lemma} \label{lem:direct-net}
    The following assertions hold:
  \begin{description}
  \item[\rm(a)] The nets $\sH_M^{\max}$  and $\sH_M^{\rm min}$ 
    on $M$ satisfy {\rm(Iso)} and {\rm(Cov)}. 
  \item[\rm(b)] The set of all open subsets $\cO \subeq M$
  for which $\sH_M^{\max}(\cO)$ is cyclic is $G$-invariant.
\item[\rm(c)] The following are equivalent:
  \begin{description}
  \item[\rm(i)] $S_W \subeq S_\sV$.
  \item[\rm(ii)] $\sH_M^{\max}(W) = \sV$. 
  \item[\rm(iii)] $\sH_M^{\max}(W)$ is standard. 
  \item[\rm(iv)] $\sH_M^{\max}(W)$ is cyclic. 
  \item[\rm(v)] $\sH_M^{\rm min}(W) = \sV$. 
  \item[\rm(vi)] $\sH_M^{\rm min}(W)$ is standard. 
  \item[\rm(vii)] $\sH_M^{\rm min}(W)$ is separating. 
  \end{description}
\end{description}
\end{lemma}

\begin{proof} (a) Isotony is clear and covariance of the maximal net
  follows from
  \[ \sH_M^{\max}(g_0\cO)
    = \bigcap_{g_0\cO \subeq gW} U(g)\sV
    = U(g_0)\bigcap_{g_0\cO \subeq gW} U(g_0^{-1} g)\sV
    = U(g_0)\sH_M^{\max}(\cO).\]
{The argument for the minimal net is similar.} 

  \nin (b) follows from covariance.

  \nin (c) (i) $\Leftrightarrow$ (ii): Clearly,
    $\sH_M^{\rm max}(W) \subeq \sV$, and equality holds if and only if
    $W \subeq gW$ implies $U(g)\sV \supeq \sV$,
    which is equivalent to $S_W^{-1} \subeq S_\sV^{-1}$, and this is equivalent
    to~(i).

 (ii) $\Rarrow$ (iii) $\Rarrow$ (iv) are trivial.

(iv) $\Rarrow$ (ii): By covariance
 and $\exp(\R h).W = W$, the subspace $\sH_M^{\max}(W) \subeq \sV$
  is invariant under the modular group $U(\exp \R h)$ of $\sV$.
  If $\sH_M^{\max}(W)$ is cyclic, then
  the Equality Lemma~\ref{lem:lo08-3.10} implies $\sH_M^{\max}(W) = \sV$.

(i) $\Leftrightarrow$ (v) is obvious. 

 (v) $\Rarrow$ (vi) $\Rarrow$ (vii) are trivial.

 (vii) $\Rarrow$ (v): By covariance
 and $\exp(\R h).W = W$, the subspace
  $\sH_M^{\rm min}(W) \supeq \sV$
  is invariant under the modular group $U(\exp \R h)$ of $\sV$.
  If $\sH_M^{\rm min}(W)$ is separating, 
then the Equality Lemma~\ref{lem:lo08-3.10} implies $\sH_M^{\rm min}(W) = \sV$.
\end{proof}

For later applications, we record the following observation. 
\begin{lemma} \label{lem:tensprostand} 
Suppose that $(U,\cH) = \otimes_{j =  1}^n (U_j, \cH_j)$ is a
tensor product of  antiunitary  representations of $G_{\tau_h}$.
Then the standard subspace $\sV = \sV(h,U) = \Fix(J e^{\pi i \cdot \partial U(h)})$
is a tensor product
\[  \sV = \sV_1 \otimes \cdots \otimes \sV_n, \]
and, for every non-empty subset $A \subeq G$,
the subset $\sV_A := \bigcap_{g \in A} U(g)\sV$ satisfies
\begin{equation}
  \label{eq:vtens}
  \sV_A \supeq \sV_{1,A} \otimes \cdots \otimes \sV_{n,A}.
\end{equation}
\end{lemma}

\begin{proof} We have $\xi \in \sV_A$ if and only if
  $U(A)^{-1} \xi \subeq \sV$.
  This shows that any
  $\xi = \xi_1 \otimes \cdots \otimes \xi_n$ with $\xi_j \in \sV_{j,A}$
  is contained in $\sV_A$, which is \eqref{eq:vtens}.
\end{proof}

The following lemma is a consequence of the naturality
of the minimal and the maximal net. 

\begin{lemma} \label{lem:direct-net-d}
  For $A := \{ g \in G \: g^{-1} \cO \subeq W \}$,
  we have   \begin{equation}
    \label{eq:sh=va}
    \sH_M^{\max}(\cO) = \sV_A := \bigcap_{g \in A} U(g)\sV, 
  \end{equation}
and the cyclicity of this subspace 
  is inherited by subrepresentations, direct sums, direct integrals
  and finite tensor products. 
\end{lemma}

\begin{proof}   For a direct sum representation
  $U = U_1 \oplus U_2$, we have  $\sV = \sV_1 \oplus \sV_2$, which leads to
 \begin{equation}
   \label{eq:v-dirsum}
\sV_A = \sV_{1,A} \oplus \sV_{2,A}    
\end{equation}
because $U(g)^{-1}(v_1, v_2) \in \sV$ is equivalent to
$U_j(g)^{-1}v_j \in \sV_j$ for $j= 1,2$. 
We thus obtain
 \[ \sH_M^{\max}(\cO) = \sH_{M,1}^{\max}(\cO) \oplus \sH_{M,2}^{\max}(\cO).\]
  This proves that cyclicity of $\sH_M^{\max}(\cO)$ is inherited
  by subrepresentations and direct sums. 
  For finite tensor products, the assertion 
  follows from Lemma~\ref{lem:tensprostand}.
  If $U = \int_X^\oplus U_x\, d\mu(x)$ is a direct integral,
  then  \eqref{eq:sh=va} and   Lemma~\ref{lem:g-inter}(a) 
imply that \begin{equation}
  \label{eq:net-dirint}
  \sH_M^{\max}(\cO) = \int_X^\oplus \sH_{M,x}^{\max}(\cO)\, d\mu(x)
\end{equation}
for direct integrals. 
So Lemma~\ref{lem:di1} in Appendix~\ref{app:D}
shows that $\sH_M^{\max}(\cO)$ is cyclic if 
every $\sH_{M,x}^{\max}(\cO)$ is cyclic in~$\cH_x$.
 \end{proof}

\begin{remark} (Inner and outer $W$-saturation of subsets) If we write
  \[ \cO^\wedge := \Big(\bigcap_{gW \supeq \cO} gW\Big)^\circ \supeq \cO
    \quad \mbox{ and } \quad 
    \cO^\vee := \bigcup_{gW \subeq \cO} gW \subeq \cO,\]
  then $\cO^\wedge$ and $\cO^\vee$ are open subsets satisfying
  $(\cO^\wedge)^\wedge = \cO^\wedge$, $(\cO^\vee)^\vee = \cO^\vee$, and 
  \begin{equation}
    \label{eq:enla}
 \sH_M^{\rm max}(\cO^\wedge) = \sH_M^{\max}(\cO) \quad \mbox{ and }  \quad   
 \sH_M^{\rm min}(\cO^\vee) = \sH_M^{\rm min}(\cO).
  \end{equation}
  So the values of the maximal net are takes on the subset of
  open subsets $\cO \subeq M$
  satisfying $\cO = \cO^\wedge$ (interiors of intersections of
  wedge regions)
  and the minimal net on those open subsets
satisfying $\cO = \cO^\vee$ (unions of wedge regions)
\end{remark}

\begin{remark} (The case where $S_W$ is a group) 
  If the semigroup $S_W$ is a group, i.e., 
  $S_W = G_W$ and $\ker(U)$ is discrete, then the inclusion 
   $S_W \subeq S_\sV$ is equivalent to
  \begin{equation}
     \label{eq:gwx}
     G_W \subeq G_\sV = G^{h,J} = \{ g \in G^h \: J U(g) J = U(g)\}
   \end{equation}
(cf.\ Exercise~\ref{exer:sym}).
 In the context of causal homogeneous spaces, the 
     definition of~$W$ as a connected component of
     $W_M^+(h)$ (Definition~\ref{def:wedge}) 
     implies that $\exp(\R h) \subeq G^h_e \subeq G_W$, and we 
have in many concrete examples that $G_W \subeq G^h$, 
and always $\L(G_W) =  \g^h$ (Proposition~\ref{prop:LSW}).
 However, $U(G_W)$ need not commute with~$J$,
 so that \eqref{eq:gwx} may fail.
 Examples arise already for $\g = \fsl_2(\R)$; see
 Remark~\ref{rem:fno23-5.13} below. 
 \end{remark}

 \begin{remark} \label{rem:fno23-5.13}
   If $\g = \fsl_2(\R)$ and $G$ is a connected Lie group with
   Lie algebra $\g$, then
  $G_{\rm ad} = \PSL_2(\R) \cong \SO_{1,2}(\R)_e$, 
  and $H_{\rm ad} = \exp(\R h)$, so that $G_{\rm ad}/H_{\rm ad} \cong \dS^2$
  (Example~\ref{ex:causal2}).     If $Z(G)$ is non-trivial, 
    then the connected components of $W_M^+(h)$ can be labeled
    by the elements of $Z(G)$ because this subgroup
    acts non-trivially on $M = G/H_{\rm ad}$, 
    leaving the positivity region $W_M^+(h)$ invariant.
    In any irreducible representation $(U,\cH)$ we have $U(Z(G)) \subeq \T$,
    but this subgroup preserves the standard subspace
    $\sV$ if and only if it is contained
    in~$\{\pm 1\}$.     
\end{remark}

The following lemma justifies the terminology ``minimal'' and ``maximal''. 
\begin{lemma} \label{lem:maxnet-larger}
  Let $(U,\cH)$ be an antiunitary representation of $G_{\tau_h}$
  and $\sH$ a net of real subspaces on open subsets of $M$, satisfying
  {\rm (Iso), (Cov)}, and for which the Bisognano--Wichmann property holds
  in the sense that $\sH(W) = \sV = \sV(h,U)$.
Then   
\begin{equation}
  \label{eq:lem.5.7}
 \sH_M^{\rm min}(\cO) \subeq \sH(\cO) \subeq \sH_M^{\rm max}(\cO)
 \qquad \mbox{ for } \ \cO \subeq M \ \mbox{open},
\end{equation}
  and equality holds for all domains of the form
  $\cO = g.W$, $g \in G$, i.e., the wedge regions in $M$. 
\end{lemma}

 \begin{proof} The three properties
   (Iso), (Cov) and $\sH(W) = \sV$ of the net $\sH$ entail
   $S_W \subeq S_\sV$ because $g.W \subeq W$ implies
    \[ U(g)\sV
      = \ U(g) \sH(W)
     \ {\buildrel{\rm(Cov)}\over= }\ \sH(g.W)
     \ {\buildrel{\rm(Iso)}\over \subeq }\ \sH(W)
     \ {\buildrel{\rm(BW)}\over \subeq }\ =  \sV.\]
 From Lemma~\ref{lem:direct-net}(c) we thus obtain
$\sH_M^{\rm max}(W)= \sH_M^{\rm min}(W) = \sV$. 
Hence
\[ \sH(gW) = U(g) \sV = \sH_M^{\rm max}(gW) = \sH_M^{\rm min}(gW) \] 
     by covariance for any $g \in G$ (Lemma~\ref{lem:direct-net}(a)).
By (Iso), $\cO \subeq g W$ implies
$\sH(\cO) \subeq \sH(gW) = U(g)\sV,$ so that
     $\sH(\cO) \subeq \sH_M^{\rm max}(\cO)$.
     Likewise, $gW \subeq \cO$ implies
     $U(g)\sV  = \sH(gW) \subeq \sH(\cO)$, and thus
     $\sH^{\rm min}_M(\cO) \subeq~\sH(\cO)$. 
\end{proof}

   \begin{remark} \mlabel{rem:pushforward-to-M}
   The construction of the minimal and the maximal net can also be carried
   out on $G$ itself with respect to $W^G = q_M^{-1}(W)$.
   It then makes sense to compare the minimal/maximal net on $M$
   with the push-forwards of the minimal/maximal net on~$G$.
   
For $\cO \subeq M$, the relation
   $q_M^{-1}(\cO) \subeq gW^G$ is equivalent to $\cO \subeq gW$,
   so that
   \[ (q_M)_*\sH^{\rm max}_G = \sH^{\rm max}_M.\]
   Likewise, $q_M^{-1}(\cO) \supeq gW^G$ is equivalent to $\cO \supeq gW$,
   which shows that
   \[ (q_M)_*\sH^{\rm min}_G = \sH^{\rm min}_M.\]

If, however, $W^G \subeq G$ is not the full inverse image of $W \subeq M$,
then these relations may fail. 
   \end{remark}

\subsection{The endomorphism semigroup of a standard subspace} 
\label{subsec:endo-v}

Motivated by the important relation $S_W \subeq S_W$
from Lemma~\ref{lem:direct-net},
we take in this subsection a closer look at the
semigroup $S_\sV$ for an  antiunitary representation $(U,\cH)$
of $G_{\tau_h}$ with discrete kernel (\cite{Ne21, Ne22}). Here
we consider the standard subspace  $\sV := \sV(h,U) \subeq \cH$
from~\eqref{eq:def-V(h,U)} and Definition~\ref{def:bgl-net}.
\index{positive cone $C_U$ of unit. rep. $U$ \scheiding } 
To describe $S_\sV$, we need the {\it positive cone}
  \begin{equation}
    \label{eq:CU}
 C_U := \{ x \in \g \: -i \cdot \partial U(x) \geq 0\}, \qquad 
 \partial U(x) = \frac{d}{dt}\Big|_{t = 0} U(\exp tx)
  \end{equation}
  of a unitary representation~$U$. It
  is a closed, convex, $\Ad(G)$-invariant cone in $\g$
  (\cite[Prop.~X.1.5]{Ne99}).

  The key point of the identity
  \[  S(h,C_U) := \{ g \in G \: \Ad(g)h \in h - C_U\}
    = \exp(C_+) G^h \exp(C_-) \]
  for
  \[   C_\pm := \pm C_U \cap \g_{\pm 1}(h) \]
in   Theorem~\ref{thm:semigroups-equal} is that it provides 
two different perspectives on the same subsemigroup 
of $G$, and this is instrumental for the
descriptions of the semigroups~$S_\sV$. 
To see this connection, we recall that
the Monotonicity Theorem \cite[Thm.~3.3]{Ne22} asserts that 
\begin{equation}
  \label{eq:montheo}
S_\sV \subeq S(h,C_U).
\end{equation}
Its proof is based on the fact that, for two standard subspaces 
$\sV_1 \subeq \sV_2$, we have 
$\log \Delta_{\sV_2} \leq \log \Delta_{\sV_1}$ 
in the sense of quadratic forms. Since these selfadjoint operators 
are typically not semibounded, the order relation 
requires some explanation, provided in an appendix to \cite{Ne22}.
Put differently, the Monotonicity Theorem asserts that the well-defined  
$G$-equivariant map 
\[ \cO_\sV = U(G)\sV \cong G/G_\sV \to \cO_h \cong G/G^h, \quad U(g)\sV \mapsto \Ad(g)h \] 
is monotone with respect to the $C_U$-order on $\g$
and the inclusion order on $\cO_\sV \subeq \Stand(\cH)$ 
(cf.\ Section~\ref{subsec:group-type}), hence the name.

\begin{theorem} \label{thm:sv}  {\rm(\cite[Thm.~3.4]{Ne22})}
  If $(U,\cH)$ is an antiunitary representation of $G_{\tau_h}$
  with discrete kernel, then 
  \[  S_\sV = \exp(C_+) G_\sV \exp(C_-)= G_\sV \exp(C_+ + C_-)
    \quad \mbox{ for } \quad
  C_\g = C_U.\]
\end{theorem}

Here the second equality follows from Theorem~\ref{thm:semigroups-equal}. 
The Borchers--Wiesbrock Theorem~\ref{thm:borch-wies}
in Appendix~\ref{subsubsec:stand-pairs} 
immediately shows that $\exp(C_+) \subeq S_\sV$.
Applying it again with $-h$ and $\sV' = \sV(-h,U)$,  we also
get $\exp(C_-) \subeq S_\sV$, which leads with
\eqref{eq:montheo} to 
\[ \exp(C_+) G_\sV \exp(C_-) \subeq S_\sV
  \subeq S(h,C_U).\] 
Therefore the main point is to show that
\[ S(h,C_U)  \subeq \exp(C_+) G^h \exp(C_-)\]
and to identify the connected components of
$G^h$ fixing $\sV$.

\begin{examplekh} \label{ex:poincare}
  (Poincar\'e group) 
In Quantum Field Theory on Minkowski space, 
the natural symmetry group  is the proper Poincar\'e group 
$P(d) \cong \R^{1,d-1} \rtimes \OO_{1,d-1}(\R)^\uparrow$ 
acting by causal isometries
on $d$-dimensional Minkowski space $M := \R^{1,d-1}$. 
Its Lie algebra is  $\g := \fp(d) \cong \R^{1,d-1} \rtimes \so_{1,d-1}(\R)$ 
and the closed forward light cone 
\begin{equation}
  \label{eq:lightcone}
C_\g := \{ (x_0, \bx) \in \R^{1,d-1} \colon x_0 \geq \sqrt{\bx^2}\} 
\end{equation}
is a pointed invariant cone in $\g = \fp(d)$. 
The generator $h \in \so_{1,d-1}(\R)$ of the Lorentz boost on the 
$(x_0,x_1)$-plane 
\[ h(x_0,x_1, \ldots, x_{d-1}) = (x_1, x_0, 0,\ldots, 0)\] 
is an Euler element 
and $\tau_h = e^{\pi i \ad h}$ defines an involution on 
$\g$, acting on the ideal $\R^{1,d-1}$ (Minkowski space) by 
\[ \tau_M(x_0, x_1, \ldots, x_{d-1}) = (-x_0, -x_1, x_2, \ldots, x_{d-1}).\] 

We apply the results in this section
to the identity component 
\[ G := P(d)_e  \cong \R^{1,d-1} \rtimes \SO_{1,d-1}(\R)_e \] 
which has trivial center $Z(G) = \{e\}$. 
A unitary representation $(U,\cH)$ 
of $G$ is called a {\it positive energy representation} 
\index{representation!positive energy \scheiding} 
if $C_\g \subeq C_U$. If $\ker(U)$ is discrete, then $C_U$ is pointed, 
and $C_\g = C_U$ follows from the fact that the only non-zero pointed
invariant cone in the Lie algebra $\g = \fp(d)$ for $d  > 2$
are $\pm C_\g$, and, for $d = 2$,
there are four pointed invariant  cones which 
are quarter planes in $\R^{1,1}$. 

The centralizer of the Lorentz boost is 
\[ \g_0 = (\{(0,0)\} \times \R^{d-2}) 
\rtimes (\so_{1,1}(\R) \oplus \so_{d-2}(\R)) 
\cong (\R^{d-2} \rtimes \so_{d-2}(\R)) \oplus \R h, \] 
and, 
\begin{equation}
  \label{eq:cpm-poinc}
 C_+ = C_\g \cap \g_1 = \R_{\geq 0} (\be_1 + \be_0)
\quad \mbox{ and } \quad 
C_- = -C_\g \cap \g_{-1} =  \R_{\geq 0}(\be_1 - \be_0).
\end{equation}

The subsemigroup
\[ S(h,C_\g) = \{ g \in G \colon h - \Ad(g)h \in C_\g\} \] 
is easy to determine. 
The relation $\Ad(g)h-h \in \R^d$ implies that $g = (v,\ell)$ with 
$\Ad(\ell)h = h$, and then 
$\Ad(g)h = \Ad(v,\bone)h = -hv \in - C_\g$ is equivalent to $hv \in C_\g$, which 
specifies the closure $\oline{W_R}$ of the standard right wedge 
\[ W_R = \{ x \in \R^{1,d-1} \colon x_1 > |x_0|\}. \] 
The two cones $C_\pm$ generate a proper Lie subalgebra of $\g$. 
We therefore obtain with Lemma~\ref{lem:4.17} 
\begin{equation}
  \label{eq:s(h,cg)}
  S(h,C_\g) = \oline{W_R} \rtimes \big(\SO_{1,1}(\R)^\uparrow \times \SO_{d-2}(\R)\big) \ {\buildrel \ref{lem:4.17} \over =} \ \{ g \in G \colon gW_R \subeq W_R\} = S_{W_R},
\end{equation}
where $\SO_{1,1}(\R)^\uparrow = \exp(\R h)$.  
We claim that, for any antiunitary positive energy representation of 
$G_{\tau_h}$ with discrete kernel, the semigroup $S_\sV$ corresponding to the standard subspace
$\sV = \sV(h,U)$ is given by 
\begin{equation}
  \label{eq:sv-minkowski}
  S_\sV = S(h,C_\g) = S_{W_R}. 
\end{equation}
To verify this claim, we first observe that
\eqref{eq:montheo} implies $S_\sV \subeq S(h,C_\g)$.
We further have 
\[ S(h,C_\g) = S_{W_R} = \exp(C_+) G_{W_R} \exp(C_-), \] 
and the group $G_{W_R}$ is connected, hence contained in $G^{h,\tau_h}
\subeq G_\sV$. Now our claim follows from
Theorem~\ref{thm:sv}. 

Assume that $d \geq 4$, so that
the simply connected covering group $q_G \: \tilde G \to G$ 
is a $2$-fold covering. Then we obtain for $\tilde G$ 
the same picture because the involution $\tau_h$ acts trivially on the 
covering group $\tilde G^h$ of~$G^h$ by \eqref{eq:s(h,cg)},
and this implies that $U(\tilde G^h)$ fixes $\sV$.

For $d = 2$, the group $G \cong \R^2 \rtimes \R$ is simply connected,
but for $d= 3$ the picture is quite different. Then
$\SO_{1,2}(\R)_e \cong \PSL_2(\R)$ and $\pi_1(G) \cong \Z$.
In this case $\tau_h$ acts by inversion on the center $Z(\tilde G)$.
So
\[ \tilde G^h \cong \R \be_2 \times (\exp(\R h) Z(\tilde G))
  \cong \R \times \R \times \Z \]
and
\[ \tilde G_\sV = \{ g \in \tilde G^h \: g\tau_h(g)^{-1} \in \ker U \}
  = \R \times \R \times \{ n \in \Z \: 2n \in \ker U\}.\]
Therefore $\tilde G_{W_R} = \tilde G^h$ is
equivalent to $U(Z(\tilde G)) \subeq \{ \pm \bone\}$. 
In this case the nets $\sH^{\rm min}$ and $\sH^{\rm max}$ 
on $M = \R^{1,d-1}$ satisfy (BW) for $W = W_R$. 
\end{examplekh}

\begin{examplekh}
 (Conformal groups $\SO_{2,d}(\R)$) 
The Lie algebra of the conformal group 
$G := \SO_{2,d}(\R)_e$ of Minkowski space is 
$\g = \so_{2,d}(\R)$, which contains the Poincar\'e--Lie algebra 
as those elements corresponding to affine vector fields on $\jV := \R^{1,d-1}$. 
For $d \geq 3$ it is a simple hermitian Lie algebra. 
All its Euler elements $h$
are mutually conjugate (Proposition~\ref{prop:herm}).
One arises from the 
element $h = \id_\jV$ corresponding to the Euler vector field on~$\jV$. 
Then $\g_j(h)$, $j = -1,0,1$, are spaces of vector fields 
on~$\jV$ which are linear (for $j = 0$), constant (for $j = 1$) and 
quadratic (for $j = -1$).

Another important example is the 
element $h_{01} \in \so_{1,d-1}(\R) \subeq \so_{2,d-1}(\R)$ 
corresponding to a Lorentz boost in the Poincar\'e--Lie algebra 
(see Example~\ref{ex:poincare}). 

We consider the minimal invariant cone $C_\g \subeq \g$ 
which intersects $\jV$ in the positive light cone 
$C_+\subeq \jV$. 
Then we obtain a complete description of 
the corresponding semigroups by
\[ S_\sV = \exp(C_+) G_\sV \exp(C_-),\]
and these semigroups 
have interior points because the cones
$C_\pm$ generate the subspaces $\g_{\pm 1}$
(see \cite{MNO26} for more details). 
\end{examplekh}

\begin{examplekh} Another interesting example which is 
neither semisimple nor an affine group is given by  
the Lie algebra 
\[ \g = \hcsp(V,\omega) := \heis(V,\omega) \rtimes \csp(V,\omega) \]
from Example~\ref{ex:hcsp}. 
Now we turn to the corresponding group and one of its 
irreducible unitary representations. 
Choosing a symplectic basis in $V$, we obtain 
an isomorphism 
\[ V \cong V_{-1} \oplus V_1 \cong \R^n \oplus \R^n\]
with the canonical symplectic 
form specified by 
$\omega((q,0), (0,p)) = \la q,p\ra$ and $\tau_V(q,p) = (-q,p)$. Let 
$\Mp_{2n}(\R)$ denote the {\it metaplectic group},
\index{group!metaplectic \scheiding} 
which is the unique non-trivial double cover of $\Sp_{2n}(\R$). 
We consider the group 
\[ G := \Heis(\R^{2n}) \rtimes_\alpha (\R^\times_+ \times \Mp_{2n}(\R)),\] 
where $\R^\times$ acts on $\Heis(\R^{2n}) = \R \times \R^{2n}$ by 
$\alpha_r(z,v) = (r^2 z, rv)$.
Its Lie algebra is $\g = \hcsp(V,\omega)$. 
Then  the Hilbert space 
\[ \cH 
:= L^2\Big(\R^\times_+, \frac{d\lambda}{\lambda}; L^2(\R^n)\Big) \cong 
L^2\Big(\R^\times_+ \times \R^n, \frac{d\lambda}{\lambda} \otimes dx\Big),\] 
carries an irreducible unitary representation of $G$, where 
$L^2(\R^n) \cong L^2(V_{-1})$ carries the oscillator representation $U_0$ 
of $\Heis(\R^{2n}) \rtimes \Mp_{2n}(\R)$ (cf.\ \cite[\S IX.4]{Ne99}). 
The Heisenberg group $\Heis(\R^{2n})$ is represented on $\cH$ by 
\begin{align*}
(U(z,0,0) f)(\lambda, x) &= e^{i\lambda^2 z} f(\lambda, x), \\ 
(U(0,q,0) f)(\lambda, x) &= e^{i \lambda \la q,x\ra} f(\lambda, x), \\ 
(U(0,0,p) f)(\lambda, x) &= f(\lambda, x - \lambda p).
\end{align*}
The group $\Mp_{2n}(\R)$ acts by the metaplectic representation 
on $L^2(\R^n)$ via 
\[ (U(g)f)(\lambda, \cdot ) := U_0(g)f(\lambda, \cdot), \] 
independently of $\lambda$. 
The one-parameter group $\R^\times_+ \cong \exp(\R h_0)$ acts by 
\[ (U'(r) f)(\lambda,x) := f(r \lambda,x) \quad \mbox{ for }\quad r > 0.\] 
We also have a conjugation $J$ on $\cH$ defined by 
\[ (Jf)(\lambda,x) := \oline{f(\lambda,-x)} \quad \mbox{ satisfying } \quad 
J U(g) J = U (\tau_h(g)),\] 
where $\tau_h$ induces on $\g$ the
involution $e^{\pi i \ad h} = (-\tau_V)\,\tilde{}\ $ (cf.~Example~\ref{ex:hcsp}). 

The positive cone $C_U \subeq \fg$ is the same as the one of the metaplectic 
representation. It intersects $\sp(V,\omega)$ in the cone 
of non-negative polynomials of degree $2$ on~$V$. This implies that 
$(C_U)_- = C_-$. To determine $(C_U)_+ = C_U \cap \g_1$, 
we observe that $\g_1$ acts on $L^2(\R^n)\cong L^2(V_-)$  
by multiplication operators. This shows that we also have $(C_U)_+ = C_+$, 
so that we can determine the semigroup $S_\sV$ for the standard subspace 
$\sV = \sV(h,U)$ as 
\[ S_\sV = \exp(C_+) G_\sV \exp(C_-),\] 
where $G_\sV = G^h$ is a double cover of $\Aff(\R^n)_e$, its inverse image 
in $\Mp_{2n}(\R)$. 
\end{examplekh}

\subsection{Regularity of unitary representations}
\label{app:regularity}

The regularity concept for an antiunitary representation
$(U,\cH)$ corresponds to the assumptions of
Euler Element Theorem~\ref{thm:2.1}, when we are already given
an Euler element and an antiunitary representation of $G_{\tau_h}$.
It is a challenging open problem
(Conjecture~\ref{conj:reg}) to show that {\bf all} antiunitary
representations of $G_{\tau_h}$ are $h$-regular, without any additional
structural assumption on~$G$. We refer to
\cite{MN24} for a detailed discussion, and to \cite{BN25} for the
case of the $4$-dimensional split oscillator group.
In this section we describe the connection between
regularity and the existence of nets of real subspaces
satisfying (Iso), (Cov), (RS) and (BW) on a very general level.

\index{representation!$h$-regular \scheiding} 

\begin{definition}\label{def:reg}
  We call an antiunitary representation 
  $(U,\cH)$ of $G_{\tau_h}$ {\it regular with respect to $h$}, or
{\it $h$-regular}, if there exists an 
$e$-neighborhood $N \subeq G$ such that
\[ \sV_N = \bigcap_{g \in N} U(g)\sV \]
is cyclic.
Replacing~$N$ by its interior, we may always assume that $N$ is
  open.
\end{definition}

\begin{conjecture} \label{conj:reg} {\rm(Regularity Conjecture)} 
  If $h \in \g$ is an Euler element, then any antiunitary representation 
$(U,\cH)$ of $G_{\tau_h}$ is $h$-regular.
\end{conjecture}

This conjecture holds for connected reductive 
groups by Corollary~\ref{cor:real-red} below 
and for several specific classes of groups and representations
(see \cite{MN24} for details).

  \begin{lemma} \label{lem:g-intera}
  For an  antiunitary  representation $(U,\cH)$ of $G_{\tau_h}$,
  the  following assertions hold:
  \begin{enumerate}
  \item[\rm(a)] If $U = U_1 \oplus U_2$ is a direct sum, then
    $U$ is $h$-regular if and only if $U_1$ and $U_2$ are $h$-regular.
  \item[\rm(b)] If $U$ is $h$-regular, then every subrepresentation is
    $h$-regular. 
  \item[\rm(c)]  Assume that $G$ has at most countably many connected
    components and let  $U = \int_X^\oplus U_m \, d\mu(m)$ be an  antiunitary  direct integral  representation of $G_{\tau_h}$, then $U$  is regular if and only if there exists an $e$-neighborhood $N \subeq G$ such that, for $\mu$-almost every $m \in X$,
    the subspace $\sV_{m,N}$ is cyclic.
      \end{enumerate}
  \end{lemma}
  
\begin{proof} (a) If $U \cong  U_1 \oplus U_2$,
  then \eqref{eq:v-dirsum} implies that
  $\sV_N = \sV_{1,N} \oplus \sV_{2,N}$ for every
  $e$-neighborhood $N \subeq G$.
  In particular, $\sV_N$ is cyclic if and only if $\sV_{1,N}$ and
  $\sV_{2,N}$ are. 

  \nin (b) follows immediately from (a). 
  
  \nin (c) Applying Lemma~\ref{lem:g-inter}(b) in Appendix~\ref{app:D} 
  to $A := N$, we obtain (c). 
\end{proof}

Note that the Regularity Characterization Theorem~\ref{thm:reg-net} below  does not require any
  assumption concerning the irreducibility of the representation.
  Proposition~\ref{prop:3.2} (cf.\ \cite[Prop.~2.26]{MN24}) is a convenient tool to reduce to
  irreducible representations. 

  \begin{proposition} \label{prop:3.2}
  Assume that $G$ has at most countably many connected
  components and that $A \subeq G$ is a subset. Then the following
  are equivalent:
  \begin{enumerate}
  \item[\rm(a)] For all antiunitary representations 
    $(U,\cH)$ of $G_{\tau_h}$, the subspace
$\sV_A := \bigcap_{g \in A} U(g)\sV$ is cyclic. 
  \item[\rm(b)] For all irreducible antiunitary representations
    $(U,\cH)$ of $G_{\tau_h}$, the subspace $\sV_A$ is cyclic. 
  \item[\rm(c)] For all irreducible unitary representations 
    $(U,\cH)$ of $G$, the subspace $\tilde\sV_A$ is cyclic
    in $\tilde \cH$, where    $\tilde \sV := \sV(h, \tilde U)$
    and $(\tilde U, \tilde \cH)$ is
    the canonical antiunitary extension of $U$ from
    {\rm Lemma~\ref{lem:3.4}} in {\rm Appendix~\ref{app:extend}}. 
  \item[\rm(d)]{\rm(Characterization in terms of unitary
      representations)} For all unitary representations 
    $(U,\cH)$ of $G$, the subspace $\tilde\sV_A$ is cyclic in $\tilde\cH$. 
  \end{enumerate}
\end{proposition}

\begin{proposition} \label{prop:loc-imp-reg}
    {\rm(Localizability implies regularity)}
  Let $\eset\not=\cO \subeq W \subeq M$ be open subsets such that
  $N := \{ g \in G \: g^{-1}\cO \subeq W \}$ is an $e$-neighborhood.
  If $(U,\cH)$ is an antiunitary representation for which
  $\sH_M^{\rm max}(W) = \sV$ and $\sH_M^{\rm max}(\cO)$ is cyclic, then $U$
  is regular.
  \end{proposition}

  \begin{proof} By assumption $\sH_M^{\rm max}(\cO)$ is cyclic, and
    we obtain with (Iso) and (Cov) (Lemma~\ref{lem:direct-net})
    the relation 
  \[ \sH_M^{\rm max}(\cO) \subeq  \bigcap_{g \in N} \sH_M^{\rm max}(gW)
    = \bigcap_{g \in N} U(g) \sH_M^{\rm max}(W) 
=  \bigcap_{g \in N} U(g) \sV = \sV_N. \]
  It follows that $\sV_N$ is cyclic. 
\end{proof}

\subsubsection{Regularity for a suitable wedge region in general groups}

The following surprising theorem show that regularity already
implies the existence of a net of real subspaces on open subsets of $G$, 
satisfying (Iso), (Cov), (RS) and (BW) for a suitably chosen 
subset $W \subeq G$. 

\begin{theorem} \label{thm:reg-net} {\rm(Regularity Characterization Theorem)} 
    Let $(U,\cH)$ be an antiunitary representation of
  $G_{\tau_h}$ and
  $\sV = \sV(h,U) \subeq \cH$ the corresponding standard subspace.
  Then there exists a net $(\sH(\cO))_{\cO \subeq G}$ on
  open subsets of $G$ satisfying
  {\rm(Iso), (Cov), (RS)}, and {\rm(BW)} for some
  open connected subset $W \subeq G$ if and only if $U$ is $h$-regular, i.e.,
  $\sV_N$ is cyclic for some $e$-neighborhood $N \subeq G$.
\end{theorem}

\begin{proof} ``$\Rarrow$'': If a net $\sH$ with the asserted properties
  exists, then $\sV = \sH(W)$, and for any relatively compact open subset $\cO
  \subeq W$ there exists an identity neighborhood $N \subeq G$ with 
  $N\cO \subeq W$. Then, for $g^{-1} \in N$, we have
  \[ U(g)^{-1}\sH(\cO) = \sH(g^{-1}.\cO)
    \subeq \sH(W)= \sV,\quad \mbox{ hence } \quad
    \sH(\cO) \subeq \sV_N,\]
  as in the proof of Proposition~\ref{prop:loc-imp-reg}.
  Now (RS) implies that $U$ is $h$-regular.

  \nin ``$\Larrow$'':  Assume that $\sV_N$ is cyclic for an
  $e$-neighborhood $N$. Pick an open connected
  identity neighborhood $N_1 \subeq N$
  with $N_1 N_1^{-1} \subeq N$. Then
  \[ W := \exp(\R h) N_1 \]
  is an open connected subset of $G$. We consider the net
  $\sH := \sH^{\rm  max}_G$, defined by
  \[ \sH(\cO) = \bigcap_{g \in G, \cO \subeq gW} U(g)\sV.\]
  This net satisfies (Iso) and (Cov) by Lemma~\ref{lem:direct-net}.

  We now verify the Reeh--Schlieder property (RS).
  So let $\eset\not=\cO \subeq G$ be an open subset.
  By (Iso) and (Cov), it suffices to show that $\sH(\cO)$ is cyclic
  if $\cO \subeq N_1$. Then $\cO \subeq gW = g \exp(\R h) N_1$ implies
  \[ g \in \cO N_1^{-1} \exp(\R h) \subeq N_1 N_1^{-1} \exp(\R h)
    \subeq N \exp(\R h),\]
  so that
  \[ \sH(\cO)
    \supeq \bigcap_{g \in N\exp(\R h)} U(g) \sV 
    = \bigcap_{g \in N} U(g) \sV = \sV_N \]
  implies that $\sH(\cO)$ is cyclic. This proves (RS).
  It follows in particular that $\sH(W)$ 
  is cyclic, so that Lemma~\ref{lem:direct-net}(c) implies 
  $\sH(W)=~\sV$. Therefore (BW) is also satisfied.
\end{proof}

\begin{remark} (Regularity vs orbit maps in $\sV$)

\nin (a)   Note that $v \in \sH^{\rm max}(\cO)$ is equivalent to
\[ g^{-1}\cO  \subeq W \quad \Rarrow \quad U(g)^{-1} v \in \sV.\] 
  If $\cO \subeq W$ is relatively compact, this condition
  holds for $g$ in an $e$-neighborhood. Therefore $\sH^{\rm max}(\cO)$
  consists of vectors $v \in \cH$ whose orbit map
  $U^v \: G \to \cH$ maps an identity neighborhood into~$\sV$
  (cf.\ Proposition~\ref{prop:4.8}(d)). Put differently,
  the subset $(U^v)^{-1}(\sV) \subeq G$ has interior points.

\nin (b) Suppose that $v \in \sV \cap \cH^\omega$ is an analytic vector
and $U(N)v \subeq \sV$ holds for an identity neighborhood $N\subeq G$.
Then the connectedness of $G$ and
uniqueness of analytic continuation imply 
  $U(G)v \subeq \sV$, i.e., $v \in \sV_G = \bigcap_{g \in G} U_g \sV$.

  If, in addition, $v$ is $G$-cyclic, then $\sV_G$ is a cyclic real subspace,
  so that its invariance under the modular group of $\sV$ implies with
  the Equality Lemma~\ref{lem:lo08-3.10} that $\sV = \sV_G$, i.e., that
  $\sV$ is $G$-invariant. If $U$ has discrete kernel, this implies that
  $h \in \fz(\g)$. Hence $\tau_h$ is trivial and therefore
  $J_\sV$ commutes with $G$. Therefore $\cH^{J_\sV} = \sV$ is a real orthogonal
  representation of $G$, and $U$ is its  complexification,
  considered as a representation of $G_{\tau_h}$ on $\cH \cong (\cH^{J_\sV})_\C$.
This is the context where $\partial U(h)$ and $J_\sV$ commute with~$G_{\tau_h}$.

  \nin (c) Another perspective on (b) is that the cyclic subrepresentation
 $U_v$  generated by any $v \in \cH^\omega \cap \sV_N$ is such that
  $\partial U_v(h)$ and $J_\sV$ commute with $G$.
  So $v$ is fixed by the normal subgroup $B$ 
  with Lie algebra
  \[ \fb := \g_1 + [\g_1,\g_{-1}] + \g_{-1}\subeq \ker(\dd U_v).\]
\end{remark}

\subsubsection{Regularity for reductive Lie groups}
\label{subsubsec:reg-reduct}

In this subsection we assume that $\g$ is reductive
and that $G$ is a corresponding connected Lie group.
We choose an involution $\theta$ on $\g$ in such a way that 
  it fixes the center pointwise and restricts to a Cartan involution
  on the semisimple Lie algebra $[\g,\g]$. Then
  the corresponding Cartan decomposition $\g = \fk \oplus \fp$ satisfies
$\fz(\g) \subeq \fk$. We write $K := G^\theta$ for the subgroup of
$\theta$-fixed points in~$G$.

For an Euler element $h \in \g$, we
write $\g = \g_1 \oplus \g_2$, where $\fz(\g) \subeq \g_1$, 
$h = h_z + h_2 \in \fz(\g) \oplus \g_2$, and $\g_2$ is minimal, i.e.,
$\g_2$ is the ideal generated by
the projection $h_2$ of $h$ to the commutator
algebra. We consider the involution $\tau$ on $\g$ with
\[ \tau\res_{\g_1} = \id_{\g_1}
  \quad \mbox{ and } \quad  
 \tau\res_{\g_2} = \tau_h \theta.\] 
We {\bf assume} that $\tau$ integrates to an involutive automorphism
$\tau^G$ of $G$.
  We write $\fh := \g^\tau$ and $\fq := \g^{-\tau} \subeq \g_2$
  for the $\tau$-eigenspaces in $\g$.
    Then there exists in $\fq$ 
    a unique maximal pointed generating $e^{\ad \fh}$-invariant
    cone $C$ containing $h_2$ in its interior
    (\cite[Thm.~4.21]{MNO23} deals with minimal cones, but
    the minimal and the maximal cone     determine each other by duality). 
    We choose an open $\theta$-invariant
    subgroup $H \subeq G^\tau$, satisfying $\Ad(H)C = C$.
This  is always the case for $H = G^\tau_e$ (the minimal choice).
    By \cite[Cor.~4.6]{MNO23}, $\Ad(H)C = C$ is 
    equivalent to $H_K= H \cap K$ fixing~$h$. 
 Then
    \begin{equation}
      \label{eq:nccb}
      M = G/H \cong G_2/H_2 \quad \mbox{ for } \quad H_2 := G_2 \cap H
    \end{equation}
    is called the corresponding non-compactly causal symmetric space
    (cf.\ Section~\ref{subsec:ncc-spaces}).
    The normal subgroups $G_1 \subeq H$ acts trivially on $M$.
    The homogeneous space $M$ carries a $G$-invariant causal
    structure, represented by the
    field $(C_m)_{m \in M}$ of closed convex cones
    $C_m \subeq T_m(M)$, which is uniquely determined by
    $C_{eH} = C \subeq \fq \cong T_{eH}(M)$
    (Subsection~\ref{sec:3.2}). By construction, 
    $eH \in W_M^+(h)$, and we write
    \begin{equation}
      \label{eq:winncc}
      W := W_M^+(h)_{eH}
    \end{equation}
    for the connected component of the base point in~$W_M^+(h)$.

\begin{definition} \label{def:local}
    We say that the (anti-)unitary representation 
  $(U,\cH)$ of $G_{\tau_h}$ is 
  {\it $(h,W)$-localizable} in those open subsets $\cO \subeq M$  
  for which  $\sH^{\rm max}_M(\cO)$ is cyclic.
\end{definition}
\index{representation!$(h,W)$-localizable \scheiding} 

\begin{theorem} \label{thm:local-reduct}
  {\rm(Localization for reductive groups)}
  If $\g$ is reductive
  and   $(U,\cH)$ is an antiunitary representation
  of $G_{\tau_h}$, then the canonical net $\sH_M^{\rm max}$ on
  the non-compactly causal symmetric space $M = G/H$
  from \eqref{eq:nccb} and $W$ from \eqref{eq:winncc} satisfies 
  {\rm(Iso), (Cov), (RS)} and {\rm(BW)}. 
\end{theorem}

\begin{proof} (cf.\ \cite{MN24})
  As the standard subspace $\sV$ is invariant under
  $G_1 \subeq G^{h, \tau_h}$, and $G_1$ acts trivially on $M$,
  the real subspaces $\sH_M^{\rm max}(\cO)$ only depend
  on $U\res_{G_2}$. We may therefore assume that $G = G_2$, i.e., that $G$ is semisimple and that $\g_0(h)$ contains no non-zero ideal.

  In view of Lemma~\ref{lem:direct-net}(c), 
  $\sV = \sH^{\rm max}(W)$ follows from the cyclicity
  of all subspaces $\sH^{\rm max}(\cO)$, $\cO \not=\eset$.
So it suffices to verify the latter. 
By Proposition~\ref{prop:3.2} and
Lemma~\ref{lem:direct-net-d}, we may further assume that
  $(U,\cH)$ is irreducible. 
  Then Theorem~\ref{thm:4.9-semisimp}
  provides a net $\sH^M_\sE$ satisfying
  (Iso), (Cov), (RS) and (BW),   and this net satisfies
  $\sH^M_\sE(\cO) \subeq \sH^{\rm  max}_M(\cO)$
  for each $\cO \subeq M$ (Lemma~\ref{lem:maxnet-larger}).
  Thus   $\sH_M^{\rm  max}(\cO)$ is cyclic.
\end{proof}

\begin{corollary}   \label{cor:real-red} {\rm(Regularity for 
    reductive groups)} 
  Let $G$ be a connected 
  reductive Lie group. 
  Then there exists an $e$-neighborhood $N \subeq G$
  such that for every separable antiunitary representation
  $(U,\cH)$ of $G_{\tau_h}$ and $\sV = \sV(h,U)$, the real subspace 
  $\sV_N = \bigcap_{g \in N} U(g) \sV$
  is cyclic. In particular, $(U,\cH)$ is $h$-regular.   
\end{corollary}

\begin{proof} Let $\cO \subeq W \subeq M = G/H$
  (with $W \subeq M$ as in Theorem~\ref{thm:local-reduct}) be an open subset
  whose closure $\oline\cO$ is relatively compact.
  In Theorem~\ref{thm:local-reduct} we have seen that
  $\sH_M^{{\rm max}}(\cO)$ is cyclic. Further
  \[ N := \{g \in G \: g\cO \subeq W\} \supeq \{g \in G \: g\oline\cO \subeq W\} \]
  is an $e$-neighborhood because $\oline\cO \subeq W$ is compact.
Therefore the $h$-regularity of $(U,\cH)$ follows from
Proposition~\ref{prop:loc-imp-reg}.
\end{proof}

\begin{theorem} \label{thm:SW-ncc-spaces}
  {\rm(Triviality of the semigroups of wedge
    regions in ncc symmetric spaces)} 
  If $G$ is a connected reductive Lie group 
  and $M = G/H$ a corresponding non-compactly causal symmetric
  space   as in \eqref{eq:nccb}, 
  with causal Euler element~$h$, and the maximal causal structure,
    then the following assertions hold:
  \begin{description}
  \item[\rm(a)] $S_W = G_W = \{ g \in G^h \: g.W = W \}$. 
  \item[\rm(b)] $S_W = G^h$ if $\g$ is simple and $Z(G)= \{e\}$. 
  \end{description}
\end{theorem}

\begin{proof} (a) First we apply the
  Localization Theorem~\ref{thm:local-reduct}
to a unitary representation with discrete kernel and $C_U =  \{0\}$; 
for instance the regular representation on $L^2(G)$. 
Then $\sH^{\rm max}(W) = \sV$ by the Localization Theorem,
and since $S_\sV = G_\sV \subeq G^h$ by Theorem~\ref{thm:sv}, 
we obtain $S_W \subeq S_\sV \subeq G^h$ from Lemma~\ref{lem:direct-net}(c).

As $W \subeq W_M^+(h)$ is a connected component and
$G^h$ preserves $W_M^+(h)$, it follows that $S_W \subeq G^h$ is the
stabilizer subgroup $(G^h)_W$ of~$W$. 

\nin (b) follows from $G_W = G^h$ in this case (\cite[Prop.~7.3]{MNO24}). 
\end{proof}

The argument in the preceding proof is somewhat unnatural 
because it uses rather deep information on unitary
representations to derive the geometric fact that
$S_W$ is a group. It would be nice to have a direct geometric
argument for this fact. 

\subsection{Left invariant nets on causal Lie groups} 
\label{subsec:liegrp}

A particularly simple situation arises
for $M = G$, with $G$ acting by left multiplication.

Let us start with an antiunitary representation $(U,\cH)$ of $G_{\tau_h}$
with discrete kernel, so that 
\[  S_\sV =G_\sV \exp(C_+ + C_-) \quad \mbox{ for } \quad
C_\pm = \pm C_U \cap \g_{\pm 1}(h)\]
follows from Theorem~\ref{thm:sv}.

We {\bf assume} that $S_\sV$ has interior points, i.e., that
$C_\pm^\circ \not=\eset$. This is always the case if the ideal
$C_U - C_U$ generates $\g$ with $h$, i.e., if
\[ \g = \R h + C_U - C_U\]
(cf.\ Lemma~\ref{lem:Project}). 

We consider the open connected subsemigroup $W := S_{\sV,e}^\circ
= G^h_e \exp(C_+^\circ + C_-^\circ)$. Then
\[ S_W = \{ g \in G \:  g.W \subeq W \} 
  = \{ g \in G \:  g.S_{\sV,e} \subeq S_{\sV,e} \}  = S_{\sV,e} \]
follows from the fact that $W$ is dense in $S_{\sV,e}$. Therefore the
equality $S_W = S_{\sV,e}$ implies that the corresponding
maximal net $\sH_G^{\rm max}$ satisfies
the Bisognano--Wichmann condition in the form
\[ \sH^{\rm max}_G(W) = \sH^{\rm max}_G(S_{\sV,e}^\circ) = \sV.\]

In addition, we obtain for the real subspace
\[ \sE :=\sV \subeq \cH^{-\infty}_{\rm KMS} \]
a left invariant isotone net $\sH^G_\sE$ on open subsets of $G$.
As $W \subeq S_{\sV,e}$ is dense, this net satisfies
\[  \sH_\sE^G(W) = \oline{U(S_{\sV,e}^\circ)\sV} = \sV,\]
which is the (BW) condition. If, in addition, $C_U^\circ \not=\eset$,
then Theorem~\ref{thm:pe-rs} in Appendix~\ref{subsubsec:reeh-schlieder}
further implies that
$\sH_\sE^G$ has the Reeh--Schlieder property (RS). As 
\[ \sH_\sE^G(\cO) \subeq \sH^{\rm max}_G(\cO), \]
it follows that $\sH^{\rm max}_G$ also has the Reeh--Schlieder property.

We thus have the following theorem: 

\begin{theorem} Suppose that $(U,\cH)$ is an
  antiunitary representation of $G_{\tau_h}$ with discrete kernel
  for which $C_\pm^\circ \not=\eset$.
  \begin{enumerate}
  \item[\rm(a)] For $\sE := \sV$, the net $\sH^G_\sE$
    satisfies {\rm(Iso), (Cov)} and {\rm(BW)} with respect to
    $W = S_{\sV,e}^\circ$. 
  \item[\rm(b)] If, in addition, $C_U^\circ \not=\eset$,
    thyen also {\rm(RS)} is satisfied. In particular $H^{\rm max}_G$
    has this property. 
  \end{enumerate}
\end{theorem}

\subsection{Causal   symmetric spaces}

Regularity of antiunitary representations has
interesting consequences for causal symmetric spaces,
which are by far not completely explored. He we state
some for non-compactly and compactly causal symmetric spaces. 

The following theorem, concerning non-compactly causal symmetric spaces,
is a consequence of the Localization
Theorem~\ref{thm:local-reduct} for reductive groups.
  
\begin{theorem} {\rm(Maximal nets on ncc spaces)} 
  If $M = G/H$ is a semisimple non-compactly causal symmetric space
and $(U,\cH)$ an antiunitary representation of $G$, then
the net $\sH^{\rm max}_M$ satisfies {\rm(Iso), (Cov), (RS)}
and {\rm (BW)} for $W$ as in \eqref{eq:winncc}. 
\end{theorem}

\begin{proof} In this case, $S_W = G_W = G^h_e H^h\subeq G^h$ is a group by
  Theorem~\ref{thm:SW-ncc-spaces}. 
Since $\tau = \tau_h \theta$ coincides on $K$ with $\tau_h$, 
we further have $H^h \subeq K^{h, \tau} \subeq K^{\tau_h}$,
so that $S_W \subeq G^{h,\tau_h} \subeq G_\sV.$ 
Therefore Lemma~\ref{lem:direct-net} shows that $\sH^{\rm max}_M$ satisfies~(BW),
  and the condition (RS) follows from   Theorem~\ref{thm:local-reduct}. 
\end{proof}

For compactly causal spaces, we presently only have the following
weaker result: 

\begin{theorem} {\rm(Maximal nets on cc spaces)}
  \label{thm:maxnet-cc}
Let $M = G/H$ be an irreducible modular
  compactly causal symmetric 
  space, $C_\g \subeq \g$ an invariant closed convex cone with
  $C = C_\g \cap \fq$,   and $(U,\cH)$ an antiunitary representation of $G$
  with discrete kernel. Then 
  the net $\sH_M^{\rm max}$ satisfies {\rm(Iso), (Cov)} 
  and {\rm (BW)}   if and only if
  \begin{description}
  \item[\rm(a)]   the positive spectrum condition $C_+ = C_\g \cap \g_1
    \subeq C_U$
  is satisfied, and 
  \item[\rm(b)] $U(G_W)$ commutes with $J$, i.e.,
  $\tau_h(g)g^{-1}\in \ker U$ for $g \in H^h$. 
  \end{description}
\end{theorem}

\begin{proof} From \cite[Thm.~9.1]{NO23a} it follows that
  the semigroup $S_W$ can be written as 
  \begin{equation}
    \label{eq:swxx}
 S_W = G^h_e H^h \exp(C_+ + C_-) \quad \mbox{ for } \quad
 C_\pm = \pm C_\g \cap \g_{\pm 1}.
  \end{equation}
  Therefore $S_W \subeq S_\sV$ is equivalent to the two conditions 
  \[ C_\pm \subeq \pm C_U \quad \mbox{ and } \quad
    U(g)J = J U(g) \quad \mbox{ for }  \quad g \in H^h.\]
  The first condition implies that $C_U \not=\{0\}$.
  As $\ker U$ is discrete and $\g$ is simple or a sum of two simple ideals
  (\cite[Rem.~4.24]{MNO23}),
  $C_U$ is a pointed generating invariant cone with $C_{U,\pm}  = C_\pm$.
  This shows that the condition $S_W \subeq S_\sV$
  (which is equivalent to (BW) for $\sH^{\rm max}_M$ by
Lemma~\ref{lem:direct-net}) implies (a) and (b).
  
  To see the converse, suppose that (a) and (b) are satisfied.
  Then $C_\pm \subeq \pm C_U $   implies that $C_{U,\pm} = C_\pm = (C_\g)_{\pm}$
  (\cite[Prop.~2.7(d)]{NO23a}). Then $S_W \subeq S_\sV$ follows from
  \eqref{eq:swxx}. This proves the theorem because the (BW) property
  of $\sH^{\rm max}_M$ is equivalent to $S_W \subeq S_{\sV}$ by
  Lemma~\ref{lem:direct-net}. 
  \end{proof}

  \begin{problem} \label{prob:cc-rs}
    Show that, in the context of the preceding theorem, (a) and (b)
    also imply the Reeh--Schlieder property (RS) of $\sH^{\rm max}_M$. 
  \end{problem}

  \begin{remark}
    Here are some remarks that may be useful to solve this problem.
    \begin{itemize}
    \item Using Proposition~\ref{prop:3.2} and that
      $\sH^{\rm  max}_M(\cO) = \sV_A$ for
      $A := \{ g \in G \: \cO \subeq g.W\}$, it suffices to solve the problem
      for irreducible representations.
    \item From \eqref{eq:SW-cc} we know that
      \[  S_W = G_W \exp(C_+ + C_-) \quad \mbox{  with  } \quad
        G_W = G^h_e H^h \quad \mbox{ and }  \quad
        G_{W,e} = G^h_e,\]
      so that $S_{W,e} = G^h_e  \exp(C_+ + C_-),
S_{W,e}^\circ = G^h_e  \exp(C_+^\circ + C_-^\circ),$  and 
      \[ W = S_{W,e}^\circ.eH \cong G^h_e \times_{H^h_e} \exp(C_+^{\circ, -\tau} + C_-^{\circ, -\tau}).\]
      For $W^G = q_M^{-1}(W)$ we then have
      \[ W^G = S_{W,e}^\circ H  \supeq S_{W,e}^\circ,\]
      which is strictly larger than $S_{W,e}^\circ$.

      From Remark~\ref{rem:pushforward-to-M} we know that
      $\sH^{\rm max}_M = (q_M)_* \sH_G^{\rm max}$, with respect to the wedge regions
      $W \subeq M$ and $W^G \subeq G$. It follows in particular that
      \[        \sH_G^{\rm max}(S_{W,e}^\circ) \subeq \sH_G^{\rm max}(W^G) = \sV.\]
      Next we argue that we actually have equality.
    \item For $g \in G$, the relation
      $S_{W,e}^\circ \subeq g W^G$ is equivalent to
      \[ q_M(S_{W,e}^\circ) = q_M(W^G) = W \subeq g.W,\]
      i.e., to $g \in S_W^{-1}$. Now $S_W \subeq S_\sV$ implies that any such $g$
      satisfies $U(g)\sV \supeq \sV$, so that 
      \begin{equation}
        \label{eq:bw2}
        \sH_G^{\rm max}(S_{W,e}^\circ) = \sV = \sH_G^{\rm max}(W^G).
      \end{equation}
\item      The net $\sH_{G,S_{W,e}^\circ}^{\rm max}$ on $G$, constructed from the
      wedge region $S_{W,e}^\circ$, which is smaller than $W^G$, also satisfies 
      \[        \sH_{G, S_{W,e}^\circ}^{\rm max}(S_{W,e}^\circ) = \sV \]
      because $S_{W,e} = G^h_e \exp(C_+ + C_-) \subeq S_\sV$.
      We thus obtain the two maximal nets
$\sH_{G, S_{W,e}^\circ}^{\rm max}$ and 
$\sH_{G, W^G}^{\rm max}$  on $G$, corresponding to
      the wedge regions $S_{W,e}^\circ$ and $W^G$. They satisfy
      the relation
      \begin{equation}
        \label{eq:twomaxnet}
        \sH_{G, S_{W,e}^\circ}^{\rm max}(\cO) \supeq \sH_{G,W^G}^{\rm max}(\cO).
      \end{equation}
      
      Note that this also follows from Lemma~\ref{lem:maxnet-larger}
      and the maximality of the net $\sH_{G,S_{W,e}^\circ}^{\rm max}$ among the nets
      mapping $S_{W,e}^\circ$ to $\sV$ (see \eqref{eq:bw2}).
    \item If, in addition, $\sH_{G, S_{W,e}^\circ}^{\rm max}(W^G) \supeq
      \sH_{G, S_{W,e}^\circ}^{\rm max}(S_{W,e}^\circ) =  \sV$ is separating,
      then the Equality Lemma~\ref{lem:lo08-3.10} and the
      $\exp(\R h)$-invariance of $W^G$ imply that both are equal.
      Now the maximality of $\sH^{\rm max}_{G,W^G}$ and
      Lemma~\ref{lem:maxnet-larger} imply that the nets
      $\sH_{G, S_{W,e}^\circ}^{\rm max}$ and $\sH_{G,W^G}^{\rm max}$ coincide.
    \end{itemize}
\end{remark}

The situation looks much better if there are nets of the form
$\sH^M_\sE$:

\begin{theorem}
  Suppose that, in addition to the assumptions of
  {\rm Theorem~\ref{thm:maxnet-cc}},   
  there exists a real linear subspace
  $\sE \subeq \cH^{-\infty}_{\rm KMS}$ which is $U^{-\infty}(H)$-invariant
  and satisfies $\sH^M_\sE(W) = \sV$. Then
  $\sH^{\rm max}_M$ also satisfies {\rm(RS)}.
\end{theorem}

\begin{proof} We have already seen in the proof of
  Theorem~\ref{thm:maxnet-cc} that $C_U^\circ\not=\eset$,
  so that Theorem~\ref{thm:pe-rs}
  in Appendix~\ref{subsubsec:reeh-schlieder}
  implies that the net $\sH^G_\sE$ on $G$ satisfies (RS),
  and so does the net $\sH^M_\sE = (q_M)_*\sH^G_\sE$
  (Remark~\ref{rem:push-forward}(b)).
  By assumption $\sH^M_\sE$ satisfties (BW), hence is contained
  in $\sH^{\rm max}_M$, and therefore the latter net also satisfies (RS).
\end{proof}

\begin{remark}
  The assumption of $\sE$ to be $U^{-\infty}(H)$-invariant
  may seem to be quite strong, but it is close to being necessary.
  In fact, what we need to argue as in the proof of the preceding
  theorem, is that $\sH^G_{\sE}(W^G) = \sH^M_\sE(W) = \sV$,
  which is equivalent to $U^{-\infty}(W^G) \sE \subeq \cH^{-\infty}_{\rm KMS}$
  by Proposition~\ref{prop:4.8}. As $eH \in \oline W$, we have
  $H \subeq \oline{W_G}$, so that the weak-$*$
  closedness of $\cH^{-\infty}_{\rm KMS}$ (Theorem~\ref{thm:BN24})
  implies that $U^{-\infty}(H)\sE \subeq \cH^{-\infty}_{\rm KMS}$.
\end{remark}

\subsection{Causal flag manifolds}
\label{subsec:causal-flag-man-part2}

Let $M = G/Q_h$ be an irreducible causal flag manifold.
The results in this section can be found in \cite{MN25}.
The fundamental group $\pi_1(M)$ is isomorphic to $\Z$
(\cite[Thm.~1.1]{Wig98} and \cite{MNO25}),
  so that there exists for every $k \in \N 
  \cup \{\infty\}$ a $k$-fold covering  $M_k$, 
where $M_\infty$ is simply connected.

\begin{itemize}
\item $M_\infty$ is a  simple space-time manifold
  in the sense of Mack/de Riese \cite{MdR07}.
  It carries a global causal order (no closed causal curves).
\item $M_k, k < \infty$, has closed causal curves, hence no global
  causal order.
\item The open embedding $\iota_M \: \jV = \g_1 \into M$ of the euclidean
  Jordan algebra $\jV$ (see the appendix to Subsection~\ref{subsec:causal-flag-man}  for  more on euclidean Jordan algebras)
  lifts to an open embedding $\iota_{M_k} \: \jV \into M_k$.
\item In $M_k$ the canonical wedge region is
  \[ W_{M_k} := \iota_{M_k}(\jV_+) \subeq M_k.\] 
  It is a connected component of the positivity domain $W_{M_k}^+(h)$
   of the Euler vector field $X_{M_k}^h$ on~$M_k$. 
   In $M_k$ the positivity domain $W_{M_k}^+(h)$ has $k$ connected components
   which are permuted by a transitive action of the group $\Z_k$ of 
   deck transformations of the covering $M_k \to M$. 
\end{itemize}

\begin{examples} (a)
  For Minkowski space $\jV = \R^{1,d-1}$, the conformal completion 
  \[ M \cong (\bS^1 \times \bS^{d-1})/\{\pm \bone\} \subeq  \bP(\R^{2,d}) \]
  is the isotropic quadric in $\bP(\R^{2,d})$
  (see \eqref{eq:quadtildev} in Example~\ref{ex:qmink})   on which $G = \SO_{2,d}(\R)_e$ acts.
  In this case, 
  \[ M_\infty \cong \R \times \bS^{d-1}.\]

  \nin (b) For the euclidean Jordan algebra
  $\jV = \Herm_r(\C)$, we have $M \cong \U_r(\C)$, on
  which $G = \SU_{r,r}(\C)$ acts by fractional linear transformations
  \[ \pmat{a & b \\ c & d}.z = (az+b)(cz+d)^{-1}.\]
  Here
  \[ \tilde \U_r(\C) \cong \R \times \SU_r(\C).\]

  \nin (c) For the euclidean Jordan algebra
  $\jV = \Sym_r(\R)$, the conformal compactification
  is the space~$M$ of Lagrangian subspaces in the symplectic
  vector space $(\R^{2r},  \omega)$, on which $G = \Sp_{2r}(\R)$
  acts naturally. Here 
  $M_\infty \cong \R \times (\SU_r(\C)/\SO_r(\R))$.
\end{examples}

To formulate existence criteria for nets on the spaces $M_k$,
we observe that the simply connected covering
group $\tilde G$ acts on every  $M_k$. 
The centralizer $\tilde G^h$ of $h$ in this group satisfies
\[ \pi_0(\tilde G^h) \cong \pi_1(M) \cong \Z.\]
We pick $g_h \in \tilde G^h$ so that its connected component
generates $\pi_0(\tilde G^h)$ and
\begin{equation}
  \label{eq:tauh-rel}
  \tau_h(g_h) = g_h^{-1}.
\end{equation}
This element can be chosen to be central in an
$\tilde\SL_2(\R)$-subgroup with $g_h^2 \in Z(\tilde G)$
(cf. \cite{MNO25}).

\begin{theorem} \label{thm:5.28} {\rm(Existence of nets)} 
  For an antiunitary representation $(U,\cH)$ of $\tilde G_{\tau_h}$
  with discrete kernel,
  a net $\sH$ on open subsets of $M_k$ satisfying
  {\rm(Iso), (Cov), (RS)} and {\rm(BW)} exists if and only if
  \begin{itemize}
  \item $U$ satisfies the positive energy condition 
    $C_+ \subeq C_U$, and, 
  \item if $k < \infty$, then $g_h^{2k} \in \ker U$. 
 \end{itemize}
\end{theorem}

\begin{proof} (Sketch; see \cite{MN25})
  In view of Lemmas~\ref{lem:direct-net} and \ref{lem:maxnet-larger},
  the existence of a net $\sH$
  satisfying (Iso), (Cov) and (BW) is equivalent to
$S_{W_{M_k}} \subeq S_\sV.$ 
These semigroups are of the form 
\[       S_{W_{M_k}} = \tilde G_{W_{M_k}} \exp(C_+ + C_-)\quad \mbox{ and} \quad
  S_\sV = \tilde G_\sV \exp(C^U_+ + C^U_-) \]
for $C^U_\pm = \pm C_U \cap         \g_{\pm 1}(h),$ 
which reduces the inclusion $S_{W_{M_k}} \subeq S_\sV.$ 
to the inclusion 
\[  \tilde G_{W_{M_k}} \subeq  \tilde G_\sV \]
and the positive energy condition
\[ C_+ \subeq C_U.\] 
As $\tilde G_{W_{M_k}} \subeq \tilde G_{W_M} = \tilde G^h$
(Proposition~\ref{prop:sw-conformal}) 
commutes with the Euler element $h$, the inclusion 
$\tilde G_{W_{M_k}} \subeq  \tilde G_\sV$ is equivalent to 
\[ \{ g \tau_h(g)^{-1} \: g \in \tilde G_{W_{M_k}}\} = g_h^{2k\Z}
    \subeq \ker U.\]

    So it only remains to verify the Reeh--Schlieder condition.
    We refer to \cite{MN25} (which builds on \cite{NO21}) for details. 
\end{proof}

The following theorem extends results by Brunetti, Guido and Longo
\cite{BGL93} for the Jordan algebra $\jV = \R^{1,d-1}$
(Minkowski space) and the group
$G = \SO_{2,d}(\R)_e$ to general causal flag manifolds. 

\begin{theorem} \label{thm:unique-add-conf}
  {\rm(Existence and uniqueness of additive nets)}
      For every antiunitary representation
      $(U,\cH)$ of $\tilde G$ satisfying the positive energy condition 
      $C_+ \subeq C_U$, there exists a unique additive net
      $\sH$ on open subsets of  $M_\infty$, satisfying
      {\rm(Iso), (Cov), (RS)} and {\rm(BW)}.\\
      On $M_k$ such nets exist  for $U$ if and only if,
      in addition, $g_h^{2k} \subeq \ker U$. 
\end{theorem}

\begin{proof} {\bf Uniqueness:}   
  On every  $M_k$, the wedge regions form a basis
  for the topology. Every additive covariant net $\sH$
  satisfying (BW) thus satisfies 
  \[ \sH(\cO) = \sH\bigg( \bigcup_{g.W_{M_k} \subeq \cO} g.W_{M_k}\bigg)  
    = \oline{\sum_{g.W_{M_k} \subeq \cO} U(g)\sH(W_{M_k})}
    = \oline{\sum_{g.W_{M_k} \subeq \cO} U(g)\sV},\]
    so that $\sH$ is determined by the representation $U$ via
    $\sH(W_{M_k}) = \sV = \sV(h,U)$.

    \nin {\bf Existence:} The argument for existence
    builds on nets for irreducible representations
    constructed in \cite{NO21} (see also \cite{MN25}), 
    and direct integral techniques. The main point is to find
    a finite-dimensional subspace $\sE \subeq \cH^{-\infty}_{\rm KMS}$
    invariant under the parabolic subgroup $Q_h$
    and that $\sH^{M_k}_\sE(W) = \sV$.
\end{proof}

\subsubsection{Locality}
\label{subsec:6.2}

Locality conditions concern open $G$-invariant subsets
$\cD_{\rm loc} \subeq M \times M$. Here are some relevant facts:
  \begin{itemize}
  \item $M \times M$ contains a unique open $G$-orbit $\cD^*$. 
  \item $M_\infty \times M_\infty$ contains infinitely
    many open $\tilde G$-orbits $(\cD^*_n)_{n \in \Z}$,
    permuted by the group $\pi_1(M) \cong \Z$, acting by deck transformations. 
  \item $M_k\times M_k$ contains $k$ open $\tilde G$-orbits $\cD^*_n$,
    $n \in \Z/k\Z$, permuted by  deck transformations 
    in $\Deck(M_k) \cong\pi_1(M)/\pi_1(M_k) \cong \Z/k\Z$.
  \end{itemize}

 Let $g_h \in \tilde G^h$ be as above and
 pick $z_\fk \in \fz(\fk$) such that $\theta := \exp(\pi \ad z_\fk)
 \in \tilde G$
is a Cartan involution.

  \begin{theorem} \label{thm:n-locality} {\rm(Locality properties of the nets)} 
    For the unique additive net associated to
  the positive energy representation $(U,\cH)$ of $\tilde G_{\tau_h}$ on
  $M_k$, $k \in\N\cup\{\infty\}$ and $n \in \Z_k$,
  the following are equivalent:
  \begin{description}
  \item[\rm(a)] The $n$-locality condition 
  \[ \cO_1 \times \cO_2 \subeq \cD^*_n \quad \Rarrow \quad
    \sH(\cO_1) \subeq \sH(\cO_2)' \]
\item[\rm(b)] $\sH(g_h^n.W_{M_\infty}) \subeq \sH(W_{M_\infty}')'$ for the
  dual wedge $W_{M_\infty}' := \theta.W_{M_\infty}$. 
\item[\rm(c)] $\exp(2\pi z_\fk) g_h^{2n}  \in \ker U$.
  \begin{footnote}
    {Note that the discrete normal subgroup $\ker U \trile G$ is central
      because $G$ is connected.}
  \end{footnote}
  \end{description}
\end{theorem}

\subsubsection{The massless spin 0 representation on Minkowski space} 

We consider Minkowski space $\jV = \R^{1,d-1}$
and its conformal compactification $M$.
  On $\tilde G = \tilde\SO_{2,d}(\R)_e$ the ``minimal'' positive
  energy representation $(U,\cH)$ is
  the extension of the Poincar\'e representation
  corresponding to {\it massless spin~$0$-particles}.
  
  It depends on the dimension $d$, to which quotient group
  of $\tilde G$ the representation $U$ descends,
  and on which covering of $M$ the net can be implemented.
  We have the following properties (cf.~\cite{BGL93, BGL02}, 
  and also \cite{MNO26}, \cite{MN25}): 
  \begin{itemize}
  \item $d-2\in 4 \Z$: $U$ is defined on the adjoint group
    $\SO_{2,d}(\R)_e/\{\pm \bone\}$ and the net lives on $M$.
  \item $d\in 4 \Z$: $U$ is defined on $\SO_{2,d}(\R)_e$ with
    $U(-1) = -\bone$ and the net lives on $M$.
  \item $d$ odd: $U$ is defined on a $2$-fold covering of $\SO_{2,d}(\R)_e$
    and the net lives on on $M_2 \cong \bS^1 \times \bS^{d-1}$. 
  \end{itemize}

  \begin{remark} The $n$-locality condition on $M_2$ (for $n = 0,1$) is
    $(n+1)d \in 2 \Z$ (cf.~Theorem~\ref{thm:n-locality}(a)).
    \begin{itemize}
    \item     For $d$ even, the net therefore is $0$- and $1$-local,
    which corresponds to spacelike and timelike locality
    on Minkowski  space.
  \item    For $d$ odd, it is only $1$-local, which corresponds to
    spacelike locality on Minkowski space.
    \end{itemize}
These locality conditions relate to support properties of the fundamental
solutions of the Klein--Gordon equation (Huygens' Principle).
We refer to \cite{MN25} for details.
\end{remark}

\subsection{Localizability for the Poincar\'e group}

The following result is well-known
in AQFT (\cite[Thm.~4.7]{BGL02}; see also \cite{Mu01, Mu03}). 
Below we derive it naturally in the
context of our theory for general Lie groups (cf.\ \cite[Thm.~4.25]{MN24}). 
It connects regularity, resp., localizability  with
  the positive energy condition.

  \begin{theorem} \label{thm:poinloc}
    {\rm(Localization for the Poincar\'e group)}
 Let $(U,\cH)$ be an (anti-)unitary representation
  of the proper Poincar\'e group
  $\cP_+ = \R^{1,d} \rtimes \cL_+$ (identified with $\cP_{\tau_h}$)
  and consider the standard boost
  $h$ and the corresponding Rindler wedge $W_R \subeq \R^{1,d}$.
  Then $(U,\cH)$ is $(h,W_R)$-localizable in
{the set of all spacelike open cones} 
  if and only if it is a positive energy representation, i.e.,
  \begin{equation}
    \label{eq:v+}
 C_U \supeq
 \oline{V_+} := \{ (x_0,\bx) \: x_0 \geq \sqrt{\bx^2} \}.
  \end{equation}
  These representations are $h$-regular. 
\end{theorem}

Note that $\Ad(\cP_+^\uparrow)$ acts transitively on the set $\cE(\fp)$
of Euler elements (Example~\ref{ex:poincare-1}),
so that the choice of a specific Euler element~$h$ is inessential. 

\begin{proof} First we show that the positive energy condition is necessary
  for localizability in spacelike cones.
  In fact, the localizability condition implies in particular
  that $\sH(W_R)$ is cyclic, so that Lemma~\ref{lem:direct-net} implies
  $S_{W_R} \subeq S_\sV$. As a consequence, 
  $\be_1 + \be_0 \in C_U$, and thus $\oline{V_+} \subeq C_U$ by
  Lorentz invariance of $C_U$. Therefore $(U,\cH)$ is a positive
  energy representation.
  
  Assume, conversely, that $(U,\cH)$ is a positive energy representation. 
  For the standard boost, we have
  $h \in \fl \cong \so_{1,d}(\R)$, and the
  restriction $(U\res_{\cL_+}, \cH)$ is $(h,W)$-localizable
  in the family of all non-empty open subsets of $\dS^d$,
  where $W = W_R \cap \dS^d$ is the canonical wedge region 
  (Theorem~\ref{thm:local-reduct}).
  
  Next we recall from Lemma~\ref{lem:4.17} 
  that
  \[ S_{W_R} = \{ g \in \cP_+^\uparrow \: gW_R \subeq W_R \}
    = \oline{W_R} \rtimes \SO_{1,d}(\R)^\up_{W_R},\]
  where
\[ \SO_{1,d}(\R)^\up_{W_R} 
  = \SO_{1,1}(\R)^\up \times \SO_{d-2}(\R)\]
is connected, hence coincides with $\cL^h_e$.
It follows that
\[ S_{W_R} = G^h_e \exp(\R_{\geq 0}(\be_1 + \be_0))
  \exp(\R_{\geq 0}(\be_1-\be_0)).\] 

In view of Theorem~\ref{thm:sv}, the positive energy condition implies 
\[ C_\pm = \R_{\geq 0}(\be_1 \pm \be_0) \subeq \oline{W_R},
  \quad \mbox{ so that } \quad S_{W_R} \subeq S_\sV.\]
By Lemma~\ref{lem:direct-net}(c), the net $\sH^{\rm max}$
satisfies $\sH^{\rm max}(W_R) = \sV$.

Now suppose that $\cC \subeq W_R$ is a spacelike cone, so that
$\cC = \R_+ (\cC \cap \dS^d),$ 
where $\cC \cap \dS^d$ is an open subset of the wedge region
$W = W_R \cap \dS^d$ 
in de Sitter space. For $g^{-1} = (v, \ell) \in \cP_+^\uparrow$, 
the condition $\cC \subeq g.W_R$ is equivalent to
\[ g^{-1}.\cC = v + \ell.\cC \subeq W_R, \]
which in turn means that $v \in \oline{W_R}$ and $\ell.\cC \subeq W_R$.
Then
\begin{equation}
  \label{eq:5.36}
  U(g) \sV  = U(\ell)^{-1} U(v)^{-1} \sV \supeq U(\ell)^{-1} \sV
\end{equation}
follows from $\oline{W_R} \subeq S_\sV$, and therefore
  \begin{align*}
 \sH^{\rm max}(\cC)
&  = \bigcap_{\cC \subeq g.W_R} U(g)\sV
\ {\buildrel \eqref{eq:5.36} \over \supseteq}\ \bigcap_{\cC \subeq \ell^{-1}.W_R} U(\ell)^{-1} \sV\\
&  = \bigcap_{\cC \cap \dS^d \subeq \ell^{-1}.(W_R \cap \dS^d)} U(\ell)^{-1} \sV
    = \sH_{U\res_{\cL}}^{\rm max}(\cC \cap \dS^d).
  \end{align*}
We conclude that, on spacelike cones with vertex in $0$,
the net $\sH^{\rm max}$ coincides with the net $\sH_{U\res_{\cL}}^{\rm max}$ on de Sitter space.
As the latter net has the Reeh--Schlieder property
by Theorem~\ref{thm:local-reduct}, and all spacelike
cones can be translated to one with vertex~$0$,
localization in spacelike cones follows. 

Finally we show that $(U,\cH)$ is regular.
For $v \in W_R$ and a pointed
spacelike cone $C$ with $v + C \subeq W$, there exists an $e$-neighborhood
 $N \subeq G$ with $v + C \subeq g.W$ for all $g \in N$. This
 implies that $\sH^{\rm max}(v + C) \subeq \sV_N$, so that
 $(U,\cH)$ is regular.
\end{proof}

\begin{definition} \label{def:minkcaus}
(a)     {\rm(Causal complement)}
Let $\jV=\RR^{1,d}$ be Minkowski space. Its causal structure allows us to
define the \textit{causal complement (or the  spacelike complement)} of
an open subset $\cO\subset \jV$ by 
\begin{equation}\label{eq:ccomp}
  \cO'=\{x\in \jV: (\forall y \in \cO)\, (y-x)^2<0\}^\circ. \end{equation}
This is the interior of the set of all points that cannot
be reached from $\cO$ with a timelike or lightlike curve. 

\nin (b)  {\rm(Spacelike cones)}
  In Minkowski space $\R^{1,d}$, we call an open subset
  $\cO$ {\it spacelike} if $x_0^2 < \bx^2$ holds for all $(x_0,\bx) \in \cO$.
  A spacelike open subset is called a {\it spacelike
    (convex) cone} if, in addition, it is a (convex) cone.
  \index{convex cone!spacelike \scheiding}
  \index{double cone \scheiding}
  
\nin (c)  {\rm(Double cone)} A {\it double cone}  is, up to Poincar\'e covariance,
the causal closure
\[ \bB_r'' = (r \be_0 - \jV_+) \cap (-r \be_0 + \jV_+) \] 
  of an open ball of the time zero hyper-plane
    $\bB_r=\{(0,\bx) \in \RR^{1,d}: \bx^2< r^2\}$.
\end{definition}

\begin{remark} 
Infinite helicity representations $(U,\cH)$ of $\cP_+$ in $\RR^{1,d}$ are 
{\bf not} localizable in double cones (Definition~\ref{def:minkcaus}).
Let  $\sV=\sV(h,U)$ for $h$ as in Example~\ref{exs:ds-ads}.
  In  \cite[Thm. 6.1]{LMR16} it is shown that, if $\cO \subeq \R^{1,d}$ is a double cone, then
  \begin{equation}\label{eq:infspin}  \sH^{\rm max}(\cO)
    = \bigcap_{\cO \subeq g.W_R} U(g) \sV = \{0\}.
  \end{equation}
  The argument for \eqref{eq:infspin} can be sketched as follows.
  Infinite spin representations are massless representations,
i.e., the support of the spectral measure of the space-time
    translation group is  
  \[ \partial \jV_+=\{(x_0,\bx) \in\RR^{1,d}:x_0^2 - \bx^2 =0,x_0 \geq 0\}.\]
Covariant nets of standard subspaces on double cones 
  in massless representations are also dilation covariant in the sense that   
  the representation of $\cP_+$ extends to the
Poincar\'e and dilation group 
$\R^{1,d} \rtimes (\R^+\times \cL)$, and that the net is also
covariant under this larger group, cf. \cite[Prop.~5.4]{LMR16}.
 When $d=3$, this follows from the fact that, due to
  Huygens' Principle, one can associate 
  a standard subspace to the forward lightcone by 
  $\sH(\jV_+)=\overline{\sum_{\cO\subset \jV_+}\sH(\cO)}$
(sum over all double cones in $\jV_+$), and
  the modular group of $\sH(\jV_+)$ is geometrically implemented by
  the dilation group.
 As massless infinite helicity representations are not
  dilation covariant, it follows that they do not
  permit localization in double cones.
Properties of the free wave equation permit to  extend this  
argument  to any space dimension $d\geq2$,   including  even space dimensions,
for which Huygens' Principle fails (\cite[Sect.~8.2]{LMR16}). 
However, in Theorem~\ref{thm:poinloc}, we recover in our general setting the result contained in \cite[Thm.~4.7]{BGL02} that all   positive energy   representations of $\cP_+$
  are localizable in spacelike cones.    
\end{remark}

\subsection{Appendices to Section~\ref{sec:5}}

\subsubsection{Standard pairs}
\label{subsubsec:stand-pairs}

\index{standar pair $(U,\sV)$ \scheiding } 
\begin{definition} 
  {\it Positive standard pairs} $(U,\sV)$ consist of a standard subspace
$\sV$ of a complex Hilbert space~$\cH$ and a 
unitary one-parameter group $(U_t)_{t \in \R}$ on $\cH$ such that 
$U_t \sV \subeq \sV$ for $t \geq 0$ and $U_t = e^{itH}$ with $H \geq 0$.
\end{definition}

\begin{theorem} \label{thm:borch-wies} {\rm(Borchers--Wiesbrock Theorem)} 
Any standard pair 
defines an antiunitary positive energy representation of 
$\Aff(\R) \cong \R \rtimes \R^\times$ by 
\begin{equation}  \label{eq:affrep1}
  U(b,e^s) := U_b \Delta_\sV^{-is/2\pi} \quad  \mbox{ and }
  \quad 
U(0,-1) := J_\sV.
\end{equation}
Conversely, all these representations define positive standard pairs.
\end{theorem}

\begin{proof} This is the Borchers--Wiesbrock Theorem 
  (\cite[Thm.~3.18]{NO17}, see also  \cite{Bo92}, \cite{Wi93a}).
\end{proof}

We refer to \cite{NOe22} for classification results
for pairs $(h,x)$ in Lie algebras, where $h$ is an Euler element
and $x$ is contained in a pointed generating invariant cone
satisfying $[h,x] = x$. 

\begin{proposition} \label{prop:trans-commute}
  Consider a Lie group $G_\sigma = G \rtimes \{ \bone, \sigma\}$,
  where $\sigma \in \Aut(G)$ is an involution. 
  Let $(U,\cH)$ be an antiunitary representation
  of $G_\sigma$. 
Suppose that $(\sV,U^j)$, $j = 1,2$, are positive standard pairs for which 
there exists a graded homomorphism $\gamma \: \R^\times\to G$ 
and $x_1, x_2 \in \g$ such that 
\[ J_\sV = U(\gamma(-1)), \quad \Delta_\sV^{-it/2\pi} = U(\gamma(e^t)), \]
and
\[ U^j(t) = U(\exp t x_j), \quad t \in \R, j =1,2.\] 
Then the unitary one-parameter groups $U^1$ and $U^2$ commute. 
\end{proposition}

\begin{proof} The positive cone $C_U\subeq \g$ of the representation $U$ 
is a closed convex $\Ad(G)$-invariant cone. As we may w.l.o.g.\ assume that 
$U$ is injective, the cone $C_U$ is pointed. 

Writing $\Delta_\sV^{-it/2\pi} = U(\exp t h)$ and 
$U^j_t = U(\exp t x_j)$ with $h,x_1, x_2 \in \g$, we have 
$[h,x_j] = x_j$ for $j =1,2$ and $x_1, x_2 \in C_U$
by \eqref{eq:affrep1}. 
If 
\[ \g_\lambda(h)= \ker(\ad h - \lambda \bone) \]  
is the $\lambda$-eigenspace of $\ad h$ in $\g$, then 
$[\g_\lambda(h),\g_\mu(h)] \subeq \g_{\lambda + \mu}(h)$, so that 
$\g_+ := \sum_{\lambda >0} \g_\lambda(h)$ is a nilpotent Lie subalgebra. 
Therefore $\fn := (C_U \cap \g_+) - (C_U \cap \g_+)$ is a nilpotent 
Lie algebra generated by the pointed invariant cone $C_U \cap \g_+$.
By \cite[Ex.~VII.3.21]{Ne99}, $\fn$ is abelian. Finally 
$x_j \in C_U \cap \g_1(h) \subeq \fn$ implies that $[x_1,x_2] =0$.
\end{proof}

One may expect that one-parameter 
groups $U^1$ and $U^2$, for which $(\sV,U^j)$ form a standard pair, 
commute. By Proposition~\ref{prop:trans-commute} this is true 
if they both come from an antiunitary representation 
of a finite-dimensional Lie group. 
The following example shows that this is not true in 
general, not even if the two one-parameter groups are conjugate under the 
stabilizer group~$\U(\cH)_\sV$. 

\begin{examplekh} On $L^2(\R)$ we consider the selfadjoint operators 
\[ (Qf)(x) = xf(x) \quad \mbox{ and }\quad (Pf)(x) = i f'(x),\]
satisfying the canonical commutation relations $[P,Q] = i \bone$. 
For both operators, the Schwartz space $\cS(\R)\subeq L^2(\R)$ is a core. 
Actually it is the space of smooth vectors for the representation 
of the $3$-dimensional Heisenberg group generated by the corresponding 
unitary one-parameter groups 
\[ (e^{itQ}f)(x) = e^{itx}f(x) \quad \mbox{ and }\quad (e^{itP}f)(x) = f(x-t).\]
Since $e^{ix^3}$ is a smooth function for which all derivatives grow at most 
polynomially, it defines a continuous multiplication operator on $\cS(\R)$ 
(\cite[Thm.~25.5]{Tr67}). Therefore the unitary operator $T := e^{iQ^3}$ 
maps $\cS(\R)$ continuously onto itself, and thus 
\[ \tilde P := TPT^{-1} = e^{iQ^3} P e^{-iQ^3} \] 
is a selfadjoint operator for which $\cS(\R)$ is a core. 
For $f \in \cS(\R)$, we obtain 
\[ (\tilde P f)(x) 
= i e^{ix^3} \frac{d}{dx} e^{-ix^3}f(x) 
= i (-i 3x^2 f(x) + f'(x)),\] 
so that $\tilde P = P + 3 Q^2$ on~$\cS(\R)$.

The two selfadjoint operators $Q$ and $e^P$ are the infinitesimal generators 
of the irreducible antiunitary representation of $\Aff(\R) = \R \rtimes \R^\times$, 
given by 
\[ U(b,e^t) = e^{ib e^P} e^{it Q} \quad \mbox{ and }\quad 
(U(0,-1) f)(x) = \oline{f(-x)}.\] 
Accordingly, the pair $(\Delta, J)$ with 
\[ \Delta = e^{-2\pi Q} \quad \mbox { and }\quad J = U(0,-1) \] 
specifies a standard subspace $\sV$ which combines with 
$U^1_t := e^{it e^P}$ to an irreducible standard pair~$(\sV,U^1)$. 
The unitary operator $T$ commutes with $\Delta$ and with $J$ because 
$JQJ = -Q$, so that $T(\sV) = \sV$. Therefore the unitary one-parameter group 
$U^2_t :=  e^{iQ^3} U^1_t e^{-iQ^3}=  e^{it e^{\tilde P}}$ 
also defines a standard pair $(\sV, U^2)$. These two one-parameter groups 
do not commute because otherwise the selfadjoint operators 
$P$ and $P + 3 Q^2$ would commute in the strong sense, hence in particular 
on their core $\cS(\R)$; contradiction. 
\end{examplekh}

\subsubsection{The Reeh--Schlieder property} 
\label{subsubsec:reeh-schlieder}

The following lemma from \cite{NO21} is a
key tool in the derivation of the Reeh--Schlieder property
for representations for which the positive cone $C_U$ has interior points. 

\begin{lemma}  \label{lem:sw64-gen2} 
Let $(U, \cH)$ be a unitary representation of the connected Lie 
group $G$ for which the positive cone $C_U$ 
has interior points. If $\xi, \eta \in \cH$ are such that the matrix coefficient 
\[ U^{\xi,\eta} \:  G\to \C, \quad g \mapsto \la \xi, U(g) \eta \ra \] 
vanishes on an open subset of $G$, then $U^{\xi,\eta} =0$ on~$G$.
\end{lemma} 

\begin{proof} Passing to the quotient Lie group $G/\ker(U)$, we may 
w.l.o.g.\ assume that $U$ is injective. Then the 
closed convex cone $C_U\subeq \g$ is pointed, and by assumption 
it is also generating, so that it defines a closed complex Olshanski semigroup 
$S_{C_U} = G \exp(iC_U)$ (Definition~\ref{def:olsh-sgrp}), 
and the matrix coefficient 
$U^{\xi,\eta}$ extends to a continuous function on $S_{C_U}$ which is holomorphic on 
the interior~$S_{C_U}^0 = G \exp(i C_U^0)$, which is a complex manifold 
(\cite[Thm.~XI.2.5, Prop.~XI.3.7]{Ne99}).

Suppose that $U^{\xi,\eta}$ vanishes on the non-empty open subset~$\cO$. 
Replacing $\eta$ by $U(g_0)\eta$ for some $g_0 \in \cO$, we may assume that 
$e \in \cO$. The exponential function 
$\exp \: \g + i C_U \to S_{C_U}$ 
is continuous, and holomorphic on $\g + i C_U^0$. 
Therefore the function $U^{\xi,\eta} \circ \exp \: \g \to \C$ 
extends to a continuous function on the closed wedge 
$\g + i C_U$, holomorphic on the interior and vanishing on 
$\exp^{-1}(\cO) \not=\eset$. 
Now \cite[Lemma~A.III.6]{Ne99} implies that 
$U^{\xi,\eta} \circ \exp = 0$. 
Therefore the regularity of the exponential 
function $\g + i C_U^0 \to S_{C_U}^0$ near $0$ implies that 
$U^{\xi,\eta} = 0$ on $S_{C_U}$, and hence its restriction to $G$ vanishes. 
\end{proof}

\begin{theorem} \label{thm:pe-rs} {\rm(Positive energy implies Reeh--Schlieder)} 
  Let $(U,\cH)$ be an antiunitary representation of $G_{\tau_h}$ for which
  $C_U^\circ \not=\eset$. We consider a net $\sH$
  of real subspaces on open subsets of the homogeneous space $M = G/H$.
  Then $\sH$ has the Reeh--Schlieder property, i.e.,
  $\sH(\cO)$ is cyclic whenever $\cO \not=\eset$, if one of the following
  conditions is satisfied:
  \begin{enumerate} 
  \item[\rm(a)] $\sH$ satisfies {\rm(Iso), (Cov), (Add)} and
    $\sH(M)$ is cyclic.
  \item[\rm(b)] There exists a real subspace $\sE \subeq \cH^{-\infty}$
    with $\sH = \sH^M_\sE$ and an open subset $W \subeq M$ with $\sH^M_\sE(W) = \sV
    = \sV(h,U)$. 
  \end{enumerate}
\end{theorem}

\begin{proof} (a) Let $\cO \subeq M$ be non-empty.
  We choose $\eset \not= \cO_1 \subeq \cO$ relatively compact and an
  $e$-neighborhood $V \subeq G$ with $V.\cO_1 \subeq \cO$.
  For $\xi \in \sH(\cO)^\bot$, $\eta \in \sH(\cO_1)$ and
  $g \in V$, we then have
  \[ U^{\xi,\eta}(g) = \la \xi, U(g) \eta \ra \in \la \xi, \sH(\cO) \ra = \{0\}.\]
  As the cone $C_U$ has interior points,
  \cite[Lemma~2.13]{NO21} now implies that $U^{\xi,\eta} = 0$.
  We conclude that $\xi\bot  U(G)\sH(\cO_1)$. Since the net $\sH$ is additive,
  the right hand side generates $\sH(M)$, so that our assumption that
  $\sH(M)$ is cyclic entails that $\xi = 0$.

  \nin (b) The nets $\sH^M_\sE$ satisfies (Iso) and (Cov) by
  Remark~\ref{rem:iso-cov-net} and it is  additive by
  Proposition~\ref{prop:hmo-add} below.
  Further $\sH^M_\sE(M) \supeq \sH^M_\sE(W) = \sV$ is cyclic. Now (b) follows from (a). 
\end{proof}

\begin{remark} If $\sH(M)$ is not cyclic, then the preceding theorem applies
  to the representation on the closed complex
  subspace spanned by $\sH(M)$, which is $U(G)$-invariant.  
\end{remark}

To shed some more light on the assumption that $\sH(M)$ is cyclic,
the following lemma is useful:

\begin{lemma} Let $(U,\cH)$ be an antiunitary
  representation of $G_{\tau_h}$.
  For a real subspace $\sE \subeq \cH^{-\infty}_{\rm KMS}$ and
  $M = G/H$, the following  are equivalent:
  \begin{description}
  \item[\rm(a)] $\sH_\sE^M(M)$ is cyclic in $\sH$. 
  \item[\rm(b)] $U^{-\infty}(C^\infty_c(G,\R))\sE$ is total in $\cH$. 
  \item[\rm(c)] $U^{-\infty}(G)\sE$ spans a weak-$*$ dense subspace of $\cH^{-\infty}$.
  \item[\rm(d)] $U^{-\infty}(G_{\tau_h})\sE$ spans a weak-$*$ dense subspace of $\cH^{-\infty}$.
  \end{description}
\end{lemma}

\begin{proof} (a) $\Leftrightarrow$ (b) follows from the fact that
  $\sH^M_\sE(M)$ is the closed real subspace of $\cH$ generated by
  $U^{-\infty}(C^\infty_c(G,\R))\sE$. 

  \nin (b) $\Rarrow$ (c) follows from the fact that the inclusion
  $\cH \into \cH^{-\infty}$ is continuous with dense range and
  that the weak-$*$ closed subspace generated by
  $U^{-\infty}(C^\infty_c(G,\R))\sE$ coincides with the one generated by
    $U^{-\infty}(G)\sE$. 

    \nin (c) $\Rarrow$ (d) is trivial. 

    \nin (d) $\Rarrow$ (c): We have to show that $J \sE = U^{-\infty}(\tau_h)\sE$
    is contained in the weak-$*$ closed subspace generated by
    $U^{-\infty}(G)\sE$. So let $\xi \in \cH^\infty$ be orthogonal to $U^{-\infty}(G)\sE$. 
    Then $\xi \bot U^{-\infty}(\exp \R h)\sE$, and since the $U_h$-orbit map
 $U^\alpha_h$ of $\alpha \in \sE$ extends holomorphically to $\cS_\pi$ with
    $U^\alpha_h(\pi i) = J \alpha$, it follows that
    $\xi \bot J\sE$. Now the assertion follows from the duality
    between $\cH^\infty$ and $\cH^{-\infty}$.

    \nin (c) $\Rarrow$ (b): Let $\xi \in \cH$ be orthogonal
    to $U(\phi)\sE$ for each $\phi \in C^\infty_c(G,\R)$.
    Then we also have $\xi \bot U(\phi)U^{-\infty}(G)\sE$,
    and now $U(\phi^*)\xi \bot U^{-\infty}(G)\sE$ implies with (c) that 
    $U(\phi^*)\xi =0$. Using for $\phi$ an approximate identity
    in $C^\infty_c(G,\R)$, we conclude that $\xi = 0$. This proves (b).     
\end{proof}

\section{Perspectives}
\label{sec:6}

In this final section we briefly discuss several issues that are 
under current
investigation and for which the existing results are much less complete.

\subsection{Additivity}
\label{subsec:additivity}

In this subsection, we take a closer look at the
additivity condition (Add) for nets of real subspaces.
We show in particular that the nets $\sH^M_\sE$ are always
additive.
For causal flag manifolds, this implies already
that nets satisfying (Iso), (Cov), (BW) and (Add)
are uniquely determined by the representation~$(U,\cH)$ of
$G_{\tau_h}$ (cf.\ Theorem~\ref{thm:unique-add-conf}). 

\index{net!additive \scheiding} 
\index{net!countably additive \scheiding} 

\begin{definition}
We call a net $\sH$ on open subsets of $M$ {\it additive}
if $\cO = \bigcup_{j \in J} \cO_j$ implies 
$\sH(\cO) = \oline{\sum_{j \in J}\sH(\cO_j)}.$
We call it {\it countably additive}, it this relation holds
for countable index sets.
\end{definition}

The following proposition shows that a large class of nets
of real subspacs is additive. 
  
\begin{proposition} \label{prop:hmo-add}
  For a real subspace $\sE \subeq \cH^{-\infty}$,
  the net $\sH^M_\sE$ is additive.
\end{proposition}

\begin{proof} Let $\cO \subeq M$ be open 
  and $(\cO_j)_{j \in J}$ an open covering of~$\cO$.
  We write
  \[ q_M \: G \to M= G/H, \qquad g \mapsto gH \]
  for the quotient map
  and consider some $\phi \in C^\infty_c(q_M^{-1}(\cO))$.

  The open subsets $q_M^{-1}(\cO_j)$ form an open cover of $q_M^{-1}(\cO)$.
  Therefore Lemma~\ref{lem:fraglem} implies the existence of
  $j_1, \ldots,j_k$ and test functions $\phi_\ell$, supported in
  $q_M^{-1}(\cO_{j_\ell})$, such that $\phi = \phi_1 + \cdots + \phi_k$.
Now 
  \[  U^{-\infty}(\phi)\sE 
  \subeq \sum_{\ell = 1}^k U^{-\infty}(\phi_\ell)\sE 
  \subeq \sum_{\ell = 1}^k \sH^M_\sE(\cO_{j_\ell}) 
  \subeq \sum_{j \in J} \sH^M_\sE(\cO_j)\]
implies that 
$\sH^M_\sE(\cO)   \subeq \sum_{j \in J} \sH^M_\sE(\cO_j),$
and additivity thus follows from isotony. 
\end{proof}

\nin {\bf Tools to verify additivity}

\begin{lemma} \label{lem:addiv-count}
  If $M$ has a countable basis for its topology,
  then every countably additive net on open subsets of $M$
  is additive.
\end{lemma}

\begin{proof} Let $(\cO_j)_{j \in J}$ be a family of open subsets of~$M$.
  Further, let
  $\fB$ be a countable basis for the topology of $M$.
  Then each $\cO_j$ is the union of the countable
  set $\fB_j$ of basis elements contained in~$\cO_j$, and therefore
  \[ \cO = \bigcup \{ \cB \: \cB \in \fB_\cO\}, \quad
  \fB_\cO := \bigcup_{j \in J} \fB_j, \]
  where $\fB_\cO$ is countable. 
Countable additivity of the net thus implies that 
  \[  \sH(\cO)
    = \oline{\sum_{\cB \in \fB_\cO} \sH(\cB)}
= \oline{\sum_{j \in J} \sum_{\cB \in \fB_j} \sH(\cB)}
= \oline{\sum_{j \in J} \sH(\cO_j)}.\]
Therefore $\sH$ is additive.
\end{proof}

\begin{remark} Every additive net is isotone because
  $\cO_1 \subeq \cO_2$ implies 
  $\cO_2 = \cO_1 \cup\cO_2$, so that additivity entails 
  \[ \sH(\cO_2) = \oline{\sH(\cO_1) + \sH(\cO_2)}
    \supeq \sH(\cO_1).\] 
\end{remark}

\begin{lemma} \label{lem:add-di} If $\sH(\cO)_{\cO \subeq M}$ is a net on
  open subsets of
  the second countable space $M$, each subspace $\sH(\cO)$ is decomposable as
  \[ \sH(\cO) = \int_X^\oplus \sH_x(\cO)\, d\mu(x), \]
  and if $\mu$-almost all the nets 
  $(\sH_x)_{x \in X}$ are additive, then $\sH$ is additive.
\end{lemma}

\begin{proof} In view of Lemma~\ref{lem:addiv-count}, it suffices
  to show that $\sH$ is countably additive.
  So let $\cO = \bigcup_{j \in J} \cO_j$ with a countable
  index set $J$.
  Then (DI3) and the additivity of the nets
  $\sH_x$ imply that
  \[ \oline{\sum_{j \in J} \sH(\cO_j)}
\ {\buildrel {\rm(DI3)}\over =}\ \int_X \oline{\sum_{j \in J} \sH_x(\cO_j)} \, d\mu(x)
    = \int_X \sH_x(\cO)\, d\mu(x) = \sH(\cO), 
  \]
  which is countable additivity. 
\end{proof}

\subsection{Locality}
\label{subsec:locality}

  \index{locality set $\cL_\sH$ of a net $\sH$ \scheiding }

\begin{definition} Let $\sH$ be a net of real subspaces
  on $M = G/H$ that is isotone and covariant. In
  $M \times M$, we defined the {\it locality set of $\sH$} by  
  \begin{align*}
    \cL_\sH
    &= \bigcup_{\sH(\cO_1) \subeq \sH(\cO_2)'} \cO_1 \times \cO_2 \subeq M \times M.
  \end{align*}
  By definition, this is an open subset, and {\rm(Cov)} implies that
  it is $G$-invariant. Moreover, it is symmetric in the sense that
  $(x,y) \in \cL_\sH$ implies $(y,x) \in \cL_\sH$.

The subset $\cL_\sH$ is non-empty if and only if the net $\sH$
satisfies the ``minimal'' locality condition that there exist
two non-empty open subsets $\cO_1, \cO_2 \subeq M$ with
$\sH(\cO_1) \subeq \sH(\cO_2)'$. 
\end{definition}

The subset $\cL_\sH \subeq M  \times M$ completely encodes the locality
properties of the net $\sH$ in terms of
a $G$-invariant subset of the set of pairs in
$M$. To connect locality properties of a net $\sH$ with
the given structures  on $M$ therefore reduces to
comparing $\cL_\sH$ with the given geometric data. 

\begin{lemma} \label{lem:3.2} If $\sH$ is 
  additive and $\cO_1, \cO_2$ are open subsets of
  $M$ with $\cO_1 \times \cO_2 \subeq \cL_\sH$, then
  \[ \sH(\cO_1) \subeq \sH(\cO_2)'.\]
\end{lemma}

This corresponds to the locality condition (Loc) in Section~\ref{sec:2}.

\begin{proof}   Since $\sH$ is additive, it is also isotone, and
  the real subspace $\sH(\cO_2)$ is generated
  by the subspaces $\sH(\cC)$, where $\cC \subeq \cO_2$ is a
  relatively compact  open  subset of $\cO_2$. So  it suffices to show that
  $\sH(\cO_1) \subeq \sH(\cC)'$ for such subsets. 

For any $(x,y) \in \cO_1 \times \oline{\cC} \subeq \cL_{\sH}$, there exist
  open subsets $\cO_x^y,  \cO_y^x \subeq M$ with
  \[ x \in \cO_x^y \subeq \cO_1, \ y \in \cO_y^x \subeq \cO_2\quad
    \mbox{ and } \quad
    \sH(\cO_x^y)\subeq    \sH(\cO_y^x)'.\]
  Then, for each $x \in \cO_1$, the sets $(\cO^x_y)_{y \in \oline{\cC}}$
  form an open covering of the compact subset $\oline{\cC} \subeq \cO_2$, 
  so that there exist finitely many points
  $y_1, \ldots, y_n \in \cO_2$ with
  \[ \cC \subeq \cO^x_{y_1} \cup \cdots \cup \cO_{y_n}^x.\]
  Then
  \[ \cO_x :=   \cO_x^{y_1} \cap \cdots \cap \cO_x^{y_n} \subeq \cO_1 \]
  is an open neighborhood of $x$ for which
  $\sH(\cO_x) \subeq \sH(\cO^x_{y_j})'$ for $j =1,\ldots, n$.
  Additivity of $\sH$ thus implies
  $ \sH(\cO_x) \subeq \sH(\cC)'$.
  Finally, we observe that the $\cO_x$ form an open cover of $\cO_1$,
  so that additivity further implies that
  $\sH(\cO_1) \subeq \sH(\cC)'$.   
\end{proof}

\begin{remark}\label{rem:3.3b} If $W$ and $W'$ are open subsets of $M$ with
  $\sH(W) = \sV$ and   $\sH(W') = \sV'$, then we have
  \begin{equation}
    \label{eq:defcLW}
    \cL_W := G.(W \times W') \cup G.(W' \times W) \subeq \cL_\sH.
  \end{equation}
  If $\sH$ is additive, it follows from Lemma~\ref{lem:3.2} that
  $\cO_1 \times \cO_2 \subeq \cL_W$ implies
  $\sH(\cO_1) \subeq \sH(\cO_2)'$.
\end{remark}

\begin{examples} (a) If $M = \R^{1,d-1}$ is Minkowski space
  and $W = W_R$ is the Rindler wedge, then
  $W' = - W_R$ and
  \[ \cL_W = G.(W \times W') \subeq M \times M \]
  is the set of spacelike pairs (cf.\ Remark~\ref{rem:poin}(d)). 
  For an open subset $\cO \subeq M$, the maximal open subset
  $\cO'$ satisfying
  \[ \cO \times \cO' \subeq \cL_W \]
  is called the {\it causal complement} of $\cO$
  (cf.\ Definition~\ref{def:minkcaus}(a)).
  \index{causal complement $\cO'$ of set $\cO$ \scheiding }

  The same picture prevails for de Sitter space
  $\dS^d \subeq \R^{1,d}$.

  \nin (b) For $M = \bS^1$, a causal flag manifold of $G = \SL_2(\R)$,
  the wedge regions are open non-dense intervals
  $W \subeq \bS^1$ (Example~\ref{ex:causal1c}).
  If $W = W_M^+(h)$, then $W' = W_M^+(-h)$ is the interior of the complement
  of $W$, and
  \[ \cL_W = G.(W \times W') = M^2 \setminus \Delta_M.\]
  So $(x,y) \in \cL_W$ if and only if $x \not=y$.

  \nin (c) In the non-compactly causal symmetric space
  $M = \SL_4(\R)/\SO_{2,2}(\R)$ (cf.\ Example~\ref{ex:Ipq}),
  not all acausal pairs are contained
  in $\cL_W$, and $M \times M$ contains several open acausal $G$-orbits
  (cf.\ \cite{NOO21, NO25}).   
\end{examples}

\begin{remark}
  In the context of abstract wedges, represented by elements of the
  set  
  \[ \cG(G_{\tau_h}) := \{ (x,\tau) \in \g \times G\tau_h\: \Ad(\tau)x = x,
    \tau^2 = e\}\]
  (cf.\ Exercise~\ref{exer:bgl}), there is a natural complementation map
  \[ (x,\tau) \mapsto (x,\tau)' := (-x, \tau).\]
  In this context, it is a natural question if
  $(h,\tau_h)' = (-h, \tau_h)$ is contained in the $G$-orbit
  of $(h,\tau_h)$. This is equivalent to the symmetry of $h$ and
  the additional condition that there exists a $g_0 \in G^{\tau_h}$ with
  $\Ad(g_0)h = -h$ (cf.\ \cite{MNO25}). 
  If this is the case and $(U,\cH)$ is an antiunitary
  representation of  $G_{\tau_h}$, then $\sV' = U(g_0) \sV$
  follows from $g_0.(h,\tau_h) = (-h, \tau_h)$. 
  For any net $\sH$ satisfying (Cov) and (BW), this implies that
  \begin{equation}
    \label{eq:Wxy}
    W \times g_0.W \subeq \cL_\sH
  \end{equation}
(cf.\ \eqref{eq:defcLW}).  
\end{remark}

  We refer to \cite{MN25} and \cite{NO25} for more detailed discussions 
  of locality properties of nets on causal flag manifolds
  and non-compactly causal symmetric spaces, respectively.
  Twisted locality conditions are discussed in \cite{MN21}
  and \cite{MNO26}.

\subsection{Representations of Lie supergroups}

Lie supergroups and their unitary representations arise
naturally in Physics in connection with supersymmetry
(cf.\ \cite{Gu75, Gu93, Gu00, Gu01}). It would be interesting to extend the
theory developed in these notes to this context, where
the Euler element $h \in \g$ is supposed to be an even element.

  \index{Lie supergroup \scheiding}

\begin{definition} A  {\it Lie supergroup} is a pair
  $(G,\g)$, where $\g = \g_{\oline 0} \oplus \g_{\oline 1}$
  is a finite-dimensional Lie superalgebra and 
$G$ is a real Lie group with Lie algebra $\g_{\oline 0}$, acting smoothly 
by automorphisms on~$\g$ via $\Ad \: G \to \Aut(\g)$, 
in such a way that the action on 
$\g_{\oline 0}$ is the adjoint action of the Lie group $G$ with
Lie algebra $\L(G) = \g_{\oline 0}$. 
\end{definition}

\begin{definition}
  A {\it unitary representation} of a Lie supergroup $(G,\g)$
  is a pair 
$(U,\beta)$, where $(U,\cH)$ is a unitary representation
of the Lie group $G$ on a graded Hilbert space
$\cH = \cH_{\oline 0} \oplus \cH_{\oline 1}$, preserving the grading,
and
\[ \beta \:  \g \to \End(\cH^\infty) \]
is a representation of the Lie superalgebra $\g$ on the space of smooth
vectors of $U$ satisfying
\[ -i \beta(x) \subeq \beta(x)^* \quad \mbox{ for }\quad x \in \g_{\oline 1} \]
(\cite{NS11}, \cite{CCTV06}).
\end{definition}

\begin{problem}
  One can associate to each real subspace
  $\sE \subeq \cH^{-\infty}$ (distribution vectors for
$U$) the closed real subspaces $\sH^G_{\sE}(\cO)$ generated by
\[ \beta(U(\g)) U^{-\infty}(C^\infty_c(\cO,\R)) \sE.\]
Does this construction lead to nets that are compatible with
fermionic nets in AQFT? Possibly one can develop a ``supersymmetric''
variant of the theory described in these notes. 
\end{problem}

Here are some relevant structures and observations. 

\begin{definition} {\rm(The $*$-monoid associated to a Lie supergroup)} 
The antilinear map \[
  \g_\C\to \g_\C\ ,\ x\mapsto x^*, \quad \mbox{defined by} \quad
x^*:=
\begin{cases}
-x&\text{ if }x\in \g_\eev,\\
-i \,x&\text{ if }x\in\g_\ood.
\end{cases}
\]
is an anti-automorphism. It 
 extends to an antilinear anti-automorphism 
\begin{equation}
\label{dfostarrr}
U(\g_\C)\to  U(\g_\C)\ ,\ 
D\mapsto D^*
\end{equation}
in a natural way. Consider the monoid $\Smi$ with underlying set $G\times U(\g_\C)$ and multiplication
\[
(D_1, g_1)(D_2, g_2)=(D_1 (g_1\cdot D_2), g_1g_2), 
\]
where $g\cdot D$ denotes the adjoint action of $g\in G$ on 
$D\in U(\g_\C)$. 
The neutral element of $\Smi$ is $1_\Smi:=(1_{U(\g_\C)}, e)$.
The map 
\[
  \Smi\to \Smi\,,\,s\mapsto s^* \quad \mbox{ defined by } \quad
  (D,g)^*:=(g^{-1} \cdot D^*, g^{-1})\]
is an involution of $\Smi$. 
Recall that $U(\g_\C)$ is an associative superalgebra. 
An element $(D,g)\in \Smi$ is called \emph{odd} 
(resp.  \emph{even}) if $D$ is an odd (resp. even) element of $U(\g_\C)$.

Replacing in this construction the elements $g \in G$
by compactly supported smooth functions on $G$, one can even
construct a graded $*$-algebra $(C^\infty_c(\cG),*)$ in such a
way that every unitary $\cG$-representation
integrates to a $*$-representation of $C^\infty_c(\cG)$ by
bounded operators. Thus one even obtains
``supergroup $C^*$-algebras''. We refer to \cite{NS16} for details. 
\end{definition}

\begin{remark}
From the table in \cite[\S 2.5]{NS11} we get some information
on which finite-dimensional simple Lie superalgebras $\g$ 
have non-trivial unitary representations.
According to \cite[\S 6]{NS11}, for any unitary representation
of $\cG = (G,\g)$, we must have the inclusion
\[  \Cone(\fg) = \cone\{ [x,x] \: x \in \g_{\oline 1} \}
  \subeq C_U = \{ x \in \g_{\oline 0} \: -i \cdot \partial U(x) \geq 0\}.\]
For a unitary representation with  discrete kernel, the cone
$C_U$ is pointed, so that the pointedness of the cone
generated by the brackets of odd elements is necessary
for the existence of non-trivial unitary representations.
In this sense \cite[Thm.~6.2.1]{NS11} compiles a negative list
of simple Lie superalgebras for which this is not the case.
\end{remark}

\subsection{The geometric structure on $M$} 

Let $(U,\cH)$ be an antiunitary representation
of $G_{\tau_h}$ and $\sV := \sV(h,U) \subeq \cH$ be the canonical
standard subspace from \eqref{eq:def-V(h,U)}.
Let $\sE \subeq \cH^{-\infty}$ be a
finite-dimensional linear subspace, invariant under the subgroup
$H \subeq G$ and $M := G/H$. Then we obtain a net $\sH^M_\sE$ on
$M$, satisfying (Iso) and (Cov).

Further 
\[ W_\sE := \{ gH \in M  \: U^{-\infty}(g)\sE
  \subeq \cH^{-\infty}_{\rm KMS}\}^\circ\]
specifies an open subset of $M$ that deserves to be called the
{\it wedge region associated  to $\sE$}, but it may be empty; depending on the
subspace~$\sE$ (cf.\ Proposition~\ref{prop:4.8}). 
\index{wedge region! in $M$ \scheiding} 

If $W_\sE \not=\eset$, then Proposition~\ref{prop:4.8} implies that  
\[ \sH^M_\sE(W_\sE) \subeq \sV \quad \mbox{ and } \quad
  \exp(\R h) W_\sE \subeq W_\sE.\]
For $\xi \in \sV^\infty := \sV \cap \cH^\infty$, we have
\[  \la \xi, i \partial U(h) \xi \ra
  = \frac{d}{dt}\Big|_{t = 0} \la \xi, e^{it \partial U(h)} \xi \ra 
  = \frac{d}{dt}\Big|_{t = 0} \la \xi, \Delta_\sV^{t/2\pi} \xi \ra 
  = \frac{d}{dt}\Big|_{t = 0} \|\Delta_\sV^{t/4\pi} \xi\|^2 
  \leq 0
\] 
because the convex function 
\[ f \: [0,2\pi] \to \R, \quad f(t) := \|\Delta_\sV^{t/4\pi} \xi \|^2 \]
takes its minimal value in $t = \pi$
and has a local maximum in $t = 0$. Here convexity follows
from the Spectral Theorem, which implies that $f$ is a Laplace
transform, and $\xi \in \sV = \Fix(J_\sV \Delta_\sV^{1/2})$ implies that
it is invariant under reflection in $\pi$.

For $\alpha \in \sE$ and $\phi \in C^\infty_c(W_\sE,\R)$, we have
$U(\phi)\alpha \in \sV^\infty$, so that we get 
\[ \la U(\phi)\alpha, i \partial U(h) U(\phi)\alpha \ra \leq 0.\] 
Letting $\phi$ tend to a point measure $\delta_g$, $g \in W_\sE$, we obtain
\[ \la \alpha, i \partial U(\Ad(g)^{-1}h)\alpha \ra \leq 0\]
in the sense of distributions on $G$. 
Maybe these inequalities can be related to the generalized
positive energy conditions appearing in \cite{JaNi24}.

\begin{remark}
In this context, it becomes apparent that the closed convex cone
$C(W_\sE) \subeq \g$, generated by
\[ \Ad(g) ^{-1} h, \quad g \in W_\sE, \]
plays an important role. It is clearly invariant under
$e^{\R \ad h}$, so that
\[   C(W_\sE) \subeq C(W_\sE)_+ + \g_0(h) - C(W_\sE)_-
  \quad \mbox{ for } \quad
  C(W_\sE)_\pm := \pm C(W_\sE) \cap \g_{\pm 1}(h) \subeq C(W_\sE)\]
(cf.\ Lemma~\ref{lem:Project}). 
\end{remark}

Here is an alternative approach.

\begin{proposition} Consider an antiunitary representation $(U,\cH)$ of $G_{\tau_h}$
  and the corresponding standard subspace $\sV := \sV(h,U)$.
    Assume that the net $\sH$ on open subsets of $M = G/H$
  satisfies {\rm(Iso)} and {\rm(Cov)}. 
  Suppose further that there exists an open subset
  $\eset \not=\cO \subeq M$
  for which $\sH(\cO)$ is cyclic and contained in~$\sV$.
  Then {\rm(BW)} holds for the open subset
  $W := \exp(\R h).\cO$.
\end{proposition}

\begin{proof} Clearly, $W$ is an $\exp(\R h)$-invariant
  open subset of $M$ and (Cov) and (Iso) imply that 
  ${\sH(W) \subeq \sV}$ is an $U(\exp \R h)$-invariant subspace.
  As $\sH(W)$ contains $\sH(\cO)$,  it is cyclic, so that 
$\sH(W) = \sV$ follows from the Equality Lemma~\ref{lem:lo08-3.10}.
\end{proof}

\begin{corollary} Assume that the net $\sH$ on open subsets of $M = G/H$
  satisfies {\rm(Iso), (Cov), (RS)} and {\rm(Add)} and
  that there exists an open subset $\cO \subeq M$ such that
  $\sH(\cO) \subeq \sV = \sV(h,U)$. Then
  the union $W^\sH$ of all such open subsets is non-empty,
  open, $\exp(\R h)$-invariant, and satisfies
  \[ \sH(W^\sH) = \sV.\] 
\end{corollary}

\begin{problem}
  Compare $W^\sH$ with $W_M^+(h)$ for nets on
  causal homogeneous spaces $M = G/H$.  
\end{problem}

\subsection{Classification of nets of real subspaces} 

We expect that there are various contexts in which
nets on $M = G/H$ could be classified.
Specifically, the (BW) property determines the net
for a given antiunitary representation $(U,\cH)$ of $G_{\tau_h}$
on all wedge regions $(g.W)_{g \in G}$ in~$M$.

For causal flag manifolds, this fact already implies that
a net satisfying (Iso), (Cov), (RS), (BW) and (Add)
 is uniquely
determined by the antiunitary representation $(U,\cH)$ of $G_{\tau_h}$
(see Subsection~\ref{subsec:additivity} and \cite{MN25} for details).

\begin{problem} Consider $G := \Aff(\R)_e$ with the non-symmetric
  Euler element $h = (0,1)$ (Example~\ref{ex:causal1b}). Here
  the intervals $(x,\infty)$, $x \in \R$, are natural
  wedge regions in $M = \R$. Given an antiunitary representation $(U,\cH)$
  of $G_{\tau_h} = \Aff(\R)$, is it possible 
  to classify all nets on open subsets of $\R$ that
  satisfy the (BW) condition?  Here additivity and locality
  conditions certainly help to reduce the problem.

  For instance, if $\sH$ is additive, then it is easy to see that the
  whole net is determined by the real subspace
  $\sH((0,1))$, assigned to the open unit interval $(0,1)$.
  So one has to determine which real subspaces arise in such nets. 
\end{problem}

  \section{Appendix}

\subsection{The category of W*-algebras}
\label{app:wstar}
  
By the Gelfand--Naimark Theorem, $C^*$-algebras
are closed $*$-subalgebras of~$B(\cH)$, 
$\cH$ a complex Hilbert space. On the other hand, we have defined
von Neumann algebras directly as $*$-subalgebras $\cM \subeq B(\cH)$
satisfying $\cM'' = \cM$. So they are in particular closed with respect
to the weak-$*$ topology on $B(\cH)$, specified by
the subspace $B_1(\cH) \subeq B(\cH)^*$ of trace class operators
and the trace pairing $(A,B) \mapsto \tr(AB)$. 
As $B(\cH) \cong B_1(\cH)^*$, the duality theory of Banach spaces easily
implies that $\cM \cong Q^*$ for $Q := B_1(\cH)/\cM^\bot$. Hence every
von Neumann algebra has a {\it predual}.

This observation can be used to specify von Neumann algebras axiomatically,
independently of an embedding in some $B(\cH)$.

\begin{definition} A $C^*$-algebra $\cM$ is called a {\it $W^*$-algebras}
  if it has a predual, i.e., there exists a closed subspace
  $\cM_* \subeq \cM^*$ with $\cM \cong (\cM_*)^*$ as Banach spaces.
\end{definition}

This approach has been pursued by S.~Sakai, and his monograph
\cite{Sa71} is an excellent reference. \cite[Cor.~1.13.3]{Sa71} asserts
in particular that $W^*$-algebras have a unique predual $\cM_*$.
Its elements are called {\it normal linear functionals}.
They are the continuous
linear functionals for the $\sigma(\cM,\cM_*)$-topology on $\cM$, i.e.,
the coarsest topology for which all functionals in $\cM_*$ are continuous.
Any normal selfadjoint functional is a difference of two positive ones, and the positive
normal functionals $\phi$ can also be characterized by the property that,
for every uniformly bounded  increasing directed subset
$(x_j)_{j \in J}$ of $\cM$, we have
\begin{equation}
  \label{eq:sup-cont}
 \phi(\sup x_j) = \sup\phi(x_j)   
\end{equation}
(\cite[Thm.~1.13.2]{Sa71}).

We also note that $W^*$-algebras have an identity,
which can be derived from the Krein--Milman Theorem because it ensures
the existence of extreme points in the unit ball of $\cM$, 
which is compact in the $\sigma(\cM,\cM_*)$-weak topology
(\cite[\S 1.7]{Sa71}).

\begin{examples} \label{exs:vN-exs}
  (a) For every complex Hilbert space $\cH$, the full operator
  algebra $B(\cH)$ is a $W^*$-algebra with predual
  $B(\cH)_* = B_1(\cH)$ (trace class operators).

  \nin (b) For every $\sigma$-finite measures space $(X,\fS, \mu)$,
  the Banach algebra $L^\infty(X,\fS,\mu)$ is a commutative $W^*$-algebra with 
  $L^\infty(X,\fS,\mu)_* \cong L^1(X,\fS,\mu)$.

  The same holds for $\ell^\infty$-direct sums
  (whose preduals are $\ell^1$-direct sums), and {\bf all commutative
    $W^*$-algebras are such sums.}
  More intrinsically, they can be described as the
  space 
  $L^\infty_{\rm loc}(X,\fS,\mu)$ of bounded, locally measurable functions
  on a semi-finite measure space $(X, \fS,\mu)$. Here
  {\it semi-finite} means that, every $E \in \fS$ with $\mu(E) = \infty$ contains a measurable   subset of finite positive  measure.
  A function $f$ is called {\it locally measurable} 
  if its restriction to all measurable subsets of finite measure is
  measurable. 
\end{examples}

\begin{definition} A {\it morphism of $W^*$-algebras} is a complex linear $*$-algebra
  morphism $\pi \: \cM \to \cN$ with 
  $\pi^*\cN_* \subeq \cM_*$, i.e., pullbacks of normal functionals are normal.
  We call these algebra morphisms {\it normal}. 
  For every complex Hilbert space $\cH$, a {\it normal representation}
  $(\pi, \cH)$ of $\cM$ is a  normal morphism $\pi \: \cM \to  B(\cH)$.   
\end{definition}

\begin{remark} \label{rem:weights-gns}
  (a) For normal states, the GNS construction produces a normal
  representation.

\nin (b)  This is more generally true for so-called semi-finite weights.
  A {\it weight}
  \[ \omega \: \cM_+ \to [0,\infty] \]
  is an additive, positively homogeneous
  function. It is called {\it normal}
  if it is  compatible with bounded sup's in the sense of \eqref{eq:sup-cont}.
  A~weight $w$ on~$\cM$ is called {\it semi-finite} if the set
	\[	\{M\in\cM_+\,\mid\,w(M)<\infty\}	\]
	generates a $*$-algebra which is $\sigma(\cM,\cM_*)$-dense in $\cM$.

The GNS construction and the Tomita--Takesaki Theorem
extend to normal weights,  and faithful normal semi-finite
weights always exist (\cite[III.2.2.26]{Bla06}). Normal
semi-finite weights are sums (in the sense of summability
of general families) of normal positive forms
(cf.~\cite{Haa75}).  As a consequence,
any von Neumann algebra $\cM$ has a standard form representation
 (cf.~\cite{Bla06}, \cite[\S 3.1]{BGN20}). 
\end{remark}

\begin{remark} \label{rem:dirsummeas} 
(a) Any $\sigma$-finite measure is semi-finite. If $X$ is a set, then 
the counting measure 
\[ \mu \: \bP(X) = 2^X\to  \N_0 \cup \{ \infty \}, \quad 
\mu(E)  := |E|  \] 
is semi-finite. It is $\sigma$-finite if and only if $X$ is countable. 

\nin (b) If $(X_j, \fS_j, \mu_j)_{j \in J}$ are semi-finite measure spaces, and 
we put 
\[ X := \coprod_{j \in J} X_j, \quad 
  \fS := \{ E \subeq X \: (\forall j \in J)\, E \cap X_j \in \fS_j\}\]
and
\[  \mu(E) := \sum_{j \in J} \mu_j(E \cap X_j),\] 
then $\fS$ is a $\sigma$-algebra on $X$, $\mu$ is a measure, and 
$(X,\fS,\mu)$ is a semi-finite measure space. 
Exercise~\ref{exer:semifin} shows that, conversely, 
up to sets of measure zero, any semi-finite measure 
space is such a direct sum of finite measure spaces. 
\end{remark}

\begin{small}
\nin {\bf Exercises for Appendix~\ref{app:wstar}}

\begin{exercise} (Direct sums of von Neumann algebras) 
Let $\cM_j \subeq B(\cH_j)$ be a family of von Neumann algebras, 
$\cH := \hat\bigoplus_{j \in J} \cH_j$ the Hilbert space direct sum 
of the $\cH_j$ and 
\[ \cM := \oline{\bigoplus}_{j \in J} \cM_j 
:= \Big\{ (M_j)_{j \in J} \in \prod_{j \in J} \cM_j \: 
\sup_{j \in J} \|M_j\| < \infty \Big\} \] 
the $\ell^\infty$-direct sum of the von Neumann algebras $\cM_j$ with the 
norm $\|M\| :=\sup_{j \in J} \|M_j\|.$ 
Show that $\cM$ can be realized in a natural way as a von Neumann 
algebra on~$\cH$.  
\end{exercise}

\begin{exercise}\label{exer:2.2.9} (Separability and $\sigma$-finiteness) 
Let $(X,\fS,\mu)$ be a measure space. 
Show that: 
\begin{description} 
\item[\rm(a)] If $f \in L^p(X,\mu)$, $1 \leq p < \infty$, then 
 the measurable subset $\{f \not=0\}$ of $X$ is $\sigma$-finite. 
\item[\rm(b)] If $\cH \subeq L^2(X,\mu)$ is a separable Hilbert subspace, 
then there exists a $\sigma$-finite measurable subset $X_0 \subeq X$ 
with the property that each $f \in \cH$ vanishes 
$\mu$-almost everywhere on $X_0^c = X \setminus X_0$. 
\end{description}
\end{exercise}

\begin{exercise} \label{exer:semifin} 
Let $(X,\fS,\mu)$ be a measure space. 
Show that there exist measurable subsets \break 
$X_j \subeq X$, $j \in J$, of finite measure such that 
\[ L^2(X,\mu) \cong \hat\bigoplus_{j \in J} L^2(X_j,\mu\res_{X_j}).\] 
Hint: Use Zorn's Lemma to find a maximal family $(X_j)_{j \in J}$ of measurable 
subsets of $X$ of finite positive measure,  
for which $\mu(X_j \cap X_k) = 0$ for $j \not=k$. 
Conclude that the corresponding subspaces $L^2(X_j,\mu\res_{X_j})$ of $L^2(X,\mu)$ 
are mutually orthogonal and that the intersection of their orthogonal 
complements is trivial. For the latter argument, use Exercise~\ref{exer:2.2.9}(a). 
\end{exercise}
\end{small}

\subsection{From unitary to antiunitary representations}
\label{app:extend} 

Antiunitary representations are somewhat harder to deal with
when it comes to direct integrals. In addition, their restriction to
$G$ may have more invariant subspaces. To deal with these issues in the
context of standard subspaces, the following lemma is a useful tool. 

\begin{lemma} {\rm(The antiunitary extension)} 
  \label{lem:3.4}
  Let $(U,\cH)$ be a unitary representation of $G$
  and write~$\oline\cH$ for the Hilbert space $\cH$, endowed with the
  opposite complex structure. Then the following assertions hold:
  \begin{enumerate}
  \item[\rm(a)] On $\tilde\cH := \cH \oplus \oline\cH$ we obtain by
    $\tilde U(g) := U(g) \oplus U(\tau_h(g))$ a unitary representation
    which extends by $\tilde U(\tau_h)(v,w) := \tilde J(v,w) := (w,v)$ to an
    antiunitary representation of $G_{\tau_h}$.
    The corresponding standard subspace $\tilde \sV := \sV(h, \tilde U)$
    coincides with the graph
    \begin{equation}
      \label{eq:tildesv}
      \tilde\sV = \Gamma(\Delta^{1/2}), 
    \end{equation}
    and its modular operator is $\tilde\Delta := \Delta \oplus \Delta^{-1}$.
  \item[\rm(b)] If $U$ extends to an antiunitary representation
    of $G_{\tau_h}$ by $J = U(\tau_h)$ on $\cH$, then the following assertions hold:
    \begin{enumerate}
    \item[\rm(1)] $\Phi \:  \cH^{\oplus 2} \to \tilde \cH, \Phi(v,w) = (v,Jw)$ 
    is a unitary intertwiner of $\tilde U$ and the antiunitary representation
    $U^\sharp$ of $G_{\tau_h}$ on $\cH^{\oplus 2}$, given by
    \[ U^\sharp\res_G = U^{\oplus 2} \quad \mbox{ and }  \quad
      U^\sharp(\tau_h)(v,w) := J^\sharp(v,w) :=  (Jw,Jv).\]
  \item[\rm(2)] The standard subspace $\sV^\sharp := \sV(h,U^\sharp)$
    coincides with the
    graph $\Gamma(T_\sV)$ of the Tomita operator
    $T_\sV = J\Delta^{1/2}$ of~$\sV$.
    \item[\rm(3)] 
    The antiunitary representation $\tilde U$ is equivalent to the antiunitary
    representation $U^{\oplus 2}$ of $G_{\tau_h}$ on~$\cH^{\oplus 2}$.
    \item[\rm(4)]  If $A \subeq G$ is a subset, then 
    $\tilde\sV_A$ is cyclic in $\tilde\cH$ if and only if
    $\sV_A$ is cyclic in $\cH$.     
    \end{enumerate}
  \end{enumerate}
\end{lemma}

\begin{proof} (\cite[Lemma~2.22]{MN24}) (a) The first assertion
  is a direct verification (cf.\ \cite[Lemma~2.10]{NO17}).
  Since
  \[ \tilde\Delta = e^{2\pi i \cdot \partial \tilde U(h)}
    = \Delta \oplus \Delta^{-1}, \]
  the description of the standard subspace $\tilde \sV = \Fix(\tilde J
  \tilde\Delta^{1/2})$ follows immediately.  

  \nin (b)  (1) Clearly, $\Phi$ is a complex linear isometry that intertwines 
  the antiunitary representation $\tilde U$ with
  the antiunitary representation $U^\sharp$.

  (2)  As $\Delta^\sharp = \Phi^{-1}\tilde\Delta \Phi = \Delta \oplus \Delta$,
  the relation
  \[ (v,w) = J^\sharp(\Delta^\sharp)^{1/2}(v,w)
    = (J \Delta^{1/2} w, J \Delta^{1/2} v) = (T_\sV w, T_\sV v) \]
  is equivalent to $w = T_\sV v$. Hence $\sV^\sharp = \Gamma(T_\sV)$.

(3) As the restrictions of $U^{\oplus 2}$ and $U^\sharp$ to $G$ coincide,
  \cite[Thm.~2.11]{NO17} implies their equivalence as antiunitary
  representations. However, in the present concrete case, it is easy to
  see an intertwining operator. The matrix
  \[  A := \frac{1}{2} \pmat{
      (1+i)\bone & (1-i)\bone \\ 
      (1-i)\bone & (1+i)\bone}
  \quad \mbox{ with } \quad A^2 = \pmat{ \0 & \bone \\ \bone & \0} \]
  defines a unitary operator on $\cH^{\oplus 2}$, commuting with $U^\sharp(G)$.
  It satisfies $J^{\oplus 2} A J^{\oplus 2} = A^* = A^{-1}$, so that
  \[ A J^{\oplus 2} A^{-1} = A^2 J^{\oplus 2} = J^\sharp.\] 

(4) If $U\res_G$ extends to an antiunitary representation~$U$ of
  $G_{\tau_h}$ on $\cH$, then (3)  implies that $\tilde U \cong U^{\oplus 2}$,
  and any equivalence $\Psi \: (\tilde U,\tilde \cH)  \to (U^{\oplus 2}, \cH^{\oplus 2})$ maps $\tilde \sV_A$ to
  $(\sV \oplus \sV)_A = \sV_A \oplus \sV_A$. 
  Therefore $\tilde \sV_A$ is cyclic if and only if
  $\sV_A$ is cyclic in~$\cH$.
\end{proof}

 The following definition extends the classical type of irreducible
  complex representations to the case where the involution on $G$ is
  non-trivial. For a unitary representation
    $(U,\cH)$, we write $(\oline U, \oline\cH)$ for 
    the unitary representation on the complex conjugate space
  $\oline\cH$ by $\oline U(g) = U(g)$. 
  We observe that, for an antiunitary representation
  $(U,\cH)$ of $G_{\tau_h}$, its {\it commutant}
  \index{commutant $S'$ of $S$ \scheiding}
  \begin{align*}
 U(G_{\tau_h})'
    &= \{ A \in B(\cH) \: (\forall g \in G_{\tau_h})
    \, A U(g) = U(g) A \} \\
    & = \{ A \in U(G)' \:  U(\tau_h) A = AU(\tau_h) \}
  \end{align*}
is only a real subalgebra of $B(\cH)$ because $U(\tau_h)$ is antilinear.

  \begin{definition} \label{def:types} (\cite[Def.~2.12]{NO17})
    Let $(U, \cH)$ be an irreducible unitary representation
    of~$G$. 
We say that $U$ is (with respect to $\tau_h$), of 
\index{representation!of real type \scheiding}
\index{representation!of complex type \scheiding}
\index{representation!of quaternionic type \scheiding}
\begin{itemize}
\item {\it real type} if there exists an antiunitary 
  involution $J$ on $\cH$ such that {$U^\sharp(\tau_h) := J$
  extends $U$ to an antiunitary representation $U^\sharp$
  of $G_{\tau_h}$ on~$\cH$,
  i.e., $J U(g) J = U(\tau_h(g))$ for $g \in G$. 
  Then the commutant of $U^\sharp(G_{\tau_h})$ is~$\R$.}
\item {\it quaternionic type} if there exists an antiunitary 
  complex structure $I$ on $\cH$ satisfying $I U(g) I^{-1}
  = U(\tau_h(g))$ for $g \in G$. Then $\oline U \circ \tau_h \cong U$, 
  $U$ has no extension  on the same space, and
  the antiunitary representation
  $(\tilde U, \tilde \cH)$ of $G_{\tau_h}$ with $\tilde U\res_G \cong
  U \oplus (\oline U\circ \tau_h)$ is irreducible with commutant~$\bH$.
\item {\it complex type} if 
  $\oline U \circ \tau_h \not\cong U$.
    This is equivalent to the non-existence of $V\in\AU(\cH)$
    such that $U(\tau_h(g)) = V U(g) V^{-1}$ for all $g \in G$, i.e.,
    to the non-existence of an antiunitary extension of $U$ to
    $G_{\tau_h}$ on $\cH$. 
  Then $(\tilde U,\tilde \cH)$ is an irreducible antiunitary 
  representation of $G_{\tau_h}$ with commutant~$\C$. 
\end{itemize}
  \end{definition}

      \begin{remark} \label{rem:tenspro-anti} (Antiunitary tensor products) 
    Let $G = G_1 \times G_2$ be a product of type~I groups  
  and $\tau$ an involutive automorphism of $G$ preserving both factors, i.e., 
$\tau = \tau_1 \times \tau_2.$ 
  We want to describe irreducible antiunitary representations
  $(U,\cH)$  of the group $G_\tau = G \rtimes \{\id_G, \tau\}$
  using \cite[Thm.~2.11(d)]{NO17}.
  
  \nin (a) The first possibility is that $U\res_G$ is irreducible, so
  that $U(G)' \cong \R$. Then
\[ (U\res_G,\cH) \cong (U_1,\cH_1) \otimes (U_2, \cH_2) \]
with irreducible unitary representations $(U_j, \cH_j)$ of $G_j$
both extending to antiunitary representations $U_j^\sharp$ of $G_j$.
{Hence both $U_1$ and $U_2$ are of real type. }

\nin (b) The second possibility is that $U\res_G$ is reducible
with $U(G)' \cong \C$ or $\bH$, so that
{\[ U\res_G \cong V \oplus (\oline V \circ \tau), \]}
where $(V,\cK)$ is an irreducible unitary representation of $G$
of complex or quaternionic type.
Now $V = U_1 \otimes U_2$, and thus
  \[ \cH \cong (\cH_1 \otimes \cH_2) \oplus
  (\oline\cH_1 \otimes \oline\cH_2), \quad 
  U\res_G \cong (U_1 \otimes U_2)  \oplus
  (\oline{U_1} \circ \tau_1 \otimes 
  \oline{U_2} \circ \tau_2).\]
If $U_j$ is of complex type, then {$\oline{U_j} \circ \tau_j \not\cong U_j$}
implies that $V$ is of complex type. If both $U_1$ and $U_2$ are
of quaternionic type, then {$\oline{U_j} \circ \tau_j \cong U_j$}
for $j = 1,2$ implies {$\oline V \circ \tau \cong V$,} so that
$V$ is of quaternionic type.
\end{remark}

\subsection{Smooth and analytic vectors}
\label{app:c}

In this appendix we collect some material on distribution vectors
and hyperfunction vectors of unitary representations
$U \: G \to \U(\cH)$. 

\subsubsection{The integrated representation}
\label{app:c1}

\begin{definition} Let $G$ be a Lie group. 
We fix a left-invariant Haar measure $\mu_G$ on $G$ 
and we often write $dg$ for $d\mu_G(g)$. 
This measure defines on $L^1(G) := L^1(G,\mu_G)$ the structure of a 
Banach-$*$ algebra by the {\it convolution product} and
\begin{equation}
  \label{eq:convol}
(\phi*\psi)(x) =\int_G \phi(g)\psi(g^{-1}x)\, d\mu_G (g),  
\quad \mbox{ and } \quad \phi^*(g) = \overline{\phi (g^{-1})}\Delta_G (g)^{-1}
\end{equation}
is the involution, where 
$\Delta_G : G\to \R_+$ is the {\it modular function} determined by
\begin{align}
  \label{eq:modfunc}
 \int_G \phi(y)\, d\mu_G(y) 
&=\int_G \phi(y^{-1})\Delta_G(y)^{-1}\, d\mu_G(y) \quad \mbox{ and } \\ 
\Delta_G(x)\int_G \phi(yx)\, d\mu_G(y)&=\int_G \phi(y)\, d\mu_G(y) \quad 
\mbox{ for } \quad \phi\in C_c (G).
\end{align}
We put $\varphi^\vee(g)=\varphi (g^{-1})\cdot\Delta_G(g)^{-1}$, so that
\begin{equation}
  \label{eq:check}
\int_G \phi(g)\, d\mu_G(g) 
= \int_G \phi^\vee(g)\, d\mu_G(g).
\end{equation}
The formulas above show that we have two isometric actions of $G$ on 
$L^1(G)$, given by 
\begin{equation}
  \label{eq:left-right-action}
(\lambda_g f)(x) = f(g^{-1}x) \quad \mbox{ and }\quad 
(\rho_g f)(x) = f(xg) \Delta_G(g).
\end{equation}
Note that 
\begin{equation}
  \label{eq:veecov}
(\lambda_g f)^* = \rho_g f^*
\quad \mbox{ and } \quad (\lambda_g f)^\vee = \rho_g f^\vee.   
\end{equation}
\end{definition}

Now let $(U,\cH)$ be a continuous unitary representation of the Lie group~$G$, 
i.e., a homomorphism $U : G\to \U(\cH), g \mapsto U(g)$ 
such that, for each $\eta \in\cH$, the  orbit map
$U^\eta (g)=U(g)\eta$ is continuous. 
For $\phi\in L^1(G)$ the operator-valued integral 
\[ U(\phi) := \int_G \phi(g) U(g)\, dg \] 
exists and is uniquely determined by 
\begin{equation}
  \label{eq:l1-est}
\la \eta, U(\phi) \zeta \ra =\int_G \phi(g)\la \eta, U(g)\zeta\ra\, dg \quad 
\mbox{ for } \quad \eta,\zeta \in \cH.
\end{equation}
Then $\|U(\phi)\|\le \|\phi\|_1$, and the 
so-obtained continuous linear map $L^1(G) \to B(\cH)$ 
is a representation of the Banach-$*$ algebra $L^1(G)$, i.e., 
$U(\phi*\psi)=U(\phi)U(\psi)$ and $U(\phi^*)=U(\phi)^*.$ 
We also note that, for $g\in G$ and $\phi \in L^1(G)$ 
\begin{equation}
  \label{eq:covl1}
U(g) U(\phi) = U(\lambda_g \phi) \quad \mbox{ and } \quad 
U(\phi)U(g)  = U(\rho_g^{-1} \phi).
\end{equation}
For $\phi_g(x) := \phi(xg)$, we then have 
$\phi_g = \Delta_G(g)^{-1} \rho_g \phi$ by \eqref{eq:left-right-action}, 
and thus by \eqref{eq:covl1} 
\begin{equation}
  \label{eq:rightrel}
U(\phi_g) = \Delta_G(g)^{-1} U(\phi) U(g^{-1}) \quad \mbox{ for } \quad 
 g\in G.
\end{equation}

\subsubsection{The space of smooth vectors and its dual}
\label{subsec:app1}

\index{vector!smooth, $\cH^{\infty}$  \scheiding }
\index{derived representation $\dd U$ \scheiding }

A {\it smooth vector} is an element $\eta\in\cH$ for which the orbit map 
\[ U^\eta : G\to \cH, \quad g \mapsto U(g)\eta \]  
is smooth. We write~$\cH^{\infty} = \cH^\infty(U)$ for the space 
of smooth vectors. It carries the {\it derived representation} 
$\dd U $ of the Lie algebra $\fg$ given by
\begin{equation}
  \label{eq:derrep}
\dd U(x)\eta =\lim_{t\to 0}\frac{U(\exp t x)\eta -\eta}{t}.
\end{equation}

  For $x \in \g$, we write
  $\partial U(x)$ for the infinitesimal generator of the one-parameter
  group $U(\exp tx)$, so that $U(\exp tx) = e^{t\partial U(x)}$. 
  As $\cH^\infty$ is dense and $U(G)$-invariant,
  $\partial U(x)$ is the closure of the operator $\dd U(x)$
  (\cite[Thm.~VIII.10]{RS73}).

We extend the representation $\dd U$ to a homomorphism 
$\dd U \:  \cU(\g) \to \End(\cH^\infty),$ 
where $\cU(\g)$ is the complex enveloping algebra of $\g$. This algebra 
carries an involution $D \mapsto D^*$ determined uniquely by 
$x^* = -x$ for $x \in \g$.
For $D \in \cU(\g)$, we obtain a seminorm on $\cH^\infty$ by 
\[p_D(\eta )=\|\dd U(D)\eta\|\quad \mbox{ for } \quad \eta \in \cH^\infty.\] 
These seminorms define a topology on $\cH^\infty$ which turns the injection 
\begin{equation}
  \label{eq:tophinfty}
 \eta \: {\cal H}^\infty \to {\cal H}^{{\cal U}(\g_\C)}, \quad 
\xi \mapsto (\dd U(D)\xi)_{D \in {\cal U}(\g_\C)}
\end{equation}
into a topological embedding, where the right-hand side carries the product 
topology (cf.\ \cite[3.19]{Mag92}).  It turns $\cH^\infty$ 
into a complete locally convex space 
for which the linear operators $\dd U(D)$, $D \in \cU(\g)$, are continuous. 
Since $\cU(\g)$ has a countable basis, 
countably many such seminorms already determine the topology, so that 
$\cH^\infty$ is metrizable. As it is also complete, it is a Fr\'echet space.
We also observe that 
the inclusion $\cH^\infty\hookrightarrow \cH$ is continuous.

The space $\cH^\infty$ of smooth vectors is $G$-invariant 
and we denote the corresponding representation by~$U^\infty$.
We thus obtain a smooth action of 
$G$ on this Fr\'echet space (\cite{Ne10}).
We have  the intertwining relation 
\begin{equation}
  \label{eq:dUAdg}
 \dd U(\Ad (g)x)= U(g) \dd U(x) U(g)^{-1} \quad \mbox{ for } \quad 
 g \in G, x \in \g.
\end{equation}
If $\varphi \in C_c^\infty (G)$ and $\xi \in\cH$, then 
$U(\varphi) \xi \in \cH^\infty$, and differentiation under the integral sign 
shows that  
\begin{equation}
  \label{eq:derrep2}
\dd U(x) U(\varphi) \xi :=U(-x^R \varphi) \xi,  
\quad \mbox{ where } \quad 
(x^R\varphi)(g) =\frac{d}{dt}\Big|_{t=0} \phi((\exp tx) g). 
\end{equation}

A sequence $(\varphi_n)_{n \in \N}$ in $C^\infty_c(G)$ is called a 
{\it $\delta$-sequence} if $\int_G \varphi_n(g)\, dg = 1$ for every 
$n \in \N$ and, for every $e$-neighborhood $U \subeq G$, we have 
$\supp(\varphi_n) \subeq U$ if $n$ is sufficiently large. 
If $(\varphi_n)_{n \in \N}$ is a 
{$\delta$-sequence}, then $U(\varphi_n)\xi \to \xi$, so that $\cH^\infty$
is dense in~$\cH$.

We write $\cH^{-\infty}$ for the space 
of continuous antilinear functionals on $\cH^\infty$. 
Its elements are called \textit{distribution vectors}.
\index{vector!distribution (of unitary rep.), $\cH^{-\infty}$ \scheiding } 
The group $G$, $\cU(\g)$ and $C^\infty_c(G)$ act on $\eta \in \cH^{-\infty}$ by
\begin{itemize}
\item $(U^{-\infty}(g)\eta ) (\xi ):= \eta  (U(g^{-1})\xi)$, $g \in G, 
\xi \in \cH^\infty$. \\ If $U \: G \to \AU(\cH)$ is an antiunitary 
representation and $U(g)$ is antiunitary, then we have to modify 
this definition slightly by 
$(U^{-\infty}(g)\eta) (\xi ):= \oline{\eta(U(g^{-1})\xi)}$. 
\item $(\dd U^{-\infty}(D) \eta ) (\xi ):= \eta(\dd U(D^*) \xi)$, $D \in \cU(\g), 
\xi \in \cH^\infty$. 
\item $U^{-\infty}(\varphi) \eta =\eta \circ U^\infty(\varphi^*)$, 
$\varphi \in C_c^\infty (G).$ 
\end{itemize} 
We have natural $G$-equivariant linear embeddings 
\begin{equation}
  \label{eq:embindistr}
\cH^\infty \into \cH
\mapright{\xi \mapsto  \la \cdot, \xi \ra} \cH^{-\infty}. 
\end{equation}

It is an important feature of \eqref{eq:embindistr} 
that the representation of $\cU(\g)$ on 
$\cH^{-\infty}$ provides an embedding of the whole Hilbert space $\cH$ 
into a larger space on which the Lie algebra acts. The following 
lemma shows that $\cH^\infty$ is the maximal $\g$-invariant subspace 
of $\cH \subeq \cH^{-\infty}$ and that the subspace $\cH$ generates 
$\cH^{-\infty}$ as a $\g$-module. 

\begin{lemma} \label{lem:charsmooth} 
The following assertions hold: 
  \begin{description}
  \item[\rm(a)] $\cH^\infty = \{ \xi \in \cH \subeq \cH^{-\infty} \: (\forall D \in \cU(\g)) 
\ \dd U^{-\infty}(D)\xi \in \cH\}.$ 
  \item[\rm(b)] ${\cal H}^{-\infty} = \Spann \big(\dd U^{-\infty}({\cal U}(\g)){\cal H}\big)$. 
  \end{description}
\end{lemma}

\begin{proof} (a) This follows by combining
  \cite[Prop.~A.1]{Oeh21}, asserting that 
\[ \cD(\partial U(x)) = \{ \xi \in \cH \: \dd U^{-\infty}(x) \xi \in \cH \}, \] 
with the fact that 
\[ \cH^\infty = \bigcap \{ \cD(\partial U(x_1) \cdots \partial U(x_n))  \: 
n \in \N, x_1, \ldots, x_n\in \g\} \] 
(\cite[Lemma~3.4]{Ne10}). 

\nin (b) Let $\eta \in {\cal H}^{-\infty}$ and  
consider $\cH^\infty$ as a subspace of the topological product
${\cal H}^{{\cal U}(\g)}$ as in \eqref{eq:tophinfty}. 
By the Hahn--Banach Extension Theorem, 
$\eta$ extends to a continuous antilinear functional 
$\tilde\eta$ on ${\cal H}^{{\cal U}(\g)}$. Since the dual of a direct product 
is the direct sum of the dual spaces, there exist 
$D_1, \ldots, D_n \in\cU(\g)$ and $\xi_1, \ldots, \xi_n \in \cH$, such that 
\[ \eta(\xi) 
= \sum_{j = 1}^n \la \xi_j,  \dd U(D_j) \xi \ra 
= \sum_{j = 1}^n \la  \dd U^{-\infty}(D_j^*) \xi_j, \xi  \ra 
\quad \mbox{ for } \quad \xi \in \cH^\infty,\] 
which means that 
$\eta =  \sum_{j = 1}^n \dd U^{-\infty}(D_j^*) \xi_j.$ 
\end{proof}

For each $\phi \in C^\infty_c(G)$, the map 
$U(\phi) \: \cH \to \cH^{\infty}$ 
is continuous, so that its adjoint defines a weak-$*$ continuous map 
$U^{-\infty}(\phi^*) \: \cH^{-\infty} \to \cH$. We actually have 
$U^{-\infty}(\phi)\cH^{-\infty} \subeq \cH^\infty$ 
as a consequence of the 
Dixmier--Malliavin Theorem \cite[Thm.~3.1]{DM78}, 
which asserts that every $\phi \in C^\infty_c(G)$ can be
written as a finite sum of functions of the form
$\phi_1 * \phi_2$ with $\phi_j \in C^\infty_c(G)$.

\subsubsection{The space of analytic vectors and its dual} 
\label{app:anavec}
\index{vector!analytic (of unitary rep.), $\cH^\omega$ \scheiding } 
\index{vector!hyperfunction (of unitary rep.), $\cH^{-\omega}$ \scheiding } 

In this subsection, we briefly discuss the space of analytic
vectors of a unitary representation of a Lie group.
Let $(U,\cH)$ be a unitary representation of
the connected real Lie group $G$. We write
\[ \cH^\omega = \cH^\omega(U)\subeq \cH \]
for the space of {\it analytic vectors}, i.e.,
those $\xi\in \cH$ for which the orbit map
$U^\xi \colon G \to \cH, g \mapsto U(g)\xi$,  is analytic.

To endow $\cH^\omega$ with a locally convex topology,
we specify subspaces $\cH^\omega_{V}$ by open convex  {$0$-neighborhoods}
$V \subeq \g$ as follows. Let $\eta_G \colon G \to G_\C$ denote the universal
complexification of $G$ and assume that $\eta_G$ has discrete kernel
(this is always the case if $G$ is semisimple or $1$-connected). 
We assume that $V$ is so small that the map
\begin{equation}
  \label{eq:eta-g-v}
 \eta_{G,V} \colon G_V := G \times V
 \to G_\C, \quad (g,x) \mapsto \eta_G(g) \exp(ix)
\end{equation}
is a covering. Then we endow $G_V$ with the unique complex manifold
structure for which $\eta_{G,V}$ is holomorphic.

We now write $\cH^\omega_V$ for the set of those analytic vectors
$\xi$ for which the orbit map $U^\xi\colon G \to \cH$ extends to a
holomorphic map
\[ U^\xi_V \colon G_V\to \cH.\]
As any such extension is $G$-equivariant by 
uniqueness of analytic continuation, it must have the form
\begin{equation}
  \label{eq:orb-map}
 U^\xi_V(g,x) = U(g) e^{\ie \cdot\partial U(x)} \xi \quad \mbox{ for }  \quad
 g \in G, x \in V,
\end{equation}
so that
$\cH^\omega_V \subeq \bigcap_{x \in V} \cD(e^{\ie \cdot \partial U(x)}).$ 

The following lemma shows that we have equality.

\begin{lemma} \label{lem:charh-om-v}
If $V \subeq \g$ is an open convex $0$-neighborhood
  for which \eqref{eq:eta-g-v} is a covering, then 
$\cH^\omega_V = \bigcap_{x \in V} \cD(e^{i\cdot \partial U(x)}).$ 
\end{lemma}

\begin{proof} (\cite[Lemma~1]{FNO25a}) It remains to show that each
  $\xi \in \bigcap_{x \in V} \cD(e^{i \cdot \partial U(x)})$ is contained
  in $\cH^\omega_V$. For that, we first observe that the
  holomorphy of the functions $z \mapsto e^{iz \partial U(x)}v$ on a
  neighborhood of the closed unit disc in $\C$ implies that the
  $\cH$-valued power series
  \[ f_\xi(x) := \sum_{n = 0}^\infty \frac{i^n}{n!} \partial U(x)^n \xi \]
  converges for each $x \in V$.
  Further, \cite[Thm.~1.1]{Go69} implies
  that $\xi \in \cH^\infty$, so that the functions
  $x \mapsto \partial U(x)^n \xi = \dd U(x)^n \xi$
  are homogeneous $\cH$-valued polynomials (cf.\ \cite{BS71}). Thus
  \cite[Thm.~5.2]{BS71} shows that the above series
  defines an analytic function   $f_\xi \colon V\to \cH$.  
  It follows in particular that $\xi$ is an analytic vector,
  and the map
  \[ U^\xi_V \colon  G_V \to \cH, \quad
  (g, x) \mapsto   U^\xi (g,x):=  U(g) e^{i \partial U(x)}\xi \]
is defined. It is clearly equivariant.
We claim that it is holomorphic. As it is locally bounded,
it suffices to show that, for each $\eta \in \cH^\omega$, the function
\[ f \colon G_V \to \C, \quad f(g,x) := \la \eta, U^\xi(g,x) \ra   \]
is holomorphic (\cite[Cor.~A.III.3]{Ne99}). We have 
\[ f(g,x) = \la U(g)^{-1}\eta, e^{i \partial U(x)}\xi \ra,  \]
and the orbit map of $\eta$ is analytic, 
therefore $f$ is real analytic. Therefore it suffices to show that it is
holomorphic on some $0$-neighborhood. This follows from the
fact that it is $G$-equivariant and coincides on some $0$-neighborhood
with the local holomorphic extension of the orbit map of~$\xi$.
Here we use that, for $x,y \in \g$ sufficiently small,
the holomorphic extension $U^\xi$ of the $\xi$-orbit map satisfies 
\[ U^\xi(\exp(x * iy)) = U(\exp x) U^\xi(\exp i y)
  = U(\exp x) f_\xi(y) = U^\xi_V(\exp x, y),\]
where $a * b = a + b + \frac{1}{2}[a,b] + \cdots$ denotes the
Baker--Campbell--Hausdorff series. 
\end{proof}

We topologize the space
$\cH^\omega_V$ by identifying it with 
$\cO(G_V, \cH)^G$, the Fr\'echet space of $G$-equivariant holomorphic maps
$F \colon G_V \to \cH$, endowed with the Fr\'echet topology of
  uniform convergence on compact subsets. Now
$\cH^\omega = \bigcup_V \cH^\omega_V$, 
and we topologize $\cH^\omega$ as the locally convex direct limit
of the Fr\'echet spaces $\cH^\omega_V$ (cf.\ \cite{GN25}, \cite{Tr67}).
If the universal complexification $\eta_G \colon  G \to  G_\C$
  is injective, we thus obtain the same topology as in \cite{GKS11}. Note that,
for any monotone
basis  $(V_n)_{n \in \N}$ of convex $0$-neighborhoods in $\g$, we
then have
\[ \cH^\omega \cong \indlim \cH^\omega_{V_n},\]
so that $\cH^\omega$ is a countable locally convex limit of
Fr\'echet spaces. As the evaluation maps
\[ \cO(G_V, \cH)^G \to \cH, \quad F \mapsto F(e,0) \] 
are continuous, the inclusion
$\iota \colon \cH^\omega \to \cH$ 
is continuous. 

We write $\cH^{-\omega}$ for the space  of continuous antilinear functionals 
$\eta\colon \cH^\omega \to \C$ (called {\it hyperfunction vectors})
and
\[ \la \cdot, \cdot \ra \colon \cH^\omega \times \cH^{-\omega} \to \C \]
for the natural sesquilinear pairing that is linear in the second argument.
We endow $\cH^{-\omega}$ with the weak-$*$ topology. 
We then have natural continuous inclusions
\[ \cH^\omega \into \cH \into \cH^{-\omega}.\] 

Our specification of the topology on $\cH^\omega$
differs from the one 
\cite{GKS11} because we do not want to assume that
the universal complexification $\eta_G \colon G \to G_\C$ is injective, 
but both constructions define the same topology.
Moreover,  the arguments in \cite{GKS11} apply with minor changes to
general Lie groups.

We actually have the following chain of complex linear embeddings 
\begin{equation}
  \label{eq:incl4}
\cH^{\omega} \subeq \cH^\infty \subeq \cH \subeq \cH^{-\infty}
\subeq \cH^{-\omega},
\end{equation}
where all inclusions are continuous
and $G$ acts on all spaces 
by representations denoted $U^\omega$, $U^\infty$, $U$, $U^{-\infty}$,
and $U^{-\omega}$, respectively.
These representations can be integrated to the
convolution algebra $C^\infty_c(G) := C^\infty_c(G,\C)$ of
test functions (cf.~\eqref{eq:convol}), for instance
\begin{equation}
  \label{eq:u-infty-phi}
  U^{-\infty}(\phi) := \int_G \phi(g)U^{-\infty}(g)\, dg,
\end{equation}
where $dg$ stands for a left Haar measure on $G$.

\subsection{Direct integral techniques}
\label{app:D}

Here we collect some material from \cite{MN24} and \cite{BN25}.
We refer to \cite{BR87} for the basics on direct integrals;
see also \cite{DD63}.

\index{real subspace!decomposable (in direct integral) \scheiding} 

Let $\cH = \int^\oplus_X \cH_m \, d\mu(m)$ be a direct
integral of Hilbert spaces on a standard measure space
$(X,\mu)$. We call a closed real subspace
$\sH \subeq \cH$ {\it decomposable} if it is of the form
\begin{equation}
  \label{eq:sh-int}
  \sH = \int_X^\oplus \sH_m\, d\mu(m),
\end{equation}
where $(\sH_m)_{m \in X}$ is a measurable field of closed real subspaces.
Now let $(\sH^k)_{k \in K}$ be an at most countable family of decomposable
real subspaces. Then we have (\cite[Lemma~B.3]{MT19}): 
\begin{description}
\item[\rm(DI1)] $\sH' = \int_X^\oplus \sH_m'\, d\mu(m)$. 
\item[\rm(DI2)] $\bigcap_{k \in K} \sH^k = \int_X^\oplus
  \bigcap_{k \in K} \sH_m^k  \, d\mu(m)$ for any at most countable set~$K$. 
\item[\rm(DI3)] $\oline{\sum_k \sH^k} = \int_X^\oplus
  \oline{\sum_k \sH_m^k}  \, d\mu(m).$ 
\end{description}

\begin{lemma} \label{lem:di1}
  The subspace $\sH$ as in \eqref{eq:sh-int} is cyclic/separating/standard
  if and only if $\mu$-almost all $\sH_m$ have this property. 
\end{lemma}

\begin{proof} (a) First we deal with the separating property.  By
  (DI2) we have
  \[ \sH \cap i\sH = \int_X^\oplus (\sH_m \cap i \sH_m)\, d\mu(m),\]
  and this space is trivial if and only if $\mu$-almost all spaces
  $\sH_m \cap i \sH_m$ are trivial, which means that
  $\sH_m$ is separating.

  \nin (b) The subspace $\sH$ is cyclic if and only if $\sH'$ is separating.
  By (DI1) and (a) this means that $\mu$-almost all $\sH_m'$ are separating,
  i.e., that $\sH_m$ is cyclic.

  \nin (c) By (a) and (b) $\sH$ is standard if and only if
  $\mu$-almost all $\sH_m$ are cyclic and separating, i.e., standard.
\end{proof}

\begin{lemma} \label{lem:dirint2} For a countable family
  $(\sH^k)_{k \in K}$ of decomposable cyclic closed real subspaces, the
  intersection $\sV := \bigcap_{k \in K} \sH^k$
  is cyclic if and only if, 
  for $\mu$-almost every $m \in X$, the subspace
$\sV_m := \bigcap_{k \in K} \sH_m^k$ is cyclic. 
\end{lemma}

\begin{proof} By (DI2), we have $\sV = \int_X^\oplus \sV_m\, d\mu(m)$,
  so that the assertion follows from Lemma~\ref{lem:di1}.   
\end{proof}

For a direct integral
\[ (U,\cH) = \int_X^\oplus (U_m, \cH_m)\, d\mu(m) \]
of antiunitary representations of $G_{\tau_h}$, the canonical
standard subspace $\sV = \sV(h, U)\subeq \cH$ from \eqref{eq:def-V(h,U)}
is specified by the decomposable operator $J \Delta^{1/2}
= U(\tau_h) e^{\pi i\, \partial U(h)}$, hence decomposable: 
\begin{equation}
  \label{eq:v-dirint}
  \sV = \int_X^\oplus \sV_m\, d\mu(m).
\end{equation}

\begin{lemma} \label{lem:g-inter}
  Assume that $G$ has at most countably many components.
  For any subset $A \subeq G$ and a real subspace $\sH \subeq \cH$, we put
  \begin{equation}
    \label{eq:VA}
    \sH_A := \bigcap_{g \in A} U(g)\sH.
  \end{equation}
  Then the following assertions hold:
  \begin{description}
  \item[\rm(a)] If $\sH$ is decomposable, then
    $\sH_A = \int_X^\oplus \sH_{m,A}\, d\mu(m)$. 
  \item[\rm(b)] $\sH_A$ is cyclic if and only if $\mu$-almost all
    $\sH_{m,A}$ are cyclic. 
  \end{description}
\end{lemma}

\begin{proof} (a) As $G$ has at most countably many components,
  it carries a separable metric, so that there exists a countable
  subset $B \subeq A$ which is dense in $A$.
  For $\xi \in \cH$, we have
  \[ \xi \in \sH_A \quad \mbox{ if and only if } \quad
    U(A)^{-1}\xi \subeq \sH.\]
  Now the closedness of $\sH$ and the density of $B$ in $A$ show that
  this is equivalent to $U(B)^{-1}\xi \subeq \sH$, i.e., to
  $\xi \in \sH_B$. This shows that $\sH_A = \sH_B$.
  We likewise obtain $\sH_{m,A} = \sH_{m,B}$ for every $m \in X$.
  Hence the assertion follows by applying (DI2) to the real subspace
  $\sH_B = \sH_A$.

\nin (b) follows from (a) and Lemma~\ref{lem:di1}.   
\end{proof}

\begin{lemma} \label{lem:denseunion}Let $\cH=\int_X^\oplus\cH_xd\mu(x)$, a  direct integral von Neumann algebra $\cA=\int_X^\oplus\cA_xd\mu(x)$ and a strongly continuous, unitary, direct integral representation of a
  Lie group $G$ with countably many connected components,
  $(U,\cH)=\int_X^\oplus(U_x,\cH_x)d\mu(x)$.
  Then, for any subset $N\subset G$, we have
  \[ \bigcap_{g\in N}\cA_g=\int_X^\oplus \bigcap_{g\in N} (\cA_g)_xd\mu(x)
    \quad \mbox{ where } \quad \cA_g=U(g)\cA U(g)^*.\] 
\end{lemma}
\begin{proof}

As $G$ has at most countably many components, 
  it carries a separable metric. Hence there exists a countable
  subset $N_0 \subeq N$ which is dense in $N$. 
  For $A \in B(\cH)$, the map
  \[ F \: G \to B(\cH), \quad F(g) = U(g) A U(g)^*,\]
  is weak operator continuous,
  so that the set of all $g \in G$ with
 $F(g) \in \bigcap_{g \in N_0} \cA_g$ 
  is a closed subset, hence contains $N$. We conclude that
  \[ \bigcap_{g \in N_0} \cA_g = \bigcap_{g \in N} \cA_g. \]
  We likewise obtain for every $x \in X$ the relation
  \[ \bigcap_{g \in N_0} \cA_{x,g}
    = \bigcap_{g \in N} \cA_{x,g} 
    \quad \mbox{ for } \quad \cA_{x,g} = U_x(g) \cA_x U_x(g)^*. \]
  From \cite[Prop.~4.4.6(b)]{BR87} we thus obtain
  \[ 
    \bigcap_{g \in N} \cA_g 
    =  \bigcap_{g \in N_0} \cA_g  
    = \int_X^\oplus  \bigcap_{g \in N_0} \cA_{x,g} \,
    d\mu(x) 
    = \int_X^\oplus  \bigcap_{g \in N} \cA_{x,g} \,
    d\mu(x).\]
  Finally, we observe that, for every $g \in G$
  \[ \cA_g
    = \int_X^\oplus (\cA_g)_x \, d\mu(x)
    = \int_X^\oplus \cA_{x,g}\, d\mu(x) \]
  follows by the uniqueness of the direct integral decomposition.
\end{proof}

\subsection{Some facts on convex cones}
\label{app:E}
  
\begin{lemma} \label{lem:coneint} {\rm(\cite[Lemma~B.1]{MNO23})}
Let $E$ be a finite-dimensional real vector space, 
$C \subeq E$ a closed convex cone and $E_1 \subeq E$ a linear subspace. 
If the interior $C^\circ$ of $C$ intersects $E_1$, then 
$C^\circ \cap E_1$ coincides with the relative interior $C_1^\circ$ 
of the cone $C_1 := C \cap E_1$ in~$E_1$.
\end{lemma}

\begin{lemma} \label{lem:ext-eigenvalue}
  Let $V$ be a finite-dimensional real vector space,
  $A \in \End(V)$ diagonalizable, and let $C \subeq V$ be a
  closed convex cone invariant under $e^{\R A}$.
 Let  $\lambda_{\rm min}$ and  $\lambda_{\rm max}$ be the minimal/maximal
 eigenvalues of $A$. For an eigenvalue $\lambda$ of $A$ we write
 $V_\lambda(A)$ for the corresponding eigenspace and
 $p_\lambda \:  V \to V_\lambda(A)$ for the projection along all other
 eigenspaces. Then
 \[ p_{\lambda_{\rm min}}(C) = C \cap V_{\lambda_{\rm min}}(A)
   \quad \mbox{ and }  \quad 
   p_{\lambda_{\rm max}}(C) = C \cap V_{\lambda_{\rm max}}(A).\]

 If $A$ has only two eigenvalues, it follows that
$C = p_{\lambda_{\rm min}}(C) \oplus p_{\lambda_{\rm max}}(C).$ 
\end{lemma}

\begin{proof} Since we can replace $A$ by $-A$, it suffices to verify the
 second assertion. So let $v \in C$ and write it as a sum
  $v = \sum_\lambda v_\lambda$ of $A$-eigenvectors. Then
  \[ v_{\lambda_{\rm max}} = \lim_{t \to \infty} e^{-t \lambda_{\rm max}}
    e^{tA} v \in C \]
  implies that $p_{\lambda_{\rm max}}(C) \subeq C \cap V_{\lambda_{\rm max}}(A),$
  and the other inclusion is trivial.   
\end{proof}

\begin{lemma} \label{lem:cone-endos}
  Let $E$ be a finite-dimensional real vector space
  and $C \subeq E$ be a closed convex cone. In
  the affine group $G := \Aff(E) \cong E \rtimes \GL(E)$, we then have 
  \begin{equation}
    \label{eq:sc=}
 S_C := \{ g \in G \: gC \subeq C \}
 = C \rtimes \{ g \in \GL(E) \: gC \subeq C\}.
  \end{equation}
  If $C$ has interior points, then $S_{C^\circ} = S_C$
  and $C = \oline{C^\circ}$.
\end{lemma}

\begin{proof} We write $g = (b,a)$ with $gx = b + ax$.
  Then $g.C \subeq C$ implies   $b = g.0 \in C$.

Moreover, for the recession cone 
\[ \lim(C) 
:= \{ x \in E \: x + C \subeq C \}  
=  \{ x \in E \: (\exists c \in C)\, c + \R_+ x \subeq C \} \] 
(\cite[Prop.~V.1.6]{Ne99}), the relation $g.C \subeq C$ implies
\[ aC = \lim(b + aC) = \lim(g.C) \subeq \lim(C) = C,\]
and this implies \eqref{eq:sc=}.

If $C$ has interior points, then
$g.C^\circ \subeq C^\circ$ and $C = \oline{C^\circ}$ imply
$g.C \subeq C$, so that $S_{C^\circ} \subeq S_C$. Conversely,
$C + C^\circ \subeq C^\circ$ implies that
$S_C \subeq S_{C^\circ}$. 
\end{proof}

\printindex 

\end{document}